\documentclass{amsart}

\usepackage[normalem]{ulem}
\usepackage{mathrsfs,amsmath,amssymb,amsthm,bbm}
\usepackage{hyperref}

\newtheorem{lem}{Lemma}[section]
\newtheorem{thrm}[lem]{Theorem}
\newtheorem{prop}[lem]{Proposition}
\newtheorem{cor}[lem]{Corollary}
\theoremstyle{definition}
\newtheorem{defn}[lem]{Definition}
\theoremstyle{remark}
\newtheorem*{rem}{Remark}

\newcommand{\eq}[2]{\begin{equation}\label{#1}#2\end{equation}}

\renewcommand{\Re}{\ensuremath{\operatorname{Re}}}
\renewcommand{\Im}{\ensuremath{\operatorname{Im}}}
\renewcommand{\epsilon}{\varepsilon}
\newcommand{\mc}{\mathcal}
\newcommand{\mb}{\mathbf}
\newcommand{\mr}{\mathrm}
\newcommand{\mf}{\mathfrak}

\newcommand{\R}{\ensuremath{\mathbb{R}}}
\newcommand{\C}{\ensuremath{\mathbb{C}}}

\newcommand{\N}{\ensuremath{\mathbb{N}}}

\newcommand{\Schwartz}{\ensuremath{\mathscr{S}}}
\newcommand{\Test}{\ensuremath{\mathcal C^\infty_c}}
\newcommand{\Cont}{\ensuremath{\mathcal C}}

\DeclareMathOperator{\sgn}{sgn}
\newcommand{\<}{\ensuremath{\langle}}
\renewcommand{\>}{\ensuremath{\rangle}}
\newcommand{\err}{\ensuremath{\mb{err}}}
\newcommand{\p}{\ensuremath{\partial}}
\newcommand{\step}[1]{\smallskip\noindent\uline{#1}}
\newcommand{\bigO}{\mc O}
\newcommand{\LHS}[1]{\mr{LHS}\eqref{#1}}
\newcommand{\RHS}[1]{\mr{RHS}\eqref{#1}}

\newcommand{\bbo}{\ensuremath{\mathbbm 1}}

\newcommand{\bpf}{\begin{proof}}
\newcommand{\epf}{\end{proof}}

\DeclareMathOperator{\tr}{tr}
\let\det=\undefined
\DeclareMathOperator{\det}{det}  

\newcommand{\eps}{\varepsilon}
\newcommand{\I}{\mf I}
\newcommand{\op}{\mr{op}}
\newcommand{\NLS}{\mr{NLS}}
\newcommand{\mKdV}{\mr{mKdV}}
\newcommand{\diff}{\mr{diff}}
\newcommand{\Bd}{B_\delta}
\newcommand{\BdS}{\Bd\cap \Schwartz}

\newcommand{\bR}{\mb R}

\newcommand{\vk}{\varkappa}
\newcommand{\norm}[1]{{\left\vert\kern-0.25ex\left\vert\kern-0.25ex\left\vert #1 \right\vert\kern-0.25ex\right\vert\kern-0.25ex\right\vert}}
\newcommand{\normN}[1]{\norm{#1}_{\NLS}}
\newcommand{\normM}[1]{\norm{#1}_{\mKdV}}
\newcommand{\normNK}[1]{\norm{#1}_{\NLS_\kappa}}
\newcommand{\normMK}[1]{\norm{#1}_{\mKdV_\kappa}}
\newcommand{\qtq}[1]{\quad\text{#1}\quad}
\newcommand{\vr}{\gamma}
\newcommand{\vj}{j_\gamma}
\newcommand{\sbrack}[1]{^{[#1]}}
\newcommand{\wrho}{\widetilde\rho}
\newcommand{\wj}{\widetilde j_\mKdV}
\newcommand{\wR}{\widetilde R}
\newcommand{\wbR}{\widetilde\bR}

\DeclareMathOperator{\sech}{sech}

\title{Sharp well-posedness for the cubic NLS and mKdV in $H^s(\R)$}

\author[B.~Harrop-Griffiths]{Benjamin Harrop-Griffiths}
\address{Benjamin Harrop-Griffiths\\
Department of Mathematics and Statistics\\
Georgetown University, Washington, DC 20057, USA}
\email{benjamin.harropgriffiths@georgetown.edu}

\author[R.~Killip]{Rowan Killip}
\address{Rowan Killip\\
Department of Mathematics\\
University of California, Los Angeles, CA 90095, USA}
\email{killip@math.ucla.edu}

\author[M.~Vi\c san]{Monica Vi\c san}
\address{Monica Visan\\
Department of Mathematics\\
University of California, Los Angeles, CA 90095, USA}
\email{visan@math.ucla.edu}

\numberwithin{equation}{section}
\allowdisplaybreaks

\begin{document}

\begin{abstract}
We prove that the cubic nonlinear Schr\"odinger equation (both focusing and defocusing) is globally well-posed in $H^s(\R)$ for any regularity $s>-\frac12$.  Well-posedness has long been known for $s\geq 0$, see \cite{MR915266}, but not previously for any $s<0$.  The scaling-critical value $s=-\frac12$ is necessarily excluded here, since instantaneous norm inflation is known to occur \cite{christ2003illposedness,MR3917712,MR3702002}.

We also prove (in a parallel fashion) well-posedness of the real- and complex-valued modified Korteweg--de Vries equations in $H^s(\R)$ for any $s>-\frac12$.  The best regularity achieved previously was $s\geq \tfrac14$; see \cite{MR1969209,MR2531556,MR1211741,MR2501679}.

To overcome the failure of uniform continuity of the data-to-solution map, we employ the method of commuting flows introduced in \cite{MR3990604}. In stark contrast with our arguments in \cite{MR3990604}, an essential ingredient in this paper is the demonstration of a local smoothing effect for both equations. Despite the non-perturbative nature of the well-posedness, the gain of derivatives matches that of the underlying linear equation.  To compensate for the local nature of the smoothing estimates, we also demonstrate tightness of orbits. The proofs of both local smoothing and tightness rely on our discovery of a new one-parameter family of coercive microscopic conservation laws that remain meaningful at this low regularity. 
\end{abstract}

\maketitle

\setcounter{tocdepth}{1}
\tableofcontents

\section{Introduction}

We consider solutions \(q\colon \R\times \R\rightarrow \C\) of the nonlinear Schr\"odinger equation
\eq{NLS}{\tag{NLS}
i\frac d{dt}q = - q'' \pm  2|q|^2q,
}
and the (complex Hirota) modified Korteweg--de Vries equation
\eq{mKdV}{\tag{mKdV}
\frac d{dt}q = -  q''' \pm 6|q|^2q',
}
with initial data \(q(0)\in H^s(\R)\).  The upper choice of signs yields the defocusing cases of these equations, while the lower signs correspond to the focusing cases.  In this paper, the symbols $\pm$ and $\mp$ will only be used in the context of this dichotomy.

By restricting \eqref{mKdV} to the case of real initial data, we recover the classical mKdV equation of Miura \cite{MR252825}:
\begin{equation}\tag{mKdV$_\R$}\label{Miura mKdV}
\frac d{dt}q = -  q''' \pm 2 (q^3)'.
\end{equation}

To treat both the defocusing and focusing versions of \eqref{NLS} and \eqref{mKdV} within the same framework, throughout this paper we adopt the notation
$$
r:=\pm \bar q.
$$
With this convention, both \eqref{NLS} and \eqref{mKdV} are Hamiltonian equations with respect to the following Poisson structure on Schwartz space: Given \(F,G\colon \Schwartz\rightarrow \C\),
\eq{PoissonBracket}{
\{F,G\} := \tfrac1i \int  \tfrac{\delta F}{\delta q}\tfrac{\delta G}{\delta r} - \tfrac{\delta F}{\delta r}\tfrac{\delta G}{\delta q} \,dx,
}
where our notation for functional derivatives is the classical one; see \eqref{FunctDeriv}.  Correspondingly, any Hamiltonian \(H\colon \Schwartz\rightarrow \R\) generates a flow, which we denote by \(e^{tJ\nabla H}\), via the equation
\eq{HFlow}{
i\dfrac d{dt} q = \frac{\delta H}{\delta r}, \qtq{or equivalently,}  i\dfrac d{dt} r = - \frac{\delta H}{\delta q} .
}
In particular, since Hamiltonians are real-valued, the relations $q=\pm\bar r$ are preserved by any such flow.

With these conventions, the equations \eqref{NLS} and \eqref{mKdV}  are the Hamiltonian flows associated to
\[
H_{\NLS} := \int q' r' + q^2r^2\,dx \qtq{and} H_{\mKdV} := \tfrac1i \int q' r'' + 3q^2rr'\,dx,
\]
respectively.  Two other important Hamiltonians are the mass and momentum,
\[
M:=\int qr\,dx \qtq{and} P=\tfrac1{i}\int qr'\,dx ,
\]
which generate phase rotations and spatial translations, respectively.  While our names for the basic conserved quantities agree with the usual parlance in the defocusing case, their signs are reversed in the focusing case; in particular, the mass becomes negative definite.  However, this sign change is offset by a corresponding sign change in the Poisson structure, so the dynamics remains those given in \eqref{NLS} and \eqref{mKdV}.

All four functions $M$, $P$, $H_{\NLS}$, and $H_{\mKdV}$ Poisson commute.   While commutation with $M$ and $P$ merely represent gauge and translation invariance, the commutativity of $H_{\NLS}$ and $H_{\mKdV}$ is surprising and a first sign of a very profound property of these equations: they are completely integrable.

One expression of this complete integrability is the existence of an infinite family of commuting flows.  Taken together, these form the AKNS--ZS hierarchy.  This name honors the authors of the seminal papers \cite{MR450815,MR0406174}.  For an authoritative introduction to this hierarchy, with particular attention to the Hamiltonian structure, we recommend \cite{MR2348643}.

The odd and even numbered Hamiltonian flows in the AKNS--ZS hierarchy behave differently under $(q,r)\mapsto(\bar q,\bar r)$.  In particular, conjugation acts as a time-reversal operator for $M$ and $H_\NLS$, but leaves the $P$ and $H_\mKdV$ flows unchanged.  This leads to a number of significant differences in our treatment of \eqref{NLS} and \eqref{mKdV}.

As we will discuss more fully below, it has been known for a long time that both \eqref{NLS} and \eqref{mKdV} are globally well-posed for sufficiently regular initial data.  In fact, the question of what constitutes sufficiently regular initial data has occupied several generations of researchers.  We are now able to give a definitive answer:

\begin{thrm}[Global well-posedness of the NLS and mKdV]\label{thrm:main}
Let \(s>-\frac12\). Then the equations \eqref{NLS} and \eqref{mKdV} are globally well-posed for all initial data in \(H^s(\R)\) in the sense that the solution map \(\Phi\) extends uniquely from Schwartz space to a jointly continuous map \(\Phi\colon \R\times H^s(\R)\rightarrow H^s(\R)\).
\end{thrm}

Here we are evidently taking the well-posedness of \eqref{NLS} and \eqref{mKdV} on Schwartz space for granted.   This has been known for a long time; see \cite{MR759907,MR641651}.

The threshold $s=-\tfrac12$ appearing in Theorem~\ref{thrm:main} is both sharp and necessarily excluded.  It is also the scaling-critical regularity. Indeed, each evolution in the AKNS-ZS hierarchy admits a scaling symmetry of the form
\begin{equation}\label{scaling}
q_\lambda(t,x) = \lambda q(\lambda^m t, \lambda x), \qtq{or equivalently,} \widehat{q_\lambda}(t,\xi) = \hat q(\lambda^m t, \xi/\lambda),
\end{equation}
where $m$ denotes the ordinal position of the Hamiltonian.  For example, $m=0$ for $M$, while \eqref{NLS} corresponds to $m=2$ and \eqref{mKdV} to $m=3$.

While a great many dispersive equations have recently been shown to be well-posed at the scaling-critical regularity, this fails for \eqref{NLS} and \eqref{mKdV}.  In fact, one has instantaneous norm inflation:  For every $s\leq-\frac12$ and $\eps>0$, there is a Schwartz solution $q(t)$ to \eqref{NLS} satisfying
\begin{align}\label{instantaneous}
\| q(0) \|_{H^s} < \eps \qtq{and} \sup_{|t|<\eps} \| q(t) \|_{H^s} > \eps^{-1}.
\end{align}
This was shown for \eqref{NLS} in \cite{christ2003illposedness,MR3917712,MR3702002}.  In Appendix~\ref{S:A} we revisit this work, giving a simplified presentation and showing that the same norm inflation holds also for \eqref{mKdV}, as well as other members of the hierarchy.  This ill-posedness effect does not seem to have been noticed before.

This norm inflation argument does not extend to \eqref{Miura mKdV}.  Nevertheless, in the appendix we show (seemingly for the first time) that a slightly weaker form of ill-posedness holds in the focusing case; see  Proposition~\ref{P:mKdV in -1/2}.  Previously,  \cite{MR1404320} showed that the data-to-solution map cannot be extended continuously to the delta-function initial data in the focusing case.  The analogous assertion for NLS (both focusing and defocusing) was proved in \cite{MR1813239}.

Let us turn our attention to the existing well-posedness theory.  The advent of Strichartz estimates \cite{MR512086} had a transformative effect on the study of nonlinear dispersive equations.  These estimates provide an elegant and efficient expression of the dispersive effect and allowed researchers to pass beyond the regularity required to make sense of the nonlinearity pointwise in time.  In \cite{MR915266}, Tsutsumi used this new tool to prove global well-posedness of \eqref{NLS} in $L^2(\R)$.

We know of no further progress in the scale of $H^s$ spaces since that time.  Here is one reason:  No ingenious harmonic analysis estimate, nor clever choice of metric, can reduce matters to a contraction mapping argument.  Such constructions lead to solutions that depend analytically on the initial data; however, in  \cite{MR2018661,christ2003illposedness,MR1813239} it is shown the the data-to-solution map cannot even be uniformly continuous on bounded subsets of $H^{s}(\R)$ when $s<0$.

Due to the derivative in the nonlinearity, Strichartz estimates alone do not suffice to understand the behavior of \eqref{mKdV}.  By bringing in local-smoothing and maximal-function estimates, Kenig, Ponce, and Vega, \cite{MR1211741}, were able to prove that \eqref{mKdV} is locally well-posed in $H^s(\R)$ for all $s\geq \frac14$.  The solution they construct depends analytically on the initial data.  Moreover, the threshold $s=\tfrac14$ is sharp if one seeks solutions that depend uniformly continuously on the initial data.  This was shown in \cite{MR2018661,MR1813239}.

In the case of \eqref{NLS}, the critical threshold for analytic well-posedness coincides with an exact conservation law, namely, that of $M(q)$.  Thus, Tsutsumi's result is automatically global in time.  Due to the absence of any obvious conservation law at regularity $s=\frac14$, it was unclear at that time whether the Kenig--Ponce--Vega solutions to \eqref{mKdV} are, in fact, global in time.  This was subsequently shown for \eqref{Miura mKdV} through the construction of suitable almost conserved quantities.  For $s>\frac14$, this was proved by Colliander--Keel--Staffilani--Takaoka--Tao \cite{MR1969209} with the endpoint added later by Guo and Kishimoto  \cite{MR2531556,MR2501679}.

With the exact threshold for analytic (or even uniformly continuous) dependence settled, the question immediately arises as to what happens at lower regularity:  What lies in the sizable gap remaining between these well-posedness results and the known breakdown of continuity at $s=-\frac12$?  This gap corresponds to regularities $-\frac12<s<0$ for \eqref{NLS} and $-\frac12<s<\frac14$ for \eqref{mKdV}.
 
For typical Schr\"odinger equations in $\R^d$ with polynomial nonlinearities, there is no gap between analytic local well-posedness and the onset of ill-posedness; see \cite{christ2003illposedness}.  Thus, it is all the more remarkable to discover a region of \emph{non-perturbative well-posedness} in this setting. This phenomenon appears to be a remarkable feature of completely integrable systems and investigating it necessitates methods that take advantage of this integrability.
 
A natural first step toward understanding solutions in this delicate region is to seek a priori $H^s$ bounds.  While boundedness of solutions would obviously follow from well-posedness, proving boundedness is typically a first step.  It is also the principal challenge in the construction of weak solutions.  On the other hand, showing impossibility of such bounds would give ill-posedness.

Early successes in this direction include \cite{MR2376575,MR2353092,MR2995102} for \eqref{NLS} and \cite{MR3058496} for \eqref{mKdV}.  Recently, the definitive result in this direction was obtained in \cite{MR3820439,MR3874652}, where exact conservation laws were constructed that control the $H^s$ norm of solutions all the way down to $s>-\tfrac12$.  Given the norm inflation discussed earlier, one cannot go any lower.

The macroscopic conservation laws constructed in \cite{MR3820439,MR3874652} interact with the scaling symmetry in a useful way; indeed, this was already employed in \cite{MR3820439} to connect differing regularities and to obtain bounds in Besov spaces.  Another important consequence of this interaction is that when $s<0$, it guarantees equicontinuity of orbits (cf. Definition~\ref{d:equicontinuity} and Proposition~\ref{prop:Equicontinuity} below).  This seems to have been first noted explicitly in \cite{MR3990604} and will play several important roles in what follows.

One example of the significance of equicontinuity is that it connects well-posedness at different regularities:  If $\sigma>s$ then existence and uniqueness of solutions with initial data in $H^s$ automatically guarantees the same for initial data in $H^\sigma$.  That the $H^s$-solution remains in $H^\sigma$ at later times follows from the existence of a priori bounds.  However, continuity of the data-to-solution map in $H^\sigma$ requires more; convergence at low regularity together with boundedness at higher regularity does not guarantee convergence at the higher regularity.  Equicontinuity in $H^\sigma$ is the simple necessary and sufficient condition for convergence in $H^\sigma$ under these circumstances.

There are two further aspects of the history we wish to discuss before describing the methods we employ: well-posedness results outside the scale of $H^s$ spaces and for these PDE posed on the torus.

By working in Fourier--Lebesgue and modulation spaces, several researchers succeeded in studying well-posedness questions outside the scale of $H^s$ spaces.  For \eqref{NLS}, for example, analytic local well-posedness was shown in almost-critical spaces by Gr\"unrock \cite{MR2181058} and Guo \cite{MR3602811}.  For \eqref{mKdV}, analogous almost-critical results in Fourier--Lebesgue spaces were obtained in \cite{MR2096258,MR2529909}.  The threshold for analytic well-posedness of \eqref{mKdV} in modulation spaces was determined in \cite{arXiv:1811.05182,arXiv:1811.04606}; however, this still does not coincide with scaling criticality.

Each of the three types of spaces (Fourier--Lebesgue, modulation, and Sobolev) has a very different character; nevertheless, each of the spaces just described can be enveloped by $H^s$ provided one takes $s>-\frac12$ sufficiently close to $-\frac12$. Conversely, both Fourier--Lebesgue and modulation spaces suppress high frequencies more strongly than negative regularity $H^s$ spaces; this substantially reduces the dangers of high-high-low interactions, which are the dominant source of instability in these models.

We are not aware of any global well-posedness results in Fourier--Lebesgue spaces close to criticality.  However, by ingeniously exploiting the way Galilei boosts interact with the conservation laws constructed in \cite{MR3820439}, Oh and Wang \cite{arXiv:1806.08761} obtained global bounds in modulation spaces, which then yield global well-posedness in these spaces.

In order to construct solutions via a contraction mapping argument, one must employ an array of subtle norms expressing the dispersive effect.  The question arises whether there might be other solutions that are continuous in $H^s$, but lie outside the auxiliary space.  This is the question of unconditional uniqueness, pioneered by Kato \cite{MR1383498,MR1403260}.  For the latest advances in this direction, see \cite{MR3073156,arXiv:1805.08410}.

We now give a quick review of what is known for \eqref{NLS} and \eqref{mKdV} posed on the circle (i.e.,~for periodic initial data).  In the Euclidean setting dispersion causes solutions to spread out.  This is impossible on the circle, there is nowhere to spread to.  Nevertheless, Bourgain \cite{MR1209299,MR1215780} proved that select Strichartz estimates do hold (expressing a form of decoherence).  As an application, these new estimates were used to prove global well-posedness of \eqref{NLS} in $L^2(\mathbb{T})$ and local well-posedness of \eqref{mKdV} in $H^{1/2}(\mathbb{T})$.  Global well-posedness of \eqref{Miura mKdV} in $H^{1/2}(\mathbb{T})$ was subsequently proved in \cite{MR1969209}.  Moreover, \cite{MR2018661} showed that these results match the threshold for analytic (or uniformly continuous) dependence on the initial data. 

For \eqref{NLS} on the circle, this $L^2$ threshold also marks the boundary for even continuous dependence on the initial data.  This was shown in \cite{MR1909648,christ2003instability,MR3801473} and represents a sharp distinction from the line case.  This `premature' breakdown of well-posedness is now understood as arising from an infinite phase rotation, which in turn suggests a suitable renormalization, namely, Wick ordering the nonlinearity.  This point of view has been confirmed in \cite{MR2333210,MR2390318,arXiv:1806.08761} where Wick-ordered NLS is shown to be globally well-posed in (almost-critical) Fourier--Lebesgue spaces where the traditional \eqref{NLS} is ill-posed.

For \eqref{Miura mKdV} on the circle, $H^{1/2}$ is not the threshold for continuous dependence. In \cite{MR2131061}, Kappeler and Topalov proved well-posedness in $L^2(\mathbb{T})$; this was shown to be sharp by Molinet~\cite{MR2927357}.  By renormalizing the nonlinearity (to remove an infinite transport term), well-posedness was then shown in \cite{MR3665197} for a larger Fourier--Lebesgue class of initial data; see also~\cite{MR4142237}.  The recent work \cite{arXiv:1911.00551} dramatically clarifies the situation regarding the full complex equation \eqref{mKdV}:  It is shown that $H^{1/2}$ \emph{is} the threshold for continuous dependence in this setting; moreover, it is shown that to go below this threshold (even in Fourier--Lebesgue spaces) a second renormalization is required.

Given the known thresholds for continuous dependence on the circle, the proof of Theorem~\ref{thrm:main} must employ some property of our equations that distinguishes the line and the circle cases!  This will be the local smoothing effect, that is, a gain of regularity locally in space on average in time.  This constitutes a significant point of departure from \cite{MR3990604}, where the arguments developed do not distinguish between the two geometries. 

The local smoothing estimates that are relevant to us involve fractional numbers of derivatives.  Correspondingly, some prudence is required in selecting the proper way to localize in space.  We do so by choosing a fixed family of Schwartz cutoff functions
\begin{align}\label{psi}
\psi(x) := \sech(\tfrac x{99}) \qtq{and} \psi_h(x) := \psi(x - h),
\end{align}
whose particular properties will allow it to be used throughout the analysis.  Corresponding to this cut-off, we define local smoothing norms by
\begin{equation}\label{E:D:LS}
\|q\|_{X^\sigma}^2 := \sup\limits_{h\in \R}\int_{-1}^1\|\psi_h^6 q\|_{H^\sigma}^2\,dt.
\end{equation}
In Lemma~\ref{lem:Reg-Gain}, we will see that this norm is strong enough to control any other choice of Schwartz-class cut-off function.

The restriction of time to the interval $[-1,1]$ in \eqref{E:D:LS} was a rather arbitrary choice; however, we see little advantage to introducing additional time parameters.  Results for alternate time intervals (or indeed other spatial intervals) can be achieved by a simple covering argument, using time- and space-translation invariance.

We are now ready to state the local smoothing estimates we prove for the solutions constructed in Theorem~\ref{thrm:main}.  As the gain in regularity differs between the two evolutions, it is easier to state our results separately:

\begin{thrm}[Local smoothing: NLS]\label{thrm:LS_N}
Fix \(-\frac12<s<0\). Given initial data $q_0\in H^s(\R)$, the corresponding solution $q(t)$ to NLS constructed in Theorem~\ref{thrm:main} satisfies
\begin{align}\label{E:NLS LS apb}
\| q \|_{X^{s+\frac12}} \lesssim \Bigl( 1 + \| q_0 \|_{H^s}\Bigr)^{\frac8{1+2s}} \| q_0 \|_{H^s};
\end{align}
moreover $q_0\mapsto q(t)$ is a continuous mapping from $H^s$ to $X^{s+\frac12}$.
\end{thrm}

\begin{thrm}[Local smoothing: mKdV]\label{thrm:LS_m}
Fix \(-\frac12<s<\frac12\). The solution $q(t)$ to mKdV with initial data $q_0\in H^s(\R)$ constructed in Theorem~\ref{thrm:main} satisfies
\begin{align}\label{E:mKdV LS apb}
\| q \|_{X^{s+1}} \lesssim \Bigl( 1 + \| q_0 \|_{H^s}\Bigr)^{\frac{11}{1+2s}} \| q_0 \|_{H^s};
\end{align}
moreover $q_0\mapsto q(t)$ is a continuous mapping from $H^s$ to $X^{s+1}$.
\end{thrm}

Estimates of this type are well-known for the underlying linear equations and readily proven either by Fourier-analytic techniques, or by explicit monotonicity identities.  In the special cases where one has a suitable microscopic conservation law, the latter technique can be adapted to nonlinear problems.  Indeed, the original local smoothing effect was the case $s=0$ of \eqref{E:mKdV LS apb}, which was proven in \cite{MR759907} by employing the microscopic conservation law
$$
\partial_t\bigl( |q|^2\bigr)  + \partial_x^3 \bigl(  |q|^2 \bigr) - 3 \partial_x \bigl(  |q'|^2 \pm |q|^4 \bigr)  =0 
$$
satisfied by solutions of \eqref{mKdV}.  The analogous microscopic conservation law for \eqref{NLS} is
$$
\partial_t 2\Im (\bar q q' ) - \partial_x^3 \bigl(  |q|^2 \bigr) + \partial_x \bigl( 4|q'|^2 \pm 2 |q|^4 \bigr)  =0,
$$
which yields \eqref{E:NLS LS apb} with $s=\frac12$.

When the sought-after regularity does not match a known conservation law, local smoothing results for nonlinear PDE have traditionally been proven perturbatively, building on the corresponding estimates for the underlying linear equation.  In particular, the arguments of \cite{MR915266} can be used to show that \eqref{E:NLS LS apb} continues to hold for $s\geq 0$.  That \eqref{E:mKdV LS apb} continues to hold for $s\geq \frac14$ was proved in \cite{MR1211741}; indeed, there the local smoothing effect was crucial to even constructing solutions.

Due to the breakdown in uniform continuity of the data-to-solution map at low regularity, we cannot expect the nonlinear flow to be well modeled by a linear flow and so some truly nonlinear technique is needed to prove Theorems~\ref{thrm:LS_N} and~\ref{thrm:LS_m}.  It is the discovery of a new one-parameter family of microscopic conservation laws for these equations that will allow us to achieve such low regularity.  As local smoothing is a linear effect, it is surprising that the loss of uniform continuity is not accompanied by any lessening of this effect --- the estimates we obtain exhibit the same derivative gain as seen for the linear equation.

As we shall see, the proof of Theorem~\ref{thrm:main} relies crucially on the local smoothing effect (though in a rather stronger form than presented in Theorems~\ref{thrm:LS_N} and~\ref{thrm:LS_m}).  With this in mind, it is natural to begin our discussion of the methods employed in this paper by describing how local smoothing is to be proved.

Local smoothing estimates also allow us to make better sense of the nonlinearity.  Note that Theorem~\ref{thrm:main} already allows us to make sense of the nonlinearity taken holistically: If $q_n$ are Schwartz solutions converging to $q$ in $L^\infty_t H^s$, then directly from the equation, we see that the corresponding sequence of nonlinearities converge, for example, as spacetime distributions.  By contrast, one may seek to make sense of the individual factors in the nonlinearity in a way that allows them to be multiplied; this is where local smoothing helps.

For example, our results show that for any $s>-1/2$, solutions of \eqref{Miura mKdV} with initial data in $H^s(\R)$ belong to $L^3_{t,x}$ on all compact regions of spacetime.  Analogously, we see that solutions to \eqref{NLS} are locally $L^3_{t,x}$ whenever $s\geq -1/6$.

\subsection{Outline of the proof} 

As we have mentioned earlier, \eqref{NLS} and \eqref{mKdV} belong to an infinite hierarchy of evolution equations whose Hamiltonians Poisson commute.  Among PDEs, this phenomenology was first discovered in the case of the Korteweg--de Vries equation \cite{PhysRevLett.19.1095}.  And it was these discoveries that Lax \cite{MR235310} elegantly codified by introducing the Lax pair formalism. (The monograph \cite{MR2348643} employs a parallel approach based around the zero-curvature condition.)

As noted above, Lax pairs for \eqref{NLS} and \eqref{mKdV} were introduced in \cite{MR450815,MR0406174}.  Several different (but equivalent) choices of these operators exist in the literature.  Our convention will be to use Lax operators
\begin{equation}\label{Intro AKNS L}
L(\vk) := \begin{bmatrix}\vk - \p & q\\-r&\vk + \p\end{bmatrix} \qtq{as well as} L_0(\vk) := \begin{bmatrix}\vk - \p & 0\\0&\vk + \p\end{bmatrix}.
\end{equation}
Here $\vk$ denotes the spectral parameter (which will always be real in this paper).  The second member of the Lax pair (traditionally denoted $P$) can be taken to be
\begin{align*}
i \begin{bmatrix}2\p^2-qr & -q\p-\p q\\ r\p+\p r & -2\p^2+qr\end{bmatrix}
\qtq{and}
\begin{bmatrix} - 4\p^3 + 3 qr\p + 3 \p qr  & 3q'\p +3\p q' \\ 3r'\p+3\p r' & -4\p^3 + 3 qr\p + 3 \p qr \end{bmatrix},
\end{align*}
for \eqref{NLS} and \eqref{mKdV}, respectively.

The Lax equation $\partial_t L = [P,L]$ guarantees that the Lax operators at different times are conjugate.  In the setting of finite matrices, this would guarantee that the characteristic polynomial of $L$ is independent of time.  In the case of \eqref{Intro AKNS L}, renormalization is required --- indeed, $L$ is not even bounded, let alone trace-class.  Such a renormalization was presented in \cite{MR3820439} based on the renormalized Fredholm determinant \(\det_2(1+A) = \det(1+A) e^{-\tr(A)}\).  Concretely, it was shown in \cite{MR3820439} that 
$$
\pm \log\det_2 [ L_0(\vk)^{-1} L(\vk;q)] 
$$
is well defined, conserved for Schwartz solutions, and coercive.  This was the origin of the coercive macroscopic conservation laws constructed in that paper.  The regularities of these laws were adjusted by integrating against a suitable measure in~$\vk$.

Unfortunately, such macroscopic conservation laws are of no use in proving local smoothing. We need not only \emph{microscopic} conservation laws, but \emph{coercive} microscopic conservation laws.  In Section~\ref{S:4}, we present our discovery of just such a density $\rho$ and its attendant currents $j$.  We feel that this is an important contribution to the much-studied algebraic theory of these hierarchies.  Moreover, it is the driver of all that follows.

We do not have a systematic way of finding microscopic conservation laws attendant to the conservation of the perturbation determinant.  If we compare the answer for KdV from \cite{MR3990604} with that developed in this paper, it is tempting to predict that it should always be a rational function of components of the diagonal Green's function.  However, we have also found the corresponding quantity for the Toda lattice~\cite{MR4320535} and in that case, it is a \emph{transcendental} function of entries in the Green's matrix.

On the other hand, the closely related one-parameter family of macroscopic conservation laws 
\begin{equation}\label{E:res trace}
\frac{\partial\ }{\partial \vk} \log\det_2 [ L_0(\vk)^{-1} L(\vk;q)] = \tr\bigl\{ L(\vk;q)^{-1} - L_0(\vk)^{-1} \bigr\}
\end{equation}
are easily seen to admit a microscopic representation based on the diagonal of the Green's function.  The associated density $\gamma$ turns out to be far inferior for what we need to do here.  Indeed, in Lemma~\ref{lem:Currents are coercive}, we will show that unfortunately, the current corresponding to $\gamma$ is not adequately coercive.  This undermines its utility for proving local smoothing.  In principle, one could recover a $\rho$-like object by integrating $\gamma$ in energy.  (This need only agree with $\rho$ up to a mean-zero function.)  In fact, we pursued this approach for a long time while still seeking the true form of $\rho$.  We can attest that this approach is extremely painful and dramatically increases the number of subtle cancellations that need to be exhibited later in the argument.

The proof of local smoothing is far and away the most lengthy and complicated part of the paper, comprising the entirety of Section~\ref{S:6} and employing crucially all of the preceding analysis.  One reason is that we actually need a two-parameter family of estimates that go far beyond the simple a priori bounds \eqref{E:NLS LS apb} and \eqref{E:mKdV LS apb}.  The role of the first of these two parameters is easy to explain at this time: it acts as a frequency threshold in the local smoothing norm.  This refinement will allow us to prove that the high-frequency contribution to the local-smoothing norm is controlled (in a very quantitative way) by the high-frequency portion of the initial data.  This is the essential ingredient in the continuity claims made in Theorems~\ref{thrm:LS_N} and~\ref{thrm:LS_m}.  (The basic question of whether such continuity holds for Kato's original estimate \cite{MR759907} seems to have been open up until now.) 

This extra frequency parameter also plays a major role in Section~\ref{sec:Tight} where it is used to show that an $H^s$-precompact set of Schwartz-class initial data leads to a collection of solutions that is $H^s$-precompact at later times.  In view of the equicontinuity of orbits mentioned earlier, this is a question of tightness.

As local smoothing estimates control the flow of the $H^s$ norm through compact regions of spacetime, it is natural to attempt to employ them to prove tightness in $H^s$.  However, it is precisely the fact that the transport of $H^s$ norm cannot exceed the total $H^s$ norm available that is used to prove Theorems~\ref{thrm:LS_N} and~\ref{thrm:LS_m}; thus these results do not provide sufficient control to yield tightness!  Our tightness result relies crucially on the extra frequency parameter to demonstrate that there is little local smoothing norm residing at high frequencies and consequently, little high-speed transport of $H^s$-norm.

The compactness result just enunciated guarantees the existence of weak solutions.  To obtain well-posedness, we must verify uniqueness (i.e., that different subsequences do not lead to different solutions), as well as continuous dependence on the initial data.  To achieve that, we will rely crucially on ideas introduced in \cite{MR3990604} and further developed in \cite{arXiv:1912.01536,arXiv:1904.11910}.

While these papers provide a useful precedent on overall strategy, they provide no guidance on how to implement it.  The first triumph of this paper is to construct the algebraic and analytic framework needed for this type of analysis in the AKNS-ZS hierarchy.  We will see that even though the two equations belong to the same hierarchy, the fundamental monotonicity laws for \eqref{NLS} and \eqref{mKdV} are different; moreover, neither equation provides significant guidance in finding the numerous cancellations necessary to treat the other.

The first step in this strategy is the introduction of regularized Hamiltonians indexed by a scalar parameter $\kappa$.  The flows induced by these Hamiltonians should (a) be readily seen to be well-posed, (b) commute with the full flows, and (c) converge to the full flows as $\kappa\to\infty$.  Such flows are introduced in Section~\ref{S:4} where they are easily proven to have properties (a) and (b).  That they enjoy property (c) in the desired topology, however, is highly non-trivial.  This is the subject of Section~\ref{S:convg}, which is the climax of this paper.

Due to their commutativity, the problem of controlling the difference between the full and regularized flows can be reduced to controlling the evolution under the difference Hamiltonian (that is, the difference of the full and regularized Hamiltonians).  In fact, this is the key insight of the commuting flow paradigm introduced in \cite{MR3990604}: instead of needing to estimate the distance between two solutions (which is rendered intractable by the breakdown of uniformly continuous dependence), one need only study a single evolution, albeit under a much more complicated flow.  

The difference flow retains all the bad behavior of the original PDE; indeed, the regularized flows are (by construction) relatively harmless.  All obstacles that prevented previous researchers from successfully analyzing solutions in this non-perturbative regime are retained.  To succeed, we will need to rely on a number of new insights; these include the new two-parameter local smoothing estimates, a novel change of unknown, and the demonstration of myriad cancellations between the full flow and its regularized counterpart.

The necessity of employing a (diffeomorphic) change of variables is common also to \cite{arXiv:1912.01536,MR3990604}.  In those works, the new variable is the diagonal Green's function.  The fact that this originates from a microscopic conservation law places one derivative in a favorable position.  Alas, all conservation laws for the NLS/mKdV hierarchy are quadratic in $q$ and so none can offer a diffeomorphic change of variables.

In place of the diagonal Green's function that proved so successful in the treatment of the KdV hierarchy, we adopt an off-diagonal entry $g_{12}(x)$ of the Green's function as our new variable.  Among its merits are the following: it has a relatively accessible time evolution; as an integral part of the definition of $\rho$, it is something for which we need to develop extensive estimates anyway; the mapping $q\mapsto g_{12}$ is a diffeomorphism; and, lastly, it gains one degree of regularity, which aids in estimating nonlinear terms.

Nevertheless, this change of variables comes with significant shortcomings. Foremost, it is not possible to control the evolution of $g_{12}$ without employing local smoothing (or some other manifestation of the underlying geometry).  For otherwise, one would obtain results for the circle that are known to be false!

At this moment, it is important to remember that we are discussing the difference flow and that our ambition is to prove that it converges to the identity as $\kappa\to\infty$.  Concomitant with this, the local smoothing effect deteriorates rapidly as $\kappa\to\infty$.  This inherent deterioration in the local smoothing estimates means that in order to treat all regularities $s>-\frac12$, we must discover every cancellation available between the full and regularized flows.  This in turn necessitates the carefully premeditated decomposition of error terms in Section~\ref{S:convg} and the stringent estimation of paraproducts in Section~\ref{S:6}.

Due to the need for local smoothing estimates, we will only be able to verify convergence locally in space.  The tightness results of Section~\ref{sec:Tight} are therefore essential for overcoming this deficiency.

In Section~\ref{S:8} we prove Theorem~\ref{thrm:main}. The tools we develop in the first seven sections allow us to prove Theorem~\ref{thrm:main} in the range $-\frac12<s<0$. This suffices for \eqref{NLS} but leaves the gap $[0,\frac14)$ for \eqref{mKdV}.  To close this gap, we construct suitable macroscopic conservation laws for both equations that allow us to prove the equicontinity of orbits in $H^s$ for $0\leq s<\frac12$ and so deduce well-posedness from that at lower regularity.  This is interesting even for \eqref{NLS}, where, for example, global in time equicontinuity of orbits in $L^2$ does not seem to have been shown previously (nor is it trivially derivable from the standard techniques).  

Section~\ref{S:9} is devoted to proving Theorems~\ref{thrm:LS_N} and \ref{thrm:LS_m}.  All the ingredients we need for the range $-\frac12<s<0$ are presented already in Section~\ref{S:6}. Thus, the majority of Section~\ref{S:9} is devoted to proving local smoothing for \eqref{mKdV} over the range $0\leq s<\frac12$ by using a new underlying microscopic conservation law. 

In closing, let us quickly recapitulate the structure of the paper that follows.  Section~\ref{S:2} discusses myriad preliminaries: settling notation, verifying basic properties of the local smoothing spaces, and proving a variety of commutator estimates.  In Section~\ref{S:3} we discuss the (matrix-valued) Green's function of the Lax operator, with particular emphasis at the confluence of the two spatial coordinates.  Section~\ref{S:4} introduces the conserved density $\rho$ and derives equations for the time evolution of this and other important quantities.  Section~\ref{S:6} proves local smoothing estimates, not only for \eqref{NLS} and \eqref{mKdV}, but also for the associated difference flows.  It is essential for what follows that these local smoothing estimates contain an additional frequency cut-off parameter.   The freedom to vary this parameter plays a crucial role, for example, in Section~\ref{sec:Tight} where these local smoothing estimates are used to control the transport of $H^s$-norm.   Section~\ref{S:convg} uses local smoothing to demonstrate the convergence of the regularized flows to the full PDEs by proving that the difference flow approximates the identity.  In Section~\ref{S:8}, we prove Theorem~\ref{thrm:main}.  Section~\ref{S:9} addresses Theorems~\ref{thrm:LS_N} and \ref{thrm:LS_m}.  Appendix~\ref{S:A} gives a new presentation of existing ill-posedness results for \eqref{NLS}, extending them to other members of the hierarchy, including \eqref{mKdV}.

\subsection*{Acknowledgements}
We would like to thank the referees for carefully reading this article and their many insightful suggestions. R.K. was supported by NSF grant DMS-1856755.  M.V. was supported by NSF grant DMS-1763074. This work was completed while B.H.-G. was a member of the Department of Mathematics at UCLA.

\section{Some notation and preliminary estimates}\label{S:2}
For the remainder of the paper, we constrain
$$
s\in (-\tfrac12,0)
$$
and all implicit constants are permitted to depend on $s$.  In view of the scaling \eqref{scaling}, it will suffice to prove all our theorems under a small-data hypothesis.  For this purpose, we introduce the notation
\begin{equation}\label{Bdelta}
\Bd := \left\{q\in H^s:\|q\|_{H^s}\leq \delta\right\}.
\end{equation}

We use angle brackets to represent the pairing:
$$
\langle f, g\rangle = \int \overline{f(x)} g(x)\,dx.
$$
In addition to being the natural inner product on (complex) $L^2(\R)$, this also informs our notions of dual space (the dual of $H^s(\R)$ is $H^{-s}(\R)$) and of functional derivatives: If $F:\Schwartz\to\C$ is $C^1$, then
\begin{equation}\label{FunctDeriv}
\tfrac{d\ }{d\theta}\Big|_{\theta=0} F(q+\theta f) = \bigl\langle \bar f, \tfrac{\delta F}{\delta q}\bigr\rangle \pm \bigl\langle f, \tfrac{\delta F}{\delta r}\bigr\rangle.
\end{equation}
For real-valued $F$, the functions $\tfrac{\delta F}{\delta q}$ and $\tfrac{\delta F}{\delta \bar  q}=\pm \tfrac{\delta F}{\delta r}$ are complex conjugates. These are functional analogues of the (Wirtinger) directional derivatives of complex analysis --- $q$ and $\bar q$ are not independent variables!

We write $\I_p$ for the $\ell^p$ Schatten class over the Hilbert space $L^2(\R)$.  For most of our analysis, the Hilbert--Schmidt class $\I_2$ will suffice.

Commensurate with our choice of time interval in \eqref{E:D:LS}, all spacetime norms will also be taken over this time interval (unless the contrary is indicated explicitly).  Thus, for any Banach space \(Z\) and \(1\leq p\leq \infty\), we define
\[
\|q\|_{L^p_t Z} := \bigl\| \|q(t) \|_Z \bigr\|_{L^p(dt;[-1,1])}.
\]

Our convention for the Fourier transform is
\begin{align*}
\hat f(\xi) = \tfrac{1}{\sqrt{2\pi}} \int_\R e^{-i\xi x} f(x)\,dx,  \qtq{whence}  \widehat{fg}(\xi) = \tfrac{1}{\sqrt{2\pi}} [\hat f * \hat g] (\xi).
\end{align*}

We shall repeatedly employ a `continuum partition of unity' device based on the cut-off $\psi_h^{12}$.  Specifically, as
\begin{equation}\label{CPUD}
\int_\R \psi(x-h)^{12} \,dh \equiv \tfrac{512}{7}, \qtq{so} f(x) = \tfrac{7}{512} \int_\R f(x) \psi_h^{12}(x)\,dh
\end{equation}
in $H^\sigma(\R)$ sense, for any $f\in H^\sigma(\R)$ and any $\sigma\in\R$.

\subsection{Sobolev spaces}
For real \(|\kappa|\geq 1\) and \(\sigma\in \R\) we define the norm
\[
\|q\|_{H^\sigma_\kappa}^2 := \int \left(4\kappa^2 + \xi^2\right)^\sigma|\hat q(\xi)|^2\,d\xi
\]
and write \(H^\sigma := H^\sigma_1\).

For $-\frac12<s<0$, elementary considerations yield
\begin{align}\label{Linfty bdd}
\|f\|_{L^\infty}\leq \|\hat{f}\|_{L^1}\leq \|f\|_{H^{s+1}_\kappa} \bigl\|(|\xi|^2+4\kappa^2)^{-\frac{s+1}2}\bigr\|_{L^2}\lesssim |\kappa|^{-(s +\frac12)}\|f\|_{H^{s+1} _\kappa}.
\end{align}
Consequently, we have the following algebra property:
\begin{equation}\label{E:algebra}
\| f g \|_{H^{s+1}_\kappa} \lesssim |\kappa|^{-(\frac12+s)} \| f \|_{H^{s+1}_\kappa} \| g \|_{H^{s+1}_\kappa}.
\end{equation}

Arguing by duality and using the fractional product rule, Sobolev embedding, and \eqref{Linfty bdd}, we may bound
\begin{align}\label{multiplier bdd on Hs}
\|qf\|_{H^s}&\lesssim \|q\|_{H^s} \|f\|_{L^\infty} + \|q\|_{L^{\frac2{1-2|s|}}} \||\nabla|^{|s|}f\|_{L^{\frac1{|s|}}}\notag\\
&\lesssim \|q\|_{H^s} \bigl[ |\kappa|^{-(s+\frac12)}\|f\|_{H^{s+1}_\kappa} + \||\nabla|^{\frac12}f\|_{ L^2}\bigr]\\
&\lesssim  |\kappa|^{-(s+\frac12)} \|q\|_{H^s}\|f\|_{H^{s+1}_\kappa}.\notag
\end{align}

\begin{lem}\label{lem:EquivNorm}
If \(s'<s\), \(|\kappa|\geq 1\), and \(q\in H^s\), then
\eq{EquivNorm}{
\|q\|_{H^s_\kappa}^2\approx_{s,s'} \int_{|\kappa|}^\infty \vk^{2(s - s')}\|q\|_{H^{s'}_\vk}^2\,\tfrac{d\vk}\vk.
}
\end{lem}
\begin{proof}
By scaling it suffices to consider the case \(\kappa = 1\). We may then write
\[
\int_1^\infty \vk^{2(s - s')}(4\vk^2 + \xi^2)^{s'}\,\tfrac{d\vk}\vk = |\xi|^{2s} \int_{\frac1{|\xi|}}^\infty \vk^{2(s - s')}(4\vk^2 + 1)^{s'}\,\tfrac{d\vk}\vk.
\]
By considering the cases \(|\xi|\leq 2\) and \(|\xi|>2\) separately, we may bound
\[
|\xi|^{2s} \int_{\frac1{|\xi|}}^\infty \vk^{2(s - s')}(4\vk^2 + 1)^{s'}\,\tfrac{d\vk}\vk\approx_{s,s'}(4 + \xi^2)^s,
\]
and the estimate \eqref{EquivNorm} then follows from the Fubini-Tonelli Theorem.
\end{proof}

\subsection{Local smoothing spaces}
It will be important to consider a one-parameter family of local smoothing norms generalizing that presented in the introduction.  To this end, given \(\kappa\geq 1\) and $\sigma\in \R$, we define the local smoothing space
\[
\|q\|_{X^\sigma_\kappa}^2 := \sup_{h\in \R} \bigl\|\tfrac{\psi_h^6q}{\sqrt{4\kappa^2 - \p^2}} \bigr\|_{L^2_t H^{\sigma+1}}^2,
\]
so that \(X^\sigma_1 = X^\sigma\), where we write \(\tfrac{\psi_h^6q}{\sqrt{4\kappa^2 - \p^2}} = (4\kappa^2 - \p^2)^{-\frac12}(\psi_h^6q)\).    At this moment, placing the inverse differential operators under their arguments (rather than in front of them) may seem clumsy; however, the mere act of writing out \eqref{gamma 4} in traditional form will quickly convince the reader of the virtue of this approach.

To ease dimensional analysis, the $X^\sigma_\kappa$ spaces have been defined to scale the same as $H^\sigma$ spaces.

Our next lemma allows us to understand the effect of changing the localizing function $\psi^6$ or the regularity $\sigma$ in the definition of the local smoothing norm:

\begin{lem}\label{lem:Reg-Gain}
Given \(\kappa\geq 1\), $\sigma\in\R$, and $\phi\in\Schwartz$,
\begin{equation}\label{E:change}
\bigl\|\tfrac{\phi q}{\sqrt{4\kappa^2 - \p^2}}\bigr\|_{L^2_t H^{\sigma+1}}\lesssim_{\phi,\sigma} \|q\|_{X^{\sigma}_\kappa}.
\end{equation}
Moreover, if \(s-1\leq \sigma'\leq\sigma \), then
\eq{Reg-Gain}{
\|q\|_{X^{\sigma'}_\kappa}\lesssim \kappa^{\frac{\sigma'-\sigma}{1 + \sigma-s}}\left(\|q\|_{X^\sigma_\kappa} + \|q\|_{L^\infty_t H^s}\right).
}
\end{lem}

\begin{proof}
We begin by discussing \eqref{E:change}.  Let $T_h:L^2\to L^2$ denote the operator with integral kernel
$$
\frac{(4+\xi^2)^{\frac{\sigma+1}2}}{(4\kappa^2 +\xi^2)^{\frac{1}2}} \widehat{\phi \psi_h^6}(\xi-\eta) \frac{(4\kappa^2+\eta^2)^{\frac{1}2}}{(4+\eta^2)^{\frac{\sigma+1}2}}.
$$
By applying Schur's test, we find that
$$
\| T_h \|_\op \lesssim \| \phi \psi_h^6\|_{H^{|\sigma|+2}} \qtq{and so}  \int_\R \| T_h \|_\op \,dh \lesssim_\phi 1.
$$
Moreover, this bound holds uniformly in $\kappa$.  Thus by employing \eqref{CPUD}, we find
$$
\bigl\|\tfrac{\phi q}{\sqrt{4\kappa^2 - \p^2}}\bigr\|_{L^2_t H^{\sigma+1}} \lesssim \int \bigl\|\tfrac{\phi\psi_h^{12} q}{\sqrt{4\kappa^2 - \p^2}}\bigr\|_{L^2_t H^{\sigma+1}}\,dh
    \lesssim  \int_\R \| T_h \|_\op \| q \|_{X^\sigma_\kappa} \,dh \lesssim_\phi \| q \|_{X^\sigma_\kappa},
$$
which settles \eqref{E:change}.

Turning to \eqref{Reg-Gain}, and setting $N=\kappa^{\frac1{1 + \sigma - s}}$, we have
\begin{align*}
\|\tfrac{\psi_h^6 q}{\sqrt{4\kappa^2 - \p^2}}\|_{L^2_t H^{\sigma'+1}}&\lesssim \kappa^{-1 + \frac{\sigma' + 1 -s}{1 + \sigma-s}} \|(\psi_h^6q)_{\leq N}\|_{L^\infty_t H^s}
    + \kappa^{\frac{\sigma'-\sigma}{1 + \sigma-s}}\bigl\|P_{>N}\tfrac{\psi_h^6q}{\sqrt{4\kappa^2 - \p^2}}\bigr\|_{L^2_t H^{\sigma+1}}.
\end{align*}
Taking the supremum over \(h\) we obtain the estimate \eqref{Reg-Gain}.
\end{proof}

Next we record several commutator-type estimates that we will use in the later sections.

\begin{lem}\label{L:bounds}
Fix \(\kappa\geq 1\). Then
\begin{align}
\bigl\|\bigl[\psi_h, \tfrac{1}{4\kappa^2-\p^2}\bigr]q\bigr\|_{H^\sigma}&\lesssim \kappa^{-3+\sigma-s} \|q\|_{H^s_\kappa} \!\qtq{for}\!\! -1\leq \sigma\leq 3+s, \label{1} \\
\bigl\|\bigl[\psi_h, \tfrac{\p^\ell}{4\kappa^2-\p^2}\bigr]q\bigr\|_{H^\sigma}&\lesssim \kappa^{-3+\ell +\sigma-s} \|q\|_{H^s_\kappa} \!\qtq{for}\!\! 1\leq \sigma+\ell\leq 3+s,\, \ell=1,2, 3,\label{1'} 
\end{align}
uniformly for $h\in \R$. Moreover, for $\ell=2,3,4$ and $2+s\leq \sigma+\ell\leq 4+s$,
\begin{align}
\bigl\|\bigl[\psi_h,\tfrac{\p^\ell}{4\kappa^2 - \p^2}\bigr]q\bigr\|_{L_t^2H^\sigma}&\lesssim \kappa^{-2+\frac12(\ell + \sigma-s)}\bigl[\|q\|_{X_\kappa^{s+1}} + \|q\|_{L_t^\infty H^s}\bigr], \label{2}
\end{align}
uniformly for $h\in \R$.
\end{lem}

\begin{proof}
The estimate \eqref{1} follows from the observation that
\begin{align}\label{del2comm}
\bigl[\psi_h, \tfrac1{4\kappa^2-\p^2}\bigr]&= \tfrac1{4\kappa^2-\p^2}\bigl(\psi_h''-2\p \psi_h'\bigr)\tfrac1{4\kappa^2-\p^2}.
\end{align}
The lower bound on $\sigma$ expresses that the maximum possible decay in $\kappa$ is $\kappa^{-4-s}$.

To handle $\ell=1,2,3$, we also use the fact that
$$
\bigl[\psi_h, \tfrac{\p^\ell}{4\kappa^2-\p^2}\bigr] = \bigl[\psi_h,  \tfrac1{4\kappa^2-\p^2}\bigr]\p^\ell+  \tfrac1{4\kappa^2-\p^2}[\psi_h, \p^\ell],
$$
from which we see that the maximum possible decay in $\kappa$ is $\kappa^{-2-s}$.

We now turn to \eqref{2} and write
\begin{align*}
\bigl[\psi_h, \tfrac{\p^\ell }{4\kappa^2-\p^2}\bigr]
&=\tfrac{\p^\ell}{(4\kappa^2-\p^2)^2}\bigl(\psi_h''-2\p \psi_h'\bigr) + \tfrac1{4\kappa^2-\p^2} \bigl[\psi_h''-2\p \psi_h', \tfrac{\p^\ell }{4\kappa^2-\p^2}\bigr]\\
&\quad + \tfrac1{4\kappa^2-\p^2}\sum_{m=0}^{\ell-1} c_m \p^{m} \psi_h^{(\ell-m)}q.
\end{align*}
Using \eqref{E:change}, this readily yields
\begin{align*}
\bigl\|\bigl[\psi_h,\tfrac{\p^\ell}{4\kappa^2 - \p^2}\bigr]q\bigr\|_{L_t^2H^\sigma}&\lesssim \|q\|_{X_\kappa^{\sigma + \ell-3}} 
\end{align*}
and \eqref{2} follows from an application of \eqref{Reg-Gain}.
\end{proof}

We also have the following estimates:
\begin{lem}\label{L:X}
Let \(\sigma>0\), \(\kappa\geq 1\), and \(f,g\in \Cont([-1,1];\Schwartz)\). If \(|\vk|\geq1\), then
\eq{X-Smoothing}{
\|(2\vk - \p)f\|_{X^\sigma_\kappa}  \lesssim |\vk|\|f\|_{X^\sigma_\kappa} + \|f\|_{X^{\sigma+1}_\kappa} \lesssim \|(2\vk - \p)f\|_{X^\sigma_\kappa} + \|f\|_{L^\infty_t H^\sigma} .
}
Further, we have the product estimates
\begin{gather}
\|fg\|_{X^\sigma_\kappa} \lesssim
    |\vk|^{-(s+\frac12)}\left(\|f\|_{X^\sigma_\kappa}\|g\|_{L^\infty_t H^{s+1}_\vk} + \|f\|_{L^\infty_t H^{s+1}_\vk}\|g\|_{X^{\sigma}_\kappa} \right),\label{X-Product-1} \\
\|fg\|_{X^\sigma_\kappa} \lesssim
    |\vk|^{-(s+\frac12)}\Bigl[\|f\|_{X^\sigma_\kappa}\|g\|_{L^\infty_t H^{s+1}_\vk}\label{X-Product-2}+\|f\|_{L^\infty_t H^s_\vk}\bigl(|\vk|\|g\|_{X^\sigma_\kappa} + \|g\|_{X^{\sigma+1}_\kappa}\bigr)\Bigr],\!\!\\
\label{E:XH algebra}
\| f g  \|_{X^\sigma_\kappa \cap L^\infty_t H^{s+1}_\vk} \lesssim |\vk|^{-(\frac12+s)} \| f  \|_{X^\sigma_\kappa \cap L^\infty_t H^{s+1}_\vk}  \| g  \|_{X^\sigma_\kappa \cap L^\infty_t H^{s+1}_\vk}.
\end{gather}
All estimates are uniform in $\kappa$ and $\vk$.
\end{lem}

\begin{proof}
By translation invariance it suffices to prove the estimates for a fixed choice of \(\psi_h\) on the left-hand side. For simplicity we take \(h=0\).

We start with \eqref{X-Smoothing}. By Plancherel, we have
\begin{align*}
&4\vk^2\bigl\|\tfrac{\psi^6f}{\sqrt{4\kappa^2 - \p^2}}\bigr\|_{L^2_t H^{\sigma+1}}^2 + \bigl\|\tfrac{\psi^6f}{\sqrt{4\kappa^2 - \p^2}}\bigr\|_{L^2_t H^{\sigma+2}}^2
    \approx  \bigl\|\tfrac{(2\vk - \p)(\psi^6f)}{\sqrt{4\kappa^2 - \p^2}}\bigr\|_{L^2_t H^{\sigma+1}}^2 .
\end{align*}
On the other hand, $(2\vk - \p)(\psi^6f) = \psi^6 (2\vk - \p)f - (\psi^6)' f$.  Thus the first inequality in \eqref{X-Smoothing} follows from \eqref{E:change};  the second is elementary.

For the product estimates \eqref{X-Product-2} and \eqref{X-Product-1}, we first decompose dyadically to obtain
\begin{equation}\label{Decomposed Product}
\|\tfrac{\psi^6fg}{\sqrt{4\kappa^2 - \p^2}}\|_{L^2_t H^{\sigma+1}}^2\approx \sum_{N} \tfrac{N^{2\sigma+2}}{\kappa^2 + N^2} \Bigl| \sum_{N_1,N_2} \bigl\|P_N\bigl[(\psi^3f)_{N_1}(\psi^3g)_{N_2}\bigr]\bigr\|_{L^2_{t,x}} \Bigr|^2.
\end{equation}
For the high-low interactions where \(N_2\ll N_1\approx N\), we use Bernstein's inequality at low frequency to bound
\begin{align*}
\|P_N\bigl[(\psi^3f)_{N_1}(\psi^3g)_{N_2}\bigr]\|_{L^2_{t,x}} &\lesssim \|(\psi^3 f)_{N_1}\|_{L^2_{t,x}}\|(\psi^3g)_{N_2}\|_{L^\infty_{t,x}}\\
&\lesssim N_2^{\frac12}(|\vk| + N_2)^{-(s+1)}\|(\psi^3 f)_{N_1}\|_{L^2_{t,x}}\|g\|_{L^\infty_t H^{s+1}_\vk}.
\end{align*}
After summing in \(N,N_1,N_2\) we obtain a contribution to \(\RHS{Decomposed Product}\) that is
\[
\lesssim |\vk|^{-2(s+\frac12)}\|f\|_{X^\sigma_\kappa}^2\|g\|_{L^\infty_t H^{s+1}_\vk}^2.
\]

For the high-high interactions where \(N\lesssim N_1\approx N_2\), we use Bernstein's inequality at the output frequency to bound
\begin{align*}
\|P_N\bigl[(\psi^3f)_{N_1}(\psi^3g)_{N_2}\bigr]\|_{L^2_{t,x}} &\lesssim N^{\frac12}\|(\psi^3f)_{N_1}(\psi^3g)_{N_2}\|_{L^2_tL^1} \\
&\lesssim N^{\frac12}(|\vk| + N_1)^{-(s+1)}\|(\psi^3 f)_{N_1}\|_{L^2_{t,x}}\|g\|_{L^\infty_t H^{s+1}_\vk}.
\end{align*}
After summation, we again obtain a contribution to \(\RHS{Decomposed Product}\) that is
\[
\lesssim |\vk|^{-2(s+\frac12)}\|f\|_{X^\sigma_\kappa}^2\|g\|_{L^\infty_t H^{s+1}_\vk}^2.
\]

For the low-high interactions where \(N_1\ll N_2\approx N\), we proceed similarly to the case of the high-low interactions, using Bernstein's inequality at low frequency to bound
\begin{align*}
\|P_N\bigl[(\psi^3f)_{N_1}(\psi^3g)_{N_2}\bigr]\|_{L^2_{t,x}} &\lesssim \|(\psi^3f)_{N_1}\|_{L^\infty_{t,x}}\|(\psi^3g)_{N_2}\|_{L^2_{t,x}}\\
&\lesssim N_1^{\frac12}(|\vk| + N_1)^{-(s+1)}\|f\|_{L^\infty_t H^{s+1}_\vk}\|(\psi^3g)_{N_2}\|_{L^2_{t,x}}.
\end{align*}
In this case, we obtain a contribution to \(\RHS{Decomposed Product}\) that is
\[
\lesssim |\vk|^{-2(s+\frac12)}\|f\|_{L^\infty_t H^{s+1}_\vk}^2\|g\|_{X^\sigma_\kappa}^2.
\]
This completes the proof of \eqref{X-Product-1}.  Alternatively, we may bound
\begin{align*}
\|P_N\bigl[(\psi^3f)_{N_1}(\psi^3g)_{N_2}\bigr]\|_{L^2_{t,x}}\lesssim N_1^{\frac12}(|\vk| + N_1)^{-s}\|f\|_{L^\infty_t H^s_\vk}\|(\psi^3g)_{N_2}\|_{L^2_{t,x}},
\end{align*}
to obtain a contribution to \(\RHS{Decomposed Product}\) of
\[
\lesssim |\vk|^{-2(s+\frac12)}\|f\|_{L^\infty_t H^s_\vk}^2\Big(|\vk|\|g\|_{X^\sigma_\kappa} + \|g\|_{X^{\sigma+1}_\kappa}\Big)^2,
\]
which completes the proof of  \eqref{X-Product-2}.

The bound \eqref{E:XH algebra} follows from \eqref{E:algebra} and \eqref{X-Product-1}.
\end{proof}

\subsection{Operator estimates}
For \(0<\sigma<1\) and \(|\kappa|\geq 1\) we define the operator \((\kappa \mp \p)^{-\sigma}\) using the Fourier multiplier \((\kappa \mp i\xi)^{-\sigma}\) where, for \(\arg z\in (-\pi,\pi]\), we define
\begin{equation}\label{z to the sigma}
z^{-\sigma} = |z|^{-\sigma}e^{-i\sigma \arg z}.
\end{equation}
We observe that with this convention, for all \(|\kappa|\geq 1\) we have
\[
\left((\kappa \mp \p)^{-\sigma}\right)^* = (\kappa \pm \p)^{-\sigma},
\]
and readily obtain the estimate
\[
\left\|(\kappa \mp \p)^{-\sigma}\right\|_{\op}\leq |\kappa|^{-\sigma}. 
\]

We will make frequent use of the following Hilbert--Schmidt estimates:

\begin{lem}\label{L:BasicBounds}
For all \(q\in H^s_\kappa(\R)\),
\begin{align}
\|(\kappa - \p)^{\frac s2-\frac14}q(\kappa + \p)^{\frac s2 - \frac14}\|_{\I_2}^2 &\lesssim \|q\|_{H^s_\kappa}^2\label{BasicBound},\\
\|(\kappa - \p)^s q(\kappa + \p)^{- \frac34 -\frac s2}\|_{\I_2}^2 &\lesssim \kappa^{-\frac12(1+2s)}\|q\|_{H^s_\kappa}^2\label{BasicBound'},\\
\|(\kappa - \p)^{-\frac12}q(\kappa + \p)^{-\frac12}\|_{\I_2}^2 &\approx \int \log\bigl(4 + \tfrac{\xi^2}{\kappa^2}\bigr)\frac{|\hat q(\xi)|^2}{\sqrt{4\kappa^2 + \xi^2}}\,d\xi. \label{LogarithmicBound}
\end{align}
Moreover, for any real $|\kappa|\geq 1$,
\begin{align}
\|(\kappa - \p)^{-(1+s)} f (\kappa + \p)^{-(1+s)}\|_{\op} &\lesssim \kappa^{-\frac12(1+2s)} \|f\|_{H^{-(1+s)}_\kappa}. \label{BasicOpBound}
\end{align}
\end{lem}

\begin{proof}
By scaling it suffices to consider \(\kappa = 1\). By Plancherel's Theorem we have
\begin{align*}
\|(1 - \p)^{-\alpha}q(1 + \p)^{-\beta}\|_{\I_2}^2 &= \tr\left\{(1 - \p^2)^{-\alpha}q(1 - \p^2)^{-\beta}\bar q\right\}\\
&= \tfrac1{2\pi}\iint \frac{|\hat q(\xi-\eta)|^2\,d\eta\, d\xi}{\left(1 + \xi ^2\right)^{\alpha}\left(1 + \eta^2\right)^{\beta}}.
\end{align*}
For the particular choices of $\alpha$ and $\beta$ relevant to \eqref{BasicBound} and \eqref{BasicBound'}, we have
\[
\int \frac1{\left(1 + (\xi + \eta)^2\right)^{\alpha}\left(1 + \eta^2\right)^{\beta}}\,d\eta \lesssim (4 + \xi^2)^s,
\]
from which we obtain \eqref{BasicBound}.  The estimate \eqref{LogarithmicBound} can be proved in a parallel manner; see~\cite[Lemma~4.1]{MR3820439}.

Arguing by duality, the key observation to prove \eqref{BasicOpBound} is that
$$
\biggl| \int f g h \,dx \biggr|  \lesssim |\kappa|^{-\frac12(1+2s)}\| f \|_{H^{-(1+s)}_\kappa} \| g \|_{H^{1+s}_\kappa}  \| h \|_{H^{1+s}_\kappa},
$$
which combines the duality of $H^\sigma_\kappa$ and $H^{-\sigma}_\kappa$ with the algebra property \eqref{E:algebra}.
\end{proof}

Our next two lemmas are devoted to similar bounds, but employing the local smoothing norm on the right-hand side.  The former employs the local smoothing norm pertinent to \eqref{NLS}, while the latter is relevant to \eqref{mKdV}.

By introducing spatial localization we obtain the following improvements:

\begin{lem}\label{lem:LambdaLoc NLS}
We have the estimates
\begin{align}
&\|(\vk - \p)^{-\frac12}(\psi_h q)(\vk + \p)^{-\frac12}\|_{L^2_t\I_2}^2\lesssim |\vk|^{-1}\kappa^{- \frac{4s}3}\Big(\|q\|_{X^{s+\frac12}_\kappa}^2 + \|q\|_{L^\infty_t H^s}^2\Big),\label{LambdaLoc NLS}\\
&\|(\vk - \p)^{-\frac12}(\psi_hq)(\vk + \p)^{-\frac12}\|_{L^4_t\I_4}^4  \label{Sharp LambdaLoc NLS} \\
&\qquad\qquad\qquad\lesssim |\vk|^{-3} \bigl[\kappa^{\frac23 - \frac{8s}3} + |\vk|^{-4s}\bigr]\|q\|_{L^\infty_tH^s}^2\Big(\|q\|_{X^{s+\frac12}_\kappa}^2 + \|q\|_{L^\infty_t H^s}^2\Big), \notag 
\end{align}
uniformly for \(|\vk|\geq \kappa^{\frac23}\geq 1\), \(q\in \Cont([-1,1];H^s)\cap X^{s+\frac12}_\kappa\), and \(h\in \R\). 
\end{lem}

\begin{proof}
By translation invariance it suffices to consider the case \(h=0\).  Given a dyadic number \(N\geq 1\) we define
\[
\Lambda_N = (\vk - \p)^{-\frac12}(\psi q)_N(\vk + \p)^{-\frac12},
\]
presaging the notation \eqref{D:LambdaGamma}.  Employing \eqref{LogarithmicBound} we may bound 
\begin{equation}\label{LL-1}
\|\Lambda_N\|_{L^2_t\I_2}^2 \lesssim \frac{ \log(4 + \tfrac{N^2}{\vk^2})}{N^{2s} (|\vk|+N)}\min\Big\{\|q\|_{L^\infty_t H^s}^2,
	\ \frac{(\kappa + N)^2}{N^{3}}\|q\|_{X^{s+\frac12}_\kappa}^2\Big\}.
\end{equation}
The estimate \eqref{LambdaLoc NLS} now follows by taking a square-root and summing over $N\in 2^\N$.

From Bernstein's inequality we have
\[
\|\Lambda_N\|_{L^\infty_t \op} \lesssim |\vk|^{-1} \|(\psi q)_N\|_{L^\infty_{t,x}}\lesssim |\vk|^{-1}N^{\frac12 -s}\|q\|_{L^\infty_tH^s},
\]
which combined with the first part of \eqref{LL-1} yields
\begin{equation}\label{LL-2}
\|\Lambda_N\|_{L^\infty_t \op} \lesssim N^{-s}\min \Big\{ |\vk|^{-1}N^{\frac12},(|\vk| + N)^{-\frac12}\big[\log(4 + \tfrac{N^2}{\vk^2})\big]^{\frac12}\Big\}\|q\|_{L^\infty_tH^s}.
\end{equation}

Thus, we may prove \eqref{Sharp LambdaLoc NLS} via first interpolating between \eqref{LL-1} and \eqref{LL-2}, and then summing over $N\in 2^\N$.  This is most easily accomplished by breaking the sum at $\kappa^{\frac23}$ and $|\vk|$.
\end{proof}

\begin{lem}\label{L:LpIp mKdV}
Fix $2\leq p <\infty$. Then
\begin{align*}
& \|(\vk - \p)^{-\frac12}(\psi_h q)(\vk + \p)^{-\frac12}\|_{L^p_t\I_p}^p \\
&\quad \lesssim |\vk|^{1-p} \Bigl[ \kappa^{\frac p2(\frac12-s)-\frac12}
		+ \bigl(1+\tfrac{\kappa^2}{\vk^2}\bigr) |\vk|^{p(\frac12-s)-3} \log^p\bigl|\tfrac{4\vk^2}{\kappa}\bigr|  \Bigr]
		\Bigl[\|q\|_{X^{s+1}_\kappa}^2 + \|q\|_{L^\infty_t H^s}^2\Bigr], \notag
\end{align*}
uniformly for \(|\vk|\geq \kappa^{\frac12}\geq 1\), \(q\in \Cont([-1,1];\Bd)\cap X^{s+1}_\kappa\), and \(h\in \R\).  Moreover, the factor $(1+\tfrac{\kappa^2}{\vk^2})$ may be deleted if $p\leq 5$.
\end{lem}

\begin{proof}
We mimic the proof of Lemma~\ref{lem:LambdaLoc NLS}, replacing \eqref{LL-1} with
\begin{align}\label{LL-3}
\|\Lambda_N\|_{L^2_t\I_2}^2 &\lesssim \frac{\log(4 + \tfrac{N^2}{\vk^2})}{N^{2s}(|\vk|+N)}
	\min\Bigl\{\|q\|_{L^\infty_t H^s}^2,\ \frac{(\kappa + N)^2}{N^{4}}\|q\|_{X^{s+1}_\kappa}^2\Bigr\}
\end{align}
and reusing \eqref{LL-2}.  We simply interpolate and then sum.  Note that the logarithmic factor is only necessary when $p(\frac12-s)\in\{3,5\}$.  When $p\leq 5$, the extra factor can be neglected due to the other summand and the constraint $\smash{|\vk|\geq \kappa^{\frac12}}$.
\end{proof}

In order to apply Lemmas~\ref{lem:LambdaLoc NLS} and~\ref{L:LpIp mKdV}, we will need to bring some power of the localizing function $\psi$ adjacent to copies of $q$ and $r$.  This is the role of the following:

\begin{lem}[Multiplicative commutators]\label{L:localize}
For \(|\vk|,|\kappa|\geq 1\),  \(\sigma\in\R\), and any integer \(|\ell|\leq 12\),  we have the following estimate uniformly for $h\in\R$ and $u\in\Schwartz$,
\begin{equation}\label{Multiplicative Commutator}
 \|\psi_h^\ell(\vk - \p)^{-1}\psi_h^{-\ell} u \|_{H^{\sigma}_\kappa} \lesssim_{\sigma} \| (\vk - \p)^{-1} u \|_{H^{\sigma}_\kappa} .
\end{equation}

Further, if \(N\geq 1\) is a dyadic integer, \(1\leq p\leq \infty\), and \(n\geq 0\), we have 
\begin{equation}\label{Localized Multiplicative Commutator}
\|\psi_h^{\ell}\tfrac{\p^n}{4\kappa^2 - \p^2}\psi_h^{-\ell}P_N\|_{L^p\rightarrow L^p}\lesssim_{n} \tfrac{N^n}{(\kappa + N)^2}.
\end{equation}
\end{lem}

\begin{proof}
By translation invariance, it suffices to consider the case \(h=0\).

Using Schur's test and the explicit kernel \eqref{G_0}, we find
\begin{equation}\label{psi up down}
\|\psi^{-\ell} (\vk - \p)^{-1}\psi^\ell\|_{L^p\to L^p}\lesssim \vk^{-1}.
\end{equation}
We will need this shortly.  It is important here that the exponential decay of the convolution kernel is faster than that of the function $\psi^\ell$.  This is a reason both for the large constant $99$ appearing in \eqref{psi} and for requiring a bound on the size of $\ell$.

We first consider the estimate \eqref{Multiplicative Commutator}. By duality, it suffices to consider the case \(\sigma\geq 0\). For \(z\in \C\), we write
\[
B_\ell(z) := (4\kappa^2 - \p^2)^{\frac{z}2}\psi^\ell (\vk - \p)^{-1}\psi^{-\ell}(\vk - \p) (4\kappa^2 - \p^2)^{-\frac{z}2},
\]
with the intention of using complex interpolation to prove $\|B_\ell(\sigma)\|\lesssim_{\sigma,\ell} 1$, which implies \eqref{Multiplicative Commutator}.  As imaginary powers of \(\kappa^2 - \p^2\) are unitary, we find 
\[
\|B_\ell(\sigma)\|_\op \leq \|B_\ell(0)\|_\op^{(m-\sigma)/m} \|B_\ell(m)\|_\op^{\sigma/m}
\]
for any integer $m\geq \sigma$.  For concreteness, we choose the least such integer.

Combining $|\psi'|\lesssim \psi$ and \eqref{psi up down} with the rewriting
\begin{equation*}
B_\ell(0)= 1 + \psi^\ell (\vk - \p)^{-1} \psi^{-\ell} \bigl[\psi^\ell \bigl(\p \psi^{-\ell}\bigr)\bigr], \qtq{yields} \|B_\ell(0)\|_\op \lesssim 1.
\end{equation*}

Turning our attention now to $B_\ell(m)$, we notice that
$$
\tilde B_\ell(m) = (2\kappa+\p)^m B_\ell(0) (2\kappa+\p)^{-m}  \qtq{satisfies} \|\tilde B_\ell(m)\|_\op =\|B_\ell(m)\|_\op ;
$$
moreover, we may expand $\tilde B_\ell(m)$ as
$$
1 + \sum \tbinom{m}{m_1,m_2,m_3,m_4}  \bigl[ \psi^{-\ell} \bigl(\p^{m_2}  \psi^\ell\bigr)\bigr]
	\bigl[ \psi^{\ell} (\vk - \p)^{-1} \psi^{-\ell} \bigr]\bigl[\psi^{\ell} \bigl( \p^{1+m_3}\psi^{-\ell} \bigr)\bigr]
	\Bigl[\tfrac{(2\kappa)^{m_1} \p^{m_4}}{(2\kappa+\p)^{m}}\Bigr],
$$
where the sum is over all decompositions $m=m_1+m_2+m_3+m_4$ using non-negative integers.  The key observation that finishes the proof is that each operator in square brackets is bounded; indeed, for every $n\geq0$ we have
\begin{align}\label{psi is harmless}
| \p^n \psi(x) | \lesssim_n \psi(x), \qtq{whence} \| \psi^{-\ell}(\p^{n}  \psi^\ell) \|_{L^\infty} \lesssim_{n} 1
\end{align}
for any integers $\ell$ and $n\geq 0$.

The proof of \eqref{Localized Multiplicative Commutator} employs similar ideas: We first write
\[
\psi^{\ell}\tfrac{\p^n}{4\kappa^2 - \p^2}\psi^{-\ell}P_N = \sum_{m=0}^n\tbinom nm \psi^\ell \tfrac1{4\kappa^2 - \p^2} (\p^m \psi^{-\ell}) (4\kappa^2 - \p^2)\tfrac{\p^{n-m}}{4\kappa^2 - \p^2}P_N,
\]
which shows that we need only prove
\eq{LpVersion}{
\|(4\kappa^2 - \p^2)  (\p^m \psi^{-\ell})  \tfrac1{2\kappa - \p}\psi^{\ell} \psi^{-\ell}\tfrac1{2\kappa + \p}\psi^{\ell}\|_{L^{p'}\rightarrow L^{p'}}\lesssim_{m}1.
}
This is easily verified, by commuting the derivatives and employing \eqref{psi is harmless} and \eqref{psi up down}.
\end{proof}

\section{The diagonal Green's functions}\label{S:3}

The role of this section is to introduce three central characters in the analysis, namely, $g_{12}$, $g_{21}$, and $\gamma$, and to develop some basic estimates for them.  What unifies these objects is that they all arise from the Green's function associated to the Lax operator $L(\kappa)$ introduced in \eqref{Intro AKNS L}.  Recall
\begin{equation}\label{Lax L'}
L(\kappa) = L_0(\kappa) +  \begin{bmatrix}0 & q\\-r&0 \end{bmatrix} \qtq{where} L_0(\kappa) := \begin{bmatrix}\kappa - \p & 0\\0&\kappa + \p\end{bmatrix}.
\end{equation}
We shall only consider $\kappa\in\R$ with $|\kappa|\geq 1$.  Note that
\begin{align}\label{L conjugation}
L(\kappa)^* = \begin{cases} -L(-\kappa) \quad & \text{in the defocusing case $r=\bar q$,}\\ -[\begin{smallmatrix} 1&\ 0\\0&-1\end{smallmatrix}]L(-\kappa) [\begin{smallmatrix} 1& 0\\0&-1\end{smallmatrix}] & \text{in the focusing case $r=-\bar q$.} \end{cases}
\end{align}
Evidently, both identities hold for $L_0$, since then $q=r=0$.

We will be constructing the Green's function, which is matrix valued, perturbatively from the case $q=r=0$.  By direct computation, one finds that
\[
R_0(\kappa) := L_0(\kappa)^{-1} = \begin{bmatrix}(\kappa - \p)^{-1} & 0\\0&(\kappa + \p)^{-1}\end{bmatrix}
\]
admits the integral kernel 
\begin{equation}\label{G_0}
G_0(x,y;\kappa) = e^{-\kappa|x - y|}\begin{bmatrix}\bbo_{\{x<y\}}&0\\0&\bbo_{\{y<x\}}\end{bmatrix} \quad\text{for $\kappa >0$}.
\end{equation}
For $\kappa<0$, we may use $G_0(x,y;-\kappa) =-G_0(y,x;\kappa)$, which follows from \eqref{L conjugation}.

Formally at least, the resolvent identity indicates that $R(\kappa):=L(\kappa)^{-1}$ can be expressed as
\begin{equation}\label{Resolvent}
R = R_0 + \sum_{\ell = 1}^\infty (-1)^{\ell}\sqrt{R_0}\left(\sqrt{R_0}(L - L_0)\sqrt{R_0}\right)^\ell\sqrt{R_0}.
\end{equation}
Here and below, fractional powers of $R_0$ are defined via \eqref{z to the sigma}.  This series forms the foundation of everything in this section; its convergence will be verified shortly as part of proving Proposition~\ref{P:R}.  With a view to this, we adopt the following notations:
\begin{equation}\label{D:LambdaGamma}
\Lambda := (\kappa - \p)^{-\frac12}q(\kappa + \p)^{-\frac12} \qtq{and} \Gamma := (\kappa + \p)^{-\frac12}r(\kappa - \p)^{-\frac12},
\end{equation}
whose significance is that
\begin{equation}\label{LambdaGammaJob}
\sqrt{R_0}(L - L_0)\sqrt{R_0} = \begin{bmatrix} 0 &\Lambda \\ -\Gamma & 0\end{bmatrix}.
\end{equation}
These operators also satisfy 
\begin{equation}\label{Lambda}
\|\Lambda\|_{\I_2} = \|\Gamma\|_{\I_2} \lesssim |\kappa|^{-(s+\frac12)}\|q\|_{H^s_\kappa},
\end{equation}
as is easily deduced from either \eqref{BasicBound} or \eqref{LogarithmicBound}.

\begin{prop}[Existence of the Green's function]\label{P:R}
There exists \(\delta>0\) so that $L(\kappa)$ is invertible, as an operator on $L^2(\R)$, for all \(q\in \Bd\) and all real \(|\kappa|\geq 1\).  The inverse $R(\kappa):=L(\kappa)^{-1}$ admits an integral kernel $G(x,y;\kappa)$ so that
\begin{equation}\label{tensor eqn}
q \mapsto G - G_0
\end{equation}
is a continuous mapping from $H^s_\kappa(\R)$ to the space of Hilbert--Schmidt operators from $H^{-\frac34 - \frac s2}_\kappa$ to $H^{\frac 34 + \frac s2}_\kappa$.  Moreover, $G-G_0$ is continuous as a function of $(x,y)\in\R^2$. Lastly,
\begin{align}
\p_xG(x,y;\kappa) &= \begin{bmatrix}\kappa&q(x)\\r(x)&-\kappa\end{bmatrix}G(x,y;\kappa) + \begin{bmatrix}-\delta(x-y)&0\\0&\delta(x-y)\end{bmatrix},\label{xID}\\
\p_yG(x,y;\kappa) &= G(x,y;\kappa)\begin{bmatrix}-\kappa&q(y)\\r(y)&\kappa\end{bmatrix} + \begin{bmatrix}\delta(x-y)&0\\0&-\delta(x-y)\end{bmatrix},\label{yID}
\end{align}
in the sense of distributions.
\end{prop}

\begin{proof} From \eqref{Lambda}, we have
$$
\bigl\| \sqrt{R_0}(L - L_0)\sqrt{R_0} \,\bigr\|_{\I_2}\leq \sqrt2 \|\Lambda\|_{\I_2} \lesssim  \|q\|_{H^s_\kappa} \lesssim \delta
$$
uniformly for $|\kappa|\geq 1$.  Thus, for $\delta>0$ sufficiently small, the series \eqref{Resolvent} converges in operator norm uniformly for $|\kappa|\geq 1$.  It is elementary to then verify that the sum acts as a (two-sided) inverse to $L(\kappa)$.
 
This argument also yields that $R-R_0 \in \I_2$.  In particular, it admits an integral kernel in $L^2(\R^2)$.  To prove \eqref{tensor eqn} is continuous, we only need to verify that the series defining $R-R_0$ converges in the sense of Hilbert--Schmidt operators from $H^{-\frac34 - \frac s2}_\kappa$ to $H^{\frac 34 + \frac s2}_\kappa$. This follows readily from \eqref{BasicBound}.

The continuity of $G-G_0$ as a function of $(x,y)$ follows from the Hilbert--Schmidt bound on \eqref{tensor eqn} because $\frac 34 + \frac s2 >\frac12$.

For regular $q$, the identities \eqref{xID} and \eqref{yID} precisely express the fact that $G$ is an integral kernel for $R(\kappa)$. The issue of how to make sense of them for irregular $q$ is settled by \eqref{tensor eqn}.
\end{proof}

From the jump discontinuities evident in \eqref{G_0}, we see that one cannot expect to restrict $G(x,y;\kappa)$ to the $x=y$ diagonal in a meaningful way.  However, as we have just shown, $G-G_0$ is continuous.  This allows us to unambiguously define the continuous functions
\begin{align*}
\vr(x;\kappa) &:= \sgn(\kappa)\bigl[ G_{11}(x,x;\kappa) + G_{22}(x,x;\kappa) \bigr]  - 1,\\
g_{12}(x;\kappa) &:= \sgn(\kappa) G_{12}(x,x;\kappa),\\
g_{21}(x;\kappa) &:= \sgn(\kappa) G_{21}(x,x;\kappa).
\end{align*}
Here, subscripts indicate matrix entries.  While the inclusion of the factor $\sgn(\kappa)$ may seem unnecessary, it has the aesthetical virtue of eliminating corresponding factors in many subsequent formulas, such as \eqref{g12-ID}--\eqref{micro As commute} below.

If \(q\in \BdS\) we may use the identities \eqref{xID} and \eqref{yID} for \(G\) to obtain
\begin{align}
\vr' &= 2\left(qg_{21} + rg_{12}\right),\label{rho-ID}\\
g_{12}' &= 2\kappa g_{12} + q[\vr + 1],\label{g12-ID}\\
g_{21}' &= - 2\kappa g_{21} + r[\vr + 1],\label{g21-ID}
\end{align}
in the sense of distributions. Combining \eqref{rho-ID}, \eqref{g12-ID}, and \eqref{g21-ID} yields the further identity
\begin{equation}\label{micro As commute}
\begin{aligned}
2(\kappa-\vk)&\bigl[ g_{12}(\kappa)g_{21}(\vk) - g_{21}(\kappa)g_{12}(\vk) \bigr]\\
&= \partial_x \bigl\{ g_{12}(\kappa)g_{21}(\vk) + g_{21}(\kappa)g_{12}(\vk) - \tfrac{[\vr(\kappa)+1][\vr(\vk)+1]}{2}\bigr\},
\end{aligned}
\end{equation}
which recurs several times in our analysis.  From \eqref{L conjugation} we also have
\begin{align}
\vr(\kappa) &= \bar \vr(-\kappa) \qtq{and} g_{12}(\kappa) = \pm \bar g_{21}(-\kappa).\label{grho-symmetries}
\end{align}

From the series representation \eqref{Resolvent} of the resolvent, we naturally can deduce corresponding series representations of $g_{12}$, $g_{21}$, and $\gamma$.  These are effectively power-series in terms of $q$ and $r$, albeit with each term being a paraproduct, rather than a monomial.  In what follows we shall often need to discuss individual terms in these series so, being sensitive to the order of such terms in $q$ and $r$, we adopt the following notations:
\begin{align}\label{g12-Series terms}
g_{12}\sbrack{2m+1}(\kappa) &:= \sgn(\kappa) (-1)^{m-1}\Big\< \delta_x,\ (\kappa - \p)^{-\frac12}\Lambda \left(\Gamma\Lambda\right)^m(\kappa + \p)^{-\frac12}\delta_x\Big\>, \\
    \label{g21-Series terms}
g_{21}\sbrack{2m+1}(\kappa) &:= \sgn(\kappa) (-1)^{m}\Big\< \delta_x,\ (\kappa + \p)^{-\frac12}\Gamma \left(\Lambda\Gamma\right)^m(\kappa - \p)^{-\frac12}\delta_x\Big\>,
\end{align}
with $g_{12}\sbrack{2m}(\kappa)=g_{21}\sbrack{2m}(\kappa):= 0$, and similarly, $\gamma\sbrack{2m+1}(\kappa):=0$ and
\begin{equation}\label{rho-Series terms}
\begin{aligned}
 \vr\sbrack{2m}(\kappa) &:= (-1)^m \sgn(\kappa)\Big\< \delta_x,\ (\kappa - \p)^{-\frac12}(\Lambda\Gamma)^m(\kappa - \p)^{-\frac12}\delta_x\Big\>\\
&\quad + (-1)^m\sgn(\kappa)\Big\< \delta_x,\  (\kappa + \p)^{-\frac12}(\Gamma\Lambda)^m (\kappa + \p)^{-\frac12}\delta_x\Big\>.
\end{aligned}
\end{equation}
In this way, we see that
\begin{equation}\label{g12-gamma-Series}
g_{12}(\kappa) = \sum_{\ell=1}^\infty g_{12}\sbrack{\ell}(\kappa),\quad g_{21}(\kappa) = \sum_{\ell=1}^\infty g_{21}\sbrack{\ell}(\kappa), \qtq{and} \gamma(\kappa) = \sum_{\ell=2}^\infty \gamma\sbrack{\ell}(\kappa) .
\end{equation}
In particular, we note that the expansion of $g_{12}$ contains only terms with $q$ appearing once more than $r$, while the expansion of $\gamma$ contains only terms of even order, with $q$ and $r$ appearing equally.  Analogous to our notation for individual terms, we write tails of these series as
$$
g_{12}\sbrack{\geq m}(\kappa) := \sum_{\ell=m}^\infty g_{12}\sbrack{\ell}(\kappa).
$$
We also extend these `square bracket' notations to algebraic combinations of these series; see, for example, \eqref{more sbrack}.

For small indices, it is possible to find explicit representations of the individual paraproducts via the explicit form of $G_0$; however, this quickly becomes overwhelming.  A more systematic approach can be based on iteration of the identities
$$
g_{12} = - (2\kappa-\p)^{-1} [q+\vr q],\quad g_{21} = (2\kappa+\p)^{-1} [r+\vr r], \qtq{and} \gamma = 2g_{12} g_{21} -\tfrac12 \gamma^2,
$$
which follow from \eqref{g12-ID}, \eqref{g21-ID}, and \eqref{QuadraticID}, respectively.  Pursuing either method, one is led to
\begin{gather}
g_{12}\sbrack{1}(\kappa) = - \tfrac q{2\kappa - \p},\qquad g_{12}\sbrack{3} (\kappa)= \tfrac2{2\kappa - \p}\big(q\cdot\tfrac r{2\kappa + \p}\cdot \tfrac q{2\kappa - \p}\big), \label{g12 1 3}\\
g_{21}\sbrack{1}(\kappa) =  \tfrac r{2\kappa + \p},\qquad\hphantom{-} g_{21}\sbrack{3} (\kappa)= \tfrac{-2}{2\kappa + \p}\big(r\cdot \tfrac q{2\kappa - \p}\cdot\tfrac r{2\kappa + \p}\big),\label{g21 1 3}
\end{gather}
as well as
\begin{align}
\vr\sbrack{2}(\kappa) &= - 2\,\tfrac q{2\kappa - \p}\cdot\tfrac r{2\kappa + \p}, \label{gamma 2}\\
\vr\sbrack{4}(\kappa) &= \tfrac q{2\kappa - \p}\cdot \tfrac 4{2\kappa + \p}\big(r\cdot\tfrac q{2\kappa - \p}\cdot \tfrac r{2\kappa + \p}\big)
        + \tfrac4{2\kappa - \p}\big(q\cdot\tfrac r{2\kappa + \p}\cdot \tfrac q{2\kappa - \p}\big)\cdot \tfrac r{2\kappa + \p} \label{gamma 4}\\
&\quad  - 2\,\tfrac q{2\kappa - \p}\cdot\tfrac r{2\kappa + \p}\cdot \tfrac q{2\kappa - \p}\cdot\tfrac r{2\kappa + \p}. \notag
\end{align}
Here dots emphasize occurrences of pointwise multiplication.

With these preliminaries out of the way, we are now ready to present some basic estimates on $g_{12}$, $g_{21}$, and $\gamma$.   Propositions~\ref{prop:g} and~\ref{prop:rho} focus on properties that hold pointwise in time; later in Lemma~\ref{L:LSgg} and Corollary~\ref{C:ET}, we employ local smoothing spaces.

\begin{prop}[Properties of \(g_{12}\) and \(g_{21}\)]\label{prop:g}
There exists \(\delta>0\) so that for all real \(|\kappa|\geq 1\) the maps \(q\mapsto g_{12}(\kappa)\) and \(q\mapsto g_{21}(\kappa)\) are (real analytic) diffeomorphisms of \(\Bd\) into \(H^{s+1}\) satisfying the estimates
\begin{equation}\label{g12-Hs}
\|g_{12}(\kappa)\|_{H^{s+1}_\kappa} + \|g_{21}(\kappa)\|_{H^{s+1}_\kappa} \lesssim \|q\|_{H^s_\kappa}.
\end{equation}
Further, the remainders satisfy the estimate
\begin{equation}\label{g12-LO}
\|g_{12}\sbrack{\geq 3}(\kappa)\|_{H^{s+1}_\kappa} + \|g_{21}\sbrack{\geq 3}(\kappa)\|_{H^{s+1}_\kappa} \lesssim |\kappa|^{-(2s+1)}\|q\|_{H^s_\kappa}^3,
\end{equation}
uniformly in \(\kappa\). Finally, if \(q\) is Schwartz then so are \(g_{12}(\kappa)\) and \(g_{21}(\kappa)\).
\end{prop}

\begin{proof}
It suffices to consider the case \(\kappa\geq 1\) as the case \(\kappa\leq -1\) is similar; moreover, by \eqref{grho-symmetries}, it suffices to consider $g_{12}(\kappa)$. Recalling \eqref{g12 1 3}, we obtain
\begin{equation}\label{linear isomorphism}
\|g_{12}\sbrack{1}(\kappa)\|_{H^{s+1}_\kappa} = \|q\|_{H^s_\kappa}.
\end{equation}
To bound the remaining terms in the series, we employ duality and Lemma~\ref{L:BasicBounds}:
\begin{align*}
\bigl|\<f, g_{12}\sbrack{\geq 3}(\kappa)\>\bigr|&\leq \|(\kappa + \p)^{-(1+s)}\bar f(\kappa - \p)^{-(1+s)}\|_{\op}\|(\kappa - \p)^s q (\kappa+\p)^{-(\frac34+\frac s2)}\|_{\I_2}^2\\
&\qquad\qquad \times\sum_{\ell=1}^\infty\|(\kappa - \p)^{-\frac14+\frac{s}2}q(\kappa + \p)^{-\frac14+\frac{s}2}\|_{\I_2}^{2\ell-1}\kappa^{-(\ell-1)(1+2s)}\\
&\lesssim |\kappa|^{-(2s+1)} \|f\|_{H^{-(1+s)}_\kappa} \|q\|_{H_\kappa^s}^3,
\end{align*}
provided $\delta>0$ is sufficiently small.  This proves \eqref{g12-LO} and completes the proof of \eqref{g12-Hs}.

We wish to apply the inverse function theorem to obtain the diffeomorphism property.  At the linearized level, we already have
\[
\tfrac{\delta g_{12}}{\delta q}(\kappa)\bigr|_{q=0} = - (2\kappa - \p)^{-1} \qtq{and}  \tfrac{\delta g_{12}}{\delta r}(\kappa)\bigr|_{q=0} = 0
\]
which is an isomorphism, as noted already in \eqref{linear isomorphism}.  At the nonlinear level,  we apply the resolvent identity, which shows that for any test function \(f\in \Schwartz\) we have
\begin{align*}
\left.\frac d{d\epsilon}\right|_{\epsilon = 0} G(x,z;q + \epsilon f) &= -\int G(x,y;q)\begin{bmatrix}0&f(y)\\\mp\bar f(y)&0\end{bmatrix}G(y,z;q)\,dy.
\end{align*}
Repeating the analysis used to prove \eqref{g12-LO}, we find
\begin{align*}
\bigl\| \tfrac{\delta g_{12}}{\delta r}(\kappa) \bigr\|_{H^s_\kappa\rightarrow H^{s+1}_\kappa} + \bigl\| \tfrac{\delta g_{12}}{\delta q}(\kappa)+ (2\kappa - \p)^{-1}  \bigr\|_{H^s_\kappa\rightarrow H^{s+1}_\kappa} 
    \lesssim \delta^2|\kappa|^{-(2s+1)} \lesssim \delta^2
\end{align*}
and so deduce that the diffeomorphism property holds for $\delta>0$ sufficiently small, which can be chosen independent of $|\kappa|\geq 1$.

Next we seek to show $g_{12} \in \Schwartz$ whenever \(q\in \BdS\), beginning with a consideration of derivatives.  For any \(h\in\R\), we have
\[
g_{12}(x+h;q) = g_{12}(x;q(\cdot + h)).
\]
In particular, differentiating \(n\) times with respect to \(h\) and evaluating at \(h=0\), we may use duality to bound
\begin{align*}
\|\partial_x^n g_{12}(\kappa)\|_{H^{s+1}} &\leq \sum\limits_{\ell=0}^\infty\sum\limits_{\substack{\sigma\in \N^{2\ell+1}\\|\sigma| = n}}\binom n\sigma |\kappa|^{-(2s+1)\ell}\prod\limits_{m=1}^{2\ell+1} \|\partial_x^{\sigma_m}q\|_{H^s}\\
&\leq \sum\limits_{\ell=0}^\infty C^{2\ell+1}(2\ell+1)^n|\kappa|^{-(2s+1)\ell} \|q\|_{H^s}^{2\ell}\|\partial_x^nq\|_{H^s}\lesssim_{n}\|\partial_x^nq\|_{H^s},
\end{align*}
where the constant \(C = C(s)> 0\) may be chosen independent of \(\kappa\). To handle spatial weights we observe that
\[
x^n(\kappa - \p)^{-1} = \sum\limits_{m=0}^n(-1)^m \frac{n!}{(n-m)!} (\kappa - \p)^{-m-1}x^{n-m}.
\]
In particular, by duality we may bound
\begin{align*}
\|x^ng_{12}(\kappa)\|_{H^{s+1}} &\leq \sum\limits_{\ell=0}^\infty \sum\limits_{m=0}^nC^{2\ell+1}\frac{n!}{(n-m)!} |\kappa|^{-m-(2s+1)\ell}\|x^{n-m}q\|_{H^s}\|q\|_{H^s}^{2\ell}\\
&\lesssim_{n} \|\<x\>^n q\|_{H^s}.
\end{align*}
Combining these, we see that if \(q\in \BdS\) then \(g_{12}(\kappa)\in \Schwartz\).
\end{proof}

\begin{prop}[Properties of \(\vr\)]\label{prop:rho}
There exists \(\delta>0\) so that for all real \(|\kappa|\geq 1\) the map \(q\mapsto \vr(\kappa)\) is bounded from \(\Bd\) to \(L^1\cap H^{s+1}\) and we have the estimates
\begin{align}
\|\vr(\kappa)\|_{H^{s+1}_\kappa} &\lesssim |\kappa|^{-(s + \frac12)}\|q\|_{H^s_\kappa}^2,\label{rho-Hs}\\
\|\vr(\kappa)\|_{L^\infty} &\lesssim |\kappa|^{-(2s +1)}\|q\|_{H^s_\kappa}^2,\label{rho-Linfty}\\
\|\vr(\kappa)\|_{L^1} &\lesssim \|q\|_{H^{-1}_\kappa}^2 + |\kappa|^{-2(2s+1)-1}\|q\|_{H^s_\kappa}^4,\label{rho-L1}\\
\|\vr\sbrack{\geq 4}(\kappa)\|_{L^1}&\lesssim |\kappa|^{-2(2s+1) - 1}\|q\|_{H^s_\kappa}^4,\label{rho-LO}
\end{align}
uniformly in \(\kappa\). Further, we have the quadratic identity
\eq{QuadraticID}{
\vr + \tfrac12\vr^2 = 2g_{12}g_{21},
}
and if \(q\) is Schwartz then so is \(\vr(\kappa)\).
\end{prop}

\begin{proof}
Once again it suffices to consider the case \(\kappa\geq1\). Using \eqref{gamma 2} and \eqref{E:algebra}, we obtain
\begin{align*}
\|\gamma\sbrack{2}\|_{H_\kappa^{s+1}}\lesssim \kappa^{-(s+\frac12)}\| q\|_{H^s_\kappa}^2.
\end{align*}
To handle $\gamma\sbrack{\geq 4}$ we use the series representation \eqref{g12-gamma-Series} and the same duality argument used to prove \eqref{g12-LO}. The estimate \eqref{rho-Linfty} then follows from \eqref{rho-Hs} via~\eqref{Linfty bdd}.

Setting $\vk=\kappa$ in \eqref{micro As commute}, we find that
$$
\partial_x \bigl\{ 2 g_{12}(x;\kappa)g_{21}(x;\kappa) - \tfrac12\vr(x;\kappa)^2 - \vr(x;\kappa) \bigr\} = 0.
$$
From \eqref{g12-Hs} and \eqref{rho-Hs}, we see that the term in braces vanishes as $|x|\to\infty$. Thus the quadratic identity \eqref{QuadraticID} follows by integration.

By using this quadratic identity, we may write
\begin{equation}\label{E:vr gr 4}
\vr\sbrack{\geq 4} = - \tfrac12 \vr^2 + 2g_{12}\sbrack{\geq 3}\cdot g_{21} + 2g_{12}\sbrack{1}\cdot g_{21}\sbrack{\geq 3}.
\end{equation}
By Proposition~\ref{prop:g} and  \eqref{rho-Hs}, we have
\begin{align*}
\|g_{12}\sbrack{\geq 3}\|_{L^2}+\|g_{21}\sbrack{\geq 3}\|_{L^2}&\lesssim |\kappa|^{-(1+s)}\bigl[\|g_{12}\sbrack{\geq 3}\|_{H^{s+1}_\kappa}+ \|g_{21}\sbrack{\geq 3}\|_{H^{s+1}_\kappa}\bigr] \lesssim |\kappa|^{-(2+3s)}\|q\|_{H^s_\kappa}^3\\
\|g_{12}\|_{L^2}+\|g_{21}\|_{L^2}&\lesssim |\kappa|^{-(1+s)} \bigl( \|g_{12}\|_{H_\kappa^{s+1}}+\|g_{21}\|_{H_\kappa^{s+1}}\bigr)\lesssim   |\kappa|^{-(1+s)}\|q\|_{H^s_\kappa} \\
\|\vr\|_{L^2}&\lesssim |\kappa|^{-(1+s)}\|\vr\|_{H^{s+1}_\kappa}\lesssim|\kappa|^{-(\frac 32 +2s)}\|q\|_{H^s_\kappa}^2.
\end{align*}
Thus
\begin{align*}
\|\vr\sbrack{\geq 4}\|_{L^1} &\lesssim \|\vr\|_{L^2}^2 + \|g_{12}\sbrack{\geq 3}\|_{L^2}\|g_{21}\|_{L^2} + \|g_{12}\sbrack{1}\|_{L^2}\|g_{21}\sbrack{\geq 3}\|_{L^2}\lesssim |\kappa|^{-(3+4s)}\|q\|_{H^s_\kappa}^4,
\end{align*}
which yields the estimate \eqref{rho-LO}. The estimate \eqref{rho-L1} then follows from applying the Cauchy-Schwarz inequality to \eqref{gamma 2}.

If \(q\in \BdS\) then from Proposition~\ref{prop:g} and the quadratic identity \eqref{QuadraticID} we see that \(\vr + \frac12\vr^2\in \Schwartz\). As \(H^{s+1}\) is an algebra, we may then bound
\begin{align*}
\|\partial_x^n \vr\|_{H^{s+1}}\left(1 - \|\vr\|_{H^{s+1}}\right) &\leq \left\|\partial_x^n\left(\vr + \tfrac12\vr^2\right)\right\|_{H^{s+1}} + \left(2^{n-1} - 1\right)\|\vr\|_{H^{n+s}}^2,\\
\|x^n \vr\|_{H^{s+1}}\left(1 - \|\vr\|_{H^{s+1}}\right) &\leq \left\|x^n\left(\vr + \tfrac12\vr^2\right)\right\|_{H^{s+1}},
\end{align*}
so using the estimate \eqref{rho-Hs} we see that \(\vr(\kappa)\in \Schwartz\), provided \(0<\delta\ll1\) is sufficiently small.
\end{proof}

Next we consider local smoothing estimates for \(g_{12} = g_{12}(\vk)\) and \(\vr = \vr(\vk)\).  We consider both \eqref{NLS} and \eqref{mKdV} here and so must allow two values for $\sigma$, namely, \(s+\frac12\) and \(s+1\). In fact, the proof below works uniformly on the interval $[s+\frac12,s+1]$.

\begin{lem}[Local smoothing estimates for \(g_{12}\), \(\vr\)]\label{L:LSgg}
Let \(\sigma \in\{ s+\frac12,s+1\}\). Then there exists \(\delta>0\) so that for all real $|\kappa|\geq 1$, $|\vk|\geq 1$, and \(q\in \Cont([-1,1];\Bd)\cap X^\sigma_\kappa\), the functions \(g_{12} = g_{12}(\vk)\) and \(\vr = \vr(\vk)\) satisfy the estimates
\begin{align}\label{g12-X}
|\vk|\|g_{12}\|_{X^\sigma_\kappa} + \|g_{12}\|_{X^{\sigma+1}_\kappa} &\lesssim \|q\|_{X^\sigma_\kappa} + \|q\|_{L^\infty_t H^s},\\
|\vk|\|g_{12}\sbrack{\geq 3}\|_{X^\sigma_\kappa} + \|g_{12}\sbrack{\geq 3}\|_{X^{\sigma+1}_\kappa} &\lesssim |\vk|^{-(2s+1)}\|q\|_{L^\infty_t H^s_\vk}^2\left(\|q\|_{X^\sigma_\kappa} + \|q\|_{L^\infty_t H^s}\right), \label{g12-X-1}\\
\label{rho-X-1}
|\vk|\|\vr\|_{X^\sigma_\kappa} + \|\vr\|_{X^{\sigma+1}_\kappa} &\lesssim |\vk|^{-(s+\frac12)}\|q\|_{L^\infty_t H^s_\vk}\left(\|q\|_{X^\sigma_\kappa} + \|q\|_{L^\infty_t H^s}\right),
\end{align}
where the implicit constants are independent of \(\kappa,\vk\).
\end{lem}

\begin{proof}
Applying the product estimate \eqref{X-Product-1} with the quadratic identity \eqref{QuadraticID} and the symmetry relation \eqref{grho-symmetries}, we may bound
\begin{align*}
|\vk|\|\vr\|_{X^\sigma_\kappa} + \|\vr\|_{X^{\sigma+1}_\kappa}&\lesssim |\vk|^{-(s+\frac12)}\|\vr\|_{L^\infty_t H^{s+1}_\vk}\left(|\vk|\|\vr\|_{X^\sigma_\kappa} + \|\vr\|_{X^{\sigma+1}_\kappa}\right)\\
&\quad +|\vk|^{-(s+\frac12)} \max_{\pm\vk}\|g_{12}\|_{L^\infty_t H^{s+1}_\vk}\max_{\pm\vk}\left(|\vk|\|g_{12}\|_{X^\sigma_\kappa} + \|g_{12}\|_{X^{\sigma+1}_\kappa}\right).
\end{align*}
In view of \eqref{rho-Hs}, taking \(0<\delta\ll1\) sufficiently small (independently of \(\vk\)) and using \eqref{g12-Hs}, we get
\begin{equation}\label{rho-X-1-halfway}
|\vk|\|\vr\|_{X^\sigma_\kappa} + \|\vr\|_{X^{\sigma+1}_\kappa}\lesssim |\vk|^{-(s+\frac12)}\|q\|_{L^\infty_t H^s_\vk} \max_{\pm\vk}\left(|\vk|\|g_{12}\|_{X^\sigma_\kappa} + \|g_{12}\|_{X^{\sigma+1}_\kappa}\right).
\end{equation}
As a consequence, the estimate \eqref{rho-X-1} follows from the estimate \eqref{g12-X}.

To prove the estimate \eqref{g12-X} we first apply the estimate \eqref{X-Smoothing} to obtain
\begin{equation}\label{AnotherReduction}
|\vk|\|g_{12}\|_{X^\sigma_\kappa} + \|g_{12}\|_{X^{\sigma+1}_\kappa}  \lesssim \|q\|_{X^\sigma_\kappa} + \|q\|_{L^\infty_t H^s} + |\vk|\|g_{12}\sbrack{\geq 3}\|_{X^\sigma_\kappa} + \|g_{12}\sbrack{\geq 3}\|_{X^{\sigma+1}_\kappa}.
\end{equation}
From the  identity \eqref{g12-ID} for \(g_{12}\), we see that $g_{12}\sbrack{\geq 3} = -(2\vk - \p)^{-1} (q\vr)$.  Thus, employing \eqref{X-Smoothing}, we find
\[
|\vk|\|g_{12}\sbrack{\geq 3}\|_{X^\sigma_\kappa} + \|g_{12}\sbrack{\geq 3}\|_{X^{\sigma+1}_\kappa}  \lesssim \|q\vr\|_{X^\sigma_\kappa} + \|q\vr\|_{L^\infty_t H^s}.
\]
To continue, we use \eqref{X-Product-2} together with \eqref{rho-Hs} and \eqref{rho-X-1-halfway} for \(\vr\) to obtain
\[
\|q\vr\|_{X^\sigma_\kappa}\lesssim |\vk|^{-(2s+1)}\|q\|_{L^\infty_t H^s_\vk}^2 \Bigl[\|q\|_{X^\sigma_\kappa} +\max_{\pm \vk}\bigl(|\vk|\|g_{12}\|_{X^\sigma_\kappa} + \|g_{12}\|_{X^{\sigma+1}_\kappa}\bigr)\Bigr].
\]
Using \eqref{multiplier bdd on Hs} and \eqref{rho-Hs}, we may bound
\begin{align*}
\|q\vr\|_{L^\infty_t H^s}&\lesssim  |\vk|^{-(2s+1)} \|q\|_{L^\infty_t H^s}\|q\|_{L^\infty_t H^s_\vk}^2.
\end{align*}
As a consequence,
\begin{align*}
&\max_{\pm \vk}\bigl(|\vk|\|g_{12}\sbrack{\geq 3}\|_{X^\sigma_\kappa} + \|g_{12}\sbrack{\geq 3}\|_{X^{\sigma+1}_\kappa}\bigr) \\
&\quad\lesssim |\vk|^{-(2s+1)} \|q\|_{L^\infty_t H^s_\vk}^2\Bigl[\|q\|_{X^\sigma_\kappa} +  \|q\|_{L^\infty_t H^s}+\max_{\pm \vk}\bigl(|\vk|\|g_{12}\|_{X^\sigma_\kappa} + \|g_{12}\|_{X^{\sigma+1}_\kappa}\bigr)\Bigr].
\end{align*}
Combining this with \eqref{AnotherReduction} and choosing \(0<\delta\ll1\) sufficiently small (independently of \(\kappa,\vk\)), we obtain \eqref{g12-X} and so also \eqref{g12-X-1}.
\end{proof}

Due to the structure of our microscopic conservation law, the functions $g_{12}$ and $\gamma$ will frequently occur in the combination \(\frac{g_{12}(\vk)}{2 + \vr(\vk)}\).  Naturally, this too may be written as a power series in $q$ and $r$ and we adapt our square brackets notation accordingly:
\[
\tfrac{g_{12}}{2 + \vr} = \big(\tfrac{g_{12}}{2 + \vr}\big)\sbrack{1} + \big(\tfrac{g_{12}}{2 + \vr}\big)\sbrack{\geq 3} = \big(\tfrac{g_{12}}{2 + \vr}\big)\sbrack{1} + \big(\tfrac{g_{12}}{2 + \vr}\big)\sbrack{3} + \big(\tfrac{g_{12}}{2 + \vr}\big)\sbrack{\geq 5},
\]
where the leading order terms are given by
\begin{equation}\label{more sbrack}
\big(\tfrac{g_{12}}{2 + \vr}\big)\sbrack{1} = \tfrac12 g_{12}\sbrack{1} \qtq{and} \big(\tfrac{g_{12}}{2 + \vr}\big)\sbrack{3} = \tfrac12 g_{12}\sbrack{3} - \tfrac14g_{12}\sbrack{1}\vr\sbrack{2},
\end{equation}
and the remainders by
\begin{align}
\big(\tfrac{g_{12}}{2 + \vr}\big)\sbrack{\geq 3} &= \tfrac12 g_{12}\sbrack{\geq 3} - \tfrac{g_{12}\vr}{2(2 + \vr)}, \label{more sbrack'}\\
\big(\tfrac{g_{12}}{2 + \vr}\big)\sbrack{\geq 5} &= \tfrac12 g_{12}\sbrack{\geq 5} - \tfrac14g_{12}\sbrack{1}\vr\sbrack{\geq 4} - \tfrac14g_{12}\sbrack{\geq 3}\vr + \tfrac{g_{12}\vr^2}{4(2 + \vr)}.  \label{more sbrack''}
\end{align}

Our earlier results yield the following information about these quantities:

\begin{cor}\label{C:ET}
Let \(\sigma \in\{ s+\frac12,s+1\}\). Then there exists \(\delta>0\) so that for all real $|\kappa|\geq 1$ and $|\vk|\geq 1$, we have the estimates
\begin{align}
|\vk|\bigl\|\tfrac{g_{12}(\vk)}{2 + \vr(\vk)}\bigr\|_{H^s} + \bigl\| \tfrac{g_{12}(\vk)}{2 + \vr(\vk)}\bigr\|_{H^{s+1}}
	&\lesssim \|q\|_{H^s},\label{ET Sob} \\
|\vk|\bigl\|\big(\tfrac{g_{12}(\vk)}{2 + \vr(\vk)}\big)\sbrack{\geq 3}\bigr\|_{H^s} + \bigl\| \big(\tfrac{g_{12}(\vk)}{2 + \vr(\vk)}\big)\sbrack{\geq 3}\bigr\|_{ H^{s+1}}
    &\lesssim |\vk|^{-(2s+1)}\|q\|_{ H^s_\vk}^2\|q\|_{H^s},\label{ET1 Sob}
\end{align}
for any $q\in B_\delta$.  Moreover, for \(q\in \Cont([-1,1];\Bd)\cap X^\sigma_\kappa\),
\begin{gather}
|\vk|\bigl\|\tfrac{g_{12}(\vk)}{2 + \vr(\vk)}\bigr\|_{X^\sigma_\kappa} + \bigl\|\tfrac{g_{12}(\vk)}{2 + \vr(\vk)}\bigr\|_{X^{\sigma+1}_\kappa} \lesssim \|q\|_{X^{\sigma}_\kappa} + \|q\|_{L^\infty_t H^s},\label{ET LS}\\
|\vk|\bigl\|\big(\tfrac{g_{12}}{2 + \vr}\big)\sbrack{\geq 3}\bigr\|_{X^\sigma_\kappa}
		+ \bigl\|\big(\tfrac{g_{12}}{2 + \vr}\big)\sbrack{\geq 3}\bigr\|_{X^{\sigma+1}_\kappa} \qquad\qquad\qquad\qquad \label{ET1 LS}\\
\qquad\qquad\qquad\qquad\qquad\qquad\lesssim |\vk|^{-(2s+1)} \|q\|_{L^\infty_t H^s_\vk}^2\Big(\|q\|_{X^{\sigma}_\kappa} + \|q\|_{L^\infty_t H^s}\Big), \notag\\
|\vk|\bigl\|\big(\tfrac{g_{12}}{2 + \vr}\big)\sbrack{\geq 5}\bigr\|_{X^\sigma_\kappa}
		+ \bigl\|\big(\tfrac{g_{12}}{2 + \vr}\big)\sbrack{\geq 5}\bigr\|_{X^{\sigma+1}_\kappa} \qquad\qquad\qquad\qquad \label{ET2 LS}\\
\qquad\qquad\qquad\qquad\qquad\qquad\lesssim |\vk|^{-2(2s+1)} \|q\|_{L^\infty_t H^s_\vk}^4\Big(\|q\|_{X^{\sigma}_\kappa} + \|q\|_{L^\infty_t H^s}\Big), \notag
\end{gather}
where $g_{12}=g_{12}(\vk)$ and $\vr=\vr(\vk)$.
\end{cor}

\begin{proof}
From \eqref{more sbrack} and \eqref{g12 1 3}, we see that
$$
|\vk|\bigl\|\big(\tfrac{g_{12}}{2 + \vr}\big)\sbrack{1}\bigr\|_{H^s} + \bigl\| \big(\tfrac{g_{12}}{2 + \vr}\big)\sbrack{1}\bigr\|_{ H^{s+1}}
\approx \bigl\| (2\vk-\p) \big(\tfrac{g_{12}(\vk)}{2 + \vr(\vk)}\big)\sbrack{1}\bigr\|_{ H^s}
\approx \|q\|_{H^s}.
$$
Thus \eqref{ET Sob} will follow once we prove \eqref{ET1 Sob}.  Moreover, using also \eqref{g12-ID}, we find
$$
(2\vk-\p) \big(\tfrac{g_{12}}{2 + \vr}\big)\sbrack{\geq 3} = - \tfrac{\vr}{2(2+\vr)} q + \tfrac{g_{12}}{(2 + \vr)^2} \gamma'
$$
and thence
\begin{align*}
\text{LHS\eqref{ET1 Sob}} &\lesssim \bigl\| \tfrac{\vr}{2+\vr} q \bigr\|_{H^s}  + \bigl\| \tfrac{g_{12}}{(2 + \vr)^2} \gamma' \bigr\|_{H^{s}} \\
&\lesssim |\vk|^{-(s+\frac12)} \|q\|_{H^s} \bigl\| \tfrac{\vr}{2+\vr} \bigr\|_{ H^{s+1}_\vk}
		+ |\vk|^{-(2s+1)} \|q\|_{H^s}\|q\|_{H^s_\vk}  \bigl\| \tfrac{g_{12}}{(2 + \vr)^2} \bigr\|_{ H^{s+1}_\vk},
\end{align*}
where the second step was an application of \eqref{multiplier bdd on Hs} and \eqref{rho-Hs}.  To handle the remaining rational functions, we expand as series and employ the algebra property \eqref{E:algebra}, together with \eqref{g12-Hs} and \eqref{rho-Hs}.  This yields \eqref{ET Sob} for $\delta>0$ sufficiently small.

Next we prove \eqref{ET1 LS}, since \eqref{ET LS} follows from this, \eqref{more sbrack}, and \eqref{X-Smoothing}.

In order to prove \eqref{ET1 LS}, we first employ \eqref{more sbrack'}.  The requisite estimate for the first term was given already in Lemma~\ref{L:LSgg}.  The second summand can be treated by combining that lemma with the algebra property \eqref{E:XH algebra}.

It remains to prove \eqref{ET2 LS}. Recalling the expansion \eqref{more sbrack''} the last two terms are easily controlled using \eqref{g12-X-1}, \eqref{rho-X-1}, \eqref{ET LS}, and Lemma~\ref{L:X}. To control the first two terms we use \eqref{g12-ID} and \eqref{E:vr gr 4}.
\end{proof}

\section{Conservation laws and dynamics}\label{S:4}

At a formal level, the logarithmic perturbation determinant $\log\det(L_0^{-1}L)$ (multiplied by $\sgn(\kappa)$) is given by
\begin{equation*}
\sgn(\kappa)\sum_{\ell=1}^\infty\frac{(-1)^{\ell-1}}\ell \tr\left\{\left(\sqrt{R_0}\left(L - L_0\right)\sqrt{R_0}\right)^\ell\right\}.
\end{equation*}
For $\ell>1$ the trace is well defined because the operator is trace class.  For $\ell=1$ this fails; however, in view of \eqref{LambdaGammaJob} it is natural to regard the trace as being zero in this case.  In fact, \eqref{LambdaGammaJob} implies that only the even $\ell$ contribute to this sum.  

With this in mind, we adopt the following as our rigorous definition of $A$:
\begin{equation}\label{A-def'}
A(\kappa; q) := \sgn(\kappa) \sum_{m=1}^\infty \tfrac{(-1)^{m-1}}{m} \tr\left\{(\Lambda\Gamma)^m \right\}.
\end{equation}
We will prove the convergence of this series in Lemma~\ref{L:A} below, as well as deriving several other basic properties.

The quantity $A$ is readily seen to be closely related to the quantity $\alpha(\kappa;q)$ that formed the center point of the analysis in \cite{MR3820439}.  Concretely, for $\kappa\geq 1$,
\begin{align}\label{alpha from A}
\alpha(\kappa;q) = \pm \Re A(\kappa;q) = \pm \tfrac 12\bigl[A(\kappa;q) - A(-\kappa;q)\bigr];
\end{align}
see \eqref{A-Symmetries} below. In that paper, it was shown that $\alpha(q)$ is preserved under the NLS and mKdV flows.  In fact, the argument given there even shows that $A(\kappa;q)$ is conserved.  However, for our purposes here, we need several stronger assertions of a similar flavor.
 
First, we need that $A(\kappa;q)$ is conserved under all flows generated by the real and imaginary parts of $A(\vk;q)$ for general $\vk$.  This is proved in Lemma~\ref{lem:PoissonBrackets} below and will yield the conservation of $\alpha$ under our  regularized Hamiltonians.  This allows us to obtain a priori bounds for these regularized flows.

Second, we rely on our discovery of a microscopic expression of the conservation of $A$; this will be essential in our development of local smoothing estimates.  The relevant density $\rho$ is introduced in Lemma~\ref{L:A}; see \eqref{micro A}.  The corresponding currents (for various flows) are collected in Corollary~\ref{C:microscopic}, building on a number of intermediate results.

\begin{lem}[Properties of \(A\)]\label{L:A}
There exists \(\delta>0\) so that for all \(q\in \Bd\) and real \(|\kappa|\geq 1\), the series \eqref{A-def'} defining \(A\) converges absolutely. Moreover,
\begin{gather}
A(\kappa) = - \bar A(-\kappa),    \label{A-Symmetries} \\
\tfrac{\delta A}{\delta q} = g_{21},\qquad \tfrac{\delta A}{\delta r} = - g_{12}, \qquad \vr' = 2\left(q\tfrac{\delta A}{\delta q} - r\tfrac{\delta A}{\delta r} \right), \label{rho from A} \\
\tfrac{\partial A}{\partial \kappa} = \int \vr(x;\kappa)\,dx \qtq{and} A(\kappa) = - \int_\kappa^{\sgn(\kappa)\infty} \int\vr(x;\vk)\,dx\,d\vk, \label{rho-to-A} \\
A  = \int \rho(x;\kappa)\, dx \qtq{where} \rho(\kappa)=\frac{qg_{21}(\kappa)-rg_{12}(\kappa)}{2+\vr(\kappa)}.  \label{micro A}
\end{gather}
\end{lem}

\begin{proof}
First, we observe that the series \eqref{A-def'} converges absolutely and uniformly for $|\kappa|\geq 1$ and $q\in \Bd$, provided $0<\delta\ll1$. This follows from the estimate \eqref{Lambda}. In the same way, convergence holds for the term-wise derivative of the series \eqref{A-def'} with respect to $\kappa$. The terms appearing are exactly those from \eqref{rho-Series terms} and \eqref{g12-gamma-Series}, and so we may deduce that
$$
\frac{\partial A}{\partial \kappa} = \int \vr(x;\kappa)\,dx.
$$
This proves the first assertion in \eqref{rho-to-A} as well as justifying \ref{E:res trace}. The second assertion of \eqref{rho-to-A} then follows since \eqref{Lambda} guarantees that $A(\kappa)\to0$ uniformly on $\Bd$ as $|\kappa|\to\infty$.

The conjugation symmetry \eqref{A-Symmetries} follows immediately from \eqref{rho-to-A} and \eqref{grho-symmetries}.

Differentiating the series \eqref{A-def'} with respect to $r$ yields the series \eqref{g12-gamma-Series} for $g_{12}$ with an additional minus sign, thus giving the second assertion in \eqref{rho from A}.  The first assertion follows in a parallel manner, or by invoking conjugation symmetry.  The third part of \eqref{rho from A} follows from the first two parts via \eqref{rho-ID}.

We now turn our attention to \eqref{micro A}.  First we must clarify what we mean by $\int \rho$. When $q\in\Schwartz$, then $\rho$ also belongs to Schwartz class (for $\delta$ small enough) and so the integral can be taken in the classical sense.  For $q\in H^s$ however, we interpret this integral via the duality between $H^s$ and $H^{-s}$, noting that
$$
q,r \in H^{s}(\R) \qtq{and} \tfrac{g_{21}(\kappa)}{2+\vr(\kappa)} , \tfrac{g_{12}(\kappa)}{2+\vr(\kappa)} \in H^{1+s}(\R) \hookrightarrow H^{-s}(\R);
$$
see Corollary~\ref{C:ET}.  By density and continuity, it suffices to verify \eqref{micro A} for $q\in\Schwartz$.

 Differentiating \eqref{g12-ID}, \eqref{g21-ID}, and \eqref{QuadraticID} with respect to $\kappa$ and then combining these with the original versions shows
\begin{align*}
\partial_x\Bigl(g_{12}\tfrac{\partial g_{21}}{\partial\kappa} - \tfrac{\partial g_{12}}{\partial\kappa} g_{21}\Bigr) & =-\vr(2+\vr) + (1+\vr)\tfrac{\partial\ }{\partial\kappa}\bigl(qg_{21}-rg_{12}\bigr) - \bigl(qg_{21}-rg_{12}\bigr) \tfrac{\partial\vr}{\partial\kappa}.
\end{align*}
Using also \eqref{rho-ID} we obtain
\begin{align*}
- \bigl( g_{12}\tfrac{\partial g_{21}}{\partial\kappa} - \tfrac{\partial g_{12}}{\partial\kappa} g_{21}\bigr) \vr'  =  -\vr(2+\vr)\tfrac{\partial\ }{\partial\kappa}\bigl(qg_{21}-rg_{12}\bigr)+ \bigl(qg_{21}-rg_{12}\bigr)(1+\vr)\tfrac{\partial\vr}{\partial\kappa}.
\end{align*}
These identities then combine to show
\begin{align*}
\partial_x\frac{g_{12}\tfrac{\partial g_{21}}{\partial\kappa} - \tfrac{\partial g_{12}}{\partial\kappa} g_{21}}{2+\vr} &= - \vr + \frac{\partial\ }{\partial\kappa}\frac{qg_{21}-rg_{12}}{2+\vr} ,
\end{align*}
which can then be integrated in $x$ to yield
\begin{align*}
\frac{\partial\ }{\partial\kappa} \int \frac{qg_{21}-rg_{12}}{2+\vr}\,dx = \int \vr\,dx = \frac{\partial A}{\partial\kappa} .
\end{align*}
The veracity of \eqref{micro A} then follows by observing that both sides of \eqref{micro A} vanish in the limit $|\kappa|\to\infty$.
\end{proof}

Next, we show that our basic Hamiltonians arise as coefficients in the asymptotic expansion of $A(\kappa)$ as $\kappa \to\infty$.  This will also be important for introducing our renormalized flows later on.

\begin{lem}\label{L:A asymptotics} For $q\in B_\delta\cap \Schwartz$,
\begin{equation}\label{A asymptotics}
A(\kappa)=\tfrac1{2\kappa} M + \tfrac{-i}{(2\kappa)^2} P + \tfrac{(-i)^2}{(2\kappa)^3} H_{\NLS} + \tfrac{(-i)^3}{(2\kappa)^4} H_{\mKdV} + O(\kappa^{-5})
\end{equation}
as an asymptotic series on Schwartz class.
\end{lem}

\begin{proof}
While the first few terms can readily be discovered by brute force, we follow a systematic method based on the biHamiltonian relations
\eq{biHam}{
\begin{aligned}
-2\kappa \tfrac{\delta A}{\delta q} &= \partial\tfrac{\delta A}{\delta q} - r\bigl[\vr+1\bigr] = \partial\tfrac{\delta A}{\delta q} + 2r\partial^{-1}\bigl(r\tfrac{\delta A}{\delta r} - q\tfrac{\delta A}{\delta q}\bigr) - r,\\
2\kappa  \tfrac{\delta A}{\delta r} &= \partial\tfrac{\delta A}{\delta r} + q\bigl[\vr+1\bigr] = \partial\tfrac{\delta A}{\delta r} - 2q\partial^{-1}\bigr(r\tfrac{\delta A}{\delta r} - q\tfrac{\delta A}{\delta q}\bigr) + q,
\end{aligned}
}
which, in view of \eqref{rho from A}, are merely a recapitulation of \eqref{g12-ID} and \eqref{g21-ID}.

By iterating \eqref{biHam}, we find
\eq{biHam2}{
\begin{aligned}
\tfrac{\delta A}{\delta r} &= - g_{12} = \tfrac{q}{2\kappa} + \tfrac{q'}{(2\kappa)^2} + \tfrac{q''-2q^2r}{(2\kappa)^3} + \tfrac{q'''-6qq'r}{(2\kappa)^4} + O(\kappa^{-5}), \\
\tfrac{\delta A}{\delta q} &=  g_{21} = \tfrac{r}{2\kappa} - \tfrac{r'}{(2\kappa)^2} + \tfrac{r''-2qr^2}{(2\kappa)^3} - \tfrac{r'''-6qrr'}{(2\kappa)^4} + O(\kappa^{-5}),
\end{aligned}
}
which can then be integrated to recover the series for $A$; indeed,
$$
A(q) = \int_0^1 \partial_\theta A(\theta q)\,d\theta
	= \int_0^1 \langle \bar q, \tfrac{\delta A}{\delta q}(\theta q)\rangle + \langle \bar r, \tfrac{\delta A}{\delta r}(\theta q)\rangle \,d\theta .
$$

In following this algorithm, we have found it convenient to successively update the asymptotic expansion of $\vr$ using \eqref{QuadraticID}, rather than compute $\partial^{-1}(r\tfrac{\delta A}{\delta r} - q\tfrac{\delta A}{\delta q})$ by laboriously finding complete derivatives.  We record here the key result:
\begin{align*}
\tfrac12\vr   = -\tfrac{qr}{(2\kappa)^2} - \tfrac{q'r-qr'}{(2\kappa)^3} & -  \tfrac{q''r - q'r' + qr''-3q^2r^2}{(2\kappa)^4} \\
	&-\tfrac{q'''r - q''r' + q'r'' - qr''' - 6qq'r^2 + 6q^2rr'}{(2\kappa)^5} + O(\kappa^{-6}).
\end{align*}
This technique is easily automated on a computer algebra system, which we have done as a check on our hand computations.
\end{proof}

Although the mechanical interpretation of the Poisson bracket \eqref{PoissonBracket} originates in real-valued observables $F$ and $G$, the definition makes sense for complex-valued functions as well.  In view of the conjugation symmetry \eqref{A-Symmetries}, the following guarantees the commutation of both the real and imaginary parts of $A$:

\begin{lem}[Poisson brackets]\label{lem:PoissonBrackets}
There exists \(\delta>0\) so that for all real \(|\kappa|,|\vk|\geq 1\) and \(q\in \BdS\) we have
\begin{equation}\label{APoissonBracket}
\{A(\kappa),A(\vk)\} = 0.
\end{equation}
\end{lem}

\begin{proof}
If \(\kappa = \vk\) there is nothing to prove. Suppose now that \(\kappa\neq\vk\).  From \eqref{rho from A} and then \eqref{micro As commute} we deduce that
\begin{align*}
\{A(\kappa),A(\vk)\} &= \tfrac1{i} \int g_{12}(\kappa)g_{21}(\vk) - g_{21}(\kappa)g_{12}(\vk) =0.\qedhere
\end{align*}
\end{proof}

As shown already in \cite{MR3820439}, the conservation of $A(\kappa)$ leads to global in time control on the $H^s$ norm.  Rather than simply recapitulate that argument, which was based on the series \eqref{A-def'}, we will present a proof that brings the density $\rho$ to center stage.  This approach will be essential later, when we introduce localizations; see Lemmas~\ref{lem:rhA} and~\ref{L:ReRho}.

\begin{prop}[A priori bound]\label{prop:alpha}
There exists \(\delta>0\) so that for all \(q\in \Bd\) and \(\kappa\geq 1\) we have
\begin{equation}\label{AlphaToNorm}
\int_\kappa^\infty \vk^{2s+1}\alpha(\vk)\,\frac{d\vk}\vk \approx_s \|q\|_{H^s_\kappa}^2.
\end{equation}

Choosing \(\delta>0\) even smaller if necessary, we deduce the a priori estimate
\begin{equation}\label{APBound}
\|q(t)\|_{H^s_\kappa}\lesssim \|q(0)\|_{H^s_\kappa} \qtq{uniformly for} q(0)\in\BdS
\end{equation}
for any Hamiltonian flow that is continuous on Schwartz class and preserves $A(\vk)$ for all $|\vk|\geq 1$.
\end{prop}

\begin{proof}
We first decompose \(\rho(\vk)= \rho\sbrack{2} (\vk)+ \rho\sbrack{\geq 4}(\vk)\) with
\begin{align}\label{rho quadratic}
\rho\sbrack{2} (\vk)&= \tfrac12\big(q\cdot\tfrac r{2\vk + \p} + \tfrac q{2\vk - \p}\cdot r\big),\\
\label{rho higher}
\rho\sbrack{\geq 4} (\vk) &= q\cdot \big(\tfrac{g_{21}(\vk)}{2+\vr(\vk)}\big)\sbrack{\geq 3} - r\cdot (\tfrac{g_{12}(\vk)}{2 + \vr(\vk)}\big)\sbrack{\geq 3}.
\end{align}

Inspired by \eqref{alpha from A}, we compute
\begin{align}\label{R151}
\pm \int \tfrac12\bigl[\rho\sbrack{2}(x;\vk)-\rho\sbrack{2}(x;-\vk)\bigr]\,dx = 2\vk \int\frac{|\hat q(\xi)|^2\,d\xi}{4\vk^2+\xi^2}
\end{align}
and so, invoking \eqref{EquivNorm}, deduce that
\begin{align}\label{R152}
\pm \int_\kappa^\infty \! \int \tfrac12\bigl[\rho\sbrack{2}(x;\vk)-\rho\sbrack{2}(x;-\vk)\bigr] \vk^{2s}\,dx\,d\vk \approx \|q\|_{H^s_\kappa}^2.
\end{align}

On the other hand, interpolating the bounds in \eqref{ET1 Sob}, we find
\begin{align}\label{rat in -s}
\bigl\| \big(\tfrac{g_{21}(\vk)}{2+\vr(\vk)}\big)\sbrack{\geq 3} \bigr\|_{H^{-s}} \lesssim  |\vk|^{-2(1+2s)} \delta \|q\|_{H^s_\vk}^2
\end{align}
and consequently,
\begin{align}\label{R153}
\int_\kappa^\infty \biggl| \int \tfrac12\bigl[\rho\sbrack{\geq 4}(x;\vk)-\rho\sbrack{\geq 4}(x;-\vk)\bigr]\vk^{2s}\,dx \biggr| d\vk \lesssim \kappa^{-(1+2s)} \delta^2 \|q\|_{H^s_\kappa}^2.
\end{align}
Thus \eqref{AlphaToNorm} follows by choosing $\delta>0$ sufficiently small.

To deduce \eqref{APBound}, we exploit continuity in time.
\end{proof}

Proposition~\ref{prop:alpha} is the key to proving equicontinuity of orbits. The proper extension of the notion of equicontinuity from the setting of the Arzel\` a--Ascoli Theorem to Sobolev spaces was discussed already by M.~Riesz~\cite{MRiesz}.

\begin{defn}[Equicontinuity]\label{d:equicontinuity}
A set \(Q\subset H^s\) is said to be \emph{equicontinuous} if
\[
\limsup_{\delta\to0}\sup_{q\in Q}\sup_{|y|<\delta}\|q(\cdot +y) - q(\cdot)\|_{H^s} = 0.
\]
\end{defn}

Beyond boundedness and equicontinuity, the other key ingredient needed for compactness is tightness; see Definition~\ref{d:tight}.

\begin{prop}[Equicontinuity of orbits]\label{prop:Equicontinuity} Suppose that \(Q\subset \BdS\) is equicontinuous in \(H^s\). Let \(H_1,H_2\) be Hamiltonians with flows that are continuous on Schwartz class and preserve \(A(\vk)\) for all \(|\vk|\geq 1\). Then the set
\[
Q^* = \left\{e^{J\nabla\left(tH_1 + \tau H_2\right)}q:q\in Q,\,t,\tau\in \R,\,\kappa\geq 1\right\}
\]
is equicontinuous in \(H^s\).
\end{prop}

\begin{proof}
By Plancherel (cf.~\cite[\S4]{MR3990604}), it is easy to show that a bounded set \(Q\subset H^s\) is equicontinuous if and only if
\[
\lim\limits_{\kappa\rightarrow+\infty}\sup\limits_{q\in Q}\|q\|_{H^s_\kappa} = 0.
\]
The result then follows directly from the estimate \eqref{APBound}.
\end{proof}

Next we address the question of how $\gamma$, $g_{12}$, and $g_{21}$ evolve when taking $A(\kappa)$ as the Hamiltonian. As a complex-valued function, \(A(\kappa)\) cannot be a true Hamiltonian. Nevertheless, there is a natural vector field associated to it by Hamilton's equations. We caution the reader that this vector field does not respect the relation \(r = \pm \bar q\). Ultimately, we would like to restrict to the real and imaginary parts of $A(\kappa)$; however, it is convenient to temporarily retain this illusory complex Hamiltonian and recover the real and imaginary parts later using \eqref{A-Symmetries}.  This context is important for our next two results: the evolution equations we derive for the $A(\kappa)$ vector field really represent a complex linear combination of the vector fields associated to the real and imaginary parts (taken separately).

\begin{prop}[Lax representation]\label{P:Lax A} For distinct $\kappa,\vk\in\R\setminus(-1,1)$,
\begin{align*}
- 2(\kappa-\vk)\begin{bmatrix}0 & \tfrac{\delta A(\kappa)}{\delta r} \\ \tfrac{\delta A(\kappa)}{\delta q} & 0 \end{bmatrix}
&= L(\vk) \begin{bmatrix}0 & \tfrac{\delta A(\kappa)}{\delta r} \\ \tfrac{\delta A(\kappa)}{\delta q}  & 0 \end{bmatrix}
        + \begin{bmatrix}0 & \tfrac{\delta A(\kappa)}{\delta r} \\ \tfrac{\delta A(\kappa)}{\delta q}  & 0 \end{bmatrix} L(\vk) \\
&\quad + \tfrac12 \left[ L(\vk),\ \begin{bmatrix}\vr(\kappa)+1 & 0 \\ 0 & -\vr(\kappa)-1 \end{bmatrix} \right].
\end{align*}
Equivalently, under the $A(\kappa)$ vector field, $U :=[\begin{smallmatrix} 1&0\\0&-1 \end{smallmatrix}] L(\vk)$ obeys
\begin{align}\label{Lax repres}
\frac{d\ }{dt} U = [ P, U] \qtq{with} P = \tfrac{1}{2i(\kappa-\vk)}\begin{bmatrix}\frac12(\vr(\kappa) + 1) & - g_{12}(\kappa)\\g_{21}(\kappa)&-\frac12(\vr(\kappa) + 1)\end{bmatrix}.
\end{align}
\end{prop}

\begin{proof}
Both identities are elementary computations using \eqref{biHam} and \eqref{rho from A}.
\end{proof}

\begin{cor}\label{C:A flow}
Fix distinct $\kappa,\vk\in\R\setminus(-1,1)$.  Then under the $A(\kappa)$ vector field,
\begin{align}\label{qr under A}
i\frac d{dt}q = - g_{12}(\kappa) \qtq{and} i\frac d{dt}r = - g_{21}(\kappa).
\end{align}
Moreover,
\begin{equation}\label{10:33}
\begin{aligned}
i \frac{d\ }{dt} g_{12}(\vk) &= \tfrac{ 1}{2(\kappa - \vk)} \bigl\{ [\vr(\kappa)+1] g_{12}(\vk) - g_{12}(\kappa)[\vr(\vk)+1] \bigr\}, \\
i \frac{d\ }{dt} g_{21}(\vk) &= \tfrac{-1}{2(\kappa - \vk)} \bigl\{ [\vr(\kappa)+1] g_{21}(\vk) - g_{21}(\kappa)[\vr(\vk)+1] \bigr\},
\end{aligned}
\end{equation}
and $\partial_t \vr(\vk)  + \partial_x \vj(\vk,\kappa) =0$ where
\begin{align}\label{A:rho dot}
\vj(\vk,\kappa) :=  \tfrac{-i}{2(\kappa - \vk)^2} \Bigl[g_{12}(\kappa)g_{21}(\vk)+g_{12}(\vk)g_{21}(\kappa) - \tfrac{\vr(\kappa)\vr(\vk) + \vr(\vk) + \vr(\kappa)}2  \Bigr].
\end{align}
Lastly, $\partial_t \rho(\vk) + \partial_x j(\vk,\kappa) = 0$ with
\begin{align}\label{j sub A}
j(\vk,\kappa):= -i \frac{g_{12}(\kappa)g_{21}(\vk)+g_{21}(\kappa)g_{12}(\vk)}{2(\kappa-\vk)(2+\vr(\vk))} + i \frac{\vr(\kappa)}{4(\kappa-\vk)}.
\end{align}
\end{cor}

\begin{proof}
The identities \eqref{qr under A} simply recapitulate \eqref{HFlow} and \eqref{rho from A}.

Combining \eqref{HFlow}, the resolvent identity, and Proposition~\ref{P:Lax A}, we have
\begin{align*}
i \frac{d\ }{dt} G(x,z;\vk) &= - \int G(x,y;\vk)\begin{bmatrix}0&\tfrac{\delta A}{\delta r}(y)\\\tfrac{\delta A}{\delta q}(y)&0\end{bmatrix}G(y,z;\vk)\,dy\\
&= \tfrac1{2(\kappa - \vk)}\left(\begin{bmatrix}0&\tfrac{\delta A}{\delta r}(x)\\\tfrac{\delta A}{\delta q}(x)&0\end{bmatrix}G(x,z;\vk) + G(x,z;\vk)\begin{bmatrix}0&\tfrac{\delta A}{\delta r}(z)\\\tfrac{\delta A}{\delta q}(z)&0\end{bmatrix}\right)\\
&\qquad + \tfrac1{4(\kappa - \vk)}[\vr(x;\kappa)+1] \begin{bmatrix}1&0\\0&-1\end{bmatrix}G(x,z;\vk) \\
&\qquad - \tfrac1{4(\kappa - \vk)}G(x,z;\vk)[\vr(z;\kappa)+1]\begin{bmatrix}1&0\\0&-1\end{bmatrix}.
\end{align*}
This quantity is actually a continuous function of $x$ and $z$ (as can be seen from the middle expression) and so we may restrict to $z=x$. Thus,
by \eqref{rho from A},
\begin{align*}
i \tfrac{d\ }{dt} G(x,x;\vk) &=
\tfrac{-1}{2(\kappa - \vk)}\begin{bmatrix}g_{12}(\kappa)g_{21}(\vk)-g_{12}(\vk)g_{21}(\kappa)\!\! & g_{12}(\kappa)[\vr(\vk)+1] \\ -g_{21}(\kappa)[\vr(\vk)+1] & \!\!g_{12}(\kappa)g_{21}(\vk)-g_{12}(\vk)g_{21}(\kappa) \end{bmatrix} \\
&\quad {} + \tfrac{\vr(\kappa)+1}{2(\kappa - \vk)} \begin{bmatrix}0&g_{12}(\vk)\\-g_{21}(\vk)&0\end{bmatrix}.
\end{align*}
This then yields \eqref{10:33} directly and \eqref{A:rho dot} by invoking  \eqref{micro As commute}.

The claim \eqref{j sub A} follows from a lengthy computation using \eqref{10:33}, \eqref{A:rho dot}, \eqref{micro As commute}, and \eqref{QuadraticID}.
\end{proof}

Corollary~\ref{C:A flow} shows that both $\rho(\vk)$ and $\vr(\vk)$ obey microscopic conservation laws.  From Lemma~\ref{L:A} we see that the corresponding macroscopic conservation laws are $A(\vk)$ and $\partial_\vk A(\vk)$, respectively; thus, these two microscopic conservation laws are closely related.  In the analysis that follows we shall rely exclusively on the conservation law associated to $\rho$, rather than $\vr$.  Let us explain why.

As we saw already in \eqref{E:res trace}, the quantity $\gamma$ arises very naturally in the theory and indeed, it was the basis of our initial investigations of the problem.  While it may be possible to build the entire theory around $\gamma$, we can attest that this approach rapidly becomes extremely tiresome.  It took us a very long time to discover the density $\rho$ that expresses the conservation of $A(\vk)$ and this innovation has immensely simplified all that follows.  A major virtue of $\rho$ compared to $\vr$ is coercivity.

The goal of our next lemma is to give a simple expression of this distinction, by looking only at the quadratic terms in the currents associated to the basic Hamiltonians appearing as coefficients in the expansion \eqref{A asymptotics}.  In particular, Lemma~\ref{lem:Currents are coercive} shows that the current $\vj$ associated with $\gamma$ is not coercive under the \eqref{mKdV} flow.  

Note that the terms in the series \eqref{A asymptotics} are alternately real and imaginary.  Correspondingly, to exploit the coercivity of $\Im j$ exhibited below, we shall need to use $\Im \rho$ when studying \eqref{NLS} and $\Re \rho$ when studying \eqref{mKdV}.  It is also instructive to remember that monotone observables must be odd (not even) under time reversal.

The identities \eqref{j2 pre-series} and \eqref{vj2 pre-series} appearing in the proof below also show us that neither $\Re j$ nor $\Re \vj$ possess any coercivity.

\begin{lem}[Coercivity of the current]\label{lem:Currents are coercive}  The coefficients in the asymptotic series
\begin{align*}
\int \Im j\sbrack{2}(\vk,\kappa)\,dx = \pm \sum_{\ell=0}^\infty \biggl\{ \frac{(-1)^\ell (2\ell+1)}{(2\kappa)^{2\ell+2}} \vk{\mkern 2mu} C_{\ell}(\vk) + \frac{(-1)^{\ell+1}(\ell+1)}{(2\kappa)^{2\ell+3}} C_{\ell+1}(\vk) \biggr\}
\end{align*}
are coercive; indeed,
\begin{align*}
C_\ell = \int \frac{2{\mkern 1mu}\xi^{2\ell} |\hat q(\xi)|^2}{4\vk^2+\xi^2}  \,d\xi.
\end{align*}
The corresponding asymptotic series for $j_\gamma$ is
\begin{align*}
\int \Im \vj\sbrack{2}(\vk,\kappa)\,dx = \pm \sum_{\ell=0}^\infty \biggl\{ \frac{(-1)^\ell (2\ell+1)}{(2\kappa)^{2\ell+2}} \frac{\partial [\vk{\mkern 2mu} C_{\ell}(\vk)]}{\partial\vk}  + \frac{(-1)^{\ell+1}(\ell+1)}{(2\kappa)^{2\ell+3}} \frac{\partial C_{\ell+1}(\vk)}{\partial\vk} \biggr\}.
\end{align*}
The coefficients appearing for even powers of \(\kappa\) are never sign definite; this undermines the utility of \(\gamma\).
\end{lem}

\begin{proof}
From \eqref{j sub A}, we readily find
\begin{align}\label{j2 pre-series}
\int j\sbrack{2}(\vk,\kappa)\,dx = \int \frac{2i\vk-\xi}{(2\kappa-i\xi)^2} \frac{\hat q(\xi) \hat r(-\xi) }{4\vk^2+\xi^2} \,d\xi,
\end{align}
from which the expansion is readily verified.  The analogous formula for $\vj$ is
\begin{align}\label{vj2 pre-series}
\int \vj\sbrack{2}(\vk,\kappa)\,dx = 2 \int \frac{4\vk\xi+i(\xi^2 -4 \vk^2)}{(2\kappa-i\xi)^2} \frac{\hat q(\xi) \hat r(-\xi) }{(4\vk^2+\xi^2)^2} \,d\xi.
\end{align}
The fact that this coincides with the $\vk$-derivative of \eqref{j2 pre-series} is not a coincidence; it reflects the first identity in \eqref{rho-to-A}. 

It easy to verify that \(\p_\vk[\vk{\mkern 2mu} C_{\ell}(\vk)]\) is never sign definite because it contains the factor \(\xi^2 - 4\vk^2\).
\end{proof}

In view of the asymptotic expansion \eqref{A asymptotics}, Corollary~\ref{C:A flow} provides an efficient method for deriving the evolutions of $g_{12}$ and $\vr$ under \eqref{NLS} and \eqref{mKdV}, although they are also readily computable directly from the definitions.

\begin{cor}[Induced flows]\label{H Dynamics for rho}
Fix $\vk\in\R\setminus(-1,1)$.\\
Under the $M$ flow (i.e., phase rotation),
\begin{align}i\frac d{dt} g_{12} &= g_{12} \qtq{and} i\frac d{dt}\vr = 0.
\end{align}
Under the $P$ flow (i.e., spatial translation),
\begin{align}
\frac d{dt}g_{12} &= g_{12}' \qtq{and} \frac d{dt}\vr = \vr'.
\end{align}
Under the $H_{\NLS}$ flow \eqref{NLS},
\begin{align}
i\frac d{dt}g_{12} &= - g_{12}'' + 4qrg_{12} + 2q^2g_{21},\label{NLS-g12-eq}\\
i\frac d{dt}\vr &= \bigl\{ 2rg_{12} - 2qg_{21} - 4\vk\vr \bigr\}'.\label{NLS-rho-eq}
\end{align}
Under the \(H_{\mKdV}\) flow \eqref{mKdV},
\begin{align}
\frac d{dt}g_{12} &= - g_{12}''' + 6qrg_{12}' + 6qg_{21}q' + 6g_{12}rq',\\
\frac d{dt}\vr &= - \vr''' + \bigl\{ 12\vk(rg_{12} - qg_{21}) - 12\vk^2\vr + 6qr(1 + \vr)\bigr\}'.\label{mKdV-rho-eq}
\end{align}
\end{cor}

These expressions highlight two phenomena that are worthy of note.  The first is that the evolution of $\gamma$ has the structure of a microscopic conservation law.  This has been discussed already, in the context of \eqref{A:rho dot}.

Although rather less obvious, these formulas also show that $g_{12}$ obeys the linearized equation around the trajectory $q$.  To explain why, let us first consider a generic one-parameter family of solutions $q(t;\tau)$ to a given PDE, say \eqref{NLS}.  Here $\tau$ is the parameter, while $t$ is time.  Evidently, the parametric derivative of $q$ obeys the linearized equation:
$$
i\partial_t \tfrac{\partial q}{\partial \tau} = -\partial^2_x \tfrac{\partial q}{\partial \tau} \pm 4 |q|^2 \tfrac{\partial q}{\partial \tau}
	\pm 2q^2 \tfrac{\partial \bar q}{\partial \tau} .
$$
This should be compared to \eqref{NLS-g12-eq}, noting the conjugation symmetry \eqref{grho-symmetries}.

Finally, to apply this general reasoning to the case at hand, we define our parametric family of solutions to the $H$-flow with initial data \(q_0\) via
$$
q(t;\tau) = \exp\bigl\{tJ\nabla H + \tau J\nabla A(\vk)\bigr\}q_0
$$
and then apply \eqref{rho from A}.

\subsection{Regularized and difference flows}

As discussed in the introduction, a key ingredient in our arguments is the decomposition of the full evolution into two commuting parts.  The first is a regularized part, that captures the dominant portion of the dynamics, while being very tame at high frequencies.  The second part, which we call the difference flow, restores the proper evolution to the high frequencies, but otherwise is very close to the identity.

The starting point for the corresponding decomposition of the Hamiltonian is \eqref{A asymptotics}, which we essentially rearrange to isolate an approximation to the true Hamiltonian.  While we wish to consider only real-valued Hamiltonians and taking real and imaginary parts of \eqref{A asymptotics} is a transparent way to do this, we should also acknowledge a more subtle point: in order to obtain local smoothing for the difference flow, it is essential that the regularized Hamiltonian retains the same conjugation/time-reversal symmetry as the full Hamiltonian.

\begin{defn} Associated to each \(\kappa\geq 1\), we define \emph{regularized Hamiltonians}
\begin{align}
H_{\NLS}^{\kappa} &= - 8\kappa^3\Re A(\kappa) + 4\kappa^2 M = - 4\kappa^3 \bigl[ A(\kappa) - A(-\kappa)\bigr] + 4\kappa^2 M  , \label{D:HNLSK}\\
H_{\mKdV}^{\kappa} &= 16\kappa^4\Im A(\kappa) + 4\kappa^2 P = -8i\kappa^4 \bigl[ A(\kappa) + A(-\kappa)\bigr] + 4\kappa^2 P , \label{D:HKdVK}
\end{align}
as functions on $\Bd\cap \Schwartz$, as well as \emph{difference Hamiltonians},
\begin{align}
H_{\NLS}^\diff = H_{\NLS}-H_{\NLS}^{\kappa} \qtq{and} H_\mKdV^\diff = H_{\mKdV} - H_\mKdV^\kappa .  \label{D:Hdiff}
\end{align}
\end{defn}

One of the key features of the regularized flows is that they are readily seen to be well-posed:

\begin{prop}[Global well-posedness of the regularized flows]\label{prop:kappa-flows}\leavevmode\hfill
There exists \(\delta>0\) so that for all \(\kappa\geq 1\) the $H_{\NLS}^{\kappa}$ and $H_{\mKdV}^{\kappa}$ flows
\begin{align}
i\frac d{dt}q &= 4\kappa^3\left(g_{12}(\kappa) - g_{12}(-\kappa)\right) + 4\kappa^2q,	\label{NLS-kappa}\tag{NLS${}_\kappa$}\\
\frac d{dt}q &= 8\kappa^4\left(g_{12}(\kappa) + g_{12}(-\kappa)\right) + 4\kappa^2q' 	\label{mKdV-kappa}\tag{mKdV${}_\kappa$}
\end{align}
are globally well-posed for initial data in \(\Bd\).  These solutions conserve \(\alpha(\vk)\) for every \(\vk\geq 1\).  Moreover, if the initial data is Schwartz then so are the corresponding solutions.
\end{prop}

\begin{proof}
The evolution equations  follow directly from \eqref{D:HNLSK} and \eqref{D:HKdVK} by applying \eqref{A-Symmetries} and \eqref{rho from A}.

Using the diffeomorphism property of the map \(q\mapsto g_{12}(\kappa)\) proved in Proposition~\ref{prop:g}, we may view the equations \eqref{NLS-kappa} and \eqref{mKdV-kappa} as ODEs in \(H^s\), the latter after making the change of variables
\[
(t,x) \mapsto (t,x - 4\kappa^2 t).
\]
Local well-posedness then follows from the Picard-Lindel\"of Theorem. Further, as the map \(q\mapsto g_{12}(\kappa)\) preserves the Schwartz class, it is clear that if \(q(0)\in \BdS\) then the corresponding solution remains Schwartz. Finally, to extend the solution globally in time we first observe that for \(q(0)\in \BdS\) we may apply Lemma~\ref{lem:PoissonBrackets} to deduce the conservation of $\alpha(\vk)$ for all \(\vk\geq 1\). Applying Proposition~\ref{prop:alpha} we may then extend the solution globally in time for \(q(0)\in \BdS\) and then for all \(q(0)\in \Bd\) by approximation.
\end{proof}

From Lemma~\ref{lem:PoissonBrackets} we see that the full and regularized Hamiltonian evolutions commute (at least on Schwartz space).  This allows us to obtain evolution equations for the difference Hamiltonians by simply combining the corresponding vector fields.  In this way, Proposition~\ref{prop:kappa-flows} together with Corollaries~\ref{C:A flow} and \ref{H Dynamics for rho} yields the following:

\begin{cor}[Difference flows]
Consider any \(\kappa,\vk\geq 1\) and any initial data in $\BdS$. Under the NLS difference flow,
\begin{gather} \label{NLS-diff}\tag{NLS-diff}
i\frac d{dt}q = - q'' + 2q^2r - 4\kappa^3\left(g_{12}(\kappa)- g_{12}(-\kappa)\right) - 4\kappa^2q, \\
\label{g12-NLS-diff}
\left\{\qquad \begin{aligned}
i\frac d{dt}g_{12}(\vk) &= - g_{12}(\vk)'' + 4qrg_{12}(\vk) + 2q^2g_{21}(\vk)\\
&\quad + \frac{2\kappa^3}{\kappa - \vk} \bigl\{ [\vr(\kappa)+1] g_{12}(\vk) - g_{12}(\kappa)[\vr(\vk)+1] \bigr\}\\
&\quad + \frac{2\kappa^3}{\kappa + \vk} \bigl\{ [\vr(-\kappa)+1] g_{12}(\vk) - g_{12}(-\kappa)[\vr(\vk)+1] \bigr\}\\
&\quad - 4\kappa^2g_{12}(\vk),
\end{aligned}\right.
\end{gather}
and under the mKdV difference flow
\begin{gather}\label{mKdV-diff}\tag{mKdV-diff}
\frac d{dt}q = - q''' + 6qrq' - 8\kappa^4\left(g_{12}(\kappa) + g_{12}(-\kappa)\right) - 4\kappa^2q', \\
\label{g12-mKdV-diff}
\left\{\qquad \begin{aligned}
\frac d{dt}g_{12}(\vk) &=- g_{12}(\vk)''' + 6qrg_{12}(\vk)' + 6qq'g_{21}(\vk) + 6rq'g_{12}(\vk)\\
&\quad + \frac{4\kappa^4}{\kappa - \vk} \bigl\{ [\vr(\kappa)+1] g_{12}(\vk) - g_{12}(\kappa)[\vr(\vk)+1] \bigr\}\\
&\quad - \frac{4\kappa^4}{\kappa + \vk} \bigl\{ [\vr(-\kappa)+1] g_{12}(\vk) - g_{12}(-\kappa)[\vr(\vk)+1] \bigr\}\\
&\quad - 4\kappa^2g_{12}(\vk)'.
\end{aligned}\right.
\end{gather}
\end{cor}

We end this section with the following result, which encapsulates the microscopic conservation law attendant to $A(\vk)$ under the various flows considered in this paper.

\begin{cor}\label{C:microscopic}
For \(\kappa,\vk\geq 1\) and initial data in $\BdS$, we have
\eq{MicroscopicA}{
\p_t\rho(\vk) + \p_x j_\star = 0,
}
for each of the NLS, mKdV, and difference flows; the currents are given by
\begin{align*}
&j_{\NLS}(\vk) = - i\left(\tfrac{ q'\cdot g_{21}(\vk) + r'\cdot g_{12}(\vk)}{2 + \vr(\vk)} - qr + 2\vk \rho(\vk)\right),\\
&j_{\mKdV}(\vk) = \tfrac{ (q'' - 2q^2r)\cdot g_{21}(\vk) - (r'' - 2r^2q)\cdot g_{12}(\vk)}{2 + \vr(\vk)} - q'r + qr' + 2i\vk j_{\NLS}(\vk),\\
&j_\NLS^\diff(\vk,\kappa) = j_{\NLS}(\vk) + 4\kappa^3\left(j(\vk,\kappa) - j(\vk,-\kappa)\right),\\
&j_\mKdV^\diff(\vk,\kappa) = j_{\mKdV}(\vk) + 8i\kappa^4\left(j(\vk,\kappa) + j(\vk,-\kappa)\right) + 4\kappa^2\rho(\vk).
\end{align*}
\end{cor}


\section{Local smoothing}\label{S:6}

The goal of this section is to prove local smoothing estimates, not only for the NLS and mKdV flows, but also for the difference flows.   To do this, we will be using an integrated form of the microscopic conservation law \eqref{MicroscopicA} for \(A(\vk)\):
\begin{equation}\label{LSD}
\int_{-1}^1\int_\R j_\star(t,x;\vk)\,\psi_h^{12}(x)\,dx\,dt = \int \big[\rho(1,x;\vk) - \rho(-1,x;\vk)\big]\,\Psi_h(x)\,dx,
\end{equation}
where $h\in\R$ is a translation parameter, $\psi_h$ is as in \eqref{psi}, 
\begin{equation}\label{phi-def}
\Psi_h(x) := \int_{-\infty}^x\psi_h^{12}(y)\,dy,
\end{equation}
and the currents are as recorded in Corollary~\ref{C:microscopic}.

Eventually, we will take a supremum over $h\in\R$ as in \eqref{E:D:LS}.  With this in mind, implicit constants in this section are always to be interpreted as independent of $h$.

Control of the local smoothing norm will originate in the coercivity of the LHS\eqref{LSD} that we have already hinted at in Lemma~\ref{lem:Currents are coercive}.  The first result in this section, Lemma~\ref{lem:Currents-quad}, shows that this coercivity of the quadratic currents survives in the presence of localization.  As noted already in Section~\ref{S:4}, we will need to take the real or imaginary part of \eqref{LSD}, depending on the flow in question. 

To continue, we will show how to control the RHS\eqref{LSD}; see Lemma~\ref{lem:rhA}.  This leaves us to control the higher order terms in the currents; this is the topic of Lemma~\ref{lem:Currents-err}.  In estimating such terms, it is convenient to combine the key norms:
\begin{align*}
\normNK{q}^2 := \|q\|_{X_\kappa^{s+\frac12}}^2 + \|q\|_{L^\infty_t H^s}^2
	\qtq{and} \normMK{q}^2 := \|q\|_{X_\kappa^{s+1}}^2 + \|q\|_{L^\infty_t H^s}^2,
\end{align*}
with the convention that a missing subscript means $\kappa=1$.

The proofs of Lemmas~\ref{lem:Currents-quad} and~\ref{lem:Currents-err} are both quite substantial.  With this in mind, we delay presenting these proofs until after giving the main results of this section, namely, Propositions~\ref{prop:NLS-LS}, \ref{prop:mKdV-LS}, \ref{P:ls NLS diff}, and~\ref{P:ls mKdV diff}.

\begin{lem}[Estimates for \(j_\star\sbrack{2}\)]\label{lem:Currents-quad}
Fix $\delta>0$ sufficiently small.  Then
\begin{align}
\Im \int j_{\NLS}\sbrack{2}(\vk)\,\psi_h^{12}\,dx &= \pm 2\|(\psi_h^{6}q)'\|_{H^{-1}_\vk}^2  + \bigO\Big(\|q\|_{H^{-\frac12}_\vk}^2\Big),\label{jNLS-quad}\\
\Re\int j_{\mKdV}\sbrack{2}(\vk)\,\psi_h^{12}\,dx &= \mp 6\vk\|(\psi_h^{6} q)'\|_{H^{-1}_\vk}^2\label{jmKdV-quad}\\
&\quad + \bigO\Big(\|(\psi_h^{6} q)'\|_{H^{-\frac12}_\vk}\|q\|_{H^{-\frac12}_\vk} + \|q\|_{H^{-\frac12}_\vk}^2\Big),\notag
\end{align}
uniformly for $q\in\BdS$, \(\vk\geq 1\), and \(h\in \R\).  Analogously,
\begin{align}
\Im \int j_\NLS^\diff{}\sbrack{2}(\vk,\kappa)\,\psi_h^{12}\,dx &= \pm 2\int \frac{\xi^4(8\kappa^2 + \xi^2)|\widehat{\psi_h^6 q}(\xi)|^2\,d\xi}{(4\vk^2 + \xi^2)(4\kappa^2 + \xi^2)^2} + \bigO\Big(\|q\|_{H^{-\frac12}_\vk}^2\Big),\!\!\!\label{jNLSdiff-quad}\\
\Re\int j_\mKdV^{\diff\;[2]}(\vk,\kappa)\,\psi_h^{12}\,dx &= \mp 2\vk\int \frac{(20\kappa^2 + 3\xi^2)\xi^4|\widehat{\psi_h^{6}q}(\xi)|^2\,d\xi}{(4\vk^2 + \xi^2)(4\kappa^2 + \xi^2)^2} \label{jmKdVdiff-quad}\\
&\quad  + \bigO\Big(\|\tfrac{(\psi_h^{6}q)''}{\sqrt{4\kappa^2 - \p^2}}\|_{H^{-\frac12}_\vk}\|q\|_{H^{-\frac12}_\vk} + \|q\|_{H^{-\frac12}_\vk}^2\Big),\notag
\end{align}
uniformly for $q\in\BdS$, \(\kappa,\vk\geq 1\), and \(h\in \R\).
\end{lem}

\begin{lem}[Estimate for \(\rho\)]\label{lem:rhA}
Let \(q\in \BdS\) and \(\Psi_h\) be defined as in \eqref{phi-def}. Then for \(\vk\geq 1\) we have the estimate
\begin{equation}\label{rhoA}
\left|\int \rho(\vk)\,\Psi_h\,dx\right| \lesssim \|q\|_{H^{-\frac12}_\vk}^2  + \vk^{-2(2s+1)}\delta^2\|q\|_{H^s}^2,
\end{equation}
where the implicit constant is independent of \(h,\vk\).
\end{lem}

\begin{proof}
As in the proof of Proposition~\ref{prop:alpha}, we write $\rho(\vk) = \rho\sbrack{2} (\vk)+ \rho\sbrack{\geq 4}(\vk)$.  From \eqref{rho quadratic}, we bound 
\[
\left|\int \rho\sbrack{2}(\vk)\,\Psi_h\,dx\right| \lesssim \|\Psi_h q\|_{H^{-\frac12}_\vk}\|q\|_{H^{-\frac12}_\vk}\lesssim \|q\|_{H^{-\frac12}_\vk}^2.
\]
Using \eqref{rho higher} and \eqref{rat in -s}, we may bound
\[
\left|\int \rho\sbrack{\geq 4}(\vk)\,\Psi_h\,dx\right| \lesssim \vk^{-2(2s+1)}\delta^2\|q\|_{H^s}^2.
\]
This completes the proof of the lemma.
\end{proof}

To control the contribution of the remaining part \(j_\star\sbrack{\geq 4}\) of the current, we use the following lemma.  The proof of this result will take up the majority of this section.

\begin{lem}[Estimates for \(j_\star\sbrack{\geq 4}\)]\label{lem:Currents-err}
Let \(q\in \Cont([-1,1];\BdS)\) with $\delta>0$ sufficiently small.  For any \(\vk\geq 1\), we have
\begin{align}
\left\|\int j_{\NLS}\sbrack{\geq 4}(\vk)\,\psi_h^{12}\,dx\right\|_{L^1_t} &\lesssim \vk^{-2(2s+1)}\delta^2\normN{q}^2,\label{jNLS-err}\\
\left\|\int j_{\mKdV}\sbrack{\geq 4}(\vk)\,\psi_h^{12}\,dx\right\|_{L^1_t}& \lesssim \big[\vk^{-1}+\vk^{-2(2s+1)}\log^4|2\vk|\big]\delta^2 \normM{q}^2 . \label{jmKdV-err}
\end{align}
Moreover, if \(\kappa\geq 8\) and \(\vk\in  [\kappa^{\frac23},\frac12\kappa]\cup[2\kappa,\infty)\), then
\eq{jNLSdiff-err}{
\begin{aligned}
\left\|\int \! j_\NLS^\diff{}\sbrack{\geq 4}(\vk,\kappa)\,\psi_h^{12}\,dx\right\|_{L^1_t} &\lesssim \bigl[\tfrac\kappa{\kappa + \vk}\kappa^{-\frac43(2s+1)}
	+ \vk^{-2(2s+1)}\bigr]\delta^2 \normNK{q}^2,
\end{aligned}
}
whereas for \(\vk\in [\kappa^{\frac12},\frac12 \kappa]\cup [2\kappa,\infty)\), we have
\eq{jmKdVdiff-err}{
\begin{aligned}
\left\|\int \! j_\mKdV^{\diff\;[\geq 4]}(\vk,\kappa)\,\psi_h^{12}\,dx\right\|_{L^1_t} &\lesssim \bigl[\tfrac\kappa{\kappa + \vk}\kappa^{-(2s+1)}
	+ \vk^{-2(2s+1)}\log|2\vk|\bigr] \delta^2 \normMK{q}^2.
\end{aligned}
}
In all cases, the implicit constant is independent of $h$, $\vk$, and $\kappa$.
\end{lem}

The restriction $\kappa\geq 8$ (rather than $\kappa\geq 1$) appearing in this proposition is imposed to avoid confusion in the meaning of the constraints on $\vk$.  It guarantees that in both cases, the first interval is nonempty. 

The fact that the $\kappa=1$ case of \eqref{jmKdVdiff-err} yields a better bound than \eqref{jmKdV-err} warrants explanation.  Ultimately, this is because LHS\eqref{jmKdVdiff-err} requires a much more detailed analysis in order to achieve a satisfactory bound.  The bound \eqref{jmKdV-err} could be improved by a parallel analysis; however, this is not needed for what follows.

With these estimates in hand we are now able to prove our local smoothing estimates:

\begin{prop}[Local smoothing for the NLS]\label{prop:NLS-LS}
There exists \(\delta>0\) so that for any \(q(0)\in \BdS\), the solution \(q(t)\) of \eqref{NLS} satisfies the estimate
\begin{equation}\label{NLS-LS}
\|q\|_{X^{s+\frac12}}^2 \lesssim \|q(0)\|_{H^s}^2.
\end{equation}

Further, we have the high frequency estimate
\begin{equation}\label{NLS-LS-HF}
\bigl\|(\psi_h^6q)'\|_{L_t^2 H_\kappa^{s-\frac12}}^2\lesssim \|q(0)\|_{H^s_\kappa}^2 + \kappa^{-(2s+1)}\delta^2 \|q(0)\|_{H^s}^2,
\end{equation}
uniformly for \(\kappa\geq 1\).
\end{prop}

\begin{proof}
Consider the imaginary part of \eqref{LSD}.  Applying the estimates \eqref{jNLS-quad} and \eqref{jNLS-err} to the LHS and the estimate \eqref{rhoA} to the RHS, we obtain
\begin{align*}
\|(\psi_h^6 q)'\|_{L^2_t H^{-1}_\vk}^2 &\lesssim  \|q\|_{L^\infty_t H^{-\frac12}_\vk}^2  +\vk^{-2(2s+1)}\delta^2\left(\|q\|_{X^{s+\frac12}}^2 +  \|q\|_{L^\infty_t H^s}^2\right),
\end{align*}
where the implicit constant is independent of \(h,\vk\). We then choose \(-\frac12<s'<s\) and apply the a priori estimate \eqref{APBound} to obtain
\begin{align*}
\|(\psi_h^6 q)'\|_{L^2_t H^{-1}_\vk}^2 &\lesssim \vk^{-(2s'+1)} \|q(0)\|_{H^{s'}_\vk}^2+ \vk^{-2(2s+1)}\delta^2\left(\|q\|_{X^{s+\frac12}}^2 + \|q(0)\|_{H^s}^2\right).
\end{align*}

Taking \(\kappa\geq 1\) and using \eqref{EquivNorm}, we obtain
\begin{equation}\label{NLS-LS-HFID}
\begin{aligned}
\|(\psi_h^6 q)'\|_{L^2_t H^{s-\frac12}_\kappa}^2 &\approx\int_\kappa^\infty \vk^{2s+1}\|(\psi_h^6 q)'\|_{L^2_t H^{-1}_\vk}^2\,\tfrac{d\vk}\vk\\
&\lesssim \|q(0)\|_{H^s_\kappa}^2 + \kappa^{-(2s+1)}\delta^2\left(\|q\|_{X^{s+\frac12}}^2 + \|q(0)\|_{H^s}^2\right).
\end{aligned}
\end{equation}

To complete the proof, we take \(\kappa = 1\) to deduce
\begin{align*}
\|\psi_h^6q\|_{L^2_t H^{s+\frac12}}^2 &\lesssim \|P_{>1}(\psi_h^6 q)'\|_{L^2_t H^{s-\frac12}}^2 + \|P_{\leq 1}(\psi_h^6 q)\|_{L^\infty_t H^s}^2 \\
&\lesssim \delta^2\|q\|_{X^{s+\frac12}}^2 + \|q(0)\|_{H^s}^2.
\end{align*}
Taking the supremum over \(h\in \R\) and choosing \(0<\delta\ll1\) sufficiently small, we obtain the estimate \eqref{NLS-LS}. The claim \eqref{NLS-LS-HF} then follows from \eqref{NLS-LS} and \eqref{NLS-LS-HFID}.
\end{proof}

An essentially identical argument yields the corresponding result for the mKdV:

\begin{prop}[Local smoothing for the mKdV]\label{prop:mKdV-LS}
There exists \(\delta>0\) so that for any \(q(0)\in\BdS\), the solution \(q(t)\) of \eqref{mKdV} satisfies the estimate
\eq{mKdV-LS}{
\|q\|_{X^{s+1}}^2\lesssim \|q(0)\|_{H^s}^2.
}

Further, we have the high frequency estimate
\eq{mKdV-LS-HF}{
\bigl\|(\psi_h^6q)'\|_{L^2_t H_\kappa^{s}}^2\lesssim \|q(0)\|_{H^s_\kappa}^2 +  \kappa^{-(2s+1)}\log^4\!|2\kappa| \, \delta^2 \|q(0)\|_{H^s}^2,
}
uniformly for \(\kappa\geq 1\).
\end{prop}

\begin{proof}
Consider the real part of \eqref{LSD}.  Applying \eqref{jmKdV-quad} and \eqref{jmKdV-err} to the LHS and applying \eqref{rhoA} to the RHS, we deduce that
\begin{align*}
\vk\|(\psi_h^6 q)'\|_{L^2_t H^{-1}_\vk}^2 &\lesssim \varepsilon\|(\psi_h^6 q)'\|^2_{L^2_t H^{-\frac12}_\vk} + (1 + \tfrac1\varepsilon)\|q\|_{L^\infty_t H^{-\frac12}_\vk}^2\\
& \quad + \big[\vk^{-1}+\vk^{-2(2s+1)}\log^4\!|2\vk|\big]\delta^2 \normM{q}^2
\end{align*}
for any $0<\varepsilon<1$.  Here the implicit constant is independent of \(h,\vk, \varepsilon\). Applying the a priori estimate \eqref{APBound}, for any \(-\frac12<s'<s\) we obtain
\begin{align*}
\vk\|(\psi_h^6 q)'\|_{L^2_t H^{-1}_\vk}^2 &\lesssim \varepsilon\|(\psi_h^6 q)'\|_{L^2_t H^{-\frac12}_\vk}^2+ (1 + \tfrac1\varepsilon)\vk^{-(2s'+1)}\|q(0)\|_{H^{s'}_\vk}^2\\
& \quad + \big[\vk^{-1}+\vk^{-2(2s+1)}\log^4\!|2\vk|\big]\delta^2\Big(\|q\|_{X^{s+1}}^2 + \|q(0)\|_{H^s}^2\Big).
\end{align*}

Using the estimate \eqref{EquivNorm}, we obtain
\begin{align*}
\|(\psi_h^6 q)'\|_{L^2_t H^s_\kappa}^2&\approx \int_{\kappa}^\infty\vk^{2s+2}\|(\psi_h^6 q)'\|_{H^{-1}_\vk}^2\,\tfrac{d\vk} \vk \\
&\lesssim \varepsilon \|(\psi_h q)'\|_{L^2_t H^s_\kappa}^2 + (1 + \tfrac1\varepsilon)\|q(0)\|_{H^s_\kappa}^2\\
& \quad + \big[\kappa^{2s}+\kappa^{-(2s+1)}\log^4\!|2\kappa|\big]\delta^2\Big(\|q\|_{X^{s+1}}^2 + \|q(0)\|_{H^s}^2\Big).
\end{align*}
Choosing \(0<\varepsilon<1\) sufficiently small to defeat the implicit constant, we get
\eq{mKdV-LS-HFID}{
\begin{aligned}
\|(\psi_h^6 q)'\|_{L^2_t H^s_\kappa}^2 &\lesssim \|q(0)\|_{H^s_\kappa}^2+ \big[\kappa^{2s}+\kappa^{-(2s+1)}\log^4\!|2\kappa|\big]\delta^2\Big(\|q\|_{X^{s+1}}^2 + \|q(0)\|_{H^s}^2\Big).
\end{aligned}
}

To complete the proof, we apply the estimate \eqref{mKdV-LS-HFID} with \(\kappa = 1\) to bound
\begin{align*}
\|\psi_h^6 q\|_{L^2_t H^{s+1}}^2 &\lesssim \|P_{>1}(\psi_h^6 q)'\|_{L^2_t H^s}^2 + \|P_{\leq 1}(\psi_h^6 q)\|_{L^\infty_t H^s}^2 \\
&\lesssim \delta^2\|q\|_{X^{s+1}}^2 + \|q(0)\|_{H^s}^2.
\end{align*}
Taking the supremum over \(h\in \R\) and choosing \(0<\delta\ll1\) sufficiently small, we obtain \eqref{mKdV-LS}.  The estimate \eqref{mKdV-LS-HF} then follows from \eqref{mKdV-LS}, \eqref{mKdV-LS-HFID}, and the observation that $\kappa^{2s}\|q(0)\|_{H^s}^2\lesssim \|q(0)\|_{H^s_\kappa}^2$.
\end{proof}

In Propositions~\ref{prop:NLS-LS} and~\ref{prop:mKdV-LS}, the parameter \(\kappa\) plays the role of a frequency threshold. The fact that we obtain decay as \(\kappa\to\infty\) will be essential both for proving tightness and for proving that the data-to-solution map is continuous in the local smoothing norm.

We now turn to proving local smoothing for the difference flows. In this context, \(\kappa\) takes on a new meaning as the parameter appearing in the regularized Hamiltonians; see~\eqref{D:Hdiff}. In this role, \(\kappa\) marks a border (in frequency space): it is only for frequencies below \(\kappa\) that the regularized and full Hamiltonian flows well-approximate one another. Correspondingly, it is only for frequencies above \(\kappa\) that we can expect to recover the full local smoothing effects documented above for \eqref{NLS} and \eqref{mKdV}.

\begin{prop}[Local smoothing for the NLS difference flow]\label{P:ls NLS diff} There exists \(\delta>0\) so that for any \(q(0)\in \BdS\) and \(\kappa\geq 8\), the solution \(q(t)\) of the NLS difference flow \eqref{NLS-diff} with parameter \(\kappa\) satisfies the estimate
\eq{NLS-diff-LS}{
\|q\|_{X^{s+\frac12}_\kappa}^2\lesssim \|q(0)\|_{H^s}^2,
}
where the implicit constant is independent of \(\kappa\).
\end{prop}

\begin{proof}
Let us write $I=[\kappa^{\frac23},\frac12\kappa]\cup[2\kappa,\infty)$, which is the region of $\vk$ over which the estimate \eqref{jNLSdiff-err} will be proved.

Taking the imaginary part of \eqref{LSD} and applying \eqref{jNLSdiff-quad}, \eqref{jNLSdiff-err}, and \eqref{rhoA}, we find
\begin{align*}
\|\tfrac{(\psi_h^6 q)''}{\sqrt{4\kappa^2 - \p^2}}\|_{L^2_t H^{-1}_\vk}^2
&\lesssim \|q\|_{L^\infty_t H^{-\frac12}_\vk}^2 + \bigl[\tfrac\kappa{\kappa + \vk}\kappa^{-\frac43(2s+1)} + \vk^{-2(2s+1)}\bigr]\delta^2\normNK{q}^2,
\end{align*}
uniformly for $\vk\in I$. Choosing \(-\frac12<s'<s\) and employing the a priori estimate \eqref{APBound}, we deduce that 
\begin{align*}
\|\tfrac{(\psi_h^6 q)''}{\sqrt{4\kappa^2 - \p^2}}\|_{L^2_t H^{-1}_\vk}^2
	&\lesssim \vk^{-(2s'+ 1)}\|q(0)\|_{H^{s'}_\vk}^2 + \bigl[\tfrac\kappa{\kappa + \vk}\kappa^{-\frac43(2s+1)} + \vk^{-2(2s+1)}\bigr]\delta^2\normNK{q}^2,
\end{align*}
uniformly for $\vk\in I$. Next we wish to integrate out $\vk$.

By Lemma~\ref{lem:EquivNorm}, we have
\begin{align*}
\|\tfrac{(\psi_h^6 q)''}{\sqrt{4\kappa^2 - \p^2}}\|_{L^2_t H^{s-\frac12}_{\kappa^{2/3}}}^2
	\approx\int_{\kappa^{\frac23}}^\infty \vk^{2s+1}\|\tfrac{(\psi_h^6 q)''}{\sqrt{4\kappa^2 - \p^2}}\|_{L^2_t H^{-1}_\vk}^2\tfrac{d\vk}\vk,
\end{align*}
from which it follows that
\begin{align}\label{a hole is okay}
\|\tfrac{(\psi_h^6 q)''}{\sqrt{4\kappa^2 - \p^2}}\|_{L^2_t H^{s-\frac12}_{\kappa^{2/3}}}^2
	&\approx\int_I \vk^{2s+1}\|\tfrac{(\psi_h^6 q)''}{\sqrt{4\kappa^2 - \p^2}}\|_{L^2_t H^{-1}_\vk}^2\tfrac{d\vk}\vk ,
\end{align}
because the integrand on the interval $[\kappa/2,2\kappa]$ is comparable to that on $[2\kappa,4\kappa]$. 

Proceeding in this way, we find that 
\begin{align*}
\|\tfrac{(\psi_h^6 q)''}{\sqrt{4\kappa^2 - \p^2}}\|_{L^2_t H^{s-\frac12}_{\kappa^{2/3}}}^2  \lesssim \|q(0)\|_{H^s}^2 + \kappa^{-\frac13(1+2s)}\delta^2\Big(\|q\|_{X^{s+\frac12}_\kappa}^2 + \|q(0)\|_{H^s}^2\Big).
\end{align*}

To complete the proof, we decompose
\begin{align*}
\|\tfrac{\psi_h^6 q}{\sqrt{4\kappa^2 - \p^2}}\|_{L^2_t H^{s+\frac32}}^2 &\lesssim \|\tfrac1{\sqrt{4\kappa^2 - \p^2}}P_{\leq\kappa^{\frac23}}(\psi_h^6 q)\|_{L^2_t H^{s+\frac32}}^2 + \|\tfrac1{\sqrt{4\kappa^2 - \p^2}}P_{>\kappa^{\frac23}}(\psi_h^6 q)\|_{L^2_t H^{s+\frac32}}^2\\
&\lesssim \|q(0)\|_{H^s}^2 + \|\tfrac{(\psi_h^6 q)''}{\sqrt{4\kappa^2 - \p^2}}\|_{L^2_t H^{s-\frac12}_{\kappa^{2/3}}}^2\\
&\lesssim \|q(0)\|_{H^s}^2 + \delta^2\|q\|_{X^{s+\frac12}_\kappa}^2.
\end{align*}
Taking the supremum over \(h\in\R\), we obtain the estimate \eqref{NLS-diff-LS} whenever \(0<\delta\ll1\) is sufficiently small, depending only on \(s\).
\end{proof}

Next, we record a corollary of Proposition~\ref{P:ls NLS diff}, which will be used in Section~\ref{S:convg}.

\begin{cor}\label{C:ls} There exists \(\delta>0\) so that for any \(q(0)\in \BdS\) and \(\kappa\geq 8\), the solution \(q(t)\) of the NLS difference flow \eqref{NLS-diff} with parameter \(\kappa\) satisfies
\begin{equation} \label{ls1}
\begin{aligned}
\sup_{h\in \R} \|P_{N} (\psi_h^6 q)\|_{L_{t,x}^2}
\lesssim  \|q(0)\|_{H^s}
\begin{cases}
N^{-s} \quad &\text{if \, $N\leq \kappa^{\frac23}$,}\\
\kappa N^{-(\frac32+s)} \quad &\text{if \, $\kappa^{\frac23}\leq N\leq \kappa$},\\
N^{-(\frac12+s)} \quad &\text{if \, $N\geq \kappa$},\\
\end{cases}
\end{aligned}
\end{equation}
uniformly for $N\geq 1$ and  \(\kappa\geq 1\).  Consequently, 
\begin{equation}\label{ls2}
\begin{aligned}
\sup_{h\in \R} \bigl\| \psi_h^6\tfrac{\p^\ell q}{4\kappa^2-\p^2}\bigr\|_{L_{t,x}^2} +\sup_{h\in \R} \bigl\| \tfrac{\p^\ell (\psi_h^6q)}{4\kappa^2-\p^2}\bigr\|_{L_{t,x}^2}\lesssim \|q(0)\|_{H^s}
\begin{cases}
\kappa^{-2+ \frac23(l-s)} \quad &\text{if \,  $\ell=0,1$,}\\
\kappa^{-(\frac12+s)} \quad &\text{if \, $\ell=2$},
\end{cases}
\end{aligned}
\end{equation}
uniformly for \(\kappa\geq 1\).
\end{cor}
\begin{proof}
The claim \eqref{ls1} follows immediately from \eqref{NLS-diff-LS} and Bernstein inequalities.  To obtain \eqref{ls2}, we decompose into Littlewood--Paley pieces, use \eqref{ls1} and Lemma~\ref{L:localize}, and then sum.
\end{proof}

\begin{prop}[Local smoothing for the mKdV difference flow]\label{P:ls mKdV diff}
There exists \(\delta>0\) so that for any \(q(0)\in \BdS\) and \(\kappa\geq8\) the solution \(q(t)\) of the mKdV difference flow \eqref{mKdV-diff} with parameter \(\kappa\) satisfies
\eq{mKdV-diff-LS}{
\|q\|_{X^{s+1}_\kappa}^2\lesssim \|q(0)\|_{H^s}^2,
}
where the implicit constant is independent of \(\kappa\).
\end{prop}

\begin{proof}
Consider the real part of \eqref{LSD}.  Applying the estimates \eqref{jmKdVdiff-quad}, \eqref{jmKdVdiff-err}, and \eqref{rhoA}, we deduce that
\begin{align*}
\vk\|\tfrac{(\psi_h^6 q)''}{\sqrt{4\kappa^2 - \p^2}}\|_{L^2_t H^{-1}_\vk}^2
&\lesssim \varepsilon\|\tfrac{(\psi_h^6 q)''}{\sqrt{4\kappa^2 - \p^2}}\|^2_{L^2_t H^{-\frac12}_\vk} + (1 + \tfrac1\varepsilon)\|q\|_{L^\infty_t H^{-\frac12}_\vk}^2\\
&\quad\  + \bigl[\tfrac\kappa{\kappa + \vk}\kappa^{-(2s+1)} + \vk^{-2(2s+1)}\log|2\vk|\bigr]\delta^2\normMK{q}^2,
\end{align*}
uniformly for $0<\varepsilon<1$ and $\vk\in I:= [\kappa^{\frac12},\frac12\kappa]\cup [2\kappa,\infty)$.

Choosing \(-\frac12<s'<s\)  and applying the a priori estimate \eqref{APBound}, this becomes
\begin{align*}
\vk\|\tfrac{(\psi_h^6 q)''}{\sqrt{4\kappa^2 - \p^2}}\|_{L^2_t H^{-1}_\vk}^2
	&\lesssim \varepsilon\|\tfrac{(\psi_h^6 q)''}{\sqrt{4\kappa^2 - \p^2}}\|_{L^2_t H^{-\frac12}_\vk}^2+ \vk^{-(2s'+1)}(1 + \tfrac1\varepsilon)\|q(0)\|_{H^{s'}_\vk}^2\\
&\quad\ + \bigl[\tfrac\kappa{\kappa + \vk}\kappa^{-(2s+1)} + \vk^{-2(2s+1)}\log|2\vk|\bigr] \delta^2\normMK{q}^2.
\end{align*}

Next we wish to integrate over $\vk\in I$.  Using Lemma~\ref{lem:EquivNorm} as in the proof of Proposition~\ref{P:ls NLS diff}, we obtain the following analogues of \eqref{a hole is okay}:
\begin{align*}
\|\tfrac{(\psi_h^6 q)''}{\sqrt{4\kappa^2 - \p^2}}\|_{L^2_t H^s_{\kappa^{1/2}}}^2
	&\approx\int_I \vk^{2s+2}\|\tfrac{(\psi_h^6 q)''}{\sqrt{4\kappa^2 - \p^2}}\|_{L^2_t H^{-1}_\vk}^2\tfrac{d\vk}\vk
		\approx\int_I \vk^{2s+1}\|\tfrac{(\psi_h^6 q)''}{\sqrt{4\kappa^2 - \p^2}}\|_{L^2_t H^{-\frac12}_\vk}^2\tfrac{d\vk}\vk.
\end{align*}

Proceeding in this way, and choosing $0<\varepsilon<1$ sufficiently small, we obtain 
\[
\bigl\|\tfrac{(\psi_h^6 q)''}{\sqrt{4\kappa^2 - \p^2}}\bigr\|_{L^2_t H^s_{\kappa^{1/2}}}^2\lesssim \|q(0)\|_{H^s}^2 + \delta^2\|q\|_{X^{s+1}_\kappa}^2.
\]

To complete the proof, we decompose
\begin{align*}
\|\tfrac{\psi_h^6 q}{\sqrt{4\kappa^2 - \p^2}}\|_{L^2_t H^{s+2}}^2 &\lesssim \|\tfrac1{\sqrt{4\kappa^2 - \p^2}}P_{\leq\kappa^{\frac12}}(\psi_h^6 q)\|_{L^2_t H^{s+2}}^2 + \|\tfrac1{\sqrt{4\kappa^2 - \p^2}}P_{>\kappa^{\frac12}}(\psi_h^6 q)\|_{L^2_t H^{s+2}}^2\\
&\lesssim \|q(0)\|_{H^s}^2 + \|\tfrac{(\psi_h^6 q)''}{\sqrt{4\kappa^2 - \p^2}}\|_{L^2_t H^{s}_{\kappa^{1/2}}}^2\\
&\lesssim \|q(0)\|_{H^s}^2 + \delta^2\|q\|_{X^{s+1}_\kappa}^2.
\end{align*}
Taking the supremum over \(h\in\R\) we obtain the estimate \eqref{NLS-diff-LS} whenever \(0<\delta\ll1\) is sufficiently small, depending only on \(s\).
\end{proof}

Proposition~\ref{P:ls mKdV diff} directly yields the following analogue of Corollary~\ref{C:ls}:

\begin{cor}\label{C:ls mKdV} There exists \(\delta>0\) so that for any \(q(0)\in \BdS\) and \(\kappa\geq 8\), the solution \(q(t)\) of the mKdV difference flow \eqref{mKdV-diff} with parameter \(\kappa\) satisfies
\begin{equation*}
\begin{aligned}
\sup_{h\in \R} \|P_{N} (\psi_h^6 q)\|_{L_{t,x}^2}
\lesssim  \|q(0)\|_{H^s}
\begin{cases}
N^{-s} \quad &\text{if \, $N\leq \kappa^{\frac12}$,}\\
\kappa N^{-(2+s)} \quad &\text{if \, $\kappa^{\frac12}\leq N\leq \kappa$},\\
N^{-(1+s)} \quad &\text{if \, $N\geq \kappa$},\\
\end{cases}
\end{aligned}
\end{equation*}
uniformly for $N\geq 1$ and  \(\kappa\geq 8\).  Consequently, 
\begin{equation*}
\begin{aligned}
\sup_{h\in \R} \bigl\| \psi_h^6\tfrac{\p^\ell q}{4\kappa^2-\p^2}\bigr\|_{L_{t,x}^2} +\sup_{h\in \R} \bigl\| \tfrac{\p^\ell (\psi_h^6q)}{4\kappa^2-\p^2}\bigr\|_{L_{t,x}^2}\lesssim \|q(0)\|_{H^s}
\begin{cases}
\kappa^{-2+ \frac12(l-s)} \quad &\text{if \,  $\ell=0,1$,}\\
\kappa^{-(1+s)} \quad &\text{if \, $\ell=2$},
\end{cases}
\end{aligned}
\end{equation*}
uniformly for \(\kappa\geq 2\).
\end{cor}

We now turn to the proof of Lemma~\ref{lem:Currents-quad}:

\begin{proof}[Proof of Lemma~\ref{lem:Currents-quad}]We introduce the paraproduct \(\bR[q,r]\) with symbol
\[
R(\xi,\eta) = \tfrac1{2(2\vk - i\xi)} + \tfrac1{2(2\vk + i\eta)}
\]
so that by \eqref{rho quadratic} we may write
\[
\rho\sbrack{2}(x;\vk)=\bR[q,r](x) = \frac1{2\pi}\int R(\xi,\eta)\hat q(\xi)\hat r(\eta) e^{ix(\xi+\eta)}\,d\xi\,d\eta.
\]
We then observe that the quadratic part of the current \(j(\vk,\kappa)\) defined in \eqref{j sub A} may be written as
\[
j\sbrack{2}(\vk,\kappa) = i\bR[\tfrac q{2\kappa - \p},\tfrac r{2\kappa + \p}].
\]
Expanding in powers of \(\kappa\), we readily obtain the expressions
\begin{align*}
j_\NLS\sbrack{2}(\vk) &= i\Big(\bR[q,r'] - \bR[q',r]\Big),\\
j_\mKdV\sbrack{2} (\vk) &= \bigl(\bR[q,r]\bigr)'' - 3 \bR[q',r'],
\end{align*}
for the NLS and mKdV flows, as well as the expressions
\begin{align*}
j_\NLS^\diff{}\sbrack{2}(\vk,\kappa) &= i\Big(\bR[q,r'] - \bR[q',r]\Big)\\
&\quad - 16\kappa^4i\Big(\bR[\tfrac q{4\kappa^2 - \p^2},\tfrac {r'}{4\kappa^2 - \p^2}] - \bR[\tfrac{q'}{4\kappa^2 - \p^2},\tfrac r{4\kappa^2 - \p^2}]\Big),\\
j_\mKdV^{\diff}{}\sbrack{2} (\vk,\kappa)&= \bigl(\bR[q,r]\bigr)'' - 3 \bR[q',r']\\
&\quad - 16\kappa^4\big(\bR[\tfrac{q}{4\kappa^2 - \p^2},\tfrac{r}{4\kappa^2 - \p^2}]\bigr)'' + 48\kappa^4 \bR[\tfrac{q'}{4\kappa^2 - \p^2},\tfrac{r'}{4\kappa^2 - \p^2}]\\
&\quad + 4\kappa^2 \bR[\tfrac{q''}{4\kappa^2 - \p^2},\tfrac{r''}{4\kappa^2 - \p^2}]
\end{align*}
for the corresponding difference flows. (Alternatively, we may use the definition of the currents from Corollary~\ref{C:microscopic} to compute the quadratic components directly.)

If we could simply replace \(q,r\) by \(\psi_h^6q,\psi_h^6r\) in these expressions, rather than integrating them against $\psi_h^{12}$, then we would obtain the leading order terms in \eqref{jNLS-quad}--\eqref{jmKdVdiff-quad}. Thus, the focal point of our analysis will be bounding the various commutator terms that arise.

\medskip\step{Proof of \eqref{jNLS-quad}.}
Using the above expression we may write
\begin{align*}
\LHS{jNLS-quad} &= \pm 2\|(\psi_h^{6}q)'\|_{H^{-1}_\vk}^2 + \Re\int \Big(\bR[q,r']\,\psi_h^{12} - \bR[\psi_h^{6}q,(\psi_h^{6} r)']\Big)\,dx\\
&\quad - \Re\int\Big(\bR[q',r]\,\psi_h^{12} - \bR[(\psi_h^{6} q)',\psi_h^{6}r]\Big)\,dx.
\end{align*}
By symmetry it suffices to bound
\begin{align*}
&\int \Big(\bR[q,r']\,\psi_h^{12} - \bR[\psi_h^{6}q,(\psi_h^{6}r)']\Big)\,dx\\
&\qquad = \int  [\psi_h^{6},\tfrac1{2(2\vk - \p)}] q\cdot \psi_h^{6} r'\,dx - \int \tfrac1{2(2\vk - \p)}(\psi_h^{6}q)\cdot (\psi_h^{6})'r\,dx\\
&\qquad\quad + \int \psi_h^{6} q\cdot [\psi_h^{6},\tfrac\p{2(2\vk + \p)}]r\,dx,
\end{align*}
which may be bounded by
\begin{align*}
&\left|\int \Big(\bR[q,r']\,\psi_h^{12} - \bR[\psi_h^{6}q,(\psi_h^{6}r)']\Big)\,dx\right|\\
&\qquad \leq \|[\psi_h^{6},\tfrac1{2(2\vk - \p)}] q\|_{H^{\frac32}_\vk}\|\psi_h^{6} r'\|_{H^{-\frac32}_\vk} + \|\tfrac1{2(2\vk - \p)}(\psi_h^{6} q)\|_{H^{\frac12}_\vk}\|(\psi_h^{6})'r\|_{H^{-\frac12}_\vk}\\
&\qquad\quad + \|\psi_h^{6} q\|_{H^{-\frac12}_\vk}\|[\psi_h^{6},\tfrac\p{2(2\vk + \p)}]r\|_{H^{\frac12}_\vk}\\
&\qquad\lesssim \|q\|_{H^{-\frac12}_\vk}^2,
\end{align*}
as required.

\medskip\step{Proof of \eqref{jNLSdiff-quad}.}
We observe that the difference
\(
j_{\NLS}\sbrack{2} - j_\NLS^\diff{}\sbrack{2}
\)
has an identical expression to \(j_{\NLS}\sbrack{2}\) with \(q\) replaced by \(\tfrac{4\kappa^2 q}{4\kappa^2 - \p^2}\). The estimate \eqref{jNLSdiff-quad} then follows from the estimate \eqref{jNLS-quad} and the estimates
\[
\|\tfrac{4\kappa^2}{4\kappa^2 - \p^2}q\|_{H^{-\frac12}_\vk}\lesssim \|q\|_{H^{-\frac12}_\vk} \qtq{and} \|\p[\psi_h^{6},\tfrac{4\kappa^2}{4\kappa^2 - \p^2}]q\|_{H^{-\frac12}_\vk}\lesssim \|q\|_{H^{-\frac12}_\vk}.
\]

\medskip\step{Proof of \eqref{jmKdV-quad}.}
Integrating by parts we find
\begin{align*}
\LHS{jmKdV-quad} &= \mp 6\vk\|(\psi_h^{6} q)'\|_{H^{-1}_\vk}^2 + \Re\int \bR[q,r]\,(\psi_h^{12})''\,dx\\
&\quad  - 3\Re\int \Big(\bR[q',r']\,\psi_h^{12}- \bR[(\psi_h^{6}q)',(\psi_h^{6}r)']\Big)\,dx.
\end{align*}
As in the proof of Lemma~\ref{lem:rhA}, the second term may be readily bounded by
\[
\left|\int \bR[q,r]\,(\psi_h^{12})''\,dx\right|\lesssim \|q\|_{H^{-\frac12}_\vk}^2.
\]
For the remaining term we write
\begin{align*}
&\int \Big(\bR[q',r']\,\psi_h^{12} - \bR[(\psi_h^{6}q)',(\psi_h^{6}r)']\Big)\,dx\\
&\qquad = \int [\psi_h^{6},\tfrac\p{2(2\vk - \p)}]q\cdot (\psi_h^{6} r)'\,dx -  \int \tfrac\p{2(2\vk - \p)}(\psi_h^{6}q)\cdot (\psi_h^{6})' r\,dx \\
&\qquad \quad - \int [\psi_h^{6},\tfrac\p{2(2\vk - \p)}]q\cdot (\psi_h^{6})' r\,dx + \int (\psi_h^{6} q)'\cdot [\psi_h^{6},\tfrac\p{2(2\vk + \p)}]r\,dx\\
&\qquad \quad -  \int (\psi_h^{6})' q\cdot \tfrac\p{2(2\vk + \p)}(\psi_h^{6} r)\,dx - \int (\psi_h^{6})' q\cdot[\psi_h^{6},\tfrac\p{2(2\vk + \p)}] r\,dx.
\end{align*}
The first three summands here may be bounded in magnitude via
\begin{align*}
\|[\psi_h^{6},\tfrac\p{2(2\vk - \p)}]q\|_{H^{\frac12}_\vk}\|(\psi_h^{6} r)'\|_{H^{-\frac12}_\vk}
	+ \|\tfrac\p{2(2\vk - \p)}&(\psi_h^{6} q)\|_{H^{\frac12}_\vk}\|(\psi_h^{6})' r\|_{H^{-\frac12}_\vk} \\
	&\qquad\lesssim \|(\psi_h^{6} q)'\|_{H^{-\frac12}_\vk} \|q\|_{H^{-\frac12}_\vk}
\end{align*}
and
\begin{align*}
\|[\psi_h^{6},\tfrac\p{2(2\vk - \p)}]q\|_{H^{\frac12}_\vk}\|(\psi_h^{6})' r\|_{H^{-\frac12}_\vk} 
&\lesssim  \|q\|_{H^{-\frac12}_\vk}^2,
\end{align*}
both of which are acceptable.  The remaining three summands can then be estimated in a parallel fashion; indeed, this is tantamount to replacing $\vk$ by $-\vk$.

\medskip\step{Proof of \eqref{jmKdVdiff-quad}.}
Integrating by parts several times we obtain the identity
\begin{align*}
\LHS{jmKdVdiff-quad} &= \mp 2\vk\int \tfrac{(20\kappa^2 + 3\xi^2)\xi^4}{(4\vk^2 + \xi^2)(4\kappa^2 + \xi^2)^2}|\widehat{\psi_h^{6}q}|^2\,d\xi\\
&\quad - 3\Re\int \Big(\bR[\tfrac{q'''}{4\kappa^2 - \p^2},\tfrac{r'''}{4\kappa^2 - \p^2}]\,\psi_h^{12} - \bR[\tfrac{(\psi_h^{6} q)'''}{4\kappa^2 - \p^2},\tfrac{(\psi_h^{6}r)'''}{4\kappa^2 - \p^2}]\Big)\,dx\\
&\quad - 20\kappa^2\Re\int \Big(\bR[\tfrac{q''}{4\kappa^2 - \p^2},\tfrac{r''}{4\kappa^2 - \p^2}]\,\psi_h^{12} - \bR[\tfrac{(\psi_h^{6}q)''}{4\kappa^2 - \p^2},\tfrac{(\psi_h^{6}r)''}{4\kappa^2 - \p^2}]\Big)\,dx\\
&\quad + \Re\int \bR[\tfrac{q''}{4\kappa^2 - \p^2},\tfrac{r''}{4\kappa^2 - \p^2}]\,(\psi_h^{12})''\,dx\\
&\quad + 20\kappa^2\Re\int \bR[\tfrac{q'}{4\kappa^2 - \p^2},\tfrac{r'}{4\kappa^2 - \p^2}]\,(\psi_h^{12})''\,dx\\
&\quad - 4\kappa^2\Re\int \bR[\tfrac{q}{4\kappa^2 - \p^2},\tfrac{r}{4\kappa^2 - \p^2}]\,(\psi_h^{12})''''\,dx.
\end{align*}
The final three terms are lower order errors that may be bounded by
\begin{align*}
&\left|\int \bR[\tfrac{q''}{4\kappa^2 - \p^2},\tfrac{r''}{4\kappa^2 - \p^2}]\,(\psi_h^{12})''\,dx\right| \lesssim \|\tfrac{q''}{4\kappa^2 - \p^2}\|_{H^{-\frac12}_\vk}^2\lesssim \|q\|_{H^{-\frac12}_\vk}^2\\
&\left|\kappa^2\int \bR[\tfrac{q'}{4\kappa^2 - \p^2},\tfrac{r'}{4\kappa^2 - \p^2}]\,(\psi_h^{12})''\,dx\right|\lesssim \kappa^2\|\tfrac{q'}{4\kappa^2 - \p^2}\|_{H^{-\frac12}_\vk}^2 \lesssim \|q\|_{H^{-\frac12}_\vk}^2,\\
&\left|\kappa^2\int \bR[\tfrac{q}{4\kappa^2 - \p^2},\tfrac{r}{4\kappa^2 - \p^2}]\,(\psi_h^{12})''''\,dx\right|\lesssim \kappa^2\|\tfrac q{4\kappa^2 - \p^2}\|_{H^{-\frac12}_\vk}^2\lesssim \kappa^{-2}\|q\|_{H^{-\frac12}_\vk}^2.
\end{align*}
For the first commutator term we estimate
\begin{align*}
&\left|\int \Big(\bR[\tfrac{q'''}{4\kappa^2 - \p^2},\tfrac{r'''}{4\kappa^2 - \p^2}]\,\psi_h^{12} - \bR[\tfrac{(\psi_h^{6} q)'''}{4\kappa^2 - \p^2},\tfrac{(\psi_h^{6}r)'''}{4\kappa^2 - \p^2}]\Big)\,dx\right|\\
&\qquad\lesssim \|[\psi_h^{6},\tfrac{\p^3}{(2\vk - \p)(4\kappa^2 - \p^2)}]q\|_{H^{\frac12}_\vk}\|\tfrac{\p^3}{4\kappa^2 - \p^2}(\psi_h^{6} q)\|_{H^{-\frac12}_\vk}\\
&\qquad \quad + \Bigl\{\|[\psi_h^{6},\tfrac{\p^3}{(2\vk - \p)(4\kappa^2 - \p^2)}]q\|_{H^{\frac12}_\vk}+ \|\tfrac{\p^3}{4\kappa^2 - \p^2}(\psi_h^{6} q)\|_{H^{-\frac12}_\vk}\Bigr\}\|[\psi_h^{6},\tfrac{\p^3}{4\kappa^2 - \p^2}]q\|_{H^{-\frac12}_\vk}\\
&\qquad\lesssim \|q\|_{H^{-\frac12}_\vk}\|\tfrac{(\psi_h^{6}q)''}{\sqrt{4\kappa^2 - \p^2}}\|_{H^{-\frac12}_\vk} + \|q\|_{H^{-\frac12}_\vk}^2.
\end{align*}
The second commutator term is bounded similarly:
\begin{align*}
&\left|\kappa^2\int \Big(\bR[\tfrac{q''}{4\kappa^2 - \p^2},\tfrac{r''}{4\kappa^2 - \p^2}]\,\psi_h^{12} - \bR[\tfrac{(\psi_h^{6}q)''}{4\kappa^2 - \p^2},\tfrac{(\psi_h^{6}r)''}{4\kappa^2 - \p^2}]\Big)\,dx\right|\\
&\qquad\qquad\qquad\qquad\qquad\qquad\qquad \lesssim \|q\|_{H^{-\frac12}_\vk}\|\tfrac{(\psi_h^{6}q)''}{\sqrt{4\kappa^2 - \p^2}}\|_{H^{-\frac12}_\vk} + \|q\|_{H^{-\frac12}_\vk}^2.
\end{align*}

This completes the proof of the lemma.
\end{proof}

We now turn to the proof of Lemma~\ref{lem:Currents-err}. Here we will use the estimates of Lemmas~\ref{lem:LambdaLoc NLS} and~\ref{L:LpIp mKdV} to obtain bounds for the tails of the series defining \(g_{12},g_{21},\vr\). However, these estimates are not sufficient to capture cancellations that occur for several quartic terms in the currents \(j_\NLS^\diff\) and \(j_\mKdV^\diff\). For this reason, we start by proving several quadrilinear estimates that are designed to capture the additional smallness that arises from these cancellations.

For any \(\kappa\geq 1\) and multi-index \(\beta\in \{0,1,2\}^4\) we introduce the class \(S(\beta;\kappa)\) of smooth symbols \(m\colon \R^4\rightarrow \C\) that may be written as
\begin{equation}\label{E:symbol class}
m(\xi;\kappa) = C\tfrac{\xi_1^{\beta_1}\xi_2^{\beta_2}\xi_3^{\beta_3}\xi_4^{\beta_4}}{(\kappa^2 + \xi_1^2)(\kappa^2 + \xi_2^2)(\kappa^2 + \xi_3^2)(\kappa^2 + \xi_4^2)},
\end{equation}
for a constant \(C\in \C\).  We write $m[f_1,\ldots,f_4]$ for the paraproduct with this symbol.

While it is often natural to consider paraproducts as multilinear operators, we shall only be applying them to \(q\) and to objects subordinate to \(q\), in the sense of \eqref{Norway}. Thus, it is more natural to view these paraproducts as polynomial-like functions of \(q\). When it comes to estimating these nonlinear expressions, the first step will always be to isolate the two highest frequency terms and use local smoothing to control them. (Integrability in time forbids using local smoothing for more than two factors.) Correspondingly, a multilinear point of view would lead to right-hand sides containing a sum over all permutations of the arguments. Here we see the virtue of phrasing them as nonlinear estimates and of subordinating their arguments to \(q\).

For the NLS we have the following lemma:
\begin{lem}[Quartic estimate for the NLS]\label{lem:NLS-para}
Let \(|\vk|\geq \kappa^{\frac23}\geq 1\) and the Schwartz functions \(q,f\in \Cont([-1,1];\BdS)\) satisfy
\eq{Norway}{
\normNK{f}\lesssim \normNK{q}.
}
Let \(m\in S(\beta;\kappa)\) where \(1\leq |\beta|\leq 5\) and at most one \(\beta_j = 2\). Then we have the paraproduct estimate
\eq{NLS-para}{
\bigl\| m[q,r,q,\psi_h^{12}\tfrac f{2\vk + \p}] \bigr\|_{L^1_{t,x}} \lesssim \kappa^{|\beta| - 7 -\frac43(2s+1)}\delta^2\normNK{q}^2,
}
where the implicit constant is independent of $\kappa$, $\vk$, and $h\in \R$, and $\psi$ is as in \eqref{psi}.
\end{lem}

\begin{proof}
By space-translation invariance, we may assume $h=0$. 

By Bernstein's inequality, for \(0\leq j\leq 2\) we may bound
\begin{align*}
\tfrac{N^j}{(\kappa + N)^2}\|(\psi^3q)_N\|_{L^2_{t,x}}& \lesssim \kappa^{-2} N^{j-s-\frac52} \min\{N^{\frac52},\kappa N,\kappa^2 \} \normNK{q}\\
\tfrac{N^j}{(\kappa + N)^2}\|\big(\psi^3\tfrac f{2\vk + \p}\big)_N\|_{L^2_{t,x}}&
	\lesssim |\vk|^{-(s+\frac12)}(|\vk| + N)^{s-\frac12} \kappa^{-2} N^{j-s-\frac52} \min\{N^{\frac52},\kappa N,\kappa^2 \} \normNK{q},
\end{align*}
which we will use to estimate high frequency terms. Using Bernstein's inequality again, we also find that for \(0\leq j\leq 2\),
\begin{align*}
&\sum\limits_{M\leq N}\tfrac{M^j}{(\kappa + M)^2}\|(\psi^3q)_M\|_{L^\infty_{t,x}}\lesssim \kappa^{-2}N^{\frac12-s} \min\{ N^j,\kappa^j\}\delta,\\
&\sum\limits_{M\leq N}\tfrac{M^j}{(\kappa + M)^2}\|\big(\psi^3\tfrac f{2\vk + \p}\big)_M\|_{L^\infty_{t,x}}\!\lesssim |\vk|^{-(s+\frac12)}(|\vk| + N)^{s-\frac12}\kappa^{-2}N^{\frac12-s} \min\{ N^j,\kappa^j\}\delta,
\end{align*}
which we will use to estimate low frequency terms.

For dyadic \(N_j\geq 1\) we write
\begin{align*}
m_{N_1,N_2,N_3,N_4} := m\bigl[\psi^{-3}(\psi^3q)_{N_1},\psi^{-3}(\psi^3r)_{N_2},\psi^{-3}(\psi^3q)_{N_3},\psi^9(\psi^3\tfrac f{2\vk + \p})_{N_4}\bigr],
\end{align*}
so that
\[
\bigl\|  m[q,r,q,\psi^{12}\tfrac f{2\vk + \p}] \bigr\|_{L^1_{t,x}} \leq \sum\limits_{N_j\geq 1} \bigl\| m_{N_1,N_2,N_3,N_4} \bigr\|_{L^1_{t,x}} .
\]
As the estimates will be symmetric in the first three terms, we may assume that \(N_1\geq N_2\geq N_3\). Our strategy will be to bound the two highest frequency terms in \(L^2_{t,x}\) to take advantage of the local smoothing norms, and the two lowest frequencies in \(L^\infty_{t,x}\).  Concretely, when $N_4\leq N_2$, we apply Lemma~\ref{L:localize}, to obtain
\begin{align*}
\bigl\| m_{N_1,N_2,N_3,N_4} \bigr\|_{L^1_{t,x}}&\lesssim\|\psi^3\tfrac{\p^{\beta_1}}{4\kappa^2 - \p^2}\psi^{-3}(\psi^3q)_{N_1}\|_{L^2_{t,x}}\|\psi^3\tfrac{\p^{\beta_2}}{4\kappa^2 - \p^2}\psi^{-3}(\psi^3q)_{N_2}\|_{L^2_{t,x}}\\
&\quad\times \|\psi^3\tfrac{\p^{\beta_3}}{4\kappa^2 - \p^2}\psi^{-3}(\psi^3q)_{N_3}\|_{L^\infty_{t,x}}\|\psi^{-9}\tfrac{\p^{\beta_4}}{4\kappa^2 - \p^2}\psi^9(\psi^3\tfrac f{2\vk + \p})_{N_4}\|_{L^\infty_{t,x}}\\
&\lesssim \tfrac{N_1^{\beta_1}N_2^{\beta_2}N_3^{\beta_3}N_4^{\beta_4}}{(\kappa + N_1)^2(\kappa + N_2)^2(\kappa + N_3)^2(\kappa + N_4)^2}\|(\psi^3q)_{N_1}\|_{L^2_{t,x}}\|(\psi^3q)_{N_2}\|_{L^2_{t,x}}\\
&\quad\times\|(\psi^3q)_{N_3}\|_{L^\infty_{t,x}}\|(\psi^3\tfrac f{2\vk + \p})_{N_4}\|_{L^\infty_{t,x}},
\end{align*}
whereas, when \(N_4>N_2\) we obtain instead
\begin{align*}
\bigl\| m_{N_1,N_2,N_3,N_4} \bigr\|_{L^1_{t,x}}&\lesssim\tfrac{N_1^{\beta_1}N_2^{\beta_2}N_3^{\beta_3}N_4^{\beta_4}}{(\kappa + N_1)^2(\kappa + N_2)^2(\kappa + N_3)^2(\kappa + N_4)^2}\|(\psi^3q)_{N_1}\|_{L^2_{t,x}}\|(\psi^3q)_{N_2}\|_{L^\infty_{t,x}}\\
&\qquad\qquad\times\|(\psi^3q)_{N_3}\|_{L^\infty_{t,x}}\|\big(\psi^3 \tfrac f{2\vk + \p}\big)_{N_4}\|_{L^2_{t,x}}.
\end{align*}
We then sum over the lowest two frequencies and invoke the estimates laid out above.  When $N_4\leq N_2$, this leads to a bound of the form
$$
\sum_{N_1\geq N_2} \Gamma_{N_1,N_2}  \delta^2 \normNK{q}^2 
$$
where $\Gamma$ is the matrix
\begin{align*}
\Gamma_{N_1,N_2} := |\vk|^{-(s+\frac12)}&(|\vk| + N_2)^{s-\frac12}\kappa^{-8}N_1^{\beta_1 - \frac52 - s}N_2^{\beta_2 - \frac32 - 3s}\\
&\times \min\{N_1^{\frac52},\kappa N_1,\kappa^2\}\min\{N_2^{\beta_3 + \beta_4 + \frac52},\kappa N_2^{\beta_3 + \beta_4 + 1},\kappa^{\beta_3 + \beta_4 + 2}\}.
\end{align*}
When on the other hand $N_4> N_2$, we are lead to a bound of the form
$$
\sum_{N_1\geq N_4} \Gamma_{N_1,N_4} \delta^2 \normNK{q}^2
		+  \sum_{N_4\geq N_1} \Bigl[\tfrac{|\vk| + N_4}{|\vk| + N_1}\Bigr]^{s-\frac12}\Gamma_{N_4,N_1} \delta^2 \normNK{q}^2 ,
$$
with corresponding permutations of the indices $\beta$.

In this way, we see that the proof can be completed by proving 
\begin{equation*}
\sum_{N \geq M } \Gamma_{N,M} \lesssim \kappa^{|\beta| - 7 -\frac83(s+\frac12)}.
\end{equation*}
As the matrix entries are monotone in $|\vk|$, it suffices to prove the bound when $|\vk|=\kappa^{2/3}$. Summing first in $N$, we are left to estimate
\begin{align}\label{NLS Gamma misery}
\sum_{M\geq 1} \kappa^{-8-\frac13(2s+1)} (\kappa^{\frac23} + M&)^{s-\frac12} M^{\beta_2-\frac32(2s+1)}\min\Bigl\{M^{\beta_3+\beta_4+\frac52}[\kappa^{\beta_1-\frac12-s}+ \kappa^{\frac23(\beta_1-s)}], \notag\\
&\kappa^2 M^{\beta_3+\beta_4+1} [\kappa^{\beta_1-\frac32-s}+ M^{\beta_1-\frac32-s}], \ \kappa^{\beta_3+\beta_4 +4} M^{\beta_1-\frac52-s}\Bigr\}.
\end{align}
From here, one need only consider the cases $M\leq \kappa^{\frac23}$, $\kappa^{\frac23}\leq M\leq \kappa$, and $M\geq \kappa$.
\end{proof}

For the mKdV we have the following variation:
\begin{lem}[Quartic estimates for the mKdV]\label{lem:mKdV-para}
Let \(m\in S(\beta;\kappa)\) with \(1\leq |\beta|\leq 8\).  For any \(|\vk|\geq \sqrt \kappa \geq 1\) and any Schwartz functions \(q,f\in \Cont([-1,1];\BdS)\) satisfying
\[
\normMK{f}\lesssim \normMK{q},
\]
we have the paraproduct estimate
\eq{mKdV-para-2}{
\begin{aligned}
\bigl\| m[q,r,q,\psi_h^{12} &\tfrac f{2\vk + \p}] \bigr\|_{L^1_{t,x}} + \bigl\| m[q,r,q,\tfrac f{2\vk + \p}]\,\psi_h^{12} \bigr\|_{L^1_{t,x}}\\
&\lesssim \Big[\tfrac{\kappa^{-2s}}{|\vk|}+|\vk|^{-2(2s+1)}\log\bigl|\tfrac{4\vk^2}{\kappa}\bigr|\Big]\kappa^{|\beta|-8}\delta^2\normMK{q}^2,
\end{aligned}
}
where the implicit constant is independent of $\kappa, \vk,$ and $h\in \R$.  Moreover, if $|\beta|\geq 2$,
\begin{align}\label{mKdV-para-2'}
\biggl|\iint m\bigl[q,r,q,\psi_h^{12} &\tfrac f{2\vk + \p}\bigr]\,dx\,dt\biggr|
	\lesssim \kappa^{|\beta|-9-2s}\delta^2\normMK{q}^2.
\end{align}
\end{lem}

\begin{rem}
As we will see in the proof, it is not essential that the first three entries in the paraproduct are exactly $q$, $r$, and $q$.  Rather, we only require that they obey the same estimates as $q$, in the manner that $f$ does.  As we shall seldom need this extra generality, we have chosen to present the lemma in this more representative form.
\end{rem}

\begin{proof}
The proof is essentially identical to that of Lemma~\ref{lem:NLS-para}. By space-translation symmetry, we may assume $h=0$.

We will reuse the $L^\infty_{t,x}$ bounds appearing in the proof of Lemma~\ref{lem:NLS-para}; however, the $L^2_{t,x}$ local smoothing bounds used to treat the high-frequency terms must be adapted to the mKdV setting.  Specifically, we will use 
\begin{align*}
\tfrac{N^j}{(\kappa + N)^2}\|(\psi^3q)_N\|_{L^2_{t,x}} &\lesssim \kappa^{-2} N^{j-3-s} \min\{N^3,\kappa N, \kappa^2 \} \normMK{q},\\
\tfrac{N^j}{(\kappa + N)^2}\|\big(\psi^3\tfrac f{2\vk + \p}\big)_N\|_{L^2_{t,x}} &\lesssim |\vk|^{-(s+\frac12)}(|\vk| + N)^{s - \frac12} \kappa^{-2} N^{j-3-s} \min\{N^3,\kappa N, \kappa^2 \} \normMK{q}.
\end{align*}

Proceeding as in the proof of \eqref{NLS-para} we take
\begin{align*}
m_{N_1,N_2,N_3,N_4} := m\bigl[\psi^{-3}(\psi^3q)_{N_1},\psi^{-3}(\psi^3r)_{N_2},\psi^{-3}(\psi^3q)_{N_3},\psi^{9}(\psi^3\tfrac f{2\vk + \p})_{N_4}\bigr],
\end{align*}
so that
\[
\|m[q,r,q,\psi^{12}\tfrac f{2\vk + \p}]\|_{L^1_{t,x}} \leq \sum\limits_{N_j\geq 1}\|m_{N_1,N_2,N_3,N_4}\|_{L^1_{t,x}},
\]
or
\begin{align*}
\tilde m_{N_1,N_2,N_3,N_4} := m\bigl[\psi^{-3}(\psi^3q)_{N_1},\psi^{-3}(\psi^3r)_{N_2},\psi^{-3}(\psi^3q)_{N_3},\psi^{-3}(\psi^3\tfrac f{2\vk + \p})_{N_4}\bigr],
\end{align*}
\[
\|m[q,r,q,\tfrac f{2\vk + \p}]\,\psi^{12}\|_{L^1_{t,x}} \leq \sum\limits_{N_j\geq 1}\|\tilde m_{N_1,N_2,N_3,N_4}\,\psi^{12}\|_{L^1_{t,x}}.
\]
As in the proof of \eqref{NLS-para}, it suffices to restrict our attention to the case \(N_1\geq N_2\geq N_3\geq N_4\). With $\ell=3,9$, we may bound
\begin{align*}
\text{LHS}\eqref{mKdV-para-2}&\lesssim \!\!\sum\limits_{N_1\geq N_2\geq N_3\geq N_4}\!\!\|\psi^3\tfrac{\p^{\beta_1}}{4\kappa^2 - \p^2}\psi^{-3}(\psi^3q)_{N_1}\|_{L^2_{t,x}}\|\psi^3\tfrac{\p^{\beta_2}}{4\kappa^2 - \p^2}\psi^{-3}(\psi^3q)_{N_2}\|_{L^2_{t,x}}\\
&\qquad\quad\times \|\psi^3\tfrac{\p^{\beta_3}}{4\kappa^2 - \p^2}\psi^{-3}(\psi^3q)_{N_3}\|_{L^\infty_{t,x}}\|\psi^{-\ell}\tfrac{\p^{\beta_4}}{4\kappa^2 - \p^2}\psi^\ell(\psi^3\tfrac f{2\vk + \p})_{N_4}\|_{L^\infty_{t,x}}\\
&\lesssim  \!\!\!\sum\limits_{N_1\geq N_2\geq N_3\geq N_4}\!\!\!\!\tfrac{N_1^{\beta_1}N_2^{\beta_2}N_3^{\beta_3}N_4^{\beta_4}}{(\kappa + N_1)^2(\kappa + N_2)^2(\kappa + N_3)^2(\kappa + N_4)^2}\|(\psi^3q)_{N_1}\|_{L^2_{t,x}}\|(\psi^3q)_{N_2}\|_{L^2_{t,x}}\\
&\qquad\quad\times\|(\psi^3q)_{N_3}\|_{L^\infty_{t,x}}\|(\psi^3\tfrac f{2\vk + \p})_{N_4}\|_{L^\infty_{t,x}}.
\end{align*}
Summing in \(N_3\geq N_4\geq 1\) we obtain a bound of a constant multiple of
\begin{align*}
&\sum\limits_{N_1\geq N_2}|\vk|^{-(s+\frac12)}(|\vk| + N_2)^{s-\frac12}\kappa^{-8}N_1^{\beta_1-3-s}N_2^{\beta_2 - 2 - 3s}\min\{N_1^3,\kappa N_1,\kappa^2\}\\
&\qquad \qquad\qquad\qquad\times\min\{N_2^{\beta_3 + \beta_4 + 3},\kappa N_2^{\beta_3 + \beta_4 + 1}, \kappa^{\beta_3 + \beta_4 + 2}\}\delta^2\normMK{q}^2.
\end{align*}
Proceeding as in Lemma~\ref{lem:NLS-para} and summing in $N_1$, we are led to control the following analogue of \eqref{NLS Gamma misery}:
\begin{align*}
\sum_{M\geq 1} \kappa^{-8} |\vk|^{-(s+\frac12)} (|\vk|& + M)^{s-\frac12} M^{\beta_2-2-3s}\min\Bigl\{M^{\beta_3+\beta_4+3}[\kappa^{\beta_1-1-s}+ \kappa^{\frac12(\beta_1-s)}], \notag\\
&\kappa^2 M^{\beta_3+\beta_4+1} [\kappa^{\beta_1-2-s}+ M^{\beta_1-2-s}], \ \kappa^{\beta_3+\beta_4 +4} M^{\beta_1-3-s}\Bigr\}.
\end{align*}
Once again, this requires consideration of individual cases. Unlike in Lemma~\ref{lem:NLS-para}, the final bound depends upon $|\vk|$ and so we cannot exploit monotonicity; thus, we need to treat separately $\kappa^{\frac12}\leq |\vk|\leq \kappa$ and $|\vk|\geq \kappa$.  Evaluating these sums carefully reveals that
\eqref{mKdV-para-2} can be improved to
\begin{align}
\text{LHS\eqref{mKdV-para-2}} \lesssim \Big[\tfrac{\kappa^{-2s}}{\kappa+|\vk|}+|\vk|^{-2(2s+1)}\log\bigl|\tfrac{4\vk^2}{\kappa}\bigr|\Big]\kappa^{|\beta|-8}\delta^2\normMK{q}^2
\end{align}
in two cases: (i) if $|\beta|\geq 3$ or (ii) if $|\beta|=2$ and no individual $\beta_j=2$.  These bounds suffice to prove \eqref{mKdV-para-2'} because if $|\beta|=2$ and some factor has two derivatives (i.e. some $\beta_j=2$), then we may integrate by parts to redistribute one of the derivatives and recover case (ii).
\end{proof}

Next we prove another pair of lemmas that will act as replacements for Lemmas~\ref{lem:LambdaLoc NLS},~\ref{L:LpIp mKdV} in certain situations:
\begin{lem}
Let \(|\vk|\geq \kappa^{\frac23}\geq 1\) and \(f_1,f_2\in \Cont([-1,1];\BdS)\) satisfy
\[
\normNK{f_j}\lesssim \normNK{q}.
\]
Then we have the estimate
\eq{L4I4-replacement NLS}{
\|(\vk - \p)^{-\frac12}(\psi_h f_1\cdot\psi_h\tfrac{f_2}{2\vk - \p})(\vk + \p)^{-\frac12}\|_{L^2_t \I_2}^2\lesssim |\vk|^{-3}\bigl[\kappa^{\frac23 - \frac{8s}3} + |\vk|^{-4s}\bigr]\delta^2\normNK{q}^2.
}
\end{lem}
\begin{proof}
By translation invariance, we may take \(h=0\).  Decomposing dyadically and using \eqref{LogarithmicBound}, yields
\[
\LHS{L4I4-replacement NLS}\lesssim \sum_{N\geq 1} (|\vk| + N)^{-1}\log\big(4 + \tfrac{N^2}{|\vk|^2}\big) \bigl\|P_N( \psi f_1\cdot \psi\tfrac {f_2}{2\vk - \p}\big)\bigr\|_{L^2_{t,x}}^2,
\]
in which we then substitute the bound 
\[
\bigl\|P_N( \psi f_1\cdot \psi\tfrac {f_2}{2\vk - \p}\big)\bigr\|_{L^2_{t,x}}\leq \biggl\| \sum_{N_1,N_2\geq 1}\big\|P_N\big( (\psi f_1)_{N_1}\cdot \big(\psi\tfrac {f_2}{2\vk - \p}\big)_{N_2}\big)\big\|_{L^2}\biggr\|_{L^2_t}.
\]

We then proceed using the Littlewood--Paley trichotomy:

\step{Case 1: \(N_2 \ll N_1 \approx N\).} Applying Bernstein's inequality, we bound
\begin{align*}
\big\|P_N\big( (\psi f_1)_{N_1}\cdot \big(\psi\tfrac {f_2}{2\vk - \p}\big)_{N_2}\big)\big\|_{L^2_{t,x}}
&\lesssim \|(\psi f_1)_{N_1}\|_{L^2_{t,x}}\|\big(\psi\tfrac {f_2}{2\vk - \p}\big)_{N_2}\|_{L^\infty_{t,x}}\\
&\lesssim N_1^{-s-\frac32}\min\{N_1^{\frac32},\kappa + N_1\} N_2^{\frac12-s}(|\vk| + N_2)^{-1}\delta \normNK{q}.
\end{align*}
Observing that for fixed \(N\geq 1\) we have
\[
\sum_{1\leq N_2\lesssim N}N_2^{\frac12-s}(|\vk| + N_2)^{-1}\lesssim |\vk|^{-1}(N\wedge |\vk|)^{\frac12-s},
\]
we are lead to estimate 
\begin{align*}
&\quad \sum_{N\geq 1}|\vk|^{-2}(N\wedge |\vk|)^{1-2s}(|\vk| + N)^{-1}\log(4 + \tfrac{N^2}{|\vk|^2}) N^{-2s-3}\min\{N^3,(\kappa + N)^2\} \\
&\qquad\qquad\lesssim |\vk|^{-3} \bigl[ \kappa^{\frac23 - \frac{8s}3}+|\vk|^{-4s} \bigr].
\end{align*}

\step{Case 2: \(N_1 \ll N_2\approx N\).} A similar argument yields the estimate
\begin{align*}
&\big\|P_N\big( (\psi f_1)_{N_1}\cdot \big(\psi\tfrac {f_2}{2\vk - \p}\big)_{N_2}\big)\big\|_{L^2_{t,x}}\lesssim N_1^{\frac12-s}(|\vk| + N_2)^{-1}N_2^{-s-\frac32}\min\{N_2^{\frac32},\kappa + N_2\} \delta \normNK{q}.
\end{align*}
This then leads us to evaluate
\[
\sum_{N\geq 1}(|\vk| + N)^{-3}\log(4 + \tfrac{N^2}{\vk^2})N^{-2(1+2s)}\min\{N^3,(\kappa + N)^2\},
\]
which yields the same bound as in Case~1.

\step{Case 3: \(N_1\approx N_2 \gtrsim N\).} Bernstein's inequality implies
\begin{align*}
\big\|P_N\big( (\psi f_1)_{N_1}\cdot &\big(\psi\tfrac {f_2}{2\vk - \p}\big)_{N_2}\big\|_{L^2}  \\
	&\lesssim N^{\frac12} \big\|(\psi f_1)_{N_1} \big\|_{L^2} \cdot N_1^{-s} (|\vk|+N_1)^{-1} \big\|(2\vk-\partial)\big(\psi\tfrac {f_2}{2\vk - \p}\big)_{N_2}\big\|_{H^{s}} .
\end{align*}
Thus, applying Cauchy--Schwarz to the sum, we obtain 
\begin{align*}
\biggl\| \sum_{N_1\approx N_2\gtrsim N}\big\|P_N\big( (\psi & f_1)_{N_1}\cdot \big(\psi\tfrac {f_2}{2\vk - \p}\big)_{N_2}\big)\big\|_{L^2}\biggr\|_{L^2_t}^2\\
&\lesssim N \sum\limits_{N_1\gtrsim N} N_1^{-4s-3}(|\vk| + N_1)^{-2}\min\{N_1^3,(\kappa + N_1)^2\} \delta^2\normNK{q}^2.
\end{align*}

We are then left to evaluate the sum
\begin{align*}
&\quad \sum_{N\geq 1} \sum_{N_1\gtrsim N} \tfrac N{|\vk| + N}\log\big(4 + \tfrac{N^2}{|\vk|^2}\big) N_1^{-4s-3}(|\vk| + N_1)^{-2}\min\{N_1^3,(\kappa + N_1)^2\} \\
&\lesssim \sum\limits_{N_1\geq 1}|\vk|^{-1}(|\vk| + N_1)^{-2} N_1^{-4s-2}\min\{N_1^3,(\kappa + N_1)^2\},
\end{align*}
which ultimately yields a contribution identical to that of Cases~1 and~2.
\end{proof}

In the case of the mKdV we have the following analogue:

\begin{lem}
Let \(|\vk|\geq \kappa^{\frac12}\geq 1\) and \(f_1,f_2,f_3\in \Cont([-1,1];\BdS)\) satisfy
\[
\normMK{f_j}\lesssim \normMK{q}.
\]
Then we have the estimates
\begin{align}
&\|(\vk - \p)^{-\frac12}(\psi_h f_1\cdot\psi_h\tfrac{f_2}{2\vk - \p})(\vk + \p)^{-\frac12}\|_{L^2_t \I_2}^2\label{L4I4-replacement mKdV}\\
&\qquad\qquad \lesssim |\vk|^{-3}\big[\kappa^{\frac12-2s} + |\vk|^{-1 - 4s}\log|2\vk|\big]\delta^2\normMK{q}^2,\notag\\
&\|(\vk - \p)^{-\frac12}(\psi_h f_1\cdot\psi_h\tfrac{f_2}{2\vk - \p}\cdot\psi_h\tfrac{f_3}{2\vk + \p})(\vk + \p)^{-\frac12}\|_{L^2_t \I_2}^2\label{L6I6-replacement mKdV}\\
&\qquad\qquad \lesssim |\vk|^{-5}\bigl[ \kappa^{1-3s} + \bigl(1+\tfrac{\kappa^2}{\vk^2}\bigr) |\vk|^{-6s} \log |2\vk|\bigr]\delta^4\normMK{q}^2.\notag
\end{align}
\end{lem}

\begin{proof}
The estimate \eqref{L4I4-replacement mKdV} follows from the same argument used to prove \eqref{L4I4-replacement NLS}; all that changes are the specific powers inside the sums.

Thus, it remains to consider the estimate \eqref{L6I6-replacement mKdV}. Proceeding as in the proof of \eqref{L4I4-replacement NLS}, we may assume that \(h=0\) and bound
\[
\LHS{L6I6-replacement mKdV}\lesssim \sum\limits_{N\geq 1}(|\vk| + N)^{-1}\log\big(4 + \tfrac{N^2}{\vk^2}\big)\big\|P_N\bigl(\psi f_1\cdot \psi\tfrac{f_2}{2\vk - \p}\cdot\psi\tfrac{f_3}{2\vk + \p}\bigr)\big\|_{L^2_{t,x}}^2 .
\]
We then decompose further by frequency, using
\begin{align*}
&\big\|P_N\big(\psi f_1\cdot \psi\tfrac{f_2}{2\vk - \p}\cdot\psi\tfrac{f_3}{2\vk + p}\big)\big\|_{L^2_{t,x}}\\
&\qquad  \leq \biggl\| \sum\limits_{N_1,N_2,N_3\geq 1} \big\|P_N\big((\psi f_1)_{N_1}\cdot \big(\psi\tfrac{f_2}{2\vk - \p}\big)_{N_2}\cdot\big(\psi\tfrac{f_3}{2\vk + \p}\big)_{N_3}\big)\big\|_{L^2} \biggr\|_{L^2_t}.
\end{align*}

As everything is symmetric under the $N_2\leftrightarrow N_3$ interchange, we may reduce matters to four possible cases:

\step{Case 1: \(\min\{N_1,N\}\geq \max\{N_2,N_3\}\).} Here we apply Bernstein's inequality to bound
\begin{align*}
&\big\|P_N\big((\psi f_1)_{N_1}\cdot \big(\psi\tfrac{f_2}{2\vk - \p}\big)_{N_2}\cdot\big(\psi\tfrac{f_3}{2\vk + \p}\big)_{N_3}\big)\big\|_{L^2_{t,x}}\\
&\qquad \leq \|(\psi f_1)_{N_1}\|_{L^2_{t,x}}\|\big(\psi\tfrac{f_2}{2\vk - \p}\big)_{N_2}\|_{L^\infty_{t,x}}\|\big(\psi\tfrac{f_3}{2\vk + \p}\big)_{N_3}\|_{L^\infty_{t,x}}\\
&\qquad \lesssim N_1^{-s-2}\min\{N_1^2,\kappa + N_1\}N_2^{\frac12-s}(|\vk| + N_2)^{-1}N_3^{\frac12-s}(|\vk| + N_3)^{-1}\delta^2 \normMK{q}.
\end{align*}
Summing in \(N_2,N_3\) and then in \(N_1\approx N\) using the Cauchy-Schwarz inequality we obtain a contribution to \(\RHS{L6I6-replacement mKdV}\) that is
\begin{align*}
&\lesssim \sum\limits_{N\geq 1}|\vk|^{-4}(N\wedge |\vk|)^{2-4s}(|\vk| + N)^{-1}\log\big(4 + \tfrac{N^2}{\vk^2}\big)N^{-2s-4}\min\{N^4,(\kappa + N)^2\} \delta^4 \normMK{q}^2 \\
&\lesssim |\vk|^{-5} \bigl[\kappa^{1-3s}+ \bigl(1+\tfrac{\kappa^2}{\vk^2}\bigr)|\vk|^{-6s} \log|2\vk| \bigr]\delta^4\normMK{q}^2.
\end{align*}

\step{Case 2: \(\min\{N_2,N\}\geq \max\{N_1,N_3\}\).} A similar argument, this time placing \(\psi\frac{f_2}{2\vk - \p}\) in \(L^2_{t,x}\) yields the estimate
\begin{align*}
&\big\|P_N\big((\psi f_1)_{N_1}\cdot \big(\psi\tfrac{f_2}{2\vk - \p}\big)_{N_2}\cdot\big(\psi\tfrac{f_3}{2\vk + \p}\big)_{N_3}\big)\big\|_{L^2_{t,x}}\\
&\qquad \lesssim N_1^{\frac12-s}(|\vk| + N_2)^{-1}N_2^{-s-2}\min\{N_2^2,\kappa + N_2\}N_3^{\frac12-s}(|\vk| + N_3)^{-1}\delta^2 \normMK{q},
\end{align*}
which yields a contribution to \(\RHS{L6I6-replacement mKdV}\) of
\[
\lesssim \sum\limits_{N\geq 1}|\vk|^{-2}(N\wedge |\vk|)^{1-2s}(|\vk| + N)^{-3}\log\big(4 + \tfrac{N^2}{\vk^2}\big)N^{-4s-3}\min\{N^4,(\kappa + N)^2\} \delta^4 \normMK{q}^2,
\]
which yields an identical contribution to Case 1.

\step{Case 3: \(\min\{N_1,N_2\}\geq \max\{N_3,N\}\).} Here we apply Bernstein's inequality at the output frequency and sum using the Cauchy-Schwarz inequality in \(N_1\approx N_2\) so that for fixed \(N\geq 1\) we obtain
\begin{align*}
&\left\|\sum\limits_{N_1\approx N_2\gtrsim N_3,N}\big\|P_N\big((\psi f_1)_{N_1}\cdot \big(\psi\tfrac{f_2}{2\vk - \p}\big)_{N_2}\cdot\big(\psi\tfrac{f_3}{2\vk + \p}\big)_{N_3}\big)\big\|_{L^2}\right\|_{L^2_t}^2\\
&\qquad\lesssim  N \left\| \sum\limits_{N_1\approx N_2\gtrsim N_3,N} \|(\psi f_1)_{N_1}\|_{L^2}\|\big(\psi\tfrac{f_2}{2\vk - \p}\big)_{N_2}\|_{L^2}\|\big(\psi\tfrac{f_3}{2\vk + \p}\big)_{N_3}\|_{L^\infty}\right\|_{L^2_t}^2\\
&\qquad \lesssim N \sum\limits_{N_1\gtrsim N}N_1^{-4s-4}\min\{N_1^4,(\kappa + N_1)^2\}(|\vk| + N_1)^{-2}|\vk|^{-2}(N_1\wedge |\vk|)^{1-2s}\delta^4 \normMK{q}^2.
\end{align*}
We then obtain a contribution to \(\RHS{L6I6-replacement mKdV}\) of
\begin{align*}
&\lesssim \sum\limits_{N\geq 1} \tfrac N{|\vk| + N}\log\big(4 + \tfrac{N^2}{\vk^2}\big) \sum\limits_{N_1\gtrsim N}N_1^{-4s-4}\min\{N_1^4,(\kappa + N_1)^2\}\\
&\qquad\qquad\qquad\times(|\vk| + N_1)^{-2}|\vk|^{-2}(N_1\wedge |\vk|)^{1-2s}\delta^4 \normMK{q}^2\\
&\lesssim \sum\limits_{N_1\geq 1} |\vk|^{-3}N_1^{-4s-3}\min\{N_1^4,(\kappa + N_1)^2\}(|\vk| + N_1)^{-2}(N_1\wedge |\vk|)^{1-2s}\delta^4\normMK{q}^2,
\end{align*}
which gives an identical contribution to Cases~1 and~2.

\step{Case 4: \(\min\{N_2,N_3\}\geq \max\{N_1,N\}\).} Arguing as in Case 3, for fixed \(N\geq 1\) we may bound
\begin{align*}
&\left\|\sum\limits_{N_2\approx N_3\gtrsim N_1,N}\big\|P_N\big((\psi f_1)_{N_1}\cdot \big(\psi\tfrac{f_2}{2\vk - \p}\big)_{N_2}\cdot\big(\psi\tfrac{f_3}{2\vk + \p}\big)_{N_3}\big)\big\|_{L^2}\right\|_{L^2_t}^2\\
&\qquad \lesssim N \sum\limits_{N_2\gtrsim N}N_2^{-6s-3}\min\{N_2^4,(\kappa + N_2)^2\}(|\vk| + N_2)^{-4}\delta^4 \normMK{q}^2,
\end{align*}
which yields a contribution to \(\RHS{L6I6-replacement mKdV}\) of
\[
\sum\limits_{N_2\geq 1} |\vk|^{-1}N_2^{-6s-2}\min\{N_2^4,(\kappa + N_2)^2\}(|\vk| + N_2)^{-4}\delta^4 \normMK{q}^2.
\]
This gives an identical contribution to the previous cases.
\end{proof}

We are now in a position to prove our main error estimates for the NLS:

\begin{lem}[Error estimates for the NLS]\label{lem:TAE-new}
There exists \(\delta>0\) so that for all real \(|\vk|\geq \kappa^{\frac23}\geq 1\), Schwartz functions \(q,f\in \Cont([-1,1];\BdS)\) satisfying
\[
\normNK{f}\lesssim \normNK{q},
\]
and \(\chi\in \{(\psi^\ell)^{(j)}: 6\leq \ell\leq 12, \, j=0,1 \}\), we have the estimates
\begin{align}
&\left\|\int \tfrac f{2\vk + \p}\, g_{12}\sbrack{\geq 3}(\pm\kappa)\,\chi_h\,dx\right\|_{L^1_t} \lesssim |\vk|^{-1}\kappa^{- 1 - \frac43(2s+1)}\delta^2\normNK{q}^2,\label{rem Dual NLS}\\
&\left\|\int f\,\big(\tfrac{g_{12}(\vk)}{2 + \vr(\vk)}\big)\sbrack{\geq 3}\,\chi_h\,dx\right\|_{L^1_t}\lesssim |\vk|^{-3} \bigl[\kappa^{\frac23 - \frac{8s}3} + |\vk|^{-4s}\bigr] \delta^2\normNK{q}^2,\label{stack Dual NLS}\\
&\left\|\vr(\pm\kappa)\sbrack{\geq 4}\,\chi_h\right\|_{L^1_{t,x}}  \lesssim \kappa^{-2 - \frac43(2s+1)}\delta^2\normNK{q}^2,\label{rho rem NLS}\\
&\left\|\int \tfrac f{2\vk + \p}\,\big(g_{12}\sbrack{\geq 3}(\kappa) + g_{12}\sbrack{\geq 3}(-\kappa)\big) \psi_h^{12}\,dx\right\|_{L^1_t}\lesssim \kappa^{- 2 - \frac43(2s+1)}\delta^2\normNK{q}^2,\label{rem Dual cancel NLS}
\end{align}
which are uniform in \(\kappa\), \(\vk\), and $h\in \R$.  As ever, $\chi_h(x):= \chi(x-h)$.
\end{lem}

\begin{proof}
By translation invariance, it suffices to consider the case $h=0$.

Our basic technique here is to expand using the series \eqref{g12-gamma-Series}, commute copies of $\psi$, and then use H\"older's inequality in trace ideals.  We first exhibit this technique to prove the auxiliary  result \eqref{AVT NLS 1} before turning our attention to the principal claims.

Given a test function \(F\in \Cont([-1,1];\Schwartz)\), using \eqref{g12-gamma-Series} we may write
\[
\sgn(\vk) \int F\,g_{12}(\vk)\sbrack{\geq 3} \psi^4 \,dx = \sum_{\ell=1}^\infty (-1)^{\ell-1}
	\tr\left\{\Lambda(\Gamma\Lambda)^{\ell}(\vk + \p)^{-\frac12}\psi^4 F(\vk - \p)^{-\frac12}\right\}.
\]
Applying Lemma~\ref{L:localize} followed by the operator estimates \eqref{Lambda} and \eqref{Sharp LambdaLoc NLS}, we obtain
\begin{align*}
&\left\|\tr\left\{\Lambda(\Gamma\Lambda)^{\ell}(\vk + \p)^{-\frac12}\psi^4 F(\vk - \p)^{-\frac12}\right\}\right\|_{L^1_t}\\
&\qquad\lesssim \|\Lambda\|_{L^\infty_t\I_2}^{2\ell-2}\|(\vk - \p)^{-\frac12}(\psi q)(\vk + \p)^{-\frac12}\|_{L^4_t\I_4}^3\|(\vk + \p)^{-\frac12}(\psi F)(\vk - \p)^{-\frac12}\|_{L^4_t\I_4}\\
&\qquad\lesssim C^\ell |\vk|^{-(2s+1)(\ell-1)-3}\bigl[\kappa^{\frac23 - \frac{8s}3} + |\vk|^{-4s}\bigr]\delta^{2(\ell-1)}\\
&\qquad\qquad\qquad\qquad\times\|q\|_{L^\infty_tH^s}^{\frac32}\normNK{q}^{\frac32}\|F\|_{L^\infty_t H^s}^{\frac12}\normNK{F}^{\frac12},
\end{align*}
where the implicit constant is independent of \(\ell\). Summing in $\ell$ and applying Young's inequality, we obtain
\begin{align}\label{AVT NLS 1}
\Bigl\|\int F\,& g_{12}\sbrack{\geq 3}(\vk)\,\psi^4 \,dx\Bigr\|_{L^1_t}\notag\\
&\lesssim |\vk|^{-3}\bigl[\kappa^{\frac23 - \frac{8s}3} + |\vk|^{-4s}\bigr]\delta\normNK{q}\Big(\|F\|_{L^\infty_t H^s}\normNK{q} + \delta\normNK{F}\Big).
\end{align}

The estimate \eqref{rem Dual NLS} follows immediately from \eqref{AVT NLS 1} by setting \(\vk=\pm \kappa\), \(F=\frac{\chi}{\psi^4}\frac f{2\vk + \p}\), and using \eqref{X-Smoothing} and \eqref{E:change} to bound $\normNK{F}\lesssim |\vk|^{-1} \normNK{q}$.

We turn now to \eqref{stack Dual NLS} and recall that
\begin{align}\label{10:06 decomp}
\big(\tfrac{g_{12}}{2 + \vr}\big)\sbrack{\geq 3}= \tfrac12 g_{12}\sbrack{\geq 3}  -\tfrac{g_{12}\vr}{2(2 + \vr)}.
\end{align}
By \eqref{AVT NLS 1}, the contribution of the first term to the left-hand side of \eqref{stack Dual NLS} is easily seen to be acceptable. To estimate the contribution of the second term on the right-hand of \eqref{10:06 decomp}, we take \(f_1 = \chi f/\psi^4\) and \(f_2 = (2\vk - \p) \tfrac{g_{12}(\vk)}{2 + \vr(\vk)}\) and apply the estimate \eqref{L4I4-replacement NLS} to bound
\[
\|(\vk \pm \p)^{-\frac12}(\psi^2 f_1\tfrac{f_2}{2\vk-\p})(\vk \pm \p)^{-\frac12}\|_{L^2_t \I_2}
		\lesssim |\vk|^{-\frac32}\bigl[\kappa^{\frac13 - \frac{4s}3} + |\vk|^{-2s}\bigr]\delta\normNK q,
\]
where we have used \eqref{X-Smoothing} with the estimates \eqref{ET Sob}, \eqref{ET LS} to bound
$$
\normNK{f_2} = \normNK{(2\vk - \p) \tfrac{g_{12}(\vk)}{2 + \vr(\vk)}}\lesssim \normNK{q}.
$$
We then use \eqref{g12-gamma-Series} to write
\begin{align*}
\int \!\tfrac{\chi f g_{12}(\vk)}{2(2 + \vr(\vk))}\, \vr\,dx
	= \sgn(\vk)\sum\limits_{\ell=1}^\infty(-1)^{\ell}\Bigl[&\tr\left\{(\Lambda\Gamma)^\ell(\vk -\p)^{-\frac12}(\psi^4 f_1\,\tfrac{f_2}{2\vk-\p})(\vk - \p)^{-\frac12}\right\}\\
+{}& \tr\left\{(\Gamma\Lambda)^\ell(\vk +\p)^{-\frac12}(\psi^4 f_1\,\tfrac{f_2}{2\vk-\p})(\vk + \p)^{-\frac12}\right\}\Bigr].
\end{align*}
Repeating our basic technique using \eqref{Sharp LambdaLoc NLS}, we obtain
\begin{align}\label{AVT NLS 2}
&\left\|\int f_1\,\tfrac{f_2}{2\vk-\p}\, \vr(\vk)\,\psi^{4 }\,dx\right\|_{L^1_t}\\
&\lesssim \|(\vk - \p)^{-\frac12}(\psi q)(\vk + \p)^{-\frac12}\|_{L^4_t\I_4}^2\|(\vk - \p)^{-\frac12}(\psi^2 f_1\,\tfrac{f_2}{2\vk-\p}) (\vk - \p)^{-\frac12}\|_{L^2_t \I_2}\notag\\
&\lesssim |\vk|^{-3}\bigl[\kappa^{\frac23 - \frac{8s}3} + |\vk|^{-4s}\bigr]\delta^2\normNK{q}^2,\notag
\end{align}
which completes the proof of \eqref{stack Dual NLS}.  The estimate \eqref{rho rem NLS} follows analogously using \eqref{Sharp LambdaLoc NLS} with \(\vk = \kappa\):
\begin{align}\label{AVT NLS 3}
\biggl\|\int F\, \vr\sbrack{\geq 4}(\kappa)\,\chi\,dx\biggr\|_{L^1_t}&\lesssim \|(\kappa - \p)^{-\frac12}(\psi q)(\kappa + \p)^{-\frac12}\|_{L^4_t\I_4}^4 \|(\kappa \pm \p)^{-\frac12}\|_\op^2 \|\chi F \psi^{-4}\|_{L^\infty_{t,x}} \notag\\
&\lesssim \kappa^{- 2 - \frac43(2s+1)}\delta^2\normNK{q}^2\|F\|_{L^\infty_{t,x}}.
\end{align}

Finally, we consider \eqref{rem Dual cancel NLS}.  Arguing in the same style, we bound
\begin{align*}
\left\|\int g_{12}\sbrack{\geq 5}(\kappa)\, \tfrac f{2\vk + \p}\psi^{12}\,dx\right\|_{L^1_t}
&\lesssim \|\Lambda\|_{L^\infty_t\I_2} \|(\kappa - \p)^{-\frac12}(\psi q)(\kappa + \p)^{-\frac12}\|_{L^4_t\I_4}^4\\
&\qquad \times \|(\kappa - \p)^{-\frac12}( \psi^8 \tfrac f{2\vk + \p})(\kappa + \p)^{-\frac12}\|_{L^\infty_t\op}\\
&\lesssim \kappa^{- 2 - \frac{11}6(2s+1)}\delta^3\normNK{q}^2\|\tfrac f{2\vk + \p}\|_{L^\infty_{t,x}}\\
&\lesssim \kappa^{- 2 - \frac{13}6(2s+1)}\delta^4\normNK{q}^2,
\end{align*}
where we have used that \(|\vk|\geq\kappa^{\frac23}\) to estimate \(
\|\tfrac f{2\vk + \p}\|_{L^\infty_{t,x}}\lesssim \kappa^{-\frac13(2s+1)}\delta\). For the remaining term we observe that integrating by parts we may write
\[
\int g_{12}\sbrack{3}(\kappa)\,\tfrac f{2\vk + \p}\,\psi^{12}\,dx = \int m[q,r,q,\psi^{12} \tfrac f{2\vk + \p}]\,dx,
\]
where the symbol
\[
m(\xi_1,\xi_2,\xi_3,\xi_4) = \tfrac2{(2\kappa + i\xi_2) (2\kappa - i\xi_3)(2\kappa + i\xi_4)}
\]
is a sum of terms in \(\kappa^{5 - |\beta|}S(\beta;2\kappa)\) for \(0\leq |\beta|\leq 5\), where at most one \(\beta_j = 2\). In particular, when considering the sum \(g_{12}\sbrack{3}(\kappa) + g_{12}\sbrack{3}(-\kappa)\) we see that the terms with even \(|\beta|\) cancel and hence
\[
\int \Big(g_{12}\sbrack{3}(\kappa) + g_{12}\sbrack{3}(-\kappa)\Big)\,\tfrac f{2\vk + \p}\,\psi^{12}\,dx = \int \widetilde m[q,r,q,\psi^{12} \tfrac f{2\vk + \p}]\,dx,
\]
where \(\widetilde m\) has symbol given by a sum of terms in \(\kappa^{5 - |\beta|}S(\beta;2\kappa)\) for \(|\beta| = 1,3,5\) and at most one \(\beta_j = 2\). Applying the estimate \eqref{NLS-para} we then obtain
\[
\left\|\int \Big(g_{12}\sbrack{3}(\kappa) + g_{12}\sbrack{3}(-\kappa)\Big)\,\tfrac f{2\vk + \p}\,\psi^{12}\,dx\right\|_{L^1_t}\lesssim \kappa^{-2 - \frac43(2s+1)}\delta^2\normNK{q}^2,
\]
which completes the proof of \eqref{rem Dual cancel NLS}.
\end{proof}

Similar arguments yield the following error estimates for the mKdV:

\begin{lem}[Error estimates for the mKdV]
There exists \(\delta>0\) so that for all real \(|\vk|\geq \kappa^{\frac12}\geq 1\), \(q,f\in \Cont([-1,1];\BdS)\) satisfying
\[
\normMK{f}\lesssim\normMK{q},
\]
and \(\chi\in \{(\psi^\ell)^{(j)}: 6\leq \ell\leq 12, \, j=0,1,2 \}\), we have the estimates
\begin{gather}
\begin{aligned}
&\biggl\|\int f\, g_{12}\sbrack{3}(\vk)\,\chi_h\,dx\biggr\|_{L^1_t} + \biggl\|\int f\, \big(\tfrac{g_{12}(\vk)}{2 + \vr(\vk)}\big)\sbrack{\geq 3}\,\chi_h\,dx\biggr\|_{L^1_t} \\
&\qquad\qquad\qquad\qquad \lesssim |\vk|^{-3}\bigl[\kappa^{\frac12 - 2s}+ |\vk|^{-1-4s}\log^4\!|2\vk|\bigr]\delta^2\normMK{q}^2,
\end{aligned} \label{stack Dual mKdV} \\ 
\begin{aligned}
&\left\|\int f\, \big(\tfrac{g_{12}(\vk)}{2 + \vr(\vk)}\big)\sbrack{\geq 5}\,\chi_h\,dx\right\|_{L^1_t} \\
&\qquad\qquad\quad \lesssim |\vk|^{-4-(s+\frac12)}\bigl[\kappa^{\frac34 - \frac{5s}2} + |\vk|^{-\frac12-5s}\log^6\!|2\vk|\bigr]\delta^4\normMK{q}^2,
\end{aligned} \label{stack Dual mKdV 2}\\
\begin{aligned}
&\biggl\|\int g_{21}\sbrack{3}(\pm\vk)\, \big(\tfrac{g_{12}(\vk)}{2 + \vr(\vk)}\big)\sbrack{\geq 3}\,\chi_h\,dx\biggr\|_{L^1_t}\\
&\qquad\qquad\qquad\qquad\qquad\qquad\quad \lesssim |\vk|^{-5} \bigl[\kappa^{1-2s}+ |\vk|^{-4s}\bigr] \delta^4\normMK{q}^2,
\end{aligned} \label{stack mKdV L2} \\
\begin{aligned}
&\left\|\int \tfrac f{2\vk + \p}\, \big(g_{12}\sbrack{\geq 3}(\kappa) + g_{12}\sbrack{\geq 3}(-\kappa)\big)\,\chi_h\,dx\right\|_{L^1_t}\\
&\qquad\qquad\qquad\qquad\qquad\qquad\lesssim \bigl[ \kappa^{-1} + |\vk|^{-1} \bigr]\kappa^{-2-(2s+1)}\delta^2\normMK{q}^2,
\end{aligned} \label{rem Dual cancel mKdV}\\
\begin{aligned}
&\left\|\int \tfrac f{2\vk + \p}\,\big(g_{12}\sbrack{\geq 3}(\kappa) - g_{12}\sbrack{\geq 3}(-\kappa) - \tfrac1{2\kappa^3}q^2r\big)\,\chi_h\,dx\right\|_{L^1_t} \\
&\qquad\qquad\qquad\qquad\qquad\qquad\qquad\qquad\qquad\lesssim \kappa^{-3-(2s+1)}\delta^2\normMK{q}^2,
\end{aligned} \label{rem Dual cancel mKdV 2}  \\
\left\|\vr(\pm\kappa)\sbrack{\geq 4}\,\chi_h\right\|_{L^1_{t,x}}  \lesssim \kappa^{-\frac52-(2s+1)}\delta^2\normMK{q}^2,\label{rho rem mKdV}\\
\left\|\int \Big(\vr(\kappa)\sbrack{\geq 4} - \vr(-\kappa)\sbrack{\geq 4}\Big)\,\psi_h^{12}\,dx\right\|_{L^1_t} \lesssim \kappa^{-3-(2s+1)}\delta^2\normMK{q}^2,\label{rem Dual cancel mKdV 3}\\
\begin{aligned}
&\left\|\int \Big(\vr(\kappa)\sbrack{\geq 4} + \vr(-\kappa)\sbrack{\geq 4} - \tfrac3{2\kappa^2}qr\big(q\cdot \tfrac r{4\kappa^2 - \p^2} + \tfrac q{4\kappa^2 - \p^2}\cdot r\big)\Big)\,\psi_h^{12}\,dx\right\|_{L^1_t} \\
&\qquad\qquad\qquad\qquad\qquad\qquad\qquad\qquad\qquad\ \lesssim \kappa^{-3-(2s+1)}\delta^2\normMK{q}^2,
\end{aligned} \label{rem Dual cancel mKdV 4}
\end{gather}
where the implicit constants are independent of \(\kappa\), \(\vk\), and \(h\in \R\).
\end{lem}

\begin{proof}
The basic technique is that used to prove Lemma~\ref{lem:TAE-new}; however, new cancellations need to be exhibited.  We begin with the estimates on $\gamma$.

Mimicking \eqref{AVT NLS 3}, but using Lemma~\ref{L:LpIp mKdV} with $p=4$ yields \eqref{rho rem mKdV}.  When taking $p=6$, we obtain instead
\begin{align}\label{gamma 6 is okay}
\bigl\| \vr(\pm\kappa)\sbrack{\geq 6}\chi \bigr\|_{L^1_{t,x}}&\lesssim [\kappa^{-5-3s}+\kappa^{-6-6s}\log(2\kappa)]\delta^4\normMK{q}^2 \\
&\lesssim \kappa^{-4-2s}\delta^4\normMK{q}^2.\notag
\end{align}
This estimate reduces \eqref{rem Dual cancel mKdV 3} and \eqref{rem Dual cancel mKdV 4} to consideration of the quartic terms, for which we turn to \eqref{gamma 4}.
Evidently, every term in \eqref{rem Dual cancel mKdV 3} and \eqref{rem Dual cancel mKdV 4} can be written as a sum of paraproducts with symbols conforming to \eqref{E:symbol class}; however, by forming these particular linear combinations, we eliminate all terms with $|\beta|=0$.  Thus we may apply Lemma~\ref{lem:mKdV-para} (with $\vk=\kappa$) and so deduce \eqref{rem Dual cancel mKdV 3} and \eqref{rem Dual cancel mKdV 4}. 

Applying our basic technique to $g_{12}$ using Lemma~\ref{L:LpIp mKdV} with \(p=4\) yields
\begin{align*}
\biggl\|\int f\, g_{12}\sbrack{3}(\vk)\,\chi_h\,dx\biggr\|_{L^1_t} + \biggl\|\int f g_{12}\sbrack{\geq 3}(\vk)\,\chi\,dx\biggr\|_{L^1_t}
	\lesssim \text{RHS\eqref{stack Dual mKdV}}.
\end{align*}
Taking $p=5$ and  using also \eqref{Lambda} yields
\begin{align*}
\biggl\|\int f g_{12}\sbrack{\geq 5}(\vk)\,\chi\,dx\biggr\|_{L^1_t}
	\lesssim \text{RHS\eqref{stack Dual mKdV 2}} ,
\end{align*}
These constitute a significant step toward proving \eqref{stack Dual mKdV} and \eqref{stack Dual mKdV 2}.  In view of \eqref{10:06 decomp}, the proof of \eqref{stack Dual mKdV} is completed by the following:
\begin{align}
\biggl\|\int f \tfrac{g_{12}}{2+\gamma} \vr \,\chi\,dx\biggr\|_{L^1_t}
	&\lesssim |\vk|^{-3} \bigl[\kappa^{\frac12 - 2s} + |\vk|^{-1-4s}\log^4|2\vk|\bigr] \delta^2 \normMK{q}^2,
\end{align}
which is a consequence of the argument used in \eqref{AVT NLS 2}, but using Lemma~\ref{L:LpIp mKdV} and \eqref{L4I4-replacement mKdV} in place of their NLS analogues.

To prove \eqref{stack Dual mKdV 2}, we use \eqref{E:vr gr 4} and $\vr\sbrack{2} = 2g_{12}\sbrack{1}g_{21}\sbrack{1}$ to  rewrite \eqref{more sbrack''} as
\begin{equation}\label{sbrack5 expansion}
\big(\tfrac{g_{12}}{2 + \vr}\big)\sbrack{\geq 5} = \tfrac12g_{12}\sbrack{\geq 5} - \tfrac{g_{12}(4 + \vr)}{4(2+\vr)}\vr\sbrack{\geq 4}
	+ \tfrac{g_{12}g_{21}}{2 + \vr}g_{12}\sbrack{\geq 3} - \tfrac12g_{12}\sbrack{1}g_{21}\sbrack{1} g_{12}\sbrack{\geq 3}
	+ \tfrac{g_{12}}{2 + \vr}g_{12}\sbrack{1}g_{21}\sbrack{\geq 3}.
\end{equation}

The contribution of the first term was handled already.

Consider now the second term in \eqref{sbrack5 expansion}.  Applying Lemma~\ref{L:X} together with the estimates \eqref{rho-Hs}, \eqref{rho-X-1}, \eqref{ET Sob}, and \eqref{ET LS}, we find that
\[
F=\chi f\tfrac{g_{12}(4 + \vr)}{2(2+\vr)\psi^5} \qtq{satisfies} \normMK{F}\lesssim |\vk|^{-(s+\frac12)}\delta\normMK{q}.
\]
Thus, applying the basic technique and using Lemma~\ref{L:LpIp mKdV} with $p=5$ shows 
\begin{align*}
\left\|\int  \tfrac{g_{12}(4 + \vr)}{2(2+\vr)}\vr\sbrack{\geq 4} \,\chi\,dx\right\|_{L^1_t} &\lesssim \text{RHS\eqref{stack Dual mKdV 2}}.
\end{align*}

The remaining three terms in \eqref{sbrack5 expansion} are handled in a parallel fashion, which we demonstrate using the first term. Set $F= f_1\,\frac{f_2}{2\vk-\p} \, \frac{f_3}{2\vk+\p}$ with \(f_1 = \chi f/\psi^6\), \(f_2 = (2\vk - \p)\frac{g_{12}}{2 + \vr}\), and \(f_3 = (2\vk + \p)g_{21}\).  Then \eqref{L6I6-replacement mKdV} implies
\begin{align*}
\|(\vk - \p)^{-\frac12}\psi^3 & F(\vk + \p)^{-\frac12}\|_{L^2_t\I_2}^2  \\
&\lesssim |\vk|^{-5}\bigl[ \kappa^{1-3s} + \bigl(1+\tfrac{\kappa^2}{\vk^2}\bigr) |\vk|^{-6s} \log |2\vk|\bigr]\delta^4\normMK{q}^2.
\end{align*}
Thus applying Lemma~\ref{L:LpIp mKdV} with $p=6$, we find
\begin{align*}
\left\|\int F\,g_{12}\sbrack{\geq 3}(\vk)\,\psi^6\,dx\right\|_{L^1_t}
	&\lesssim  |\vk|^{-5}\bigl[ \kappa^{1-3s} + \bigl(1+\tfrac{\kappa^2}{\vk^2}\bigr) |\vk|^{-6s} \log^6\! |2\vk|\bigr]\delta^4\normMK{q}^2,
\end{align*}
which is no larger than RHS\eqref{stack Dual mKdV 2}.  This completes the proof of \eqref{stack Dual mKdV 2}.  

We turn now to \eqref{stack mKdV L2}.  Combining \eqref{stack Dual mKdV 2} with \eqref{g21 1 3} and Lemma~\ref{L:X} yields
\begin{align*}
\biggl\|\int g_{21}\sbrack{3}(\pm\vk)\, \big(\tfrac{g_{12}}{2 + \vr}\big)\sbrack{\geq 5}(\vk)\,\chi\,dx\biggr\|_{L^1_t}
	&\lesssim |\vk|^{-1-(1+2s)} \cdot\text{RHS\eqref{stack Dual mKdV 2}} \lesssim \text{RHS\eqref{stack mKdV L2}}.
\end{align*}
To continue,  we employ \eqref{more sbrack}.  From Lemma~\ref{L:LpIp mKdV} and \eqref{BasicBound}, we find that
\begin{align*}
\biggl|\iint g_{21}\sbrack{3}\,F\psi^3\,dx\,dt\biggr|
	&\lesssim\|(\vk-\p)^{-\frac12}\psi q(\vk + \p)^{-\frac12}\|_{L^6_t\I_6}^3\|(\vk - \p)^{-\frac12}F(\vk + \p)^{-\frac12}\|_{L^2_t\I_2}\\
&\lesssim |\vk|^{-3}\bigl[\kappa^{\frac12 - \frac{3s}2}+ (1+\tfrac{\kappa}{|\vk|}\bigr)|\vk|^{-3s}\log^3\!|2\vk| \bigr] \delta^2\normMK{q}\|F\|_{L^2_{t,x}}
\end{align*}
and consequently, that 
\begin{align*}
\biggl\|\int g_{21}\sbrack{3}(\pm\vk)\, g_{12}\sbrack{3}(\vk)\,\chi\,dx\biggr\|_{L^1_t}
&\lesssim\bigl\|g_{21}\sbrack{3}(\pm\vk)\psi^3\bigr\|_{L^2_{t,x}}^2 \lesssim \text{RHS\eqref{stack mKdV L2}}.
\end{align*}

On the other hand, using Lemma~\ref{L:LpIp mKdV}, \eqref{gamma 2}, and \eqref{L6I6-replacement mKdV}, we get
\begin{align*}
&\biggl\|\int g_{21}\sbrack{3}(\pm\vk)\, g_{12}\sbrack{1}(\vk) \vr\sbrack{2}(\vk)\,\chi\,dx\biggr\|_{L^1_t}\\
&\lesssim\|(\vk-\p)^{-\frac12}(\psi q)(\vk + \p)^{-\frac12}\|_{L^6_t\I_6}^3\|(\vk - \p)^{-\frac12}\psi^3 g_{12}\sbrack{1}(\vk) \vr\sbrack{2}(\vk)(\vk + \p)^{-\frac12}\|_{L^2_t\I_2}\\
&\lesssim |\vk|^{-6}\bigl[ \kappa^{1-3s} + \bigl(1+\tfrac{\kappa^2}{\vk^2}\bigr) |\vk|^{-6s} \log^6\! |2\vk|\bigr]\delta^4\normMK{q}^2
		\lesssim \text{RHS\eqref{stack mKdV L2}}.
\end{align*}
This completes the proof of \eqref{stack mKdV L2}.

It remains to prove \eqref{rem Dual cancel mKdV} and \eqref{rem Dual cancel mKdV 2}.  We begin by reducing matters to the quartic terms.  As $|\vk|\geq \sqrt\kappa$, so \(\|\tfrac f{2\vk + \p}\|_{L^\infty_{t,x}}\lesssim \kappa^{-\frac12(s+\frac12)}\delta\).  Thus we find
\[
\left\|\int \tfrac f{2\vk + \p}\cdot g_{12}\sbrack{\geq 5}(\pm\kappa)\,\chi\,dx\right\|_{L^1_t}\lesssim \kappa^{-\frac92-3s}\delta^4\normMK{q}^2,
\]
by applying Lemma~\ref{L:LpIp mKdV} with $p=5$.

Regarding the quartic terms, we observe that
\[
\int \tfrac f{2\vk + \p}\cdot \Big(g_{12}\sbrack{3}(\kappa) - \tfrac1{4\kappa^3}q^2r\Big)\,\chi\,dx = \int m[q,r,q,\chi \tfrac f{2\vk + \p}]\,dx,
\]
where the lowest order terms cancel to give a symbol \(m\) that is a sum of terms in \(\kappa^{5 - |\beta|}S(\beta;2\kappa)\) for \(1\leq |\beta|\leq 8\).  Thus, \eqref{mKdV-para-2} may be applied, which then yields \eqref{rem Dual cancel mKdV}.  To obtain \eqref{rem Dual cancel mKdV 2} we use \eqref{mKdV-para-2'} instead.  This is possible due to the absence of any $|\beta|=1$ terms in the multiplier.
\end{proof}

We are finally in a position to undertake the proof of Lemma~\ref{lem:Currents-err}:

\begin{proof}[Proof of Lemma~\ref{lem:Currents-err}] We consider each of the currents in turn.

\medskip\step{Proof of \eqref{jNLS-err}.}
From Corollary~\ref{C:microscopic} and \eqref{rho higher},
\[
j_{\NLS}\sbrack{\geq 4} = - i (2\vk + \p)q\cdot \big(\tfrac{g_{21}}{2 + \vr}\big)\sbrack{\geq 3} + i (2\vk - \p) r\cdot \big(\tfrac{g_{12}}{2 + \vr}\big)\sbrack{\geq 3}.
\]
Writing
$$
\psi_h^6 (2\vk + \p)q= 2\vk (\psi_h^6 q) - (\psi_h^6)' q + (\psi_h^6 q)'_{\leq \vk} +(\psi_h^6 q)'_{>\vk}
$$
and invoking \eqref{stack Dual NLS} and \eqref{ET1 LS}, we estimate 
\begin{align*}
&\left\|\int j_{\NLS}\sbrack{\geq 4}(\vk)\,\psi_h^{12}\,dx\right\|_{L^1_t}\\
&\lesssim \vk^{-2(2s+1)} \delta^2\normN{q}^2 + \|(\psi_h^6 q)'_{>\vk}\|_{L_t^2 H^{-(s+\frac32)}} \bigl\| \psi_h^6 \big(\tfrac{g_{21}}{2 + \vr}\big)\sbrack{\geq 3}\bigr\|_{L_t^2 H^{s+\frac32}}\\
&\lesssim \vk^{-2(2s+1)} \delta^2\normN{q}^2,
\end{align*}
which completes the proof of \eqref{jNLS-err}.

\medskip\step{Proof of \eqref{jmKdV-err}.}
From Corollary~\ref{C:microscopic}, we compute
\begin{align*}
j_{\mKdV}\sbrack{\geq 4} &= (4\vk^2 + 2\vk \p + \p^2)q\cdot\big(\tfrac{g_{21}}{2 + \vr}\big)\sbrack{\geq 3} - (4\vk^2 - 2\vk\p + \p^2)r\cdot\big(\tfrac{g_{12}}{2 + \vr}\big)\sbrack{\geq 3}\\
&\quad  - 2q^2 r\tfrac{g_{21}}{2+\vr} + 2r^2q\tfrac{g_{12}}{2 + \vr}.
\end{align*}

Focusing on the first line in our expression for $j_{\mKdV}\sbrack{\geq 4}$, we write
\begin{align}\label{lo+hi}
\psi_h^3 &(4\vk^2 + 2\vk \p + \p^2)q\\
&= 4\vk^2(\psi_h^3 q) + 2\vk(\psi_h^3 q)'+(\psi_h^3 q)''- 2\vk (\psi_h^3)' q + (\psi_h^3)'' q - 2 [(\psi_h^3)' q]' \notag.
\end{align}
Thus, using Bernstein's inequality and \eqref{ET1 LS}, we estimate 
\begin{align*}
&\left\|\int P_{>\vk} \bigl[\psi_h^3 (4\vk^2 + 2\vk \p + \p^2)q\bigr] \cdot \big(\tfrac{g_{12}(\vk)}{2 + \vr(\vk)}\big)\sbrack{\geq 3}\,\psi_h^{9}\,dx\right\|_{L^1_t}\\
&\quad\lesssim \vk^{-(1+2s)} \|\psi_h q\|_{L_t^2 H^{s+1}}  \bigl\|\big(\tfrac{g_{21}(\vk)}{2 + \vr(\vk)}\big)\sbrack{\geq 3}\bigr\|_{X^{s+2}}\lesssim \vk^{-2(2s+1)} \delta^2\normM{q}^2.
\end{align*}
On the other hand, an application of \eqref{stack Dual mKdV} yields
\begin{align*}
&\left\| \int P_{\leq \vk} \bigl[\psi_h^3 (4\vk^2 + 2\vk \p + \p^2)q\bigr] \cdot \big(\tfrac{g_{12}(\vk)}{2 + \vr(\vk)}\big)\sbrack{\geq 3} \,\psi_h^{9}\,dx\right\|_{L^1_t}\\
&\qquad\qquad\lesssim \big[\vk^{-1}+ \vk^{-2(2s+1)}\log^4\!|2\vk|\big] \delta^2\normM{q}^2.
\end{align*}

We now demonstrate how to estimate the contribution of the final two terms in our expression for $j_{\mKdV}\sbrack{\geq 4}$, using the former as our example.  We first decompose into frequencies, as follows:
\[
\left\|\int q^2r\,\tfrac{g_{21}(\vk)}{2+\vr(\vk)}\,\psi_h^{12}\,dx\right\|_{L^1_t}  \lesssim\sum\limits_{N_j\geq 1}\left\|\int (\psi_h^{3}q)_{N_1}(\psi_h^{3}r)_{N_2}(\psi_h^{3}q)_{N_3}\big(\psi_h^{3}\tfrac{g_{21}}{2+\vr}\big)_{N_4}\,dx\right\|_{L^1_t},
\]
where the two highest frequencies must be comparable.  By exploiting symmetries, we may reduce consideration to two cases, namely, $N_1\sim N_2\geq N_3\vee N_4$ and $N_1\sim N_4 \gtrsim N_2\geq N_3$.

To estimate the low frequencies we use
\begin{align}
\|(\psi_h^{3}q)_{N}\big\|_{L^\infty_{t,x}}&\lesssim N^{\frac12-s}\delta,\label{qinfty}\\
\big\|\big(\psi_h^{3}\tfrac{g_{21}}{2+\vr}\big)_{N}\big\|_{L^\infty_{t,x}}&\lesssim (\vk+N)^{-1}N^{\frac12-s}\delta,\label{ginfty}
\end{align}
which follow from Bernstein's inequality and \eqref{ET Sob}.  To estimate the high frequencies we use
\begin{align*}
\|(\psi_h^{3}q)_{N}\|_{L^2_{t,x}}&\lesssim N^{-1-s}\normM{q},\\
\big\|\big(\psi_h^{3}\tfrac{g_{21}}{2+\vr}\big)_{N}\big\|_{L^2_{t,x}}&\lesssim (\vk+N)^{-1}N^{-1-s}\normM{q},
\end{align*}
which follow from Bernstein's inequality and \eqref{ET LS}. Estimating the two lowest frequency terms in \(L^\infty_{t,x}\) and the two highest frequency terms in \(L^2_{t,x}\) we obtain
\begin{align*}
\left\|\int q^2r\tfrac{g_{21}(\vk)}{2+\vr(\vk)}\,\psi_h^{12}\,dx\right\|_{L^1_t}\lesssim \big[\vk^{-1} +\vk^{-2(2s+1)}\log|2\vk|\big]\delta^2\normM{q}^2.
\end{align*}
This completes the proof of \eqref{jmKdV-err}.

\medskip\step{Proof of \eqref{jNLSdiff-err}.}
Recall that \(\vk\in [\kappa^{\frac23},\frac12\kappa]\cup[2\kappa,\infty)\). We decompose
\begin{align*}
j_{\NLS}^\diff{}\sbrack{\geq 4} &= \underbrace{-i\big(1 - \tfrac{\kappa^2}{\kappa^2 - \vk^2}\tfrac{4\kappa^2}{4\kappa^2 - \p^2}\big)(2\vk + \p)q\cdot\big(\tfrac{g_{21}(\vk)}{2 + \vr(\vk)}\big)\sbrack{\geq 3}}_{\err_1}\\
&\quad \underbrace{+ i\big(1 - \tfrac{\kappa^2}{\kappa^2 - \vk^2}\tfrac{4\kappa^2}{4\kappa^2 - \p^2}\big)(2\vk - \p)r\cdot\big(\tfrac{g_{12}(\vk)}{2 + \vr(\vk)}\big)\sbrack{\geq 3}}_{\err_2}\\
&\quad \underbrace{- \Big(\tfrac{2i\kappa^3}{\kappa - \vk}g_{12}(\kappa)\sbrack{\geq 3} + \tfrac{2i\kappa^3}{\kappa + \vk}g_{12}(-\kappa)\sbrack{\geq 3}\Big)\cdot\tfrac{g_{21}(\vk)}{2 + \vr(\vk)}}_{\err_3}\\
&\quad \underbrace{- \Big(\tfrac{2i\kappa^3}{\kappa - \vk}g_{21}(\kappa)\sbrack{\geq 3} + \tfrac{2i\kappa^3}{\kappa + \vk}g_{21}(-\kappa)\sbrack{\geq 3}\Big)\cdot\tfrac{g_{12}(\vk)}{2 + \vr(\vk)}}_{\err_4}\\
&\quad \underbrace{+ \tfrac{i\kappa^3}{\kappa - \vk}\vr(\kappa)\sbrack{\geq 4} + \tfrac{i\kappa^3}{\kappa + \vk}\vr(-\kappa)\sbrack{\geq 4}}_{\err_5}
\end{align*}
and note that by symmetry, it suffices to consider the contributions of the terms \(\err_j\) with $j=1,3,5$.

For \(\err_1\), we first write
\[
\big(1 - \tfrac{\kappa^2}{\kappa^2 - \vk^2}\tfrac{4\kappa^2}{4\kappa^2 - \p^2}\big)(2\vk + \p)q = -\tfrac{4\kappa^2\vk^2}{\kappa^2- \vk^2}\tfrac{2\vk + \p}{4\kappa^2 - \p^2}q - \tfrac{(2\vk + \p)\p^2}{4\kappa^2 - \p^2}q.
\]

Using Lemma~\ref{L:localize} together with \eqref{ET1 LS}, we estimate the contribution of the high frequencies as follows:
\begin{align*}
&\tfrac{\kappa^2\vk^2}{\kappa^2+ \vk^2}\left\|\int \psi^3_h\tfrac1{2\kappa + \p}\psi_h^{-3}\bigl(\psi_h^3\tfrac{2\vk + \p}{2\kappa - \p}q\bigr)_{>\vk}\cdot\big(\tfrac{g_{21}(\vk)}{2 + \vr(\vk)}\big)\sbrack{\geq 3}\,\psi_h^9\,dx\right\|_{L^1_t}\\
&\quad\lesssim \tfrac{\kappa^2\vk^2}{\kappa^2+ \vk^2} \bigl\| \bigl(\psi_h^3\tfrac{2\vk + \p}{2\kappa - \p}q\bigr)_{>\vk}\bigr\|_{L^2_t H^{-\frac52-s}} \vk^{-(1+2s)}\delta^2\normNK{q}  \lesssim\vk^{-2(1+2s)}\delta^2\normNK{q}^2,\\
&\left\|\int \psi_h^3\tfrac1{2\kappa + \p}\psi_h^{-3}\bigl(\psi_h^3\tfrac{(2\vk + \p)\p^2}{2\kappa - \p}q\bigr)_{>\vk}\cdot\big(\tfrac{g_{21}(\vk)}{2 + \vr(\vk)}\big)\sbrack{\geq 3}\,\psi_h^9\,dx\right\|_{L^1_t}\\
&\quad\lesssim  \bigl\| \bigl(\psi_h^3\tfrac{(2\vk + \p)\p^2}{2\kappa - \p}q\bigr)_{>\vk}\bigr\|_{L^2_t H^{-\frac52-s}} \vk^{-(1+2s)}\delta^2\normNK{q}  \lesssim\vk^{-2(1+2s)}\delta^2\normNK{q}^2 .
\end{align*}
The two low-frequency terms are estimated using Lemma~\ref{L:localize} and \eqref{stack Dual NLS}:
\begin{align*}
&\tfrac{\kappa^2\vk^2}{\kappa^2+ \vk^2}\left\|\int \psi_h^3\tfrac1{2\kappa + \p}\psi_h^{-3}\bigl(\psi_h^3\tfrac{2\vk + \p}{2\kappa - \p}q\bigr)_{\leq \vk}\cdot\big(\tfrac{g_{21}(\vk)}{2 + \vr(\vk)}\big)\sbrack{\geq 3}\,\psi_h^9\,dx\right\|_{L^1_t} \\
&\qquad\qquad\qquad\qquad\lesssim \tfrac{\kappa^2\vk^2}{\kappa^2+ \vk^2} \cdot \tfrac{\vk}{\kappa^2}
		\cdot |\vk|^{-3} \bigl[\kappa^{\frac23 - \frac{8s}3} + |\vk|^{-4s}\bigr] \delta^2\normNK{q}^2
\end{align*}
and similarly,
\begin{align*}
&\left\|\int \psi_h^3\tfrac1{2\kappa + \p}\psi_h^{-3}\bigl(\psi_h^3\tfrac{(2\vk + \p)\p^2}{2\kappa - \p}q\bigr)_{\leq \vk}\cdot\big(\tfrac{g_{21}(\vk)}{2 + \vr(\vk)}\big)\sbrack{\geq 3}\,\psi_h^9\,dx\right\|_{L^1_t} \\
&\qquad\qquad\qquad\qquad\lesssim \min\bigl\{\tfrac{\vk^3}{\kappa^2},\vk\bigr\} \cdot |\vk|^{-3} \bigl[\kappa^{\frac23 - \frac{8s}3} + |\vk|^{-4s}\bigr] \delta^2\normNK{q}^2.
\end{align*}

Collecting these estimates, we deduce that
\[
\left\|\int \err_1\,\psi_h^{12}\,dx\right\|_{L^1_t} \lesssim \bigl[\tfrac\kappa{\kappa + \vk}\kappa^{-\frac43(2s+1)} + \vk^{-2(2s+1)}\bigr]\delta^2\normNK{q}^2.
\]

To estimate the contribution of \(\err_3\) we define
\(
f = (2\vk + \p)\big(\tfrac {g_{21}(\vk)}{2 + \vr(\vk)}\big),
\)
and apply the estimates \eqref{ET Sob} and \eqref{ET LS} to see that
\[
\|f\|_{L^\infty_t H^s}\lesssim \delta \qtq{and} \normNK{f}\lesssim \normNK{q}.
\]
We then write
\begin{align*}
\err_3 &= -\tfrac{2i\kappa^3\vk}{\kappa^2 - \vk^2}\Big(g_{12}(\kappa)\sbrack{\geq 3} - g_{12}(-\kappa)\sbrack{\geq 3}\Big)\cdot \tfrac f{2\vk + \p}\\
&\quad -\tfrac{2i\kappa^4}{\kappa^2 - \vk^2}\Big(g_{12}(\kappa)^{{[\geq 3]}} + g_{12}(-\kappa)^{{[\geq 3]}}\Big)\cdot \tfrac f{2\vk + \p}.
\end{align*}
Applying the estimate \eqref{rem Dual NLS} to the first term and the estimate \eqref{rem Dual cancel NLS} to the second, we obtain 
\[
\left\|\int \err_3\,\psi_h^{12}\,dx\right\|_{L^1_t}\lesssim \tfrac\kappa{\kappa + \vk}\kappa^{-\frac43(2s+1)}\delta^2\normNK{q}^2.
\]

Finally, using \eqref{rho rem NLS} we estimate the contribution of \(\err_5\) by
\[
\left\|\int \err_5\,\psi_h^{12}\,dx\right\|_{L^1_t} \lesssim \tfrac{\kappa}{\kappa + \vk} \kappa^{- \frac43(2s+1)}\delta^2\normNK{q}^2,
\]
which completes the proof of \eqref{jNLSdiff-err}.

\medskip\step{Proof of \eqref{jmKdVdiff-err}.}
Recall that \(\vk\in [\kappa^{\frac12},\frac12\kappa]\cup[2\kappa,\infty)\). We decompose 
\begin{align*}
&j_\mKdV^{\diff\;[\geq 4]} = \underbrace{(1 - \tfrac{\kappa^2}{\kappa^2 - \vk^2}\tfrac{4\kappa^2}{4\kappa^2 - \p^2})(2\vk + \p)q'\cdot\big(\tfrac{g_{21}(\vk)}{2 + \vr(\vk)}\big)\sbrack{\geq 3}}_{\err_1}\\
&\quad \underbrace{+ (1 - \tfrac{\kappa^2}{\kappa^2 - \vk^2}\tfrac{4\kappa^2}{4\kappa^2 - \p^2})(2\vk - \p)r'\cdot\big(\tfrac{g_{12}(\vk)}{2 + \vr(\vk)}\big)\sbrack{\geq 3}}_{\err_2}\\
&\quad \underbrace{+ \big(\tfrac{4\kappa^4}{\kappa - \vk}g_{12}\sbrack{\geq 3}(\kappa) - \tfrac{4\kappa^4}{\kappa + \vk}g_{12}\sbrack{\geq 3}(-\kappa) - \tfrac{2\kappa^2}{\kappa^2 - \vk^2}q^2r\big)\cdot\tfrac{g_{21}(\vk)}{2 + \vr(\vk)}}_{\err_3}\\
&\quad \underbrace{+\big(\tfrac{4\kappa^4}{\kappa - \vk}g_{21}\sbrack{\geq 3}(\kappa) - \tfrac{4\kappa^4}{\kappa + \vk}g_{21}\sbrack{\geq 3}(-\kappa) + \tfrac{2\kappa^2}{\kappa^2 - \vk^2} r^2q\big)\cdot\tfrac{g_{12}(\vk)}{2 + \vr(\vk)}}_{\err_4}\\
&\quad \underbrace{-\Big[\tfrac{2\kappa^4}{\kappa - \vk}\vr\sbrack{\geq 4}(\kappa) - \tfrac{2\kappa^4}{\kappa + \vk}\vr\sbrack{\geq 4}(-\kappa) - \tfrac{3\kappa^2\vk}{\kappa^2 - \vk^2}qr\big(q\tfrac r{4\kappa^2 - \p^2} + r\tfrac q{4\kappa^2 - \p^2}\big)\Big]}_{\err_5}\\
&\quad \underbrace{+ \tfrac{4\vk^4}{\kappa^2 - \vk^2}r\cdot\big(\tfrac{g_{12}(\vk)}{2 + \vr(\vk)}\big)\sbrack{\geq 5}+ \tfrac{2\vk^2}{\kappa^2 - \vk^2}q^2r\cdot\bigl(\tfrac{g_{21}(\vk)}{2 + \vr(\vk)}\bigr)\sbrack{\geq 3}}_{\err_6}\\
&\quad \underbrace{- \tfrac{3\kappa^2\vk}{\kappa^2 - \vk^2}q^2r\tfrac r{4\kappa^2 - \p^2}  - \tfrac{\vk^4}{\kappa^2 - \vk^2}rg_{12}\sbrack{1}(\vk)\vr\sbrack{2}(\vk) - \tfrac{2\vk^4}{\kappa^2-\vk^2}qg_{21}\sbrack{3}(\vk)+ \tfrac{\vk^2}{\kappa^2 - \vk^2}q^2rg_{21}\sbrack{1}(\vk)}_{\err_7}\\
&\quad\underbrace{ - \tfrac{4\vk^4}{\kappa^2 - \vk^2}q\cdot\big(\tfrac{g_{21}(\vk)}{2 + \vr(\vk)}\big)\sbrack{\geq 5}- \tfrac{2\vk^2}{\kappa^2 - \vk^2}r^2q\cdot\bigl(\tfrac{g_{12}(\vk)}{2 + \vr(\vk)}\bigr)\sbrack{\geq 3}}_{\err_8}\\
&\quad\underbrace{- \tfrac{3\kappa^2\vk}{\kappa^2 - \vk^2}r^2q\tfrac q{4\kappa^2 - \p^2} + \tfrac{\vk^4}{\kappa^2 - \vk^2}qg_{21}\sbrack{1}(\vk)\vr\sbrack{2}(\vk) + \tfrac{2\vk^4}{\kappa^2-\vk^2}r\cdot g_{12}\sbrack{3}(\vk)- \tfrac{\vk^2}{\kappa^2 - \vk^2}r^2qg_{12}\sbrack{1}(\vk)}_{\err_9}.
\end{align*}
While the validity of this equality is, of course, elementary, the particular grouping of terms (and the addition of an extra term in $\err_5$ that is then subtracted in $\err_7$ and $\err_9$) represents a very delicate accounting for numerous cancellations.

As we will see, each term in this expansion individually yields an acceptable contribution to \eqref{jmKdVdiff-err}.  We will treat $\err_1, \err_3, \err_5, \err_6$, and $\err_7$ in turn.  The remaining terms are covered by this analysis and conjugation symmetry. 

For \(\err_1\), we first write
\[
(1 - \tfrac{\kappa^2}{\kappa^2 - \vk^2}\tfrac{4\kappa^2}{4\kappa^2 - \p^2})(2\vk + \p)q' = -\tfrac{4\kappa^2\vk^2}{\kappa^2- \vk^2}\tfrac{(2\vk + \p)\p}{4\kappa^2 - \p^2}q - \tfrac{(2\vk + \p)\p^3}{4\kappa^2 - \p^2}q.
\]
Proceeding as in the proof of \eqref{jNLSdiff-err} and using \eqref{ET1 LS}, we estimate the contribution of the second term as follows:
\begin{align*}
&\left\|\int \psi_h^3\tfrac1{2\kappa+\p}\psi_h^{-3}\bigl(\psi_h^3\tfrac{(2\vk + \p)\p^3}{2\kappa - \p}q\bigr)_{>\vk}\,\big(\tfrac{g_{21}(\vk)}{2 + \vr(\vk)}\big)\sbrack{\geq 3}\,\psi_h^{9}\,dx\right\|_{L^1_t}\\
&\qquad\lesssim \bigl\| \bigl(\psi_h^3\tfrac{(2\vk + \p)\p^3}{2\kappa - \p}q\bigr)_{>\vk}\bigr\|_{L^2_t H^{-(3+s)}}\bigl\| \big(\tfrac{g_{12}(\vk)}{2 + \vr(\vk)}\big)\sbrack{\geq 3}\bigr\|_{X_\kappa^{2+s}}\lesssim \vk^{-2(2s+1)}\delta^2\normMK{q}^2,\\
&\left\|\int \psi_h^3\tfrac1{2\kappa+\p}\psi_h^{-3}\bigl(\psi_h^3\tfrac{(2\vk + \p)\p^3}{2\kappa - \p}q\bigr)_{\leq\vk}\,\big(\tfrac{g_{21}(\vk)}{2 + \vr(\vk)}\big)\sbrack{\geq 3}\,\psi_h^{9}\,dx\right\|_{L^1_t}\\
&\qquad\lesssim \bigl\| \bigl(\psi_h^3\tfrac{(2\vk + \p)\p^3}{2\kappa - \p}q\bigr)_{\leq\vk}\bigr\|_{L^2_t H^{-(2+s)}} \bigl\| \big(\tfrac{g_{12}(\vk)}{2 + \vr(\vk)}\big)\sbrack{\geq 3}\bigr\|_{X_\kappa^{1+s}}\lesssim\vk^{-2(1+2s)}\delta^2\normMK{q}^2.
\end{align*}

To estimate the term with fewer derivatives, we write $\big(\tfrac{g_{21}}{2 + \vr}\big)\sbrack{\geq 3} =\big(\tfrac{g_{21}}{2 + \vr}\big)\sbrack{\geq 5}+ \big(\tfrac{g_{21}}{2 + \vr}\big)\sbrack{3}$.  Arguing as above and using \eqref{ET2 LS} in place of \eqref{ET1 LS}, we get
\begin{align*}
&\tfrac{\kappa^2\vk^2}{\kappa^2+ \vk^2}\left\|\int \psi_h^3\tfrac1{2\kappa+\p}\psi_h^{-3}\bigl(\psi_h^3\tfrac{(2\vk + \p)\p}{2\kappa - \p}q\bigr)_{>\vk}\,\big(\tfrac{g_{21}(\vk)}{2 + \vr(\vk)}\big)\sbrack{\geq 5}\,\psi_h^{9}\,dx\right\|_{L^1_t}\\
&\qquad\lesssim \tfrac{\kappa^2\vk^2}{\kappa^2+ \vk^2}\bigl\| \bigl(\psi_h^3\tfrac{(2\vk + \p)\p}{2\kappa - \p}q\bigr)_{>\vk}\bigr\|_{L^2_t H^{-(3+s)}}\bigl\| \big(\tfrac{g_{12}(\vk)}{2 + \vr(\vk)}\big)\sbrack{\geq 5}\bigr\|_{X_\kappa^{2+s}}\\
&\qquad \lesssim \vk^{-3(2s+1)}\delta^2\normMK{q}^2,
\end{align*}
while using \eqref{stack Dual mKdV 2}, we estimate
\begin{align*}
&\tfrac{\kappa^2\vk^2}{\kappa^2+ \vk^2}\left\|\int \psi_h^3\tfrac1{2\kappa+\p}\psi_h^{-3}\bigl(\psi_h^3\tfrac{(2\vk + \p)\p}{2\kappa - \p}q\bigr)_{\leq\vk}\,\big(\tfrac{g_{21}(\vk)}{2 + \vr(\vk)}\big)\sbrack{\geq 5}\,\psi_h^{9}\,dx\right\|_{L^1_t}\\
&\qquad \lesssim \tfrac{|\vk|^{-(s+\frac12)}}{\kappa^2+\vk^2} \bigl[\kappa^{\frac34 - \frac{5s}2} +|\vk|^{-(\frac12+5s)}\log^6\!|2\vk|\bigr]\delta^4\normMK{q}^2\lesssim \RHS{jmKdVdiff-err}.
\end{align*}
 
It remains to estimate the contribution of the quartic terms, which we expand using \eqref{more sbrack} and treat the two parts separately.

Setting $m_1=\frac{i\xi_1}{4\kappa^2+\xi_1^2}$ and $m_2=\frac{-\xi_1^2}{4\kappa^2+\xi_1^2}$, and using \eqref{g21 1 3} and \eqref{gamma 2}, we have
$$
\tfrac{(2\vk + \p)\p q}{4\kappa^2 - \p^2}  \, g_{21}\sbrack{1}(\vk) \vr\sbrack{2}(\vk)  = -\tfrac{4}{\vk}m_1\bigl[ q, \tfrac{\vk\, r}{2\vk + \p}, \tfrac{\vk\, q}{2\vk - \p}, \tfrac{r}{2\vk + \p}\bigr] - \tfrac{2}{\vk^2}m_2\bigl[ q, \tfrac{\vk\, r}{2\vk + \p}, \tfrac{\vk\, q}{2\vk - \p}, \tfrac{r}{2\vk + \p}\bigr].
$$
Applying both \eqref{mKdV-para-2} and \eqref{mKdV-para-2'} from Lemma~\ref{lem:mKdV-para}, we deduce that 
\begin{align*}
&\tfrac{\kappa^2\vk^2}{\kappa^2+ \vk^2}\left\|\int \tfrac{(2\vk + \p)\p}{4\kappa^2 - \p^2}q\,\big[g_{21}\sbrack{1}\vr\sbrack{2}\big](\vk)\,\psi_h^{12}\,dx\right\|_{L^1_t} \lesssim \RHS{jmKdVdiff-err}.
\end{align*}

For the remaining quartic term we first use \eqref{g21-ID} to write 
\begin{align*}
 \psi_h^{12}g_{21}\sbrack{3}(\vk)
 &= \tfrac{1}{2\vk+\p}\bigl[ r\vr\sbrack{2}(\vk) \psi_h^{12}\bigr] + \tfrac{1}{2\vk+\p} \bigr[(\psi_h^{12})' g_{21}\sbrack{3}(\vk)\bigr]\\
  &= \tfrac{1}{2\vk+\p}\bigl[ r\vr\sbrack{2}(\vk) \psi_h^{12}\bigr] + \tfrac{1}{(2\vk+\p)^2}\bigl[ r\vr\sbrack{2}(\vk) (\psi_h^{12})'\bigr]+ \tfrac{1}{(2\vk+\p)^2} \bigr[(\psi_h^{12})'' g_{21}\sbrack{3}(\vk)\bigr]
\end{align*}
and so
\begin{align*}
&\left\|\int \tfrac{(2\vk + \p)\p q}{4\kappa^2 - \p^2}\, g_{21}\sbrack{3}(\vk)\psi_h^{12}\,dx\right\|_{L^1_t}\\
&\leq\left\|\int \tfrac{(2\vk + \p)\p q}{(2\vk-\p)(4\kappa^2 - \p^2)} r\vr\sbrack{2}(\vk) \psi_h^{12}\,dx\right\|_{L^1_t}+ \left\|\int \tfrac{(2\vk + \p)\p q}{(2\vk-\p)^2(4\kappa^2 - \p^2)}\, r\vr\sbrack{2}(\vk) (\psi_h^{12})'\,dx\right\|_{L^1_t} \\
&\quad+ \left\|\int \tfrac{(2\vk + \p)\p q}{(2\vk-\p)^2(4\kappa^2 - \p^2)}\, g_{21}\sbrack{3}(\vk) (\psi_h^{12})'' \,dx\right\|_{L^1_t}.
\end{align*}
Using \eqref{mKdV-para-2} again, we see that the contribution arising from the first two terms above is acceptable.  For the last term, we estimate
\begin{align*}
&\tfrac{\kappa^2\vk^2}{\kappa^2+ \vk^2}\left\|\int \tfrac{(2\vk + \p)\p q}{(2\vk-\p)^2(4\kappa^2 - \p^2)}\, g_{21}\sbrack{3}(\vk) (\psi_h^{12})'' \,dx\right\|_{L^1_t}\\
&\, \lesssim \tfrac{\kappa^2\vk^2}{\kappa^2+ \vk^2} \bigl\| \tfrac{(2\vk + \p)\p}{(2\vk-\p)^2(4\kappa^2 - \p^2)}q\bigr\|_{L^\infty_t H^{-(1+s)}} \bigl\|g_{21}\sbrack{3}(\vk)\bigr\|_{L^\infty_t H^{1+s}}\lesssim |\vk|^{-2(2s+1)} \delta^2\normMK{q}^2.
\end{align*}

Collecting the estimates above, we obtain
\[
\left\|\int \err_1\,\psi_h^{12}\,dx\right\|_{L^1_t}\lesssim \RHS{jmKdVdiff-err}.
\]

For \(\err_3\) we start by writing
\begin{align*}
\err_3 &= \tfrac{4\kappa^5}{\kappa^2 - \vk^2}\big(g_{12}(\kappa)\sbrack{\geq 3} - g_{12}(-\kappa)\sbrack{\geq 3} - \tfrac1{2\kappa^3}q^2r\big)\,\tfrac{g_{21}(\vk)}{2 + \vr(\vk)}\\
&\quad + \tfrac{4\kappa^4\vk}{\kappa^2 - \vk^2}\big(g_{12}(\kappa)\sbrack{\geq 3} + g_{12}(-\kappa)\sbrack{\geq 3}\big)\,\tfrac{g_{21}(\vk)}{2 + \vr(\vk)}.
\end{align*}
We then apply the estimates \eqref{rem Dual cancel mKdV} and \eqref{rem Dual cancel mKdV 2} with \(f=(2\vk+\p)\frac{g_{21}(\vk)}{2 + \vr(\vk)}\) together with \eqref{ET Sob} and \eqref{ET LS} to bound
\[
\left\|\int \err_3\,\psi_h^{12}\,dx\right\|_{L^1_t}\lesssim \tfrac\kappa{\kappa + \vk}\kappa^{-(2s+1)}\delta^2\normMK{q}^2.
\]

For \(\err_5\) we may write
\begin{align*}
\err_5 &= - \tfrac{2\kappa^5}{\kappa^2 - \vk^2}\Big(\vr(\kappa)\sbrack{\geq 4} - \vr(-\kappa)\sbrack{\geq 4}\Big) \\
&\quad - \tfrac{2\kappa^4\vk}{\kappa^2 - \vk^2}\Big(\vr(\kappa)\sbrack{\geq 4} + \vr(-\kappa)\sbrack{\geq 4} - \tfrac3{2\kappa^2}qr\big(q\cdot\tfrac r{4\kappa^2 - \p^2} + r\cdot\tfrac q{4\kappa^2 - \p^2}\big)\Big) 
\end{align*}
and then use \eqref{rem Dual cancel mKdV 3} and \eqref{rem Dual cancel mKdV 4} to bound
\[
\left\|\int \err_5\,\psi_h^{12}\,dx\right\|_{L^1_t}\lesssim \tfrac\kappa{\kappa + \vk}\kappa^{-(2s+1)}\delta^2\normMK{q}^2.
\]

For \(\err_6\) we first apply the estimate \eqref{stack Dual mKdV 2} to bound
\begin{align*}
&\left\|\int \tfrac{4\vk^4}{\kappa^2 - \vk^2}r\cdot\big(\tfrac{g_{12}(\vk)}{2 + \vr(\vk)}\big)\sbrack{\geq 5}\,\psi_h^{12}\,dx\right\|_{L^1_t}
\\
&\qquad\lesssim \tfrac{\vk^{-(s+\frac12)}}{\kappa^2 + \vk^2}\bigl[\kappa^{\frac34-\frac{5s}2}+ \vk^{-(\frac12+5s)}\log^6|2\vk|\bigr]\delta^4\normMK{q}^2\lesssim \RHS{jmKdVdiff-err}.
\end{align*}

Next, we use \eqref{g12 1 3} and \([\psi_h^{12},\tfrac{\p}{2\vk - \p}] = -\frac{2\vk}{2\vk - \p}(\psi_h^{12})'\frac1{2\vk - \p}\) to write
\begin{align*}
\psi_h^{12}\,q^2r &= \psi_h^{12}\,4\vk^3g_{12}\sbrack{3}(\vk) + \tfrac{4\vk^3}{2\vk - \p}\bigl[(\psi_h^{12})'g_{12}\sbrack{3}(\vk)\bigr]\\
&\quad  - \tfrac{\p}{2\vk - \p}\bigl[\psi_h^{12}q\tfrac{2\vk r}{2\vk + \p}\tfrac{2\vk q}{2\vk - \p}\bigr] + \psi_h^{12}\,q\tfrac{r'}{2\vk + \p}\tfrac{2\vk q}{2\vk - \p} - \psi_h^{12}\,qr\tfrac{q'}{2\vk - \p}.
\end{align*}
From Corollary~\ref{C:ET} and elementary manipulations, we have
\[
\normMK{ \tfrac1{2\vk + \p}  \big(\tfrac{g_{21}(\vk)}{2 + \vr(\vk)}\big)\sbrack{\geq 3} }\lesssim \vk^{-3-2s}\delta^2\normMK{q}.
\]
Thus, by taking \(m_1(\xi) = \frac{i\xi_2}{2\kappa + i\xi_2}\), \(m_2(\xi) = \frac{i\xi_3}{2\kappa - i\xi_3}\), and applying \eqref{stack mKdV L2} to the first term, \eqref{stack Dual mKdV} to the second, and \eqref{mKdV-para-2}, \eqref{ET1 Sob}, \eqref{ET1 LS} to the remaining terms,  we have
\begin{align*}
&\left\|\int \tfrac{2\vk^2}{\kappa^2 - \vk^2}q^2r\,\big(\tfrac{g_{21}(\vk)}{2 + \vr(\vk)}\big)\sbrack{\geq 3}\,\psi_h^{12}\,dx\right\|_{L^1_t}\\
&\qquad\lesssim\tfrac{\vk^5}{\kappa^2 + \vk^2}\left\|\int g_{12}\sbrack{3}(\vk)\,\big(\tfrac{g_{21}(\vk)}{2 + \vr(\vk)}\big)\sbrack{\geq 3}\,\psi_h^{12}\,dx\right\|_{L^1_t}\\
&\qquad\quad + \tfrac{\vk^5}{\kappa^2 + \vk^2}\left\|\int  (\psi_h^{12})'g_{12}\sbrack{3}(\vk) \cdot \tfrac 1{2\vk + \p} \bigl(\tfrac{g_{21}(\vk)}{2 + \vr(\vk)}\bigr)\sbrack{\geq 3}\,dx\right\|_{L^1_t}\\
&\qquad\quad + \tfrac{\vk^3}{\kappa^2 + \vk^2}\left\|\int m_1\Bigl[q,\tfrac{2\kappa + \p}{2\vk + \p}\bigl(\tfrac{g_{21}(\vk)}{2 + \vr(\vk)}\bigr)\sbrack{\geq 3},\tfrac{2\vk q}{2\vk - \p},\tfrac{r}{2\vk + \p}\Bigr]\,\psi^{12}_h\,dx\right\|_{L^1_t}\\
&\qquad\quad + \tfrac{\vk^2}{\kappa^2 + \vk^2}\left\|\int m_1\Bigl[q,\tfrac{2\kappa + \p}{2\vk + \p}r,\tfrac{2\vk q}{2\vk - \p},\bigl(\tfrac{g_{21}(\vk)}{2 + \vr(\vk)}\bigr)\sbrack{\geq 3}\Bigr]\,\psi_h^{12}\,dx\right\|_{L^1_t}\\
&\qquad\quad + \tfrac{\vk^2}{\kappa^2 + \vk^2}\left\|\int m_2\Bigl[q,r,\tfrac{2\kappa - \p}{2\vk - \p}q,\bigl(\tfrac{g_{21}(\vk)}{2 + \vr(\vk)}\bigr)\sbrack{\geq 3}\Bigr]\,\psi_h^{12}\,dx\right\|_{L^1_t} \lesssim\RHS{jmKdVdiff-err}.
\end{align*}

For \(\err_7\) we first observe that 
\begin{align*}
&-\tfrac{2\vk^4}{\kappa^2-\vk^2}\int qg_{21}\sbrack{3}(\vk)\,\psi_h^{12}\, dx\\
&\qquad= \tfrac{4\vk^4}{\kappa^2-\vk^2}\int \tfrac{q}{2\vk-\p}\tfrac{r}{2\vk+\p}\tfrac{q}{2\vk-\p}\, r\,\psi_h^{12}\,dx - \tfrac{2\vk^4}{\kappa^2-\vk^2}\int \tfrac{q}{2\vk-\p}g_{21}\sbrack{3}(\vk)\,(\psi_h^{12})'\, dx.
\end{align*}
Applying \eqref{stack Dual mKdV}, the second integral contributes a constant multiple of
$$
\tfrac{1}{\kappa^2+\vk^2}\bigl[ \kappa^{\frac12-2s} +\vk^{-(1+4s)}\log^4|2\vk| \bigr]\delta^2\normMK{q}^2\lesssim\RHS{jmKdVdiff-err}.
$$
Thus the remaining quartic terms are
\begin{align*}
- \tfrac{3\kappa^2\vk}{\kappa^2 - \vk^2}q^2r\tfrac r{4\kappa^2 - \p^2} + \tfrac{2\vk^4}{\kappa^2-\vk^2} \tfrac{q}{2\vk-\p}\tfrac{r}{2\vk+\p}\tfrac{q}{2\vk-\p}\, r
+\tfrac{\vk^2}{\kappa^2-\vk^2}q^2r\tfrac{r}{2\vk+\p}.
\end{align*} 
A quick computation shows that
\begin{align*}
&- \tfrac{3\kappa^2\vk}{\kappa^2 - \vk^2}q^2r\tfrac r{4\kappa^2 - \p^2}+\tfrac{\vk^2}{\kappa^2-\vk^2}q^2r\tfrac{r}{2\vk+\p}\\
&\qquad =-\tfrac{2\kappa^2\vk^2}{\kappa^2 - \vk^2} q^2\tfrac{r}{4\kappa^2 - \p^2}\tfrac r{2\vk + \p} -\tfrac{6\kappa^2\vk^2}{\kappa^2 - \vk^2} m_1\bigl[q,\tfrac{2\kappa+\p}{2\vk+\p}r,q,\tfrac r{2\vk+\p}\bigr]\\
&\qquad\quad - \tfrac{3\kappa^2\vk}{\kappa^2 - \vk^2}m_2\bigl[q,\tfrac{r'}{2\vk+\p},q,\tfrac r{2\vk+\p}\bigr]-\tfrac{\vk^2}{\kappa^2-\vk^2}m_3\bigl[q,r,q,\tfrac{r}{2\vk+\p}\bigr],
\end{align*} 
where $m_1(\xi)= \frac{i\xi_2}{2\kappa+i\xi_2} \frac{1}{4\kappa^2 + \xi_4^2}$, $m_2(\xi) = \frac{i\xi_4}{4\kappa^2 + \xi_4^2}$ and $m_3(\xi)= \frac{(i\xi_2)^2}{4\kappa^2+\xi_2^2}$.  To continue, we observe that
\begin{align*}
&-\tfrac{2\kappa^2\vk^2}{\kappa^2 - \vk^2} q^2\tfrac{r}{4\kappa^2 - \p^2}\tfrac r{2\vk + \p} + \tfrac{2\vk^4}{\kappa^2-\vk^2} \tfrac{q}{2\vk-\p}\tfrac{r}{2\vk+\p}\tfrac{q}{2\vk-\p}\,r\\
&\qquad=\tfrac{4\kappa^2\vk^2}{\kappa^2 - \vk^2} m_4\bigl[\tfrac{2\kappa-\p}{2\vk-\p}q,r,\tfrac{2\vk q}{2\vk-\p},\tfrac{r}{2\vk+\p}\bigr]
-\tfrac{2\kappa^2\vk^2}{\kappa^2 - \vk^2} m_4\bigl[\tfrac{2\kappa-\p}{2\vk-\p}q,r,\tfrac{q'}{2\vk-\p},\tfrac{r}{2\vk+\p}\bigr]\\
&\qquad \quad -\tfrac{\vk^2}{2(\kappa^2-\vk^2)}m_3\bigl[\tfrac{2\vk q}{2\vk-\p},r,\tfrac{2\vk q}{2\vk-\p},\tfrac{r}{2\vk+\p}\bigr],
\end{align*} 
where $m_4(\xi) = \frac{i\xi_1}{2\kappa-i\xi_1} \frac{1}{4\kappa^2+\xi_2^2}$.  Applying the estimate \eqref{mKdV-para-2}, we obtain
\begin{align*}
&\tfrac{\kappa^2\vk^2}{\kappa^2 +\vk^2}\left\|\int m_1\bigl[q,\tfrac{2\kappa+\p}{2\vk+\p}r,q,\tfrac r{2\vk + \p}\bigr]\,\psi_h^{12}\,dx\right\|_{L^1_t}\\
&\qquad\qquad + \tfrac{\kappa^2\vk}{\kappa^2 +\vk^2}\left\|\int m_2\bigl[q,\tfrac{r'}{2\vk+\p},q,\tfrac r{2\vk + \p}\bigr]\,\psi_h^{12}\,dx\right\|_{L^1_t}\\
&\qquad\qquad+\tfrac{\vk^2}{\kappa^2 +\vk^2}\left\|\int m_3\bigl[q,r,q,\tfrac{r}{2\vk+\p}\bigr]\,\psi_h^{12}\,dx\right\|_{L^1_t}\\
&\qquad\qquad + \tfrac{\vk^2}{\kappa^2 +\vk^2}\left\|\int m_3\bigl[\tfrac{\vk q}{2\vk-\p},r,\tfrac{\vk q}{2\vk-\p},\tfrac{r}{2\vk+\p}\bigr]\,\psi_h^{12}\,dx\right\|_{L^1_t} \\
&\qquad\qquad+ \tfrac{\kappa^2\vk^2}{\kappa^2 +\vk^2}\left\|\int m_4\bigl[\tfrac{2\kappa-\p}{2\vk-\p}q,r,\tfrac{\vk q}{2\vk-\p},\tfrac{r}{2\vk+\p}\bigr]\,\psi_h^{12}\,dx\right\|_{L^1_t} \\
&\qquad\qquad+ \tfrac{\kappa^2\vk^2}{\kappa^2 +\vk^2}\left\|\int m_4\bigl[\tfrac{2\kappa-\p}{2\vk-\p}q,r,\tfrac{q'}{2\vk-\p},\tfrac{r}{2\vk+\p}\bigr]\,\psi_h^{12}\,dx\right\|_{L^1_t} \lesssim \RHS{jmKdVdiff-err}.
\end{align*}

Collecting all our bounds, we obtain the estimate \eqref{jmKdVdiff-err}.
\end{proof}


\section{Tightness}\label{sec:Tight}~
Let \(\chi \in \Test\) be an even non-negative function supported in \(\{|x|\leq 1\}\) with \(\|\chi\|_{L^1} = 1\) and define
\[
\phi(x) = \int_0^{|x|} \chi(y-2)\,dy.
\]
For \(R\geq 1\) we define the rescaled function \(\phi_R(x) = \phi(\frac xR)\).  Notice that $\phi_R$ plays the role of a smooth cut-off to large $|x|$ and so leads naturally to the following formulation of tightness:

\begin{defn}\label{d:tight}
A bounded subset \(Q\subset H^s\) is \emph{tight} in \(H^s\) if
\[
\phi_Rq\rightarrow 0\text{ in }H^s\text{ as }R\rightarrow \infty, \text{ uniformly for }q\in Q.
\]
\end{defn}

We first prove that tightness of \(q\) implies tightness of \(g_{12}\):

\begin{lem}
For $\delta>0$ sufficiently small,
\begin{align}
\|\phi_Rg_{12}\|_{H^{s+1}_\vk} &\lesssim \|\phi_Rq\|_{H^s_\vk} + (|\vk| R)^{-1}\|q\|_{H^s_\vk},\label{g12-Tight}\\
\|\phi_Rg_{12}\sbrack{\geq 3}\|_{H^{s+1}_\vk} &\lesssim |\vk|^{-(2s+1)}\delta^2\Big(\|\phi_Rq\|_{H^s_\vk} + (|\vk| R)^{-1}\|q\|_{H^s_\vk}\Big),\label{g12-err-Tight}\\
\|\phi_R\big(\tfrac{g_{12}}{2 + \vr}\big)\sbrack{\geq 3}\|_{H^{s+1}_\vk} &\lesssim |\vk|^{-(2s+1)}\delta^2\Big(\|\phi_Rq\|_{H^s_\vk} + (|\vk| R)^{-1}\|q\|_{H^s_\vk}\Big).\label{rA-err-Tight}
\end{align}
uniformly for \(|\vk|\geq 1\),  \(R\geq 1\), and \(q\in \Bd\).  Here \(g_{12} = g_{12}(\vk)\) and  \(\vr = \vr(\vk)\).
\end{lem}
\begin{proof}
Using the identity \eqref{g12-ID} we write
\[
\phi_Rg_{12} = - \tfrac1{2\vk - \p}\big(\phi_Rq(1 + \vr)\big) - \tfrac1{2\vk - \p}\big(\phi_R'g_{12}\big),
\]
so the estimate \eqref{g12-Tight} follows from the estimates \eqref{g12-Hs} and \eqref{rho-Hs}.

Similarly, the estimate \eqref{g12-err-Tight} follows from the identity
\[
\phi_Rg_{12}\sbrack{\geq 3} = - \tfrac1{2\vk - \p}\big(\phi_Rq\vr\big) - \tfrac1{2\vk - \p}\big(\phi_R'g_{12}\sbrack{\geq 3}\big)
\]
and the estimates \eqref{g12-LO} and \eqref{rho-Hs}. The estimate \eqref{rA-err-Tight} is then a corollary of the estimates \eqref{g12-Tight}, \eqref{g12-err-Tight}, \eqref{E:algebra}, and \eqref{more sbrack'}.
\end{proof}

We will prove tightness for solutions of \eqref{NLS} and \eqref{mKdV} by considering the equation satisfied by \(\Re\rho(\vk)\).  Our next lemma shows that this is a suitable quantity to consider.  The utility of this density should not be conflated with that of the currents used to prove the local smoothing effect.  In particular, in the \eqref{NLS} setting, it is the \emph{imaginary} part of $\rho$ that is used to prove local smoothing.

\begin{lem}\label{L:ReRho}
For $\delta$ sufficiently small we have
\eq{ReRho}{
\|\phi_Rq\|_{H^s}^2\approx \int_1^\kappa \vk^{2s+1}\Big(\pm \Re\int \rho\,\phi_R^2\,dx\Big)\,\tfrac{d\vk}{\vk} + \bigO\Big(R^{-2}\|q\|_{H^s}^2 + \|q\|_{H^s_\kappa}^2\Big),
}
uniformly for  \(q\in \BdS\), \(R\geq 1\),  and \(\kappa\geq 1\).
\end{lem}

\begin{proof}
As in the proof of Lemma~\ref{lem:Currents-quad}, we write
\[
\rho\sbrack{2} = \bR[q,r] = \tfrac12\big(q\cdot\tfrac r{2\vk + \p} + \tfrac q{2\vk - \p}\cdot r\big)
\]
and compute that
\[
\Re \int \bR[\phi_R q,\phi_R r]\,dx = \pm 2\vk\|\phi_Rq\|_{H^{-1}_\vk}^2.
\]
Applying Lemma~\ref{lem:EquivNorm} we then obtain
\[
\|\phi_Rq\|_{H^s}^2\approx \int_1^\kappa \vk^{2s+1}\Big(\pm \Re \int \bR[\phi_R q,\phi_R r]\,dx\Big)\,\tfrac{d\vk}{\vk} + \bigO\Big(\|q\|_{H^s_\kappa}^2\Big).
\]

It remains to bound the contribution of the difference
\begin{align*}
\int \rho\,\phi_R^2\,dx - \int \bR[\phi_R q,\phi_R r]\,dx &= \int \Big(\bR[q,r]\,\phi_R^2 - \bR[\phi_R q,\phi_R r]\Big)\,dx\\
&\quad  + \int \Big(q\cdot \big(\tfrac{g_{21}}{2 + \vr}\big)\sbrack{\geq 3} - r\cdot \big(\tfrac{g_{12}}{2 + \vr}\big)\sbrack{\geq 3}\Big)\,\phi_R^2\,dx.
\end{align*}
For the first term we bound
\begin{align*}
\left|\int \Big(\bR[q,r]\,\phi_R^2 - \bR[\phi_R q,\phi_R r]\Big)\,dx\right|
&\lesssim \|\phi_R q\|_{H_\vk^s} \bigl\| \bigl[ \phi_R, \tfrac1{2\vk-\p}\bigr]q\bigr\|_{H_\vk^{-s}}\\
&\lesssim \vk^{-(1+2s)} (\vk R)^{-1}\|\phi_R q\|_{H^s_\vk}\|q\|_{H^s_\vk}.
\end{align*}
For the second term we apply the estimate \eqref{rA-err-Tight} and Young's inequality to bound
\begin{align*}
\left|\int q\, \big(\tfrac{g_{21}}{2 + \vr}\big)\sbrack{\geq 3}\,\phi_R^2\,dx\right| &\lesssim \vk^{-(2s+1)}\|\phi_R q\|_{H^s_\vk}\|\phi_R\big(\tfrac{g_{21}}{2 + \vr}\big)\sbrack{\geq 3}\|_{H^{s+1}_\vk}\\
&\lesssim \vk^{-2(2s+1)}\delta^2\|\phi_Rq\|_{H^s_\vk}^2 + \vk^{-2(2s+1)}(\vk R)^{-2}\delta^2\|q\|_{H^s_\vk}^2.
\end{align*}
As a consequence, we may integrate to obtain
\begin{align*}
&\int_1^\kappa\vk^{2s+1}\left|\int \rho\,\phi_R^2\,dx - \Re \int \bR[\phi_R q,\phi_R r]\,dx\right|\,\tfrac{d\vk}\vk\\
&\qquad\qquad\qquad\lesssim R^{-1}\|\phi_Rq\|_{H^s}\|q\|_{H^s} + \delta^2\|\phi_Rq\|_{H^s}^2 + \delta^2R^{-2}\|q\|^2_{H^s},
\end{align*}
from which we derive the estimate \eqref{ReRho} by taking $\delta$ sufficiently small.
\end{proof}

We now arrive at the center-piece of this section:

\begin{prop}[Tightness of the flows]\label{prop:Tightness}
For \(\delta>0\) sufficiently small, the following holds:  If \(Q\subset \BdS\) is tight and equicontinuous in \(H^s\) then
\[
\Big\{q(t) = e^{tJ\nabla H_\star}q:q\in Q,\;t\in [-1,1]\Big\}\quad\text{is tight in $H^s$}.
\]
Here \(\star=\NLS,\mKdV\).
\end{prop}

We will prove this result for each of the two flows separately.  One element common to both is the following: For \(\sigma = s+\frac12\) or \(\sigma = s+1\), we have
\begin{align}
\|\psi_h^6(\phi_R^2)' F\|_{L_h^1L_t^2H^\sigma}&\lesssim \|\psi_h^3(\phi_R^2)'\|_{L^1_hH^1}\|\psi_h^3F\|_{L^\infty_hL^2_tH^\sigma} \lesssim \|F\|_{X^\sigma}. \label{easy2}
\end{align}

\begin{proof}[Proof of Proposition~\ref{prop:Tightness} for \eqref{NLS}]
Taking \(t\in [-1,1]\) and \(R\geq 1\), we multiply the equation \eqref{MicroscopicA} by \(\phi_R^2\), take the real part, and integrate by parts to obtain
\eq{Im-rho-NLS}{
\Re\int \bigl[\rho(t) - \rho(0)\bigr]\, \phi_R^2\,dx = \Re\int_0^t \int j_{\NLS}\,(\phi_R^2)' \,dx\,d\tau.
}
Choosing \(\kappa\geq 1\) and applying the estimate \eqref{ReRho} and the a priori estimate \eqref{APBound}, we obtain
\begin{align*}
\|\phi_R q\|_{L^\infty_tH^s}^2 &\lesssim \|\phi_R q(0)\|_{H^s}^2 + R^{-2}\|q(0)\|_{H^s}^2 + \|q(0)\|_{H^s_\kappa}^2\\
&\quad + \int_1^\kappa \vk^{2s+1} \biggl\| \int j_{\NLS}\,(\phi_R^2)' \,dx \biggr\|_{L^1_t}\,\tfrac{d\vk}\vk.
\end{align*}

Integrating by parts and using \eqref{rho-ID} we may write
\begin{align*}
\int j_{\NLS}\,(\phi_R^2)'\,dx  &= -i \int q \,\Big( (2\vk - \p)\tfrac{g_{21}}{2 + \vr} - \tfrac12r\Big)\,(\phi_R^2)'\,dx \\
&\quad + i \int r \,\Big( (2\vk + \p)\tfrac{g_{12}}{2 + \vr} + \tfrac12q\Big)\,(\phi_R^2)'\,dx\\
&\quad - \tfrac{i}{2} \int \log[2 + \vr] \,(\phi_R^2)'''\,dx.
\end{align*}

For the final term we may apply \eqref{rho-Linfty} and \eqref{APBound} to obtain
\begin{align*}
\biggl\| \int \log[2 + \vr] \,(\phi_R^2)'''\,dx\biggr\|_{L^1_t}  \lesssim R^{-2} .
\end{align*}

The remaining two terms are treated identically, so it suffices to consider the first.  We decompose
\begin{align}\label{decomp 3:16}
q=\tfrac{4\kappa^2q}{4\kappa^2-\p^2} - \tfrac{\p^2q}{4\kappa^2-\p^2}
\end{align}
and estimate the contribution of the low frequency term via
\begin{align*}
&\biggl\| \int\tfrac{4\kappa^2q}{4\kappa^2-\p^2} \,\Big( (2\vk - \p)\tfrac{g_{21}}{2 + \vr} -\tfrac12 r\Big)\,(\phi_R^2)'\,dx\biggr\|_{L^1_t}\\
&\qquad\qquad\lesssim \bigl\|\tfrac{\kappa^2q}{4\kappa^2-\p^2}\bigr\|_{L_t^\infty H^{-s}} \bigl\|(\phi_R^2)'\big( (2\vk - \p)\tfrac{g_{21}}{2 + \vr} -\tfrac12 r\big)\bigr\|_{L_t^\infty H^{s}}\\
&\qquad\qquad\lesssim \kappa^{-2s}\|q\|_{L_t^\infty H^{s}} \bigl\|(\phi_R^2)'\bigr \|_{H^1} \Big(\|(2\vk - \p)\tfrac{g_{21}}{2 + \vr}\|_{L_t^\infty H^s} + \|q\|_{L_t^\infty H^s}\Big)\\
&\qquad\qquad\lesssim R^{-\frac12}\kappa^{-2s}\|q(0)\|_{H^s}^2.
\end{align*}

To continue, we use \eqref{CPUD} to express the high frequency term via
\begin{align*}
&\int\tfrac{\p^2q}{4\kappa^2-\p^2} \,\Big( (2\vk - \p)\tfrac{g_{21}}{2 + \vr} - \tfrac 12r\Big)\,(\phi_R^2)'\,dx \\
&\qquad = \tfrac{7}{512} \iint \Bigl(\bigl[\psi_h^6, \tfrac{\p^2}{4\kappa^2-\p^2}\bigr]q + \tfrac{\p^2(\psi_h^6q)}{4\kappa^2-\p^2}\Bigr) \,\psi_h^6\,\Big( (2\vk - \p)\tfrac{g_{21}}{2 + \vr} -\tfrac12 r\Big)\,(\phi_R^2)'\,dx\,dh.
\end{align*}
For the commutator term, we apply the local smoothing estimates \eqref{ET LS} and \eqref{NLS-LS},  together with \eqref{1'} and \eqref{easy2} to bound
\begin{align*}
&\biggl\| \int \bigl[\psi_h^6, \tfrac{\p^2}{4\kappa^2-\p^2}\bigr]q \,\psi_h^6\,\Big( (2\vk - \p)\tfrac{g_{21}}{2 + \vr} -\tfrac12 r\Big)\,(\phi_R^2)'\,dx\biggr\|_{L^1_{t,h}}\\
&\qquad \lesssim \bigl\|\bigl[\psi_h^6, \tfrac{\p^2}{4\kappa^2-\p^2}\bigr]q\bigr\|_{L^\infty_{t,h} H^{-(s+\frac12)}}\bigl\|\psi_h^6(\phi_R^2)'\big( (2\vk - \p)\tfrac{g_{21}}{2 + \vr} -\tfrac12 r\big)\bigr\|_{L^1_hL_t^2H^{s+\frac12}}\\
&\qquad \lesssim \kappa^{-(2s+\frac32)} \|q(0)\|_{H^s}\Big(\bigl\|(2\vk - \p)\tfrac{g_{21}}{2 + \vr} \bigr\|_{X^{s+\frac12}} + \|q\|_{X^{s+\frac12}}\Big)\\
&\qquad \lesssim \kappa^{-(2s+\frac32)}\|q(0)\|_{H^s}^2 .
\end{align*}
For the remaining term, we use \eqref{ET LS}, \eqref{NLS-LS}, \eqref{NLS-LS-HF}, and \eqref{easy2}, as follows:
\begin{align*}
&\biggl\| \int \tfrac{\p^2(\psi_h^6q)}{4\kappa^2-\p^2} \,\psi_h^6\,\Big( (2\vk - \p)\tfrac{g_{21}}{2 + \vr} -\tfrac12 r\Big)\,(\phi_R^2)'\,dx \biggr\|_{L^1_{t,h}}\\
&\qquad \lesssim \bigl\|\tfrac{\p^2(\psi_h^6q)}{4\kappa^2-\p^2}\bigr\|_{L_h^\infty L^2_t H^{-(s+\frac12)}}\bigl\|\psi_h^6(\phi_R^2)'\big( (2\vk - \p)\tfrac{g_{21}}{2 + \vr} -\tfrac12 r\big)\bigr\|_{L^1_hL_t^2H^{s+\frac12}}\\
&\qquad \lesssim \kappa^{-(2s+1)} \|(\psi_h^6q)'\|_{L_h^\infty L^2_t H_\kappa^{s-\frac12}}\Big(\bigl\|(2\vk - \p)\tfrac{g_{21}}{2 + \vr} \bigr\|_{X^{s+\frac12}} + \|q\|_{X^{s+\frac12}}\Big)\\
&\qquad \lesssim \kappa^{-(2s+1)}\|q(0)\|_{H^s_\kappa}\|q(0)\|_{H^s} +\kappa^{-\frac32(2s+1)} \delta\|q(0)\|_{H^s}^2.
\end{align*}

Combining these bounds we see that for any \(\kappa\geq 1\) we have the estimate
\begin{align*}
\|\phi_R q\|_{L^\infty_t H^s}^2 &\lesssim \|\phi_R q(0)\|_{H^s}^2 +\|q(0)\|_{H^s_\kappa}^2 + \delta \|q(0)\|_{H^s_\kappa} \\
&\quad + \kappa^{2s+1} R^{-2} + \bigl( \kappa R^{-\frac12} + \kappa^{-(s+\frac12)}\bigr)\delta^2.
\end{align*}
Taking the supremum over \(q(0)\in Q\) and using that \(Q\) is tight we obtain
\begin{align*}
\limsup_{R\rightarrow\infty}\sup_{q(0)\in Q}\|\phi_Rq\|_{L^\infty_t H^s}^2
		&\lesssim \sup\limits_{q(0)\in Q} \delta \|q(0)\|_{H^s_\kappa}+ \kappa^{-(s+\frac12)} \delta^2.
\end{align*}
Using that \(Q\) is equicontinuous, the result follows by sending \(\kappa\rightarrow \infty\).
\end{proof}

\begin{proof}[Proof of Proposition~\ref{prop:Tightness} for \eqref{mKdV}]
Mimicking the argument given in the \eqref{NLS} case reduces matters to proving a suitable $L^1_t$ estimate for
\begin{align}\label{6:j_mKdV}
\int & j_{\mKdV}\,(\phi_R^2)'\,dx \\
&= \int q'\,\Big( (2\vk - \p)\tfrac{g_{21}}{2 + \vr}  - r\Big)\,(\phi_R^2)'\,dx
		+ \int r' \,\Big( (2\vk + \p)\tfrac{g_{12}}{2 + \vr} + q\Big)\,(\phi_R^2)'\,dx \notag\\
&\quad +  \int q\,\big(\tfrac{g_{21}}{2 + \vr}\big)' (\phi_R^2)''\,dx
		- \int r \,\big(\tfrac{g_{12}}{2 + \vr}\big)' (\phi_R^2)''\,dx
		- 2\vk \int qr \,(\phi_R^2)'\,dx \notag\\
&\quad -2  \int \rho \,qr \,(\phi_R^2)'\,dx  + \int \rho \,(4\vk^2 + \p^2)(\phi_R^2)'\,dx .\notag
\end{align}

From Corollary~\ref{C:ET}, \eqref{X-Product-2}, \eqref{APBound}, and \eqref{mKdV-LS} we have
\begin{align}\label{space-time-rho}
\| \rho \|_{L^\infty_t H^s} + \| \rho \|_{X^{s+1}} \lesssim \|q(0)\|_{H^s}^2 .
\end{align}
Thus, we may estimate the final term in \eqref{6:j_mKdV} as follows:
\begin{align*}
\biggl\| \int \rho \,(4\vk^2 + \p^2)(\phi_R^2)'\,dx\biggr\|_{L^1_t} 
	\lesssim\big(\vk^2 R^{-\frac12} + R^{-\frac52}\big)\|q(0)\|_{H^s}^2.
\end{align*}

To estimate the contribution of the remaining terms, we rely on the decomposition \eqref{decomp 3:16}.  We first bound the low-frequency contribution to each of the terms in \eqref{6:j_mKdV}, before treating the high-frequency terms.  From \eqref{ET Sob} and \eqref{APBound}, we have
\begin{align*}
&\biggl\|\int \tfrac{4\kappa^2q'}{4\kappa^2-\p^2}\,\Big( (2\vk - \p)\tfrac{g_{21}}{2 + \vr}- r\Big)\,(\phi_R^2)'\,dx\biggr\|_{L^1_t}\\
&\qquad\lesssim \bigl\|\tfrac{\kappa^2q'}{4\kappa^2-\p^2}\bigr\|_{L_t^\infty H^{-s}}\bigl\|(\phi_R^2)'\bigr\|_{H^1} \bigl\| (2\vk - \p)\tfrac{g_{21}}{2 + \vr}  - r\bigr\|_{L_t^\infty H^{s}} \lesssim \kappa^{1-2s} R^{-\frac12}\|q(0)\|_{H^s}^2.
\end{align*}
Arguing similarly, we also obtain
\begin{align*}
\biggl\| \int  \tfrac{4\kappa^2q}{4\kappa^2-\p^2}\,\big(\tfrac{g_{21}}{2 + \vr}\big)' \,(\phi_R^2)''\,dx\biggr\|_{L^1_t}&\lesssim \kappa^{-2s} R^{-\frac32}\|q(0)\|_{H^s}^2,\\
\vk\biggl\| \int  \tfrac{4\kappa^2q}{4\kappa^2-\p^2}\,r \,(\phi_R^2)'\,dx\biggr\|_{L^1_t}&\lesssim \vk\kappa^{-2s} R^{-\frac12}\|q(0)\|_{H^s}^2.
\end{align*}

For the penultimate term in \eqref{6:j_mKdV}, we decompose both $q$ and $r$ according to \eqref{decomp 3:16}:
\begin{align*}
\biggl\| \int  \tfrac{4\kappa^2q}{4\kappa^2-\p^2}\, \tfrac{4\kappa^2 r}{4\kappa^2-\p^2}\, \rho (\phi_R^2)'\,dx\biggr\|_{L^1_t}
		&\lesssim \bigl\|\tfrac{\kappa^2 q}{4\kappa^2-\p^2}\bigr\|_{L^\infty_t H^{s+1}}^2 \bigl\|\rho \bigr\|_{L^\infty_t H^s} \bigl\|(\phi_R^2)'  \bigr\|_{H^1} \\
&\lesssim \kappa^2 R^{-\frac12} \|q(0)\|_{H^s}^4.
\end{align*}

To estimate the contribution of the high-frequency term in the decomposition \eqref{decomp 3:16}, we use \eqref{CPUD}.  For example, we write
\begin{align*}
& \biggl\| \int\tfrac{\p^3q}{4\kappa^2-\p^2} \,\Big( (2\vk - \p)\tfrac{g_{21}}{2 + \vr} - \tfrac 12r\Big)\,(\phi_R^2)'\,dx\biggr\|_{L^1_t} \\
&\qquad \lesssim  \biggl\| \int \Bigl(\bigl[\psi_h^6, \tfrac{\p^3}{4\kappa^2-\p^2}\bigr]q + \tfrac{\p^3(\psi_h^6q)}{4\kappa^2-\p^2}\Bigr) \,\psi_h^6\,\Big( (2\vk - \p)\tfrac{g_{21}}{2 + \vr} -\tfrac12 r\Big)\,(\phi_R^2)'\,dx\biggr\|_{L^1_{t,h}}.
\end{align*}
Using \eqref{ET LS}, \eqref{mKdV-LS}, \eqref{easy2}, and \eqref{1'}, we get
\begin{align*}
&\biggl\| \int \bigl[\psi_h^6, \tfrac{\p^3}{4\kappa^2-\p^2}\bigr]q\,\psi_h^6\,\Big( (2\vk - \p)\tfrac{g_{21}}{2 + \vr}  - r\Big)\,(\phi_R^2)'\,dx\biggr\|_{L^1_{t,h}} \\
&\qquad\lesssim  \bigl\|\bigl[\psi_h^6, \tfrac{\p^3}{4\kappa^2-\p^2}\bigr]q\bigr\|_{L_h^\infty L_t^\infty H^{-(s+1)}}\bigl\|\psi_h^6(\phi_R^2)'\big( (2\vk - \p)\tfrac{g_{21}}{2 + \vr} -\tfrac12 r\big)\bigr\|_{L_h^1 L_t^2 H^{s+1}}\\
& \qquad\lesssim \kappa^{-(2s+1)}\|q(0)\|_{H^s}\|q(0)\|_{H^s_\kappa}.
\end{align*}
Using also \eqref{mKdV-LS-HF}, we estimate
\begin{align*}
&\left| \int\int_0^t\int\tfrac{\p^3(\psi_h^6q)}{4\kappa^2-\p^2}\,\psi_h^6\,\Big( (2\vk - \p)\tfrac{g_{21}}{2 + \vr}  - r\Big)\,(\phi_R^2)'\,dx\,d\tau\,dh\right|\\
&\qquad\lesssim \bigl\| \tfrac{\p^3(\psi_h^6q)}{4\kappa^2-\p^2}\bigr\|_{L_h^\infty L_t^2 H^{-(s+1)}} \bigl\| \psi_h^6(\phi_R^2)'\big( (2\vk - \p)\tfrac{g_{21}}{2 + \vr}  - r\big)\bigr\|_{L_h^1 L_t^2 H^{s+1}}\\
&\qquad \lesssim \kappa^{-(2s+1)} \|(\psi_h^6q)'\|_{L_h^\infty L^2_t H_\kappa^{s}} \|q(0)\|_{H^s} \\
& \qquad\lesssim \kappa^{-(2s+1)}\|q(0)\|_{H^s_\kappa}\|q(0)\|_{H^s}+ \kappa^{-\frac32(1+2s)} \log^2(2\kappa) \|q(0)\|_{H^s}^3.
\end{align*}
Arguing similarly, we also obtain
\begin{align*}
&\biggl\| \int \bigl[\psi_h^6, \tfrac{\p^2}{4\kappa^2-\p^2}\bigr]q\,\psi_h^6\,\big(\tfrac{g_{21}}{2 + \vr}\big)' \,(\phi_R^2)''\,dx\biggr\|_{L^1_{t,h}}
	\lesssim R^{-1}\kappa^{-2(1+s)}\|q(0)\|_{H^s}\|q(0)\|_{H^s_\kappa},  \\
&\vk\biggl\| \int \bigl[\psi_h^6, \tfrac{\p^2}{4\kappa^2-\p^2}\bigr]q \,\psi_h^6\,r(\phi_R^2)'\,dx\biggr\|_{L^1_{t,h}}
	\lesssim \vk\kappa^{-2(1+s)}\|q(0)\|_{H^s}\|q(0)\|_{H^s_\kappa}
\end{align*}
and
\begin{align*}
&\biggl\| \int  \tfrac{\p^2(\psi_h^6q)}{4\kappa^2-\p^2}\,\psi_h^6\,\big(\tfrac{g_{21}}{2 + \vr}\big)' \,(\phi_R^2)''\,dx \biggr\|_{L^1_{t,h}}\\
&\qquad\qquad \lesssim R^{-1}\kappa^{-2(1+s)}\|q(0)\|_{H^s_\kappa}\|q(0)\|_{H^s}+ R^{-1}\kappa^{-(\frac52+3s)}\log^2(2\kappa)\|q(0)\|_{H^s}^3,\\
&\vk\biggl\| \int \int_0^t\int  \tfrac{\p^2(\psi_h^6q)}{4\kappa^2-\p^2} \,\psi_h^6\,r(\phi_R^2)'\,dx\biggr\|_{L^1_{t,h}}\\
&\qquad\qquad\lesssim \vk\kappa^{-2(1+s)}\|q(0)\|_{H^s_\kappa}\|q(0)\|_{H^s}+\vk\kappa^{-(\frac52+3s)}\log^2(2\kappa)\|q(0)\|_{H^s}^3.
\end{align*}

This leaves us to handle the high-frequency contribution to the penultimate term in \eqref{6:j_mKdV}, which involves the combination
\begin{align*}
\tfrac{4\kappa^2q}{4\kappa^2-\p^2} \cdot \Bigl(\bigl[\psi_h^6, \tfrac{\p^2}{4\kappa^2-\p^2}\bigr]r + \tfrac{\p^2(\psi_h^6 r)}{4\kappa^2-\p^2}\Bigr) \psi_h^6
	+ \Bigl(\bigl[\psi_h^6, \tfrac{\p^2}{4\kappa^2-\p^2}\bigr]q + \tfrac{\p^2(\psi_h^6q)}{4\kappa^2-\p^2}\Bigr) \cdot r\,\psi_h^6 .
\end{align*}
We illustrate the estimation of these contributions using the latter summand.  Using \eqref{1'}, \eqref{easy2}, and \eqref{space-time-rho}, we get
\begin{align*}
&\biggl\| \int \bigl[\psi_h^6, \tfrac{\p^2}{4\kappa^2-\p^2}\bigr]q\cdot r\,\psi_h^6\,\rho \,(\phi_R^2)'\,dx\biggr\|_{L^1_{t,h}} \\
&\qquad\qquad  \lesssim  \bigl\| \bigl[\psi_h^6, \tfrac{\p^2}{4\kappa^2-\p^2}\bigr]q\bigr\|_{L^\infty_{t,h} H^{-s}}
	\| r \|_{L^\infty_{t} H^s} \bigl\| \psi_h^6 (\phi_R^2)' \rho \bigr\|_{L^1_h L_t^2 H^{s+1}}\\
&\qquad\qquad  \lesssim \kappa^{-(1+2s)} \|q(0)\|_{H^s_\kappa}\|q(0)\|_{H^s}^3, \\
&\biggl\| \int \tfrac{\p^2(\psi_h^6q)}{4\kappa^2-\p^2}\,r \,\rho\,\psi_h^6 \,(\phi_R^2)'\,dx \biggr\|_{L^1_{t,h}} \\
&\qquad\qquad\lesssim \bigl\| \tfrac{\p^2(\psi_h^6q)}{4\kappa^2-\p^2}\bigr\|_{L_h^\infty L_t^2 H^{-s}} \|q\|_{L_t^\infty H^s} \|\rho\|_{X^{s+1}}\\
&\qquad\qquad\lesssim \kappa^{-(1+2s)}\|q(0)\|_{H^s_\kappa}\|q(0)\|_{H^s}^3+\kappa^{-\frac32(1+2s)}\log^2(2\kappa)\|q(0)\|_{H^s}^5.
\end{align*}

The proof may now be completed exactly as in the NLS case.
\end{proof}

\section{Convergence of the difference flows}\label{S:convg}
Our main goal in this section is to prove the following:
\begin{prop}[Difference flow approximates the identity]\label{prop:diff-flows}
Let $\delta>0$ be sufficiently small and fix \(\star \in\{ \NLS,\mKdV\}\).  Given \(Q\subset \BdS\) that is equicontinuous in \(H^s\) and \(\vk\geq 4\), we have
\[
\psi_h^{12} g_{12}\bigl(\vk;e^{tJ\nabla(H_\star - H_{\star}^{\kappa})}q\bigr)\rightarrow \psi_h^{12}  g_{12}(\vk;q)\text{ in }\Cont([-1,1];H^{s+1})\text{ as }\kappa\rightarrow\infty,
\]
uniformly for \(q\in Q\) and $h\in\R$.
\end{prop}

\begin{proof}[Proof for \eqref{NLS-diff}]
Applying Proposition~\ref{prop:Equicontinuity} we see that
\[
Q^* = \{e^{tJ\nabla(H_{\NLS} - H_{\NLS}^{\kappa})}q:q\in Q,\, t\in \R\}
\]
is equicontinuous in \(H^s\). By Proposition~\ref{prop:g}, for any $\vk\geq 1$, the map \(q\mapsto g_{12}(\vk)\) is a diffeomorphism from \(\Bd\rightarrow H^{s+1}\); moreover, this map commutes with spatial translations.  Thus the set
\[
\{g_{12}(\vk;q):q\in Q^*\} \qtq{and so also} \{\psi_h^{12}g_{12}(\vk;q):q\in Q^*, \, h\in \R\}
\]
is equicontinuous in \(H^{s+1}\). As a consequence, it suffices to show that
\eq{DiffFlowGoal}{
\lim\limits_{\kappa\rightarrow\infty}\sup\limits_{q\in Q}\sup\limits_{h\in \R}\left\| i\frac d{dt}\left(\psi_h^{12} g_{12}(\vk;e^{tJ\nabla(H_\NLS-H_{\NLS}^{\kappa})}q)\right)\right\|_{L^1_tH^{-4}} = 0.
}

Using the identities \eqref{g12-ID} for \(g_{12}\) and \eqref{rho-ID} for \(\vr\) we may write
\[
-g_{12}(\vk)'' + 2qrg_{12}(\vk) +   2q^2g_{21}(\vk) = - 4\vk^2g_{12}(\vk) -[1 + \vr(\vk)] \,(2\vk + \p)q.
\]
Thus, we may rewrite \eqref{g12-NLS-diff} as
\[
i\frac d{dt}g_{12}(\vk) = \sum\limits_{j=1}^{11}\err_j,
\]
where we define
\begin{align*}
&\err_1 = \tfrac{4\vk^4}{\kappa^2 - \vk^2}g_{12}(\vk),\qquad \err_2 = \tfrac{4\kappa^2\vk^2}{\kappa^2 - \vk^2}[1 + \vr(\vk)]\tfrac{(2\vk + \p)q}{4\kappa^2 - \p^2},\\
&\err_3 = [1 + \vr(\vk)]\tfrac{(2\vk + \p)\p^2 q}{4\kappa^2-\p^2}, \qquad\err_4=-\tfrac{32\kappa^4\vk^2}{\kappa^2-\vk^2}g_{12}(\vk)\,\tfrac{q}{4\kappa^2-\p^2}\,\tfrac{r}{4\kappa^2 - \p^2},\\
&\err_5 =\tfrac{16\kappa^4\vk}{\kappa^2 - \vk^2}g_{12}(\vk)\Bigl[\tfrac{q}{4\kappa^2 - \p^2}\,\tfrac{\p r}{4\kappa^2 - \p^2} -\tfrac{\p q}{4\kappa^2 - \p^2}\,\tfrac{r}{4\kappa^2 - \p^2}\Bigr],\\
&\err_6= \bigl[\tfrac{8\kappa^4}{\kappa^2 - \vk^2}+16\kappa^2\bigr]g_{12}(\vk) \,\tfrac{\p q}{4\kappa^2 - \p^2}\,\tfrac{\p r}{4\kappa^2 - \p^2},\\
&\err_7 = - 8\kappa^2g_{12}(\vk)\p^2\Bigl[\tfrac{q}{4\kappa^2 - \p^2}\, \tfrac{r}{4\kappa^2 - \p^2}\Bigr],\\
&\err_8 =2g_{12}(\vk)\,\tfrac{\p^2q}{4\kappa^2 - \p^2}\,\tfrac{\p^2r}{4\kappa^2 - \p^2},\\
&\err_9 = \tfrac{2\kappa^3\vk}{\kappa^2-\vk^2}[1+\vr(\vk)]\,     \Bigl[-g_{12}\sbrack{\geq3}(\kappa)+g_{12}\sbrack{\geq3}(-\kappa) \Bigr],\\
&\err_{10} = - \tfrac{2\kappa^4}{\kappa^2 - \vk^2}[1+\vr(\vk)]\, \Bigl[ g_{12}\sbrack{\geq3}(\kappa)+g_{12}\sbrack{\geq3}(-\kappa) \Bigr],\\
&\err_{11} = g_{12}(\vk)\Bigl[\tfrac{2\kappa^3}{\kappa - \vk}\vr(\kappa)\sbrack{\geq 4} + \tfrac{2\kappa^3}{\kappa + \vk}\vr(-\kappa)\sbrack{\geq 4}\Bigr].
\end{align*}

It remains to bound each of the terms \(\err_j\).  We will rely on the a priori estimate \eqref{APBound} and the local smoothing estimate \eqref{NLS-diff-LS}, which yield
\begin{align}\label{apriori}
\normNK{q} = \|q\|_{L^\infty_t H^s}+ \|q\|_{X^{s+\frac12}_\kappa} \lesssim \|q(0)\|_{H^s}.
\end{align}
We will also employ the estimates recorded in Corollary~\ref{C:ls}, as well as the bounds
\begin{align}
\|q\|_{X^{-(s+\frac12)}_\kappa} \lesssim \kappa^{-\frac23(2s+1)}\|q(0)\|_{H^s},\label{X0}
\end{align}
\begin{align}
\|g_{12}(\vk)\|_{X^{s+\frac32}_\kappa} \lesssim\|q(0)\|_{H^s}\qtq{and} \|\vr(\vk)\|_{X^{s+\frac32}_\kappa}\lesssim\|q(0)\|_{H^s}^2,\label{X1}
\end{align}
which follow from \eqref{Reg-Gain}, \eqref{apriori}, \eqref{g12-X}, and \eqref{rho-X-1}.

As $\vk$ is fixed, we allow implicit constants to depend on this parameter. Throughout the proof, we will take $\kappa\geq 2\vk$. When it is convenient to argue by duality, we will write $\phi$ for a generic function in $L^\infty_t H^4$ of unit norm.

\step{Estimate for \(\err_1\).} We apply the estimate \eqref{g12-Hs} to bound
\[
\left\|\psi_h^{12}\err_1\right\|_{L^1_tH^{-4}}\lesssim \kappa^{-2}\|g_{12}(\vk)\|_{L^\infty_tH^{s+1}}\lesssim \kappa^{-2}\|q(0)\|_{H^s}.
\]

\step{Estimate for \(\err_2\).} Similarly, using duality and \eqref{rho-Hs}, we may bound
\begin{align*}
\left\|\psi_h^{12} \err_2\right\|_{L^1_tH^{-4}}
&\lesssim \bigl\|\psi_h^{12}[1+\vr(\vk)]\bigr\|_{L^\infty_tH^{s+1}}\bigl\|\tfrac{(2\vk + \p)q}{4\kappa^2 - \p^2}\bigr\|_{L^\infty_tH^{-(s+1)}}\lesssim\kappa^{-2(1+s)}\|q(0)\|_{H^s}.
\end{align*}

\step{Estimate for \(\err_3\).}  We estimate
\begin{align*}
\bigl\|\psi_h^{12} [1+\vr(\vk)]\tfrac{(2\vk + \p)\p^2 q}{4\kappa^2-\p^2}\bigr\|_{L^1_tH^{-4}}
&\lesssim \bigl\| \psi_h^6[1+\vr(\vk)] \tfrac{(2\vk + \p)\p^2 (\psi_h^6q)}{4\kappa^2-\p^2}\bigr\|_{L^1_tH^{-4}}\\
& \quad + \bigl\|\psi_h^6[1+\vr(\vk)] [\psi_h^6,\tfrac{(2\vk + \p)\p^2 }{4\kappa^2-\p^2}]q\bigr\|_{L^1_tH^{-4}}.
\end{align*}
We will bound both of these terms using duality. Using \eqref{X0} and \eqref{X1}, we get
\begin{align*}
\bigl\| \psi_h^6[1+\vr(\vk)] \tfrac{(2\vk + \p)\p^2 (\psi_h^6q)}{4\kappa^2-\p^2}\bigr\|_{L^1_tH^{-4}}
&\lesssim \sup_\phi \|\phi[1+\vr(\vk)]\|_{X^{s+\frac32}_\kappa} \|q\|_{X^{-(s+\frac12)}_\kappa}\\
&\lesssim \bigl[\kappa^{-1} + \|q(0)\|_{H^s}^2\bigr] \kappa^{-\frac23(2s+1)}\|q(0)\|_{H^s}\\
&\lesssim \kappa^{-\frac23(2s+1)}\|q(0)\|_{H^s}.
\end{align*}
Using instead \eqref{rho-Hs} and \eqref{1'}, we may bound
\begin{align*}
&\bigl\|\psi_h^6[1+\vr(\vk)] [\psi_h^6,\tfrac{(2\vk + \p)\p^2 }{4\kappa^2-\p^2}]q\bigr\|_{L^1_tH^{-4}}\\
&\quad\lesssim\sup_\phi \|\psi_h^6\phi [1+\vr(\vk)]\|_{L_t^\infty H^{s+1}} \bigl\| [\psi_h^6,\tfrac{(2\vk + \p)\p^2 }{4\kappa^2-\p^2}]q\bigr\|_{L_t^\infty H^{-(s+1)}}\lesssim \kappa^{-(2s+1)}\|q(0)\|_{H^s}.
\end{align*}

\step{Estimate for \(\err_4\).} Using \(L^1\subset H^{-4}\) and \(H^{s+1}\subset L^\infty \) together with \eqref{g12-Hs}, we get
\begin{align*}
\|\psi_h^{12} \err_4\|_{L^1_tH^{-4}}&\lesssim \kappa^{2}\|g_{12}(\vk)\|_{L^\infty_{t,x}}\bigl\|\tfrac q{4\kappa^2 - \p^2}\bigr\|_{L^\infty_tL^2}^2\lesssim \kappa^{-2(1+s)} \|q(0)\|_{H^s}^3.
\end{align*}

\step{Estimate for \(\err_5\).} Arguing as for $\err_4$, we may bound
\begin{align*}
\|\psi_h^{12} \err_5\|_{L^1_tH^{-4}}\lesssim \kappa^2 \|g_{12}(\vk)\|_{L^\infty_{t,x}}\bigl\|\tfrac q{4\kappa^2 - \p^2}\bigr\|_{L^\infty_tL^2}\bigl\|\tfrac {\p q}{4\kappa^2 - \p^2}\bigr\|_{L^\infty_tL^2}\lesssim \kappa^{-(1+2s)}\|q(0)\|_{H^s}^3.
\end{align*}

\step{Estimate for \(\err_6\).}  Using that \(L^1\subset H^{-4}\) and Corollary~\ref{C:ls} we may bound
\begin{align*}
\|\psi_h^{12} \err_6\|_{L^1_tH^{-4}} &\lesssim \kappa^2\|g_{12}(\vk)\|_{L^\infty_{t,x}}  \bigl\| \psi_h^6\tfrac{\p q}{4\kappa^2-\p^2}\bigr\|_{L_{t,x}^2}^2
\lesssim \kappa^{-\frac23(2s+1)}\|q(0)\|_{H^s}^3.
\end{align*}

\step{Estimate for \(\err_7\).}  We estimate
\begin{align*}
\|\psi_h^{12} \err_7\|_{L^1_tH^{-4}} &\lesssim \kappa^2\bigl\|\psi_h^6g_{12}(\vk)\p^2 \bigl[\psi_h^6\tfrac{q}{4\kappa^2 - \p^2}\, \tfrac{r}{4\kappa^2 - \p^2}\bigr]\bigr\|_{L^1_tH^{-4}}\\
&\quad+\kappa^2\bigl\|\psi_h^6g_{12}(\vk)[\psi_h^6,\p^2] \bigl[\tfrac{q}{4\kappa^2 - \p^2}\, \tfrac{r}{4\kappa^2 - \p^2}\bigr]\bigr\|_{L^1_tH^{-4}}.
\end{align*}
Using that \(L^1\subset H^{-4}\), we estimate the commutator term by
\begin{align*}
&\kappa^2\bigl\|\psi_h^6g_{12}(\vk)[\psi_h^6,\p^2] \bigl[\tfrac{q}{4\kappa^2 - \p^2}\, \tfrac{r}{4\kappa^2 - \p^2}\bigr]\bigr\|_{L^1_tH^{-4}}\\
&\qquad \lesssim \kappa^2\|g_{12}(\vk)\|_{L^\infty_{t,x}} \bigl\|\tfrac{q}{4\kappa^2 - \p^2}\bigr\|_{L_t^\infty H^1} \bigl\|\tfrac{q}{4\kappa^2 - \p^2}\bigr\|_{L_t^\infty L^2}\lesssim \kappa^{-(1+2s)}\|q(0)\|_{H^s}^3.
\end{align*}

To estimate the remaining term, we argue by duality. Using \eqref{X1}, we have
\begin{align}\label{E:err7'NLS}
&\kappa^2\bigl\|\psi_h^6g_{12}(\vk)\p^2 \bigl[\psi_h^6\tfrac{q}{4\kappa^2 - \p^2}\, \tfrac{r}{4\kappa^2 - \p^2}\bigr]\bigr\|_{L^1_tH^{-4}}\\
&\qquad\lesssim \kappa^2\|q(0)\|_{H^s} \bigl\| \sqrt{4\kappa^2-\p^2} \bigl[\psi_h^6\tfrac{q}{4\kappa^2 - \p^2}\, \tfrac{r}{4\kappa^2 - \p^2}\bigr]\bigr\|_{L_t^2 H^{-(s+\frac12)}} \notag
\end{align}
Employing Lemma~\ref{L:localize}, breaking into Littlewood--Paley pieces, and using Bernstein's inequality, we deduce that
\begin{align*}
\text{LHS\eqref{E:err7'NLS}} &\lesssim \|q(0)\|_{H^s}\!\!\! \sum_{N_2\leq N_1}\!\!\!\tfrac{\kappa^2  (\kappa+N_1+N_2)}{(\kappa+N_1)^2 (\kappa+N_2)^2}  \|(\psi_h^3q)_{N_1}\|_{L_{t,x}^2} N_2^{\frac12-s}\|(\psi_h^3q)_{N_2}\|_{L^\infty_t H^s} .
\end{align*}
Invoking Corollary~\ref{C:ls} and evaluating the resulting sum, we ultimately find
$$
\err_7\lesssim \kappa^{-\frac23(1+2s)}\|q(0)\|_{H^s}^3.
$$

\step{Estimate for \(\err_8\).}  Using \(L^1\subset H^{-4}\) and Corollary~\ref{C:ls}, we bound
\begin{align*}
\|\psi_h^{12} \err_8\|_{L^1_tH^{-4}} &\lesssim\|g_{12}(\vk)\|_{L^\infty_{t,x}}  \bigl\| \psi_h^6\tfrac{\p^2 q}{4\kappa^2-\p^2}\bigr\|_{L_{t,x}^2}^2
\lesssim \kappa^{-(2s+1)}\|q(0)\|_{H^s}^3.
\end{align*}

\step{Estimate for \(\err_9\).}  Using \eqref{rho-Hs}, \eqref{rho-X-1}, and \eqref{rem Dual NLS}, we may bound
\begin{align*}
\|\psi_h^{12} \err_9\|_{L^1_tH^{-4}}
&\lesssim \kappa^{-\frac43(2s+1)} \sup_\phi\normNK{(2\vk+\p)\bigl(\phi[1+\vr(\vk)]\bigr)} \|q(0)\|_{H^s}^3\\
&\lesssim \kappa^{-\frac43(2s+1)}\|q(0)\|_{H^s}^3.
\end{align*}

\step{Estimate for \(\err_{10}\).} Arguing as for $\err_9$ and using \eqref{rem Dual cancel NLS} in place of \eqref{rem Dual NLS}, we find
\begin{align*}
\|\psi_h^{12} \err_{10}\|_{L^1_tH^{-4}}
&\lesssim \kappa^{-\frac43(2s+1)}\sup_\phi\normNK{(2\vk+\p)\bigl(\phi[1+\vr(\vk)]\bigr)} \|q(0)\|_{H^s}^3\\
&\lesssim \kappa^{-\frac43(2s+1)}\|q(0)\|_{H^s}^3.
\end{align*}

\step{Estimate for \(\err_{11}\).}  Using \eqref{g12-Hs} and \eqref{rho rem NLS}, we obtain
\begin{align*}
\|\psi_h^{12} \err_{11}\|_{L^1_tH^{-4}} &\lesssim \kappa^2 \|g_{12}\|_{L_{t,x}^\infty} \bigl\| \vr(\pm\kappa)\sbrack{\geq 4} \psi_h^{12}\bigr\|_{L_{t,x}^1}
\lesssim \kappa^{-\frac43(2s+1)}\|q(0)\|_{H^s}^5.
\end{align*}

Collecting all our estimates for the error terms yields \eqref{DiffFlowGoal}.
\end{proof}

\medskip

\begin{proof}[Proof for \eqref{mKdV-diff}] It suffices to show the following analogue of \eqref{DiffFlowGoal}:
\eq{DiffFlowGoal2}{
\lim\limits_{\kappa\rightarrow\infty}\sup\limits_{q\in Q}\sup\limits_{h\in \R}\left\| \frac d{dt}\left(\psi_h^{12} g_{12}(\vk;e^{tJ\nabla(H_\mKdV-H_{\mKdV}^{\kappa})}q)\right)\right\|_{L^1_tH^{-4}} = 0.
}

Using the identities \eqref{g12-ID} for \(g_{12}\) and \eqref{rho-ID} for \(\vr\) we may write
\begin{align*}
&-g_{12}(\vk)''' + 6qrg_{12}(\vk)' +  6qq'g_{21}(\vk) + 6rq'g_{12}(\vk)-4\kappa^2g_{12}(\vk)' + \tfrac{8\kappa^4\vk}{\kappa^2-\vk^2}g_{12}(\vk)\\
&= g_{12}(\vk) \bigl[ \tfrac{8\vk^5}{\kappa^2-\vk^2} +4\vk qr + 2rq' - 2qr' \bigr]- [1 + \vr(\vk)] \bigl[ q'' +2\vk q' + 4(\kappa^2+\vk^2)q -2q^2r\bigr].
\end{align*}
As a consequence, we may write \eqref{g12-mKdV-diff} as
\[
\frac d{dt}g_{12}(\vk) = \sum\limits_{j=1}^{16}\err_j,
\]
where we define
\begin{align*}
&\err_1 = \tfrac{8\vk^5}{\kappa^2 - \vk^2}g_{12}(\vk),\qquad \err_2 = \tfrac{8\kappa^2\vk^3}{\kappa^2 - \vk^2}[1 + \vr(\vk)]\tfrac{(2\vk + \p)q}{4\kappa^2 - \p^2} + 4\vk^2[1 + \vr(\vk)]\tfrac{\p^2q}{4\kappa^2 - \p^2},\\
&\err_3 = [1 + \vr(\vk)]\tfrac{(2\vk + \p)\p^3 q}{4\kappa^2-\p^2}, \qquad\err_4=-\tfrac{64\kappa^4\vk^3}{\kappa^2-\vk^2}g_{12}(\vk)\,\tfrac{q}{4\kappa^2-\p^2}\,\tfrac{r}{4\kappa^2 - \p^2},\\
&\err_5 =\tfrac{32\kappa^4\vk^2}{\kappa^2 - \vk^2}g_{12}(\vk)\Bigl[\tfrac{q}{4\kappa^2 - \p^2}\,\tfrac{\p r}{4\kappa^2 - \p^2} -\tfrac{\p q}{4\kappa^2 - \p^2}\,\tfrac{r}{4\kappa^2 - \p^2}\Bigr],\\
&\err_6= \tfrac{16\kappa^4\vk}{\kappa^2 - \vk^2}g_{12}(\vk) \,\tfrac{\p q}{4\kappa^2 - \p^2}\,\tfrac{\p r}{4\kappa^2 - \p^2},\\
&\err_7 = -16\kappa^2\vk g_{12}(\vk)\Bigl[\tfrac{ \p^2q}{4\kappa^2 - \p^2}\, \tfrac{r}{4\kappa^2 - \p^2}+\tfrac{q}{4\kappa^2 - \p^2}\, \tfrac{\p^2r}{4\kappa^2 - \p^2}\Bigr] ,\\
&\err_8 =4\vk g_{12}(\vk)\,\tfrac{\p^2q}{4\kappa^2 - \p^2}\,\tfrac{\p^2r}{4\kappa^2 - \p^2},\\
&\err_9 = 16\kappa^2g_{12}(\vk)\Bigl[\tfrac{\p^2 q}{4\kappa^2 - \p^2}\, \tfrac{\p r}{4\kappa^2 - \p^2}-\tfrac{\p q}{4\kappa^2 - \p^2}\, \tfrac{\p^2 r}{4\kappa^2 - \p^2}\Bigr] ,\\
&\err_{10}=-8\kappa^2g_{12}(\vk) \partial \Bigl[\tfrac{\p^2 q}{4\kappa^2 - \p^2}\, \tfrac{ r}{4\kappa^2 - \p^2} - \tfrac{ q}{4\kappa^2 - \p^2}\, \tfrac{\p^2 r}{4\kappa^2 - \p^2}\Bigr] ,\\
&\err_{11}= 2g_{12}(\vk)\Bigl[\tfrac{\p^3 q}{4\kappa^2 - \p^2}\, \tfrac{\p^2 r}{4\kappa^2 - \p^2}-\tfrac{\p^2 q}{4\kappa^2 - \p^2}\, \tfrac{\p^3r}{4\kappa^2 - \p^2}\Bigr],\\
&\err_{12} =- \tfrac{4\kappa^5}{\kappa^2-\vk^2}[1+\vr(\vk)]\, \Bigl[ g_{12}(\kappa)\sbrack{\geq 3}-g_{12}(-\kappa)\sbrack{\geq 3}-\tfrac1{2\kappa^3} q^2r  \Bigr],\\
&\err_{13} = - \tfrac{4\kappa^4\vk}{\kappa^2 - \vk^2}[1+\vr(\vk)]\, \Bigl[ g_{12}(\kappa)\sbrack{\geq 3}+g_{12}(-\kappa)\sbrack{\geq 3}  \Bigr],\\
&\err_{14} =  \tfrac{4\kappa^4\vk}{\kappa^2 - \vk^2}g_{12}(\vk)\Bigl[\vr(\kappa)\sbrack{\geq 4} + \vr(-\kappa)\sbrack{\geq 4}\Bigr],\\
&\err_{15} =  \tfrac{4\kappa^5}{\kappa^2 - \vk^2}g_{12}(\vk)\Bigl[\vr(\kappa)\sbrack{\geq 4} -\vr(-\kappa)\sbrack{\geq 4}\Bigr],\\
&\err_{16} =  -\tfrac{2\vk^2}{\kappa^2 - \vk^2}[1+\vr(\vk)] q^2r.
\end{align*}

To bound the error terms, we will rely on the a priori estimate \eqref{APBound} and the local smoothing estimate \eqref{mKdV-diff-LS}, which yield
\begin{align}\label{apriorim}
\normMK{q} = \|q\|_{L^\infty_t H^s}+ \|q\|_{X^{s+1}_\kappa} \lesssim \|q(0)\|_{H^s}.
\end{align}
We will also employ the estimates recorded in Corollary~\ref{C:ls mKdV}, as well as the bounds
\begin{align}
\|q\|_{X^{-s}_\kappa} \lesssim \kappa^{-\frac12(2s+1)}\|q(0)\|_{H^s},\label{Xqm}
\end{align}
\begin{align}
\|g_{12}(\vk)\|_{X^{s+2}_\kappa} \lesssim\|q(0)\|_{H^s}\qtq{and} \|\vr(\vk)\|_{X^{s+2}_\kappa}\lesssim\|q(0)\|_{H^s}^2,\label{Xgm}
\end{align}
which follow from \eqref{Reg-Gain}, \eqref{apriorim}, \eqref{g12-X}, and \eqref{rho-X-1}.

We will allow implicit constants to depend on $\vk$. Throughout the proof, we will take $\kappa\geq 2\vk$. As before, when arguing by duality, we write $\phi$ for a function in $L^\infty_t H^4$ of unit norm.

\step{Estimate for \(\err_1\).} We apply the estimate \eqref{g12-Hs} to bound
\[
\left\|\psi_h^{12}\err_1\right\|_{L^1_tH^{-4}}\lesssim \kappa^{-2}\|g_{12}(\vk)\|_{L^\infty_tH^{s+1}}\lesssim \kappa^{-2}\|q(0)\|_{H^s}.
\]

\step{Estimate for \(\err_2\).} Similarly, using duality and \eqref{rho-Hs}, we may bound
\begin{align*}
&\left\|\psi_h^{12} \err_2\right\|_{L^1_tH^{-4}}\\
&\qquad\lesssim \bigl\|\psi_h^{12}[1+\vr(\vk)]\bigr\|_{L^\infty_tH^{s+1}}\bigl[\bigl\|\tfrac{(2\vk + \p)q}{4\kappa^2 - \p^2}\bigr\|_{L^\infty_tH^{-(s+1)}}+\bigl\|\tfrac{\p^2q}{4\kappa^2 - \p^2}\bigr\|_{L^\infty_tH^{-(s+1)}}\bigr]\\
&\qquad\lesssim\kappa^{-(1+2s)}\|q(0)\|_{H^s}.
\end{align*}

\step{Estimate for \(\err_3\).}  We estimate
\begin{align*}
\bigl\|\psi_h^{12} \err_3\bigr\|_{L^1_tH^{-4}}
&\lesssim \bigl\| \psi_h^6[1+\vr(\vk)] \tfrac{(2\vk + \p)\p^3 (\psi_h^6q)}{4\kappa^2-\p^2}\bigr\|_{L^1_tH^{-4}}\\
& \quad + \bigl\|\psi_h^6[1+\vr(\vk)] [\psi_h^6,\tfrac{(2\vk + \p)\p^3 }{4\kappa^2-\p^2}]q\bigr\|_{L^1_tH^{-4}}.
\end{align*}
We will bound both of these terms using duality. Using \eqref{Xqm} and \eqref{Xgm}, we get
\begin{align*}
\bigl\| \psi_h^6[1+\vr(\vk)] \tfrac{(2\vk + \p)\p^3 (\psi_h^6q)}{4\kappa^2-\p^2}\bigr\|_{L^1_tH^{-4}}
&\lesssim \sup_\phi\|\phi[1+\vr(\vk)]\|_{X^{s+2}_\kappa} \|q\|_{X^{-s}_\kappa}\\
&\lesssim \bigl[\kappa^{-1} + \|q(0)\|_{H^s}^2\bigr]  \kappa^{-\frac12(2s+1)}\|q(0)\|_{H^s}\\
&\lesssim \kappa^{-\frac12(2s+1)}\|q(0)\|_{H^s}.
\end{align*}
To estimate the commutator term, we use \eqref{2} and \eqref{rho-Hs}, as follows:
\begin{align*}
&\bigl\|\psi_h^6[1+\vr(\vk)] [\psi_h^6,\tfrac{(2\vk + \p)\p^3 }{4\kappa^2-\p^2}]q\bigr\|_{L^1_tH^{-4}}\\
& \lesssim \sup_\phi\|\psi_h^6\phi [1+\vr(\vk)]\|_{L_t^\infty H^{s+1}} \bigl\| [\psi_h^6,\tfrac{(2\vk + \p)\p^3 }{4\kappa^2-\p^2}]q\bigr\|_{L_t^2 H^{-(s+1)}}\lesssim \kappa^{-\frac12(2s+1)}\|q(0)\|_{H^s}.
\end{align*}

Collecting our estimates we obtain
\[
\|\psi_h^{12} \err_3\|_{L^1_tH^{-4}}\lesssim \kappa^{-\frac12(2s+1)}\|q(0)\|_{H^s}.
\]

\step{Estimate for \(\err_4\).} Using \(L^1\subset H^{-4}\) and \(H^{s+1}\subset L^\infty\) together with \eqref{g12-Hs}, we get
\begin{align*}
\|\psi_h^{12} \err_4\|_{L^1_tH^{-4}}&\lesssim \kappa^{2}\|g_{12}(\vk)\|_{L^\infty_{t,x}}\bigl\|\tfrac q{4\kappa^2 - \p^2}\bigr\|_{L^\infty_tL^2}^2\lesssim \kappa^{-2(1+s)} \|q(0)\|_{H^s}^3.
\end{align*}

\step{Estimate for \(\err_5\).} Arguing as for $\err_4$, we may bound
\begin{align*}
\|\psi_h^{12} \err_5\|_{L^1_tH^{-4}}\lesssim \kappa^2 \|g_{12}(\vk)\|_{L^\infty_{t,x}}\bigl\|\tfrac q{4\kappa^2 - \p^2}\bigr\|_{L^\infty_tL^2}\bigl\|\tfrac {\p q}{4\kappa^2 - \p^2}\bigr\|_{L^\infty_tL^2}\lesssim \kappa^{-(1+2s)}\|q(0)\|_{H^s}^3.
\end{align*}

\step{Estimate for \(\err_6\).}  Using \(L^1\subset H^{-4}\) and Corollary~\ref{C:ls mKdV}, we may bound
\begin{align*}
\|\psi_h^{12} \err_6\|_{L^1_tH^{-4}} &\lesssim \kappa^2\|g_{12}(\vk)\|_{L^\infty_{t,x}}  \bigl\| \psi_h^6\tfrac{\p}{4\kappa^2-\p^2}q\bigr\|_{L_{t,x}^2}^2
\lesssim \kappa^{-(1+s)}\|q(0)\|_{H^s}^3.
\end{align*}

\step{Estimate for \(\err_7\).}  Arguing as for $\err_6$, we may bound
\begin{align*}
\|\psi_h^{12} \err_7\|_{L^1_tH^{-4}}
&\lesssim \kappa^2\|g_{12}(\vk)\|_{L^\infty_{t,x}} \bigl\|\tfrac q{4\kappa^2 - \p^2}\bigr\|_{L^\infty_tL^2} \bigl\| \psi_h^6\tfrac{\p^2}{4\kappa^2-\p^2}q\bigr\|_{L_{t,x}^2}\\
&\lesssim \kappa^{-(1+2s)}\|q(0)\|_{H^s}^3.
\end{align*}

\step{Estimate for \(\err_8\).}  Arguing as for $\err_6$ again, we bound
\begin{align*}
\|\psi_h^{12} \err_8\|_{L^1_tH^{-4}} &\lesssim\|g_{12}(\vk)\|_{L^\infty_{t,x}}  \bigl\| \psi_h^6\tfrac{\p^2 q}{4\kappa^2-\p^2}\bigr\|_{L_{t,x}^2}^2
\lesssim \kappa^{-2(1+s)}\|q(0)\|_{H^s}^3.
\end{align*}

\step{Estimate for \(\err_9\).}  Using \eqref{1'} and then \eqref{apriorim} yields
\begin{align}\label{E:e9m}
\bigl\|\psi_h^6 \tfrac{\p q}{4\kappa^2 - \p^2}\bigr\|_{L_t^2 H^{s+1}}
	&\lesssim \kappa^{-1} \| q \|_{L^\infty_t H^s} + \kappa^{-1}\bigl\| \tfrac{\psi^6 q}{\sqrt{4\kappa^2-\p^2}} \|_{L_t^2 H^{s+2}} \lesssim \kappa^{-1} \|q(0)\|_{H^s}.
\end{align}
Analogously, but also breaking at frequency $N=\sqrt{\kappa}$, we find
\begin{align}
\bigl\|\psi_h^6 \tfrac{\p^2 q}{4\kappa^2 - \p^2}\bigr\|_{L_t^2 H^{-1-s}}
	&\lesssim \kappa^{-2-2s}\| q \|_{L_t^\infty H^{s}} + \kappa^{-2} N^{1-2s} \| (\psi_h^6 q)_{\leq N} \|_{L^\infty_t H^s} \notag\\
& \quad {}+ \kappa^{-1} N^{-(1+2s)} \| (\psi_h^6 q)_{> N} \|_{X_\kappa^{s+1}} \label{E:e9m'}\\
&\lesssim \kappa^{-1-\frac12(1+2s)} \|q(0)\|_{H^s}.\notag
\end{align}
Combining these bounds, we deduce that
\begin{align*}
\|\psi_h^{12} \err_9\|_{L^1_tH^{-4}} &\lesssim \kappa^2\|g_{12}\|_{L^\infty_t H^{s+1}}  \bigl\|\psi_h^6 \tfrac{\p q}{4\kappa^2 - \p^2}\bigr\|_{L_t^2 H^{s+1}}
	\bigl\|\psi_h^6 \tfrac{\p^2 q}{4\kappa^2 - \p^2}\bigr\|_{L_t^2 H^{-1-s}}  \\
&\lesssim \kappa^{-\frac12{(1+2s)}}\|q(0)\|_{H^s}^3.
\end{align*}

\step{Estimate for \(\err_{10}\).}  Our goal here is to employ \eqref{mKdV-para-2'}.  Given $\phi\in L^\infty_t H^4$, we have
\begin{align*}
[\phi \psi_h^{12} g_{12} ]' \tfrac{\p^2 q }{4\kappa^2 - \p^2}\, \tfrac{r}{4\kappa^2 - \p^2}
	= m\big[ q,r, \tfrac{(\psi_h^{12}\phi)'}{\psi_h^{12}}, \psi_h^{12} \tfrac{(2\vk+\p)g_{12}}{2\vk+\p}\bigr]
		+ m\big[ q,r, g_{12}', \psi_h^{12} \tfrac{(2\vk+\p)\phi}{2\vk+\p}\bigr],
\end{align*}
where the paraproduct $m$ has symbol
$$
m(\xi_1,\ldots,\xi_4) = \tfrac{\xi_1^2}{(4\kappa^2+\xi_1^2)(4\kappa^2+\xi_2^2)} .
$$

In this way, we see that
\begin{align*}
\kappa^2 \biggl\| \int \phi \psi_h^{12} g_{12}  \cdot \p \bigl[\tfrac{\p^2 q }{4\kappa^2 - \p^2}\, \tfrac{r}{4\kappa^2 - \p^2}\bigr]\,dx \biggr\|_{L^1_{t}}
		\lesssim \kappa^{-(1+2s)} \| \phi\|_{L^\infty_t H^4} \|q(0)\|_{H^s}^3
\end{align*}
and thence that
$$
\|\psi_h^{12} \err_{10}\|_{L^1_tH^{-4}} \lesssim \kappa^{-(1+2s)} \|q(0)\|_{H^s}^3.
$$

\step{Estimate for \(\err_{11}\).}  Arguing as for \eqref{E:e9m}, we first use \eqref{1'} and \eqref{apriorim} to see that
\begin{align*}
\bigl\| \psi_h^6 \tfrac{\p^2 q}{4\kappa^2 - \p^2} \bigr\|_{L^2_t H^{s+1}} \lesssim \| q\|_{L^\infty_t H^s} + \| q \|_{X^{s+1}_\kappa} \lesssim \|q(0)\|_{H^s} 
\end{align*}
and
\begin{align*}
\bigl\| \psi_h^6  \tfrac{\p^3  q}{4\kappa^2 - \p^2}\bigr\|_{L^2_t H^{-(s+1)}} \lesssim \kappa^{-(1+2s)}\| q\|_{L^\infty_t H^s} + \kappa^{-(1+2s)} \| q \|_{X^{s+1}_\kappa}
	\lesssim \kappa^{-(1+2s)}  \|q(0)\|_{H^s}.
\end{align*}
Thus,
\begin{align*}
\|\psi_h^{12} \err_{11}\|_{L^1_tH^{-4}} &\lesssim  \|g_{12}(\vk)\|_{L_t^\infty H^{s+1}} \bigl\| \psi_h^6 \tfrac{\p^2 q}{4\kappa^2 - \p^2} \bigr\|_{L^2_t H^{s+1}}
		\bigl\| \psi_h^6  \tfrac{\p^3  q}{4\kappa^2 - \p^2}\bigr\|_{L^2_t H^{-(s+1)}} \\
&\lesssim \kappa^{-(1+2s)}\|q(0)\|_{H^s}^3.
\end{align*}

\step{Estimate for \(\err_{12}\).}  We first note that \eqref{rho-Hs} and \eqref{rho-X-1} imply
\begin{align*}
\normMK{(2\vk+\p)\bigl(\phi[1+\vr(\vk)]\bigr)} \lesssim 1 + \normMK{q} \lesssim 1
\end{align*}
for any $\phi\in L^\infty_t H^4$ of unit norm.  Thus, it follows from \eqref{rem Dual cancel mKdV 2} that
\begin{align*}
\|\psi_h^{12} \err_{12}\|_{L^1_tH^{-4}} &\lesssim \kappa^{-(2s+1)}\|q(0)\|_{H^s}^3.
\end{align*}

\step{Estimate for \(\err_{13}\).} Arguing as for $\err_{12}$ and using \eqref{rem Dual cancel mKdV} in place of \eqref{rem Dual cancel mKdV 2}, we get
\begin{align*}
\|\psi_h^{12} \err_{13}\|_{L^1_tH^{-4}} &\lesssim \kappa^{-(2s+1)}\|q(0)\|_{H^s}^3.
\end{align*}

\step{Estimate for \(\err_{14}\).}  Using \(L^1\subset H^{-4}\) together with \eqref{g12-Hs} and \eqref{rho rem mKdV}, we get
\begin{align*}
\|\psi_h^{12} \err_{14}\|_{L^1_tH^{-4}} &\lesssim \kappa^2 \|g_{12}\|_{L_{t,x}^\infty} \bigl\| \vr(\pm\kappa)\sbrack{\geq 4} \psi_h^{12}\bigr\|_{L_{t,x}^1}
\lesssim \kappa^{-\frac12-(2s+1)}\|q(0)\|_{H^s}^5.
\end{align*}

\step{Estimate for \(\err_{15}\).} The argument here is essentially a recapitulation of the proof of \eqref{rem Dual cancel mKdV 3}.  For example, from \eqref{gamma 6 is okay}, we have 
\begin{align*}
\kappa^3 \bigl\|\psi_h^{12} g_{12}(\vk) \vr\sbrack{\geq 6}(\pm\kappa)\bigr\|_{L^1_tH^{-4}}
	&\lesssim \kappa^{-(1+2s)} \|g_{12}\|_{L_{t,x}^\infty}  \|q(0)\|_{H^s}^6 \lesssim \kappa^{-(1+2s)}\|q(0)\|_{H^s}^7.
\end{align*}

In order to repeat the treatment of the $\gamma\sbrack{4}$ terms given previously, we need one additional piece of information, namely, that $f$ defined by
\begin{align*}
\tfrac{f}{2\kappa-\p} = \phi\, g_{12}(\vk) \tfrac{q}{2\kappa-\p}
\end{align*}
satisfies
\begin{align*}
\| f\|_{L^\infty_t H^s} \lesssim \| q\|_{L^\infty_t H^s} \qtq{and} \normMK{f} \lesssim \normMK{q}
\end{align*}
for every  $\phi\in L^\infty_t H^4$ of unit norm.  These assertions follow readily from \eqref{g12-Hs}, \eqref{X-Product-2}, and \eqref{g12-X}.  Thus we may conclude that
\begin{align*}
\|\psi_h^{12} \err_{15}\|_{L^1_tH^{-4}} &\lesssim \kappa^{-(2s+1)}\|q(0)\|_{H^s}^5.
\end{align*}

\step{Estimate for \(\err_{16}\).} Breaking at frequency $N=\sqrt \kappa$ and using \eqref{apriorim}, we find
\begin{align*}
\| \psi_h^6 q \|_{L_t^2 H^{s+1}}&\lesssim N \|q\|_{L_t^\infty H^s} + \tfrac{\kappa}{N} \|q\|_{X_\kappa^{s+1}}\lesssim \kappa^{\frac12}\|q(0)\|_{H^s}.
\end{align*}

Thus, arguing by duality and using \eqref{g12-Hs}, we estimate 
\begin{align*}
\|\psi_h^{12} \err_{16}\|_{L^1_tH^{-4}}&\lesssim \kappa^{-2} \|q\|_{L_t^\infty H^s} \|\psi_h^6 q\|_{L_t^2 H^{s+1}}^2 \sup_\phi\|\phi[1+\vr(\vk)]\|_{L_t^\infty H^{s+1}} \\
&\lesssim \kappa^{-1}\|q(0)\|_{H^s}^3.
\end{align*}

Combining our estimates for all the error terms we deduce \eqref{DiffFlowGoal}, which then completes the proof of the mKdV case of Proposition~\ref{prop:diff-flows}.
\end{proof}


\section{Well-posedness}\label{S:8}
In this section we prove Theorem~\ref{thrm:main}. While we have already established the necessary prerequisites to obtain global well-posedness in \(H^s\) for $-\frac12<s<0$, we begin this section with one additional equicontinuity result that will be applied to yield well-posedness  at higher regularity. 

This equicontinuity relies on a certain macroscopic conservation law, which we introduce through its density 
\[
\wrho(\vk) := qr - 2\vk\rho(\vk).
\]
That this density satisfies a conservation law follows readily from Corollary~\ref{C:microscopic} and the conservation of mass. The associated microscopic conservation law will be essential for proving local smoothing estimates at positive regularity in the next section.

For later reference, we note that
\begin{align}\label{wrho2}
\wrho\sbrack{2}(\vk) = \tfrac12\Bigl(q \cdot \tfrac{r'}{2\vk+ \p} -\tfrac{q'}{2\vk - \p}\cdot r \Bigr).
\end{align}

Using $\wrho$, we prove the following analogue of Proposition~\ref{prop:alpha}:
\begin{prop}\label{p:higher} Let \(0\leq \sigma< \frac12\). Then there exists \(\delta>0\) so that for any \(q(0)\in \Schwartz\) satisfying \(\|q(0)\|_{L^2}\leq \delta\), the solution \(q(t)\) of \eqref{NLS} or \eqref{mKdV} satisfies
\eq{equi +}{
\|q(t)'\|_{H^{\sigma-1}_\kappa}^2 \lesssim \|q(0)'\|_{H^{\sigma-1}_\kappa}^2 + \kappa^{2\sigma-1}\delta^2\|q(0)\|_{L^2}^2,
}
uniformly for \(t\in \R\) and \(\kappa\geq 1\).
\end{prop}

\begin{proof}
Using \eqref{g12-Hs}, \eqref{g12-LO}, and \eqref{rho-Linfty}, we get
\begin{align*}
\|g_{12}(\vk)\|_{L^2}\lesssim \vk^{-1} \|q\|_{L^2}, \quad \|g_{12}\sbrack{\geq 3}(\vk)\|_{L^2}\lesssim \vk^{-2} \|q\|_{L^2}^3, \quad \|\vr(\vk)\|_{L^\infty}\leq \vk^{-1}\|q\|_{L^2}^2.
\end{align*}
Consequently, using \eqref{more sbrack'} and \eqref{rho higher} we obtain
\eq{ET1 Sob +}{
\|\big(\tfrac{g_{12}(\vk)}{2 + \gamma(\vk)}\big)\sbrack{\geq 3}\|_{L^2}\lesssim \vk^{-2}\delta^2\|q\|_{L^2} \qtq{and so} \|\wrho\sbrack{\geq 4}(\vk)\|_{L^1} \lesssim \vk^{-1}\delta^2\|q\|_{L^2}^2
}
whenever \(0<\delta\ll1\) is sufficiently small. Employing \eqref{wrho2} and \eqref{ET1 Sob +} we get
\[
\pm \Re\int\wrho(x;\vk)\,dx = \|q'\|_{H^{-1}_\vk}^2 + \bigO\bigl(\vk^{-1}\delta^2\|q\|_{L^2}^2\bigr).
\]
If \(\sigma = 0\), we simply set $\vk=\kappa$. If \(0<\sigma< \tfrac12\), we apply the estimate \eqref{EquivNorm} to obtain
\[
\int_\kappa^\infty \vk^{2\sigma}\Bigg(\pm \Re\int\wrho(x;\vk)\,dx\Bigg)\,\tfrac{d\vk}\vk \approx \|q'\|_{H^{\sigma-1}_\kappa}^2 + \bigO\bigl(\kappa^{2\sigma-1}\delta^2\|q\|_{L^2}^2\bigr).
\]
As the mass and left-hand sides in these estimates are conserved under both \eqref{NLS} and \eqref{mKdV}, the claim \eqref{equi +} now follows.
\end{proof}

\begin{proof}[Proof of Theorem~\ref{thrm:main}] In view of the history discussed in the introduction, it suffices to treat regularities $-\frac12<s<0$ for \eqref{NLS} and $-\frac12<s<\frac14$ for \eqref{mKdV}. With the tools at our disposal, we are able to give a uniform treatment of both equations over the range $(-\frac12, \frac12)$, so this is what we do. As the arguments for \eqref{NLS} and \eqref{mKdV} are identical, we provide details in the case of \eqref{NLS}.

We first consider initial data \(q\in H^s\), where \(-\frac12<s<0\). Let \(0<\delta\ll1\) be sufficiently small and, rescaling according to \eqref{scaling}, assume that \(q\in \Bd\). Let \(\{q_n\}_{n\geq 1}\subset\BdS\) so that \(q_n\rightarrow q\) in \(H^s\) as \(n\rightarrow\infty\).

In view of Propositions~\ref{prop:Equicontinuity} and \ref{prop:Tightness}, the set
\[
Q:= \left\{e^{tJ\nabla H_\NLS}q_n:n\geq 1,\;t\in [-1,1]\right\}
\]
is equicontinuous and tight in \(H^s\). Further, by Proposition~\ref{prop:alpha} we may find some \(C = C(s)\geq 1\) so that \(Q\subset B_{C\delta}\cap\Schwartz\).

For fixed \(\vk\geq 4\) let \(g_{12}(\cdot) = g_{12}(\vk;\cdot)\) and \(\kappa\geq 2\vk\). Let \(R\geq 1\), \(\phi_R\) be as in Section~\ref{sec:Tight}, and \(\chi_R\in \Schwartz\) be a non-negative function so that \(1\leq \phi_R^2 + \chi_R^2\). We then bound
\begin{align*}
&\|g_{12}(e^{tJ\nabla H_{\NLS}}q_n) - g_{12}(e^{tJ\nabla H_{\NLS}}q_m)\|_{L^\infty_tH^{s+1}}\\
&\qquad \lesssim \|g_{12}(e^{tJ\nabla H_{\NLS}^\kappa}q_n) - g_{12}(e^{tJ\nabla H_{\NLS}^\kappa}q_m)\|_{L^\infty_tH^{s+1}}\\
&\qquad\qquad + \sup\limits_{q\in Q^*}\|\chi_Rg_{12}(e^{tJ\nabla (H_{\NLS} - H_{\NLS}^\kappa)}q) - \chi_Rg_{12}(q)\|_{L^\infty_tH^{s+1}}\\
&\qquad\qquad + \sup\limits_{n\geq 1}\|\phi_Rg_{12}(e^{tJ\nabla H_{\NLS}}q_n)\|_{L^\infty_tH^{s+1}},
\end{align*}
where the set
\[
Q^* := \left\{e^{J\nabla\left(tH_\NLS + sH_{\NLS}^{\kappa}\right)}q_n:n\geq 1,\;\kappa\geq 2\vk,\;t,s\in [-1,1]\right\}.
\]
By Propositions~\ref{prop:alpha} and~\ref{prop:kappa-flows} we have \(Q^*\subset B_{C\delta}\cap\Schwartz\), while by Proposition~\ref{prop:Equicontinuity}, \(Q^*\) is equicontinuous in \(H^s\).

By Proposition~\ref{prop:kappa-flows} and the diffeomorphism property of Proposition~\ref{prop:g} we have
\[
\lim\limits_{n,m\rightarrow\infty}\|g_{12}(e^{tJ\nabla H_{\NLS}^{\kappa}}q_n) - g_{12}(e^{tJ\nabla H_{\NLS}^{\kappa}}q_m)\|_{L^\infty_tH^{s+1}} = 0.
\]
Using Proposition~\ref{prop:diff-flows} we obtain
\begin{align*}
&\lim\limits_{\kappa\rightarrow\infty}\sup\limits_{q\in Q^*}\|\chi_Rg_{12}(e^{tJ\nabla (H_{\NLS} - H_{\NLS}^{\kappa})}q) - \chi_Rg_{12}(q)\|_{L^\infty_tH^{s+1}}\\
&\qquad \lesssim_R \lim\limits_{\kappa\rightarrow\infty}\sup\limits_{q\in Q^*}\sup\limits_{h\in \R}\|\psi_h^{12}g_{12}(e^{tJ\nabla (H_{\NLS} - H_{\NLS}^{\kappa})}q) - \psi_h^{12}g_{12}(q)\|_{L^\infty_tH^{s+1}} = 0.
\end{align*}
Finally, from the estimate \eqref{g12-Tight} and the fact that \(Q\subset B_{C\delta}\cap\Schwartz\) is tight we have
\[
\lim\limits_{R\rightarrow\infty}\sup\limits_{n\geq 1}\|\phi_Rg_{12}(e^{tJ\nabla H_{\NLS}}q_n)\|_{L^\infty_tH^{s+1}} = 0.
\]

Thus, \(\{g_{12}(e^{tJ\nabla H_{\NLS}}q_n)\}\) is Cauchy in \(\Cont([-1,1];H^{s+1})\) and from the diffeomorphism property we conclude that \(\{e^{tJ\nabla H_{\NLS}}q_n\}\) is Cauchy in \(\Cont([-1,1];H^s)\). This yields local well-posedness of \eqref{NLS} in \(H^s\) on the time interval \([-1,1]\).

From the estimate \eqref{APBound} with \(\kappa=1\) we obtain the estimate
\[
\|e^{tJ\nabla H_{\NLS}}q\|_{H^s}\leq C\|q\|_{H^s},
\]
uniformly for \(t\in \R\) and \(q\in \BdS\). Using this bound we may iterate the local well-posedness argument to complete the proof of global well-posedness in \(H^s\).

Now consider initial data \(q\in H^\sigma\), where \(0\leq \sigma<\frac12\). Let \(0<\delta\ll1\) be sufficiently small and \(\{q_n\}_{n\geq 1}\) be a sequence of Schwartz functions so that \(q_n\to q\) in \(H^\sigma\) as \(n\to\infty\). After possibly rescaling, assume that \(\|q_n\|_{L^2}\leq \delta\) for all \(n\geq 1\).

Applying our well-posedness result with \(s = -\frac14\), the sequence of solutions \(\{e^{tJ\nabla H_{\NLS}}q_n\}\) is Cauchy in \(\Cont([-1,1];H^{-\frac14})\). Applying the estimate \eqref{equi +}, we see that the corresponding set \(Q\) is equicontinuous in \(H^\sigma\) and hence the sequence \(\{e^{tJ\nabla H_{\NLS}}q_n\}\) is also Cauchy in \(\Cont([-1,1];H^\sigma)\). This gives local well-posedness in~\(H^\sigma\).

Employing the estimate \eqref{equi +} with \(\kappa = 1\), and the conservation of mass, we obtain the estimate
\[
\|e^{tJ\nabla H_{\NLS}}q\|_{H^\sigma}\lesssim \|q\|_{H^\sigma},
\]
uniformly for \(t\in \R\) and \(q\in \Schwartz\) satisfying \(\|q\|_{L^2}\leq \delta\). This suffices to complete the proof of global well-posedness in \(H^\sigma\).
\epf

\section{Proof of Theorems \ref{thrm:LS_N} and \ref{thrm:LS_m}}\label{S:9}

In this section we prove Theorems~\ref{thrm:LS_N} and~\ref{thrm:LS_m}. We start by considering \eqref{NLS}:

\begin{proof}[Proof of Theorem~\ref{thrm:LS_N}] The estimate \eqref{E:NLS LS apb} follows from \eqref{NLS-LS} and rescaling. It remains to prove the continuity statement in Theorem~\ref{thrm:LS_N}.

Let \(0<\delta\ll1\) be sufficiently small and, by rescaling, assume the initial data \(q(0)\in \Bd\). Let \(\{q_n(0)\}_{n\geq 1}\subseteq\BdS\) so that \(q_n(0)\to q(0)\) in \(H^s\) as \(n\to\infty\) and denote the corresponding solutions by \(q(t) = e^{tJ\nabla H_{\NLS}}q(0)\) and \(q_n(t) = e^{tJ\nabla H_{\NLS}}q_n(0)\). It suffices to prove that \(q_n\rightarrow q\) in \(X^{s+\frac12}\) as \(n\to \infty\).

Decomposing into low and high frequencies, we may bound
\begin{align*}
\|q_n - q_m\|_{X^{s+\frac12}} &\leq \|P_{\leq \kappa}(q_n - q_m)\|_{L^\infty_tH^{s+\frac12}} + 2\sup\limits_{n\geq 1}\sup\limits_{h\in \R}\|P_{>\kappa}(\psi_h^6q_n)\|_{L^2_tH^{s+\frac12}}\\
&\lesssim \sqrt{\kappa}\|q_n - q_m\|_{L^\infty_tH^s} + \sup\limits_{n\geq 1}\sup\limits_{h\in \R}\|(\psi_h^6q_n)'\|_{L^2_tH^{s-\frac12}_\kappa}.
\end{align*}
As the set \(\{q_n(0)\}_{n\geq 1}\) is equicontinuous in \(H^s\), we may apply \eqref{NLS-LS-HF} from Proposition~\ref{prop:NLS-LS} to obtain
\[
\lim\limits_{\kappa\rightarrow\infty}\sup\limits_{n\geq 1}\sup\limits_{h\in \R}\|(\psi_h^6q_n)'\|_{L^2_tH^{s-\frac12}_\kappa} = 0.
\]
Finally, from Theorem~\ref{thrm:main} we have \(q_n\rightarrow q\) in \(\Cont([-1,1];H^s)\) as \(n\to\infty\), which completes the proof that \(q_n\rightarrow q\) in \(X^{s+\frac12}\).
\end{proof}

The corresponding result for \eqref{mKdV}, Theorem~\ref{thrm:LS_m}, is proved almost identically: When \(-\frac12<s<0\), we replace Proposition~\ref{prop:NLS-LS} by Proposition~\ref{prop:mKdV-LS}, whereas at higher regularity we use the following:

\begin{prop}\label{P:ls}
Let \(0\leq \sigma<\frac12\). Then there exists \(\delta>0\) so that for any \(q(0)\in \Schwartz\) satisfying \(\|q(0)\|_{L^2}\leq \delta\), the solution \(q(t)\) of \eqref{mKdV} satisfies the estimate
\eq{mKdV-LS +}{
\|q\|_{X^{\sigma+1}}\lesssim \|q(0)\|_{H^\sigma}.
}

Further, we have the high frequency estimate
\eq{mKdV-LS-HF +}{
\|(\psi_h^6q)''\|_{L^2_tH^{\sigma-1}_\kappa}^2\lesssim \|q(0)'\|_{H^{\sigma-1}_\kappa}^2 + \kappa^{2\sigma-1}\|q(0)\|_{L^2}^2,
}
uniformly for \(h\in \R\) and \(\kappa\geq 1\).
\end{prop}

To prove Proposition~\ref{P:ls}, we use the microscopic conservation law for \(\wrho(\vk)\),
\[
\p_t\wrho + \p_x\wj = 0,
\]
where the current
\[
\wj (\vk):= (qr)'' - 3(q'r' + q^2r^2) - 2\vk j_{\mKdV}(\vk).
\]
We will first establish analogues of \eqref{jmKdV-quad}, \eqref{rhoA}, and \eqref{jmKdV-err}.  We then use these as in the proof of Proposition~\ref{prop:mKdV-LS} to derive \eqref{mKdV-LS +} and \eqref{mKdV-LS-HF +}.

We start with the analogues of the estimates \eqref{jmKdV-quad} and \eqref{rhoA}. 

\begin{lem}\label{L1}
Let \(q\in \Schwartz\) satisfy \(\|q\|_{L^2}\leq \delta\) and let \(\Psi_h\) be as in \eqref{phi-def}. Then 
\begin{align}
\left|\Re\int \wrho(x;\vk)\,\Psi_h(x)\,dx\right|&\lesssim \|q'\|_{H_\vk^{-1}}^2 + \vk^{-1}\|q\|_{L^2}^2,\label{rhoA+}\\
\Re \int \wj\sbrack 2(x;\vk)\,\psi_h^{12}(x)\,dx & = \mp 3\|(\psi_h^6q)''\|_{H^{-1}_\vk}^2 +\bigO\Big( \|q'\|_{H^{-1}_\vk}^2+ \vk^{-2}\|q\|^2_{L^2}\Big)\label{jmKdV-quad +}\\
&\quad + \bigO\Bigl(\|(\psi_h^6 q)''\|_{H^{-1}_\vk}\bigl(\|q'\|_{H^{-1}_\vk} + \vk^{-1}\|q\|_{L^2}\bigr)\Bigr),\notag
\end{align}
uniformly for  \(\vk\geq 1\) and $h\in\R$.
\end{lem}

\begin{proof}
Using \eqref{wrho2}, we estimate
\begin{align*}
\left|\Re\int \wrho\sbrack 2(x;\vk)\,\Psi_h(x)\,dx\right| &=\left|\Re \int \tfrac{\xi^2}{4\vk^2+\xi^2} \hat q(\xi) \overline{\widehat{q\Psi_h}}(\xi)\, d\xi\right|\\
&\lesssim \|q'\|_{H^{-1}_\vk}\|(\Psi_h q)'\|_{H^{-1}_\vk}\lesssim \|q'\|_{H^{-1}_\vk}^2 + \vk^{-1}\|q'\|_{H^{-1}_\vk}\|q\|_{L^2}.
\end{align*}
Combining this with \eqref{ET1 Sob +} yields \eqref{rhoA+}.

We turn now to \eqref{jmKdV-quad +}.  The quadratic part of the current satisfies
\[
\wj\sbrack 2(\vk) = \bigl(\wbR[q,r]\bigr)'' - 3\wbR[q',r'],
\]
where the paraproduct \(\wbR[q,r]\) has symbol
\[
\wR(\xi,\eta) := \tfrac{-i\xi}{2(2\vk - i\xi)} + \tfrac{i\eta}{2(2\vk + i\eta)}.
\]
Notice also that \eqref{wrho2} shows
\[
\wrho\sbrack 2(x;\vk) = \wbR[q,r](x) = \tfrac1{2\pi}\int \wR(\xi,\eta)\hat q(\xi)\hat r(\eta) e^{ix(\xi + \eta)}\,d\xi\,d\eta.
\]

Taking the real part we have
\[
\Re\int \wbR[(\psi_h^6q)',(\psi_h^6r)']\,dx = \pm \|(\psi_h^6q)''\|_{H^{-1}_\vk}^2,
\]
and hence we may write
\begin{align}
\label{smore}\Re\int \wj\sbrack 2(\vk)\,\psi_h^{12}\,dx &= \mp 3\|(\psi_h^6q)''\|_{H^{-1}_\vk}^2 + \Re \int \wbR[q,r](\psi^{12}_h)''\,dx\\
&\quad - 3\Re\int \Bigl(\psi_h^{12}\wbR[q',r'] - \wbR[(\psi_h^6q)',(\psi_h^6r)']\Bigr)\,dx.\notag
\end{align}

Proceeding as in the proof of \eqref{rhoA+}, we may bound the second term on \(\RHS{smore}\) by
\[
\left|\Re\int \wrho\sbrack 2(x;\vk)\,(\psi^{12}_h)''(x)\,dx\right|\lesssim \|q'\|_{H^{-1}_\vk}^2 + \vk^{-2}\|q\|_{L^2}^2.
\]

The remaining term on \(\RHS{smore}\) is given by
\begin{align*}
&-3 \Re\int \Bigl(\psi_h^{12}\wbR[q',r'] - \wbR[(\psi_h^6q)',(\psi_h^6r)']\Bigr)\,dx\\
&\quad= 3 \Re \int \Bigl( [\psi_h^6, \tfrac{\p^3}{4\vk^2 - \p^2}]q\cdot (\psi_h^6 r)' - [\psi_h^6, \tfrac{\p^3}{4\vk^2 - \p^2}]q\cdot (\psi_h^6)' r -\tfrac{(\psi^6_hq)'''}{4\vk^2-\p^2}(\psi_h^6)' r\Bigr)\,dx.
\end{align*}
Integrating by parts, we may bound
\begin{align*}
&\left|\Re\int\Bigl(\psi_h^{12}\wbR[q',r'] - \wbR[(\psi_h^6q)',(\psi_h^6r)']\Bigr)\,dx\right|\\
&\quad \lesssim\|(\psi_h^6 q)''\|_{H^{-1}_\vk} \bigl(\|q'\|_{H^{-1}_\vk} + \vk^{-1}\|q\|_{L^2}\bigr) +\|q'\|_{H^{-1}_\vk}^2 + \vk^{-2}\|q\|^2_{L^2},
\end{align*}
which completes the proof of \eqref{jmKdV-quad +}.
\end{proof}

It remains to prove an analogue of the estimate \eqref{jmKdV-err}. To this end, we denote
\[
\normM{q}^2 := \|q\|_{X^1}^2 + \|q\|_{L^\infty_tL^2}^2,
\]
which corresponds to the local smoothing norm in the case \(s = 0\).

\begin{lem}\label{L2}
Let \(q\in \Cont([-1,1];\Schwartz)\) satisfy \(\|q(0)\|_{L^2}\leq \delta\). We have
\eq{jmKdV-err +}{
\left\|\Re\int \wj\sbrack{\geq 4}(\vk)\,\psi_h^{12}\,dx\right\|_{L^1_t}\lesssim \vk^{-1}\delta^2\normM{q}^2,
}
uniformly for  \(\vk\geq 1\) and $h\in\R$.
\end{lem}
\begin{proof}
We first establish several variants of the estimates in Corollary~\ref{C:ET}, inspired by the decomposition \eqref{jdecomp+} below. Using that
\begin{align}\label{GN}
\|f\|_{L^\infty}\lesssim \|f'\|_{L^2}^{\frac12} \|f\|_{L^2}^{\frac12},
\end{align}
we obtain
\eq{oomph}{
\|\psi^3_h q\|_{L^4_tL^\infty_x} \lesssim\|(\psi^3_h q)'\|_{L^2_{t,x}}^{\frac12}\|\psi^3_h q\|_{L^\infty_tL^2_x}^{\frac12}\lesssim \delta^{\frac12} \normM{q}^{\frac12}.
}
Thus, using \eqref{g12 1 3} and \eqref{psi up down}, we may bound
\begin{align*}
\|\psi^3_h g_{12}\sbrack 1(\vk)\|_{L^4_tL^\infty} \lesssim \vk^{-1}\delta^{\frac12} \normM{q}^{\frac12}.
\end{align*}
From \eqref{rho-Linfty} we get
\begin{align}\label{rho-Linfty'}
\|\vr(\vk)\|_{L_{t,x}^\infty}\lesssim \vk^{-1}\delta^2,
\end{align}
and thence using \eqref{psi up down} again, we find
\begin{align*}
\|\psi^3_h g_{12}\sbrack {\geq 3}(\vk)\|_{L^4_tL^\infty}&\lesssim \vk^{-1}\|\psi^3_h q\|_{L^4_tL^\infty_x} \|\vr(\vk)\|_{L_{t,x}^\infty}\lesssim \vk^{-2}\delta^{\frac52}\normM{q}^{\frac12}.
\end{align*}
From the identity \eqref{QuadraticID} and the estimate \eqref{rho-Linfty'}, taking \(0<\delta\ll1\) sufficiently small we obtain
\[
\|\psi^6_h\vr(\vk)\|_{L^2_tL^\infty} \lesssim \|\psi^3_hg_{12}\|_{L^4_tL^\infty}\|\psi^3_hg_{21}\|_{L^4_tL^\infty} \lesssim \vk^{-2}\delta \normM{q}.
\]
Consequently, using \eqref{g12-ID} we get
\begin{align*}
\|\psi^6_h g_{12}\sbrack{\geq 3}(\vk)\|_{L^2_tH^1_\vk}
&\lesssim \|q\|_{L^\infty_tL^2}\|\psi^6_h\vr(\vk)\|_{L^2_tL^\infty}+\bigl\| (\psi_h^6)'\tfrac{q\vr(\vk)}{2\vk-\p}\bigr\|_{L^2_{t,x}} \\
&\lesssim \vk^{-2}\delta^2\normM q + \vk^{-1} \|q\|_{L^\infty_t L^2} \|\vr(\vk)\|_{L_{t,x}^\infty}\lesssim \vk^{-2}\delta^2\normM q.
\end{align*}
Recalling the identity \eqref{more sbrack'} and using \eqref{rho-ID} to write \(\vr'\) in terms of \(q,r,g_{12},g_{21}\), we may apply these estimates to obtain
\eq{ET1 LS +}{
\|\psi^6_h\big(\tfrac{g_{12}(\vk)}{2 + \vr(\vk)}\big)\sbrack{\geq 3}\|_{L^2_tH^1_\vk} \lesssim\vk^{-2}\delta^2\normM q.
}

Using \eqref{g12-ID} and \eqref{psi up down} again, we may bound
\[
\|\psi_h^9 g_{12}\sbrack{\geq3}(\vk)\|_{L^{\frac43}_tL^\infty}\lesssim \vk^{-1}\|\psi_h^3q\|_{L^4_tL^\infty}\|\psi_h^6\vr(\vk)\|_{L^2_tL^\infty}\lesssim \vk^{-3}\delta^{\frac32}\normM{q}^{\frac32},
\]
Using the identity \eqref{E:vr gr 4}, we estimate
\begin{align*}
\|\psi_h^{12}\vr\sbrack{\geq 4}\|_{L^1_tL^\infty} &\lesssim \|\psi_h^6\vr\|_{L^2_tL^\infty}^2  +  \|\psi_h^9g_{12}\sbrack{\geq 3}\|_{L^{\frac43}_tL^\infty}\|\psi_h^3g_{21}\|_{L^4_tL^\infty} +\|\psi_h^3g_{12}\sbrack 1\|_{L^4_tL^\infty}\|\psi_h^9g_{21}\sbrack{\geq 3}\|_{L^{\frac43}_tL^\infty}\\
&\lesssim \vk^{-4}\delta^2 \normM{q}^2.
\end{align*}
Applying \eqref{g12-ID} once again we obtain
\begin{align*}
\|\psi_h^{12}g_{12}\sbrack{\geq 5}(\vk)\|_{L^1_tH^1_\vk} &\lesssim \|q\|_{L^\infty_tL^2}\|\psi_h^{12}\vr\sbrack{\geq 4}(\vk)\|_{L^1_tL^\infty}\lesssim \vk^{-4}\delta^3\normM{q}^2.
\end{align*}
Finally, we use the identity \eqref{more sbrack''} with the above estimates, as well as \eqref{rho-ID} to replace \(\vr'\), to obtain
\eq{ET2 LS +}{
\|\psi_h^{12}\big(\tfrac{g_{12}(\vk)}{2 + \vr(\vk)}\big)\sbrack{\geq 5}\|_{L^1_tH^1_\vk} \lesssim  \vk^{-4}\delta^3\normM{q}^2.
}

We turn now to estimating the current.  Using Corollary~\ref{C:microscopic}, we have
\begin{align}\label{jdecomp+}
\wj\sbrack{\geq 4} (\vk)&= 
-2\vk q''\cdot\big(\tfrac{g_{21}}{2 + \vr}\big)\sbrack{\geq 3} + 2\vk r''\cdot \big(\tfrac{g_{12}}{2 + \vr}\big)\sbrack{\geq 3}\\
&\quad - 4\vk^2(2\vk + \p) q\cdot\big(\tfrac{g_{21}}{2 + \vr}\big)\sbrack{\geq 3} + 4\vk^2(2\vk - \p) r\cdot \big(\tfrac{g_{12}}{2 + \vr}\big)\sbrack{\geq 3}\notag\\ 
&\quad + 4\vk q^2r\tfrac{g_{21}}{2 + \vr} - 4\vk r^2q\tfrac{g_{12}}{2 + \vr} - 3q^2r^2.\notag
\end{align}

For the first two terms we apply the estimate \eqref{ET1 LS +} to bound
\begin{align*}
&\left\|\int 2\vk q''\cdot\big(\tfrac{g_{21}}{2 + \vr}\big)\sbrack{\geq 3}\,\psi_h^{12}\,dx\right\|_{L^1_t}\\
&\qquad\lesssim \vk \|\psi_h^6 q''\|_{L^2_tH^{-1}_\vk} \|\psi_h^6\big(\tfrac{g_{21}}{2 + \vr}\big)\sbrack{\geq 3}\|_{L^2_tH^1_{\vk}}\\
&\qquad\lesssim \vk^{-1}\delta^2\normM q\bigl(\|(\psi_h^6 q)''\|_{L^2_tH^{-1}_\vk} + \|q'\|_{L^\infty_tH^{-1}_\vk}+ \vk^{-1}\|q\|_{L_t^\infty L^2}\big),
\end{align*}
which is acceptable.

We bound the sextic and higher order contributions of the remaining terms using \eqref{ET2 LS +} and \eqref{ET1 LS +}, as follows:
\begin{align*}
\left\|\int 4\vk^2(2\vk + \p) q\cdot  \big(\tfrac{g_{21}}{2 + \vr}\big)\sbrack{\geq 5}\ \psi_h^{12}\,dx\right\|_{L^1_t}
& \lesssim \vk^2 \|q\|_{L^\infty_tL^2}\|\psi_h^{12}\big(\tfrac{g_{12}}{2 + \vr}\big)\sbrack{\geq 5}\|_{L^1_tH^1_\vk}\\
&\lesssim \vk^{-2}\delta^4\normM{q}^2,\\
\left\|\int 4\vk q^2r\cdot  \big(\tfrac{g_{21}}{2 + \vr}\big)\sbrack{\geq 3}\ \psi_h^{12}\,dx\right\|_{L^1_t}
& \lesssim  \vk\|q\|_{L^\infty_tL^2}\|\psi_h^3q\|_{L^4_tL^\infty}^2\|\psi_h^6\big(\tfrac{g_{12}}{2 + \vr}\big)\sbrack{\geq 3}\|_{L^2_{t,x}}\\
&\lesssim \vk^{-2}\delta^4\normM{q}^2.
\end{align*}

It remains to consider the contributions of
\begin{align*}
\err_1 &:= 4\vk q^2r\cdot\big(\tfrac{g_{21}}{2 + \vr}\big)\sbrack 1 - 4\vk r^2q\cdot\big(\tfrac{g_{12}}{2 + \vr}\big)\sbrack 1 - 2q^2r^2,\\
\err_2 &:= - 4\vk^2(2\vk + \p) q\cdot\big(\tfrac{g_{21}}{2 + \vr}\big)\sbrack{3} + 4\vk^2(2\vk - \p) r\cdot \big(\tfrac{g_{12}}{2 + \vr}\big)\sbrack{3} - q^2r^2.
\end{align*}

For $\err_1$ we use the identity \eqref{more sbrack} to write
\[
\Re \err_1 = q^2r \tfrac{r''}{4\vk^2 - \p^2} + qr^2\tfrac{q''}{4\vk^2 - \p^2},
\]
so we may bound
\begin{align*}
&\left\|\Re \int \err_1\ \psi_h^{12}\,dx \right\|_{L^1_t}\\
&\qquad \lesssim \|q\|_{L^\infty_tL^2}\|\psi_h^3q\|_{L^4_tL^\infty}^2\Bigl(\vk^{-1}\|(\psi_h^6 q)''\|_{L^2_tH^{-1}_\vk} + \bigl\|[\psi_h^6, \tfrac{\p^2}{4\vk^2-\p^2}]q\bigr\|_{L^2_{t,x}}\Bigr)\\
&\qquad \lesssim \delta^2\normM q \Bigl(\vk^{-1}\|(\psi_h^6 q)''\|_{H^{-1}_\vk} + \vk^{-2}\normM{q} \Bigr),
\end{align*}
which is acceptable.

Recalling the identities \eqref{more sbrack}, \eqref{g12 1 3}, \eqref{g21 1 3}, \eqref{gamma 2}, we may integrate by parts to obtain
\begin{align*}
\int \err_2\,\psi_h^{12}\,dx &=\int 4\vk^2(2\vk + \p)q\cdot[\psi_h^{12},\tfrac1{2\vk +\p} ] \bigl( r\cdot \tfrac q{2\vk - \p}\cdot \tfrac r{2\vk + \p}\bigr)\,dx\\
&\quad + \int 4\vk^2(2\vk - \p)r\cdot [\psi_h^{12},\tfrac1{2\vk - \p}]\bigl(q\cdot \tfrac r{2\vk + \p}\cdot \tfrac q{2\vk - \p}\bigr)\,dx\\
&\quad + \int 6\vk^2 \Bigl(\tfrac {q'}{2\vk - \p}\cdot r - \tfrac {r'}{2\vk + \p}\cdot q\Bigr) \tfrac r{2\vk + \p}\cdot  \tfrac q{2\vk - \p}\,\psi_h^{12}\,dx\\
&\quad+\int 2\vk qr\Bigl(\tfrac{q'}{2\vk - \p}\cdot\tfrac r{2\vk + \p}- \tfrac q{2\vk - \p}\cdot\tfrac{r'}{2\vk + \p}\Bigr)\psi_h^{12}\,dx\\
&\quad - \int 2\vk^2 \Bigl(q' \cdot \tfrac r{2\vk + \p}- r' \cdot \tfrac q{2\vk - \p}\Bigr) \tfrac r{2\vk + \p}\cdot \tfrac q{2\vk - \p}\,\psi_h^{12}\,dx\\
&\quad + \int qr\cdot\tfrac{q'}{2\vk - \p}\cdot\tfrac{r'}{2\vk + \p}\,\psi_h^{12}\,dx.
\end{align*}
We then bound each of these terms by applying \eqref{psi up down} with \eqref{oomph} as follows:
\begin{align*}
&\left\|\int 4\vk^2(2\vk + \p)q\cdot[\psi_h^{12},\tfrac1{2\vk +\p} ] \bigl( r\cdot \tfrac q{2\vk - \p}\cdot \tfrac r{2\vk + \p}\bigr)\,dx\right\|_{L^1_t}\\
&\qquad \lesssim \vk^2 \|(2\vk + \p)q\|_{L^\infty_tH^{-1}_\vk}\|[\psi_h^{12},\tfrac1{2\vk +\p} ] \bigl( r\tfrac q{2\vk - \p} \tfrac r{2\vk + \p}\bigr)\|_{L^2_tH^1_\vk}\\
&\qquad \lesssim \vk^{-1}\|q\|_{L^\infty_tL^2}^2\|\psi_h^3q\|_{L^4_tL^\infty}^2\lesssim \vk^{-1}\delta^2\normM{q}^2,
\\
&\left\|\int 6\vk^2 \tfrac {q'}{2\vk - \p}\cdot r \cdot \tfrac q{2\vk - \p}\cdot  \tfrac r{2\vk +\p}\,\psi_h^{12}\,dx\right\|_{L^1_t} +\left\|\int 2\vk qr\tfrac{q'}{2\vk - \p}\tfrac r{2\vk + \p}\,\psi_h^{12}\,dx\right\|_{L^1_1}\\
&\qquad + \left\|\int 2\vk^2 q' \cdot \tfrac r{2\vk + \p}\cdot \tfrac q{2\vk - \p}\cdot \tfrac r{2\vk + \p}\,\psi_h^{12}\,dx\right\|_{L^1_t}\\
&\qquad\qquad  \lesssim \vk^{-1}\|\psi^3_h q\|_{L^4_tL^\infty}^2\|\psi_h^6q'\|_{L^2_{t,x}}\|q\|_{L^\infty_tL^2} \lesssim \vk^{-1}\delta^2\normM{q}^2,
\\
&\left\|\int qr\tfrac{q'}{2\vk - \p}\tfrac{r'}{2\vk + \p}\,\psi_h^{12}\,dx\right\|_{L^1_t}\\
&\qquad\lesssim \vk^{-1}\|\psi^3 q\|_{L^4_tL^\infty}^2\|\psi^6_h q'\|_{L^2_{t,x}}\|\tfrac{q'}{2\vk - \p}\|_{L^\infty_tL^2}\lesssim \vk^{-1}\delta^2\normM{q}^2,
\end{align*}
with identical estimates for the symmetric terms.

Combining the estimates for \(\wj\sbrack{\geq 4}\) we obtain the estimate \eqref{jmKdV-err +}.
\end{proof}

\bpf[Proof of Proposition~\ref{P:ls}]
We now argue as in the proof of Proposition~\ref{prop:mKdV-LS}, with \(\rho,j_\mKdV\) replaced by \(\wrho,\wj\) respectively, and the estimates \eqref{rhoA}, \eqref{jmKdV-quad}, \eqref{jmKdV-err} replaced by the estimates \eqref{rhoA+}, \eqref{jmKdV-quad +}, \eqref{jmKdV-err +} respectively, to obtain
\begin{align*}
\|(\psi_h^6 q)''\|_{L^2_tH^{-1}_\vk}^2 &\lesssim \epsilon\|(\psi_h^6 q)''\|_{L^2_tH^{-1}_\vk}^2 + (1 + \tfrac1\epsilon)\big(\|q'\|_{L^\infty_tH^{-1}_\vk}^2 + \vk^{-1}\|q\|_{L^\infty_tL^2}^2\big)\\
&\quad + \vk^{-1}\delta^2\normM{q}^2,
\end{align*}
where the implicit constant is independent of \(h,\vk,\epsilon\). Taking $\epsilon$ sufficiently small to defeat the implicit constant above and using \eqref{equi +} and the conservation of mass, we may bound
\begin{align*}
\|(\psi_h^6 q)''\|_{L^2_tH^{-1}_\vk}^2 &\lesssim \|q(0)'\|_{H^{-1}_\vk}^2+\vk^{-1}\|q(0)\|_{L^2}^2 + \delta^2\vk^{-1}\|q\|_{X^1}^2.
\end{align*}
Arguing as in Proposition~\ref{prop:mKdV-LS} and using the conservation of mass to bound the low frequencies, we obtain the estimates \eqref{mKdV-LS +} and \eqref{mKdV-LS-HF +} in the case \(\sigma = 0\).

If \(0<\sigma<\frac12\) we first use \eqref{mKdV-LS +} with \(\sigma = 0\) to bound $\|q\|_{X^1}$ and then integrate using \eqref{EquivNorm} to obtain
\begin{align*}
\|(\psi_h^6 q)''\|_{L^2_tH^{\sigma-1}_\kappa}^2 &\approx \int_\kappa^\infty \vk^{2\sigma}\|(\psi_h^6 q)''\|_{L^2_tH^{-1}_\vk}^2\,\tfrac{d\vk}\vk 
\lesssim \|q(0)'\|_{H^{\sigma-1}_\kappa}^2+ \kappa^{2\sigma-1}\|q(0)\|_{L^2}^2,
\end{align*}
and the proof of the estimates \eqref{mKdV-LS +} and \eqref{mKdV-LS-HF +} is completed similarly.
\epf

\appendix
\section{Ill-posedness}\label{S:A}

The key observation that drives everything in this section is the following:

\begin{lem}\label{L:A1}If $\psi:\R\to\C$ is a Schwartz function and $\psi_\lambda(x):=\lambda \psi(\lambda x)$, then
\begin{align}
\int \psi(x)\,dx &=0 \qtq{implies} \|  \psi_\lambda \|_{H^{\sigma}(\R)}^2
	\lesssim_\psi \begin{cases}  1 & : \sigma=-\tfrac12 \\ \lambda^{-1} + \lambda^{1+2\sigma} & :\sigma<-\frac12\end{cases} \label{mean0} \\
\intertext{whereas}
\int \psi(x)\,dx &\neq 0 \qtq{implies} \| \psi_\lambda \|_{H^\sigma (\R)}^2
	\gtrsim_\psi \begin{cases} \log \lambda& :\sigma=-\tfrac12 \\ 1 & :\sigma<-\frac12 \end{cases} \label{mean1}
\end{align}
uniformly for $\lambda\geq 2$.
\end{lem}

This follows from direct computation. Better bounds are possible in the $\sigma<-\frac12$ case of \eqref{mean0}, but simplicity is preferable.

In order to exploit Lemma~\ref{L:A1}, we need solutions for our flows that initially have mean zero, but later have non-zero mean.  For just \eqref{NLS} or \eqref{mKdV}, this is trivial.  However, we wish to consider all evolutions in the hierarchy simultaneously (excepting translation and phase rotation).  For this reason, it is convenient to work with the generating function $A(\kappa)$ for the Hamiltonians and then expand in inverse powers of $\kappa$.  Under the $A(\kappa)$ flow, 
\begin{equation}
\tfrac{d}{dt} \int q \,dx = i \int g_{12}(x;\kappa)\,dx. \qtq{Moreover,} \int g_{12}\sbrack{1}(x;\kappa)\,dx = - \tfrac1{2\kappa}\int q \,dx.
\end{equation}
These assertions follow from \eqref{qr under A} and \eqref{g12 1 3}, respectively. Delving further, shows
\begin{align}\label{E:A.4}
\int g_{12}\sbrack{3}(x;\kappa)\,dx = \tfrac{\pm 1}{4\kappa^3\sqrt{2\pi}} \sum_{\ell=0}^\infty 
		\int  \bigl(\tfrac{i}{2\kappa}\bigr)^\ell \hat q(\eta-\xi) \hat q(\xi) \overline{\hat q(\eta)} \, \tfrac{\xi^{\ell+1}-\eta^{\ell+1}}{\xi-\eta}\,d\xi\,d\eta.
\end{align}

\begin{prop}\label{P:illposed}
Both \eqref{NLS} and \eqref{mKdV} exhibit instantaneous inflation of the $H^{\sigma}$ norm, in the sense of \eqref{instantaneous}, for every $\sigma\leq -\frac12$.  Indeed, this also holds for all higher flows in the hierarchy (focusing or defocusing).
\end{prop}

\begin{proof}
We first consider a fixed Schwartz solution $u:\R\times\R\to\C$ of our chosen equation.  For even numbered Hamiltonians of the hierarchy, such as \eqref{NLS}, we choose initial data $\widehat{u_0}(\xi) = a \xi^2 e^{-\xi^2}$, where $a>0$ will be chosen small shortly.  For odd numbered Hamiltonians, such as \eqref{mKdV}, we choose $\widehat{u_0}(\xi) = a [\xi^2+\xi^3] e^{-\xi^2}$.
The key criterion for selecting these initial data and for choosing $a>0$ is that
\begin{equation}\label{E:good u}
\int u(0,x)\,dx =0 \qtq{but}  \int u(t_1,x)\,dx \neq 0
\end{equation}
for some $t_1>0$ and any sufficiently small $a>0$.  The existence of such a $t_1$ will follow if we show non-vanishing of the cubic terms in the time derivative of $\int u$ at time $t=0$.  This is precisely the role of \eqref{E:A.4}.

For even numbered Hamiltonians (i.e. $\ell$ even) the integrand in \eqref{E:A.4} is sign definite and so \eqref{E:good u} is clear.  For odd numbered Hamiltonians, we first symmetrize under $\eta\leftrightarrow\xi$ and then under simultaneous inversion in $\eta$ and $\xi$; this then leads to an integrand with a sign-definite imaginary part.

In the case $\sigma=-\tfrac12$, we choose $q=a u_\lambda$ using the rescaling of $u$ given by \eqref{scaling}: One chooses $a$ small to guarantee that the initial data has size $\eps$ and then $\lambda$ large to guarantee that $\lambda^{-m}t_1<\eps$ and that the norm exceeds $\eps^{-1}$ at this time.

When $\sigma<-\tfrac12$, we need an extra idea:  Consider the solution $q$ with initial data
$$
q(0,x) = \sum_{n=1}^N a u_\lambda(0,x+nL).
$$
Note that $\sum a  u_\lambda(t,x+nL)$ is almost a solution and becomes more so as $L\to\infty$.  As all equations in the hierarchy are known to admit a perturbation theory in high regularity spaces \cite{MR2653659,MR1195480}, we know that the approximate solution differs little from $q(t,x)$ uniformly for $t\in[0,\lambda^{-m}t_1]$ provided we take $L$ large enough.  The ill-posedness result now follows by choosing $N$ and $L$ large enough to guarantee large norm at time $\lambda^{-m}t_1$ and ensuring that $\lambda$ is large enough to place this time in $[0,\eps]$ and to make the norm small at time $t=0$.
\end{proof}

Evidently this argument cannot be applied to \eqref{Miura mKdV}, because $\int q$ is conserved.  Nevertheless, we are able to show the following form of norm inflation in the focusing case:

\begin{prop}\label{P:mKdV in -1/2}
For any sequence of times $t_n\to 0$, there is a sequence of (real-valued) Schwartz-class solutions $q_n$ to focusing \eqref{Miura mKdV} 
that satisfy 
\begin{align}\label{E:mKdV in -1/2}
\| q_n(0) \|_{\dot H^{-\frac12}} \lesssim 1 \qtq{and} \| q_n(t_n) \|_{H^{-\frac12}} \to \infty.
\end{align}
Moreover, instantaneous norm inflation in the sense of \eqref{instantaneous} holds when $\sigma<-\frac12$.
\end{prop}

Note that at time zero, these solutions belong to the \emph{homogeneous} Sobolev space; this is a stronger requirement than belonging to $H^{-1/2}(\R)$ and enforces that $\int q_n(0,x)\,dx = 0$. (Unlike for \eqref{NLS} and \eqref{mKdV}, this mean-zero property is preserved by \eqref{Miura mKdV}.)  Nevertheless, the (weaker) inhomogeneous norm diverges.

For $\sigma=-\frac12$, we show norm inflation for initial data of size one, rather than for arbitrarily small initial data.  It is only in this sense that the result is weaker than \eqref{instantaneous}.
\begin{proof}
All that is required is a careful inspection of the two-soliton solutions
$$
u(t,x) := 6\frac{ 2\cosh(x-t)-\cosh(2x-8t) }{\,\cosh(3x-9t)+9\cosh(x-7t)-8}  \qtq{and} u_\lambda(t,x) = \lambda u(\lambda^3 t,\lambda x).
$$

Fix $\sigma<-\frac12$.  As rescalings of a single mean-zero Schwartz function,
$$
\| u_\lambda (0) \|_{\dot H^{-\frac12}(\R)}^2 \approx 1 \qtq{and}  \| u_\lambda (0) \|_{H^\sigma(\R)}^2 \lesssim \lambda^{-1} +  \lambda^{1+2\sigma} 
$$
uniformly for $\lambda\geq1$.  On the other hand, for $\lambda_n^2 t_n^{ } \gtrsim 1$, we see that the solution resolves into two sign-definite Schwartz solitons (each of width $\approx \lambda_n^{-1}$) separated by a distance $\approx \lambda_n^2 t_n^{ }$.  In this way, one readily shows that
$$
\| u_{\lambda_n} (t_n) \|_{H^{-\frac12}(\R)}^2 \approx \log( \lambda_n )
	\qtq{and}  \| u_{\lambda_n} (t_n) \|_{H^\sigma(\R)}^2 \gtrsim 1 .
$$
The claim \eqref{E:mKdV in -1/2} follows at once by choosing $\lambda_n$ appropriately.  Norm inflation in $H^\sigma(\R)$ follows via the same summation device employed in the proof Proposition~\ref{P:illposed}.
\end{proof}

\bibliographystyle{habbrv}
\bibliography{refs}

\end{document}